\documentclass[reqno,10pt]{amsart}
\usepackage[margin=1in]{geometry}

\usepackage{graphicx}
\usepackage{amsmath,amssymb}
\usepackage{xcolor}

\newtheorem{theorem}{Theorem} 
\newtheorem{definition}{Definition}
\newtheorem{lemma}{Lemma}
\newtheorem{proposition}{Proposition}

\usepackage{enumerate}
\usepackage{ amssymb }
\usepackage{cancel}

\makeatletter
\renewcommand*\env@matrix[1][*\c@MaxMatrixCols c]{%
  \hskip -\arraycolsep
  \let\@ifnextchar\new@ifnextchar
  \array{#1}}
\makeatother

\let\e=\varepsilon

\let\p=\partial

\let\O=\Omega

\numberwithin{equation}{section}

\let\hide\iffalse

\newcommand{\R}{\mathbb{R}}

\renewcommand{\S}{\mathbb{S}}

\newcommand{\be}{\begin{equation}}
\newcommand{\bm}{\begin{multline}}
\newcommand{\ee}{\end{equation}}
\newcommand{\dd}{\mathrm{d}}

\newcommand{\xb}{x_{\mathbf{b}}}

\newcommand{\tb}{t_{\mathbf{b}}}
\newcommand{\vb}{v_{\mathbf{b}}}

\newcommand{\Bes}{\begin{eqnarray*}}
\newcommand{\Ees}{\end{eqnarray*}}
\newcommand{\Be}{\begin{equation} }
\newcommand{\Ee}{\end{equation}}
\newcommand{\Bs}{\begin{split}}
\newcommand{\Es}{\end{split}}

\def\p{\partial}

\def\O{\Omega}
\def\R{\mathbb{R}}

\def\B{\begin{equation}}
\def\E{\end{equation}}
\def\BN{\begin{eqnarray*}}
\def\EN{\end{eqnarray*}}

\begin{document}

\title{Rarefied gas dynamics with external fields under specular reflection boundary condition}
\author{Yunbai Cao}

 \address{Department of Mathematics, University of Wisconsin, Madison, WI 53706 USA}
\email{ycao35@wisc.edu}
\begin{abstract}
We consider the Boltzmann equation with external fields in strictly convex domains with the specular reflection boundary condition. We construct classical $C^1$ solutions away from the grazing set under the assumption that the external field is $C^2$ and the normal derivative of the field is positive and bounded away from $0$. 
\end{abstract}
\maketitle

\tableofcontents

\section{Introduction}
Kinetic theory studys the time evolution of a large number of particles modeled by a distribution function in the phase space: $F(t,x,v)$ for $(t,x,v) \in [0, \infty) \times  {\O} \times \R^{3}$, where $\O$ is an open bounded subset of $\R^{3}$. Dynamics and collision processes of dilute charged particles with a field $E$ can be modeled by the Vlasov-Boltzmann equation
%
\Be \label{Boltzmann_E}
\p_t F + v\cdot \nabla_x F + E \cdot \nabla_v F = Q(F,F).
\Ee
The collision operator measures ``the change rate'' in binary collisions and takes the form of
\Be\begin{split}\label{Q}
Q(F_{1},F_{2}) (v)&: = Q_\mathrm{gain}(F_1,F_2)-Q_\mathrm{loss}(F_1,F_2)\\
&: =\int_{\R^3} \int_{\S^2} 
B(v-u) \cdot \omega) [F_1 (u^\prime) F_2 (v^\prime) - F_1 (u) F_2 (v)]
 \dd \omega \dd u,
\end{split}\Ee   
where $u^\prime = u - [(u-v) \cdot \omega] \omega$ and $v^\prime = v + [(u-v) \cdot \omega] \omega$.

Here, $B(v-u,\omega) = |v - u|^\kappa q_0( \frac{v-u}{|v -u |} \cdot \omega ) $, $0 \le \kappa \le 1 $ (hard potential), and $0 \le q_0( \frac{v-u}{|v -u |} \cdot \omega ) \le C |\frac{v-u}{ |v -u | } \cdot \omega | $ (angular cutoff).

The collision operator enjoys collision invariance: for any measurable function $G$,  
\Be\label{collison_invariance}
\int_{\R^{3}} \begin{bmatrix}1 & v & \frac{|v|^{2}-3}{2}\end{bmatrix} Q(G,G) \dd v = \begin{bmatrix}0 & 0 & 0 \end{bmatrix} .
\Ee

It is well-known that a global Maxwellian $\mu$ 
satisfies $Q(\mu,\mu)=0$ where
\Be\label{Maxwellian}
\mu(v):= \frac{1}{(2\pi)^{3/2}} \exp\bigg(
 - \frac{|v |^{2}}{2 }
 \bigg).
\Ee

Throughout this paper we assume that $\Omega$ is a bounded open subset of $\mathbb R^3$ and there exists a $C^3$ function $\xi: \mathbb R^3 \to \mathbb R$ such that $\Omega = \{ x \in \mathbb R^3: \xi(x) < 0 \}$, and $\partial \Omega = \{ x\in \mathbb R^3 : \xi(x) = 0 \}$. Moreover we assume the domain is \textit{strictly convex}:
\Be \label{convex}
\sum_{i,j} \partial_{ij} \xi(x) \zeta_i \zeta_j \ge C_\xi |\zeta|^2 \, \text{ for all } \, \zeta \in \mathbb R^3 \text{ and for all } x\in \bar \Omega = \Omega \cup \partial \Omega.
\Ee
We assume that 
\Be \label{gradientxinot0}
\nabla \xi(x) \neq 0  \text{ when } |\xi(x) | \ll 1,
\Ee
and we define the outward normal as $n(x) = \frac{ \nabla \xi(x) }{ | \nabla \xi (x) |}$ at the boundary.
%
The boundary of the phase space $
\gamma := \{ (x,v) \in \partial \Omega \times \mathbb R^3 \}$ can be decomposed as 
\begin{equation} \begin{split}
\gamma_- = \{ (x,v) \in \partial \Omega \times \mathbb R^3 : n(x) \cdot v < 0 \}, &\quad (\text{the incoming set}),
\\ \gamma_+ = \{ (x,v) \in \partial \Omega \times \mathbb R^3 : n(x) \cdot v > 0 \}, &\quad (\text{the outcoming set}),
\\ \gamma_0 = \{ (x,v) \in \partial \Omega \times \mathbb R^3 : n(x) \cdot v = 0 \}, &\quad (\text{the grazing set}).
\end{split} \end{equation}
In general the boundary condition is imposed only for the incoming set $\gamma_-$ for general kinetic PDEs. In this paper we consider a so-called specular reflection boundary condition
 \Be \label{SpecR}
 F(t,x,v) = F(t,x, R_x v ) \text{ on } (x,v) \in \gamma_-, \text{ where } R_x v := v -2 n(x) (n(x) \cdot v ) .
 \Ee
Physically this represents when a gas particle hits the boundary, it bounces back with the opposite normal velocity and the same tangential velocity, just like a billiard. Previous studies on the Boltzmann equation with specular reflection boundary conditions can be found in \cite{GKTT1, Guo10, KBR, KL1, KL2}. For other important physical boundary conditions, such as the diffuse boundary condition, we refer \cite{VPBKim, CaoSIAM, CKQ, GKTT1, Guo10} and the references therein. 

Due to the importance of the Boltzmann equation in the mathematical theory and application,
there have been explosive research activities in analytic study of the equation. Notably the nonlinear energy method has led to solutions of many open problems including global strong solution of Boltzmann equation coupled with either the Poisson equation or the Maxwell system for electromagnetism when the initial data are close to the Maxwellian $\mu$ in periodic box (no boundary). See \cite{Guo_M} and the references therein. 
%
%
%
%
In many important physical applications, e.g. semiconductor and tokamak, the charged dilute
gas is confined within a container, and its interaction with the boundary plays a crucial role both in physics and mathematics. 

However, in general, higher regularity may not be expected for solutions of the Boltzmann equation in physical bounded domains. Such a drastic difference of solutions with boundaries had been
demonstrated as the formation and propagation of discontinuity in non-convex domains \cite{Kim11, EGKM}, and a non-existence of some second order derivatives at the boundary in convex domains \cite{GKTT1}. Evidently the nonlinear energy method is not generally available to the boundary problems.
In order to overcome such critical difficulty, Guo developed a $L^2$-$L^\infty$ framework in \cite{Guo10} to study global solutions of the Boltzmann equation with various boundary conditions. The core of the method lays in a direct approach (without taking derivatives) to achieve a pointwise bound using trajectory of the transport operator, which leads substantial development in various directions including \cite{EGKM2, EGKM, GKTT1, GKTT2, KBOX}.
In \cite{GKTT1}, with the acid of some distance function towards the grazing set, the weighted classical $C^1$ solutions of Boltzmann equation ($E\equiv0$ in (\ref{Boltzmann_E})) was constructed under various boundary conditions.


%
%

In this paper, we extend a result of \cite{GKTT1} to the Boltzmann equation (\ref{Boltzmann_E}) with a given external field $(E\neq 0 )$ satisfying a crucial sign condition on the boundary:
\Be \label{signEonbdry}
E(t,x) \cdot n(x) > C_E > 0 \quad   \text{ for all } t \text{ and all } x \in \partial \Omega.
\Ee
One of the major difficulties when dealing with a field $E \neq 0$ is that trajectories are curved and behave in a very complicated way when they hit the boundary.

Let's clarify some notations. For any function $z(x,v) :  \bar \Omega \times \mathbb R^3 \to \mathbb R$, denote
\Be \notag
\| z \|_\infty = \sup_{(x,v) \in \Omega \times \mathbb R^3} |z(x,v) |.
\Ee
And for any function $g(t,x) : [0,T] \times \bar \Omega  \to \mathbb R$, denote
\[
\| g \|_{L^\infty_{t,x} }  = \sup_{(t,x) \in  [0,T] \times \bar \Omega} |g(t,x) |, \text{ and } \| g \|_{C^n_{t,x} } = \sum_{ 0 \le \alpha + \beta \le n }  \sup_{(t,x) \in  [0,T] \times \bar \Omega}  | \p_t^\alpha \p_x^\beta  g (t,x ) |.
\]

Our main result is a weighted $C^1$ estimate for the solution of (\ref{Boltzmann_E}) with specular boundary condition (\ref{SpecR}) in a short time. To state the main result, we introduce a distance function $\alpha(t,x,v)$ towards the grazing set $\gamma_0$:
\Be \label{alphatilde}
 \alpha(t,x,v) \sim \bigg[ |v \cdot \nabla \xi (x)| ^2 + \xi (x)^2 - 2 (v \cdot \nabla^2 \xi(x) \cdot v ) \xi(x) - 2(E(t,\overline x ) \cdot \nabla \xi (\overline x ) )\xi(x) \bigg]^{1/2}
\Ee
for $x \in \Omega$ close to boundary, where $\overline x := \{ \bar x \in \p \Omega :  d(x,\bar x ) = d(x, \partial \Omega) \}$ is uniquely defined. The precise definition of $\alpha$ can be found in (\ref{alphadef}). Note that $\alpha \vert_{\gamma_-} \sim | n(x) \cdot v |$. Similar distance functions towards $\gamma_0$ were used in \cite{GKTT1,Guo_V,Hwang}. With the weight $\alpha$, we establish the main theorem:

\begin{theorem}[Weighted $C^1$ Estimate]  \label{C1Main} 
Suppose $E$ satisfies the sign condition (\ref{signEonbdry}), and
\Be \label{c1bddforthepotentail}
 \| E \|_{C^2_{t,x}} < \infty.
\Ee
Assume $F_0 = \sqrt \mu f_0 \ge 0 $, $f_0 \in W^{1,\infty}(\Omega \times \mathbb R^3 ) $, and for $ 2 < \beta < 3 $, $0< \theta < \frac{1}{4}$, and $b> 1$,
\[
 \left \|  \frac{ {\alpha}^{\beta-1}   }{\langle v\rangle^{b }} \partial_{x} f_{0} \right\|_{\infty} + \left \|   \frac{     \alpha^{\beta-2} }{\langle v\rangle^{b-2 }} \partial_{v} f_{0}  \right\|_{\infty} + \left\| e^{\theta |v|^2 } f_0 \right \|_\infty < \infty,
 \]
and the compatibility condition 
\begin{equation} \label{compatibility}
f_0(x,v)= f_0(x,R_{x}v) \ \ \ \text{on} \ (x,v) \in \gamma_{-}. 
\end{equation}
Then there exists a unique solution $F(t) = \sqrt \mu f(t)$ for $0 \le t \le T$ with $T  \ll 1 $ to the system \eqref{Boltzmann_E}, \eqref{SpecR} that satisfies, for all $0 \le t \le T$,
\[
\left \| e^{-\varpi \langle v \rangle t } \frac{\alpha^\beta}{\langle v \rangle^{b+1} } \nabla_x f (t) \right \|_\infty + \left \| e^{-\varpi \langle v \rangle t } \frac{\alpha^{\beta-1}}{\langle v \rangle^{b-1} } \nabla_v f (t) \right \|_\infty \lesssim_{}  \left \|  \frac{ {\alpha}^{\beta-1}   }{\langle v\rangle^{b }} \partial_{x} f_{0} \right\|_{\infty} + \left \|   \frac{     \alpha^{\beta-2} }{\langle v\rangle^{b-2 }} \partial_{v} f_{0}  \right\|_{\infty} + P \left( \left\| e^{\theta |v|^2 } f_0 \right \|_\infty \right)
\]
for some polynomial $P$. Furthermore, if $f_0 \in C^1$, then $f \in C^1$ away from the grazing set $\gamma_0$.
\end{theorem}

The proof of Theorem \ref{C1Main} devotes a nontrivial extension of the result in \cite{GKTT1}. The idea is to use Duhamel's formula to expand $f$ along the characteristics to the initial data and then take derivatives. To do this, we need to define the generalized characteristics as follows: 
\begin{definition}
For any $(t,x,v)\in [0,T] \times \Omega \times \mathbb R^3$, let $(X(s;t,x,v),V(s;t,x,v))$ denotes the characteristics
\Be\label{hamilton_ODE}
\frac{d}{ds} \left[ \begin{matrix}X(s;t,x,v)\\ V(s;t,x,v)\end{matrix} \right] = \left[ \begin{matrix}V(s;t,x,v)\\ 
E (s, X(s;t,x,v))\end{matrix} \right]  \ \ \text{ for }   0 \le  s ,  t \le T  ,
\Ee
with $(X(t;t,x,v), V(t;t,x,v)) =  (x,v)$.

We define \textit{the backward exit time} $\tb(t,x,v)$ as   
\Be\label{tb}
\tb (t,x,v) := \sup \{s \geq 0 : X(\tau;t,x,v) \in \O \ \ \text{for all } \tau \in (t-s,t) \}.
\Ee
Furthermore, we define $\xb (t,x,v) := X(t-\tb(t,x,v);t,x,v)$, and $\vb(t,x,v) := V(t-\tb(t,x,v);t,x,v)$.

Now let $(t^{0}, x^{0}, v^{0}) = (t,x,v).$ We define the specular cycles, for $\ell\geq 0,$
\[
(t^{\ell+1}, x^{\ell+1},v^{\ell+1}) = (t^{\ell}-t_{\mathbf{b}}(t^\ell, x^{\ell}, v^{\ell}), x_{\mathbf{b}}(t^\ell,x^{\ell},v^{\ell}),   v_{\mathbf b}(t^\ell, x^{\ell}, v^{\ell}) - 2n(x^{\ell+1} )(v_{\mathbf b}(t^\ell, x^{\ell}, v^{\ell}) \cdot n(x^{\ell+1}))).
\]
And we define the generalized characteristics as
\begin{equation} \label{cycles} 
\begin{split}
X_{\mathbf{cl}}(s;t,x,v)   \ = \ \sum_{\ell} \mathbf{1}_{[t^{\ell+1},t^{%
\ell})}(s) X(s;t^\ell, x^\ell , v^\ell ),  \ \ 
V_{\mathbf{cl}}(s;t,x,v)   \ = \ \sum_{\ell} \mathbf{1}_{[t^{\ell+1},t^{%
\ell})}(s)  V(s;t^\ell, x^\ell , v^\ell ).
\end{split}%
\end{equation}
\end{definition}

The key component of the proof is to estimate the derivatives of the backward trajectory
\[
 \frac{ \p \left( X_{\mathbf{cl}}(s;t,x,v), V_{\mathbf{cl}}(s;t,x,v) \right)}{\p \left( x , v  \right)}.
 \]
 This is done through the matrix method where we estimate the multiplication of $\ell^*(s;t,x,v)$ many Jacobian matrices 
\Be \label{lmatrices}
\prod_{\ell = 0}^{\ell^*(s;t,x,v) }\frac{\partial ( t^{\ell+1}, x^{\ell+1}, v^{\ell+1})}{\partial (t^\ell,x^\ell,v^\ell)}.
\Ee
Here $\ell^*(s;t,x,v)$ is the number of bounces it takes for the backward trajectory to reach time $s$ from time $t$, which can be shown to have order $\ell^*(s;t,x,v) \sim \frac{ |t-s||v|}{\alpha(t,x,v)}$. And for each bounce, we can calculate the Jacobian matrix $\frac{\partial ( t^{\ell+1}, x^{\ell+1}, v^{\ell+1})}{\partial (t^\ell,x^\ell,v^\ell)}$ explicitly.

One major difficulty here, comparing to the Boltzmann equation ($E= 0$ in \eqref{Boltzmann_E}),  is the field $E$ is time dependent, thus the characteristics ODE \eqref{hamilton_ODE} is not autonomous. This results the $\frac{\p (x^{\ell+1},v^{\ell+1})}{\p t^\ell} $ derivatives in the first column of the matrix $\frac{\partial ( t^{\ell+1}, x^{\ell+1}, v^{\ell+1})}{\partial (t^\ell,x^\ell,v^\ell)}$  is not trivally equal to 0, and need careful analysis. 

We estimate \eqref{lmatrices} by diagonalizing each matrix and multiplying them together. Here another difficulty arises as the derivatives $\frac{\p (n(x^{\ell+1}) \cdot v^{\ell+1} )}{\p( t^{\ell}, x^\ell )}$ can only be bounded as $| \frac{\p (n(x^{\ell+1}) \cdot v^{\ell+1} )}{\p( t^{\ell}, x^\ell )}| \lesssim | t^{\ell} - t^{\ell+1} |$. And this bound will result the multiplication of the $\ell^*$ many eigenvalues of the matrices to behave as
\[
 \prod_{\ell = 0}^{\ell^*(s;t,x,v) } \text{ eig } \left| \frac{\partial ( t^{\ell+1}, x^{\ell+1}, v^{\ell+1})}{\partial (t^\ell,x^\ell,v^\ell)}  \right |  \sim  (1+\sqrt{\alpha})^{\ell^*} \sim  (1+\sqrt{\alpha})^{\frac{1}{\alpha}} \to \infty
\]
as $\alpha \to 0$. Where $\alpha = \alpha(t,x,v) $ in \eqref{alphatilde}. This blow up will result the bound on \eqref{lmatrices} becomes too singular and makes it impossible for us to close the estimate. In order to overcome such a difficulty we utilize a crucial cancellation property \eqref{Fperpcancel}, and find that as long as the external field $E$ satisfies the regularity assumption
\Be \label{Ereg}
\| E(t,x) \|_{C^2_{t,x} } < \infty,
\Ee
we can improve the estimate and achieve the bound $| \frac{\p (n(x^{\ell+1}) \cdot v^{\ell+1} )}{\p( t^{\ell}, x^\ell )} | \lesssim | t^{\ell} - t^{\ell+1} |^2$. This extra smallness turns out to be just enough to control the accumulation in the many multiplications of eigenvalues:
\Be
 \prod_{\ell = 0}^{\ell^*(s;t,x,v) } \text{ eig } \left| \frac{\partial ( t^{\ell+1}, x^{\ell+1}, v^{\ell+1})}{\partial (t^\ell,x^\ell,v^\ell)}  \right |  \sim  { (1 + \alpha)}^{\frac{1}{\alpha}} < C.
\Ee
With this bound and additional cancellations between two adjacent matrices \eqref{vlplus1vl}, we carefully analyze the multiplications of the matrices and eventually achieve the key estimate in Theoreom \ref{theorem_Dxv}. 

Let's also address some other important differences when comparing the equation \eqref{Boltzmann_E} with the Boltzmann equation ($E = 0 $). Because of the presence of the field $E$ and its sign condition \eqref{signEonbdry}, we can achieve a better bound on the time gap
\[
|t^{\ell} -t^{\ell+1} | \lesssim |n(x^\ell) \cdot v^{\ell+1} |
\]
when $v$ is small \eqref{tbest}. This is because when the velocity is small, the field would always ``push" the trajectory back to the boundary in a short time. This fact would eventually allow us to get the bound
\[
|\nabla _{v}X_{\mathbf{cl}}(s;t,x,v)|   \lesssim \frac{1}{\langle v \rangle}
\]
in Theorem \ref{theorem_Dxv}, which does not blow up when $|v| \to 0$, and this turns out to be necessary for us to close the estimate.

When taking derivatives to the Duhamel's formula of $f(t,x,v)$ in \eqref{deriv_spec}, if $E \neq 0$, an extra term would come up as \eqref{IIIe}. In order to bound this term we have to additionally estimate the derivatives $\p_x t^\ell$ and $\p_v t^\ell$, for any $1 \le \ell \le \ell^*$. Those estimates are consequences of the matrix method and are obtained in \eqref{ptell} and \eqref{ptell2}:
\[
| \p_x t^\ell | \lesssim \frac{1}{\alpha^2} , \quad  | \p_v t^\ell | \lesssim \frac{1}{\alpha}.
\]
It's also important to note that in \eqref{IIIe}, we have $ | R_{x^{\ell}}v^{\ell} - v^\ell  |= 2 | (n(x^\ell ) \cdot v^\ell ) | \sim \alpha$. Thus the extra regularity we get by multiplying $\alpha^\beta$ to  $\p_x f$ and $\alpha^{\beta -1} $ to $\p_v f$ will bound the term as
\[
 \sum_{1 \le \ell \le \ell^*} \left( \alpha^\beta | \p_x t^\ell |  +  \alpha^{\beta -1} |\p_v t^\ell | \right) \max_{\ell} | R_{x^{\ell}}v^{\ell} - v^\ell  |\lesssim \frac{1}{\alpha} \left( \alpha^\beta \frac{1}{\alpha^2} + \alpha^{\beta -1} \frac{1}{\alpha}  \right) \alpha \lesssim \alpha^{\beta -2 } < C,
\]
as long as $\beta > 2 $. 

\section{Local existence and in-flow problems with external fields}
In this section we state some standard results which we will need to prove Theorem \ref{C1Main}. Let $F(t,x,v) = \sqrt \mu f(t,x,v) $. Then the corresponding problem to \eqref{Boltzmann_E}, \eqref{SpecR} becomes
\begin{equation}\label{boltzamnn_f}
 \partial_{t}f +v\cdot \nabla_{x} f + E \cdot \nabla_v f - \frac{v}{2} \cdot E f   = \Gamma_{\text{gain } }(f,f) - \nu(\sqrt{\mu} f)f.
\end{equation}
Here
\begin{equation} \label{nu}
\begin{split}
\nu( \sqrt{\mu}f)(v)
:= \frac{1}{\sqrt{\mu(v)}}Q_{\text{loss}} (\sqrt{\mu}f,\sqrt{\mu}f)(v) =  \int_{\mathbb{R}^{3}} \int_{\mathbb{S}^{2}} |v-u|^{\kappa
}q_{0}\big( \frac{v-u}{|v-u|} \cdot \omega\big)\sqrt{\mu(u)}
f (u) \mathrm{d}\omega \mathrm{d}u,
\end{split}
\end{equation}
and the gain term of the nonlinear Boltzmann operator is given by
\begin{equation}\label{gain_Gamma}
\begin{split}
\Gamma_{\text{gain}}(f_{1},f_{2})(v)   &:= \frac{1}{\sqrt{\mu }} Q_{\text{gain}} (\sqrt{\mu}f_{1}, \sqrt{\mu}f_{2})(v)\\
& =
\int_{\mathbb{R}^{3}} \int_{\mathbb{S}^{2}} |v-u|^{\kappa
}q_{0}\big( \frac{v-u}{|v-u|} \cdot \omega\big)\sqrt{\mu(u)} f_{1}(u^{\prime}) f_{2}(v^{\prime}) \mathrm{d}\omega\mathrm{d}u.
\end{split}
\end{equation}
And the specular reflection boundary condition in terms of $f$ is
\begin{equation} \label{specularBC}
f(t,x,v)= f(t,x,R_{x}v), \ \ \ \text{on} \ (x,v) \in \gamma_{-}. 
\end{equation}

We first state a local existence result which is standard:

\begin{lemma} \label{localexslemma}[Local Existence]
Suppose $\| E \|_{L^\infty_{t,x}} < \infty$, and $\| e^{\theta |v|^2} f_0 \|_\infty < \infty$, $0< \theta < \frac{1}{4}$. And $f_0$ satisfy the compatibility condition \eqref{compatibility}. Then there exists $0 < T \ll 1$ small enough such that $f \in L^\infty ([0,T) \times \Omega \times \mathbb R^3)$ solves the equation (\ref{boltzamnn_f}) with specular boundary condition (\ref{specularBC}).
%
\end{lemma}

\begin{proof}
Let $f^0 = \sqrt \mu$. We start with the sequence for $m \ge 0$
\begin{equation} \label{VPBsq1}
(\partial_t + v \cdot \nabla_x +E \cdot \nabla_v  - \frac{v}{2} \cdot  E + \nu( \sqrt \mu f^m) ) f^{m+1} = \Gamma_{\text{gain}} (f^m,f^m),
\end{equation}
with the initial data $f^m(0,x,v) = f_0(x,v)$, and boundary condition for all $(x,v) \in \gamma_-$ be
\Be \label{seqbd}
\begin{split}
f^1(t,x,v) & =  f_0(x, R_x v ),
\\ f^{m+1}(t,x,v) & =f^m(t,x,R_x v ) , \, m \ge 1.
\end{split}
\Ee
Then (see Lemma 7 in \cite{GKTT1} for example)
\Be \label{uniformlinfinitybddnoselfgeneratedpotential}
\sup_m \sup_{0 \le t \le T} \| e^{\theta' |v|^2 } f^m(t) \|_\infty \lesssim \| e^{\theta |v|^2} f_0 \|_\infty < \infty,
\Ee
where $\theta' = \theta - T$.
From (\ref{uniformlinfinitybddnoselfgeneratedpotential}) we have up to a subsequence we have the weak-$\ast$ convergence:
\Be \label{weakstarcovlinear}
 e^{\theta' |v|^2} f^m(t,x,v) \overset{\ast}{\rightharpoonup} e^{\theta' |v|^2} f(t,x,v)
 \Ee
 in $L^\infty ([0,T) \times \Omega \times \mathbb R^3) \cap L^\infty ([0,T) \times \gamma) $ for some $f$. And it's easy to show $f$ is the solution of (\ref{boltzamnn_f}) with specular boundary condition (\ref{specularBC}).
\end{proof}

We need some bound on the derivatives of the nonlocal term:
\begin{lemma} \label{lemma_operator}
Let $[Y,W]=[Y(x,v), W(x,v)]\in \Omega\times \mathbb{R}^{3}$. For $0<\theta < \frac{1}{4}$ and $\partial_{\mathbf{e}}
\in \{\p_{t}, \nabla_{x}, \nabla_{v}\},$
%
\begin{equation}\notag
\begin{split} 
    &  |\partial_{\mathbf{e}} \Gamma_{\mathrm{gain}} (g,g)(Y , W ) | \\
\lesssim & \  |\partial_{\mathbf{e}} Y|||  e^{\theta|v|^{2}} g||_{\infty} \int_{\mathbb{R}^{3}} \frac{e^{- {C_{\theta}|u-W|^{2}} }}{|u-W|^{2-\kappa}} |\nabla_{x}g(Y,u)| \mathrm{d}u  \\
&  + \ 
 |\partial_{\mathbf{e}} W|||  e^{\theta|v|^{2}} g||_{\infty} \int_{\mathbb{R}^{3}}  \frac{e^{- {C_{\theta}|u-W|^{2}} }}{|u-W|^{2-\kappa}} |\nabla_{v}g(Y,u)| \mathrm{d}u + \langle v\rangle^{\kappa} e^{-  {\theta} |v|^{2}} |\partial_{\mathbf{e}} W|  ||  e^{\theta|v|^{2}}  g||_{\infty}^{2}. 
 \end{split}
\end{equation}
\end{lemma}
\begin{proof}
See \cite{GKTT1}.
\end{proof}

We need a result for the corresponding inflow problem to \eqref{boltzamnn_f}. Consider
\begin{equation} \label{traninflowfixE}
\partial_t f + v \cdot \nabla_x f+ E \cdot \nabla_v f + \nu f = H,
\end{equation}
where $H=H(t,x,v)$ and $\nu = \nu (t,x,v)$ are given. Let $\tau_1(x)$ and $\tau_2(x)$ bet unit tangential vector to $\partial \Omega$ satisfying 
\begin{equation} \label{tangentialderivativedef}
\tau_1(x) \cdot n(x) = 0 = \tau_2 (x) \cdot n(x) \text{ and } \tau_1(x) \times \tau_2(x) = n(x).
\end{equation}
And let $\partial_{\tau_i} g$ be the tangential derivative at direction $\tau_i$ for $g$ defined on $\partial \Omega$. Define 
\Be \label{defofgradientxg}
 \begin{split}
\nabla_x g =  \sum_{i =1 }^2  \tau_i \partial_{\tau_i } g  - \frac{ n }{ n \cdot \vb } \Big \{ \partial_t g + \sum_{i =1}^2 ( \vb \cdot \tau_i ) \partial_{\tau_i } g  + \nu g - H +  E \cdot \nabla_v g  \Big \}.
 \end{split} 
 \Ee

\begin{proposition} \label{inflowexistence1}
Assume the compatibility condition
\[ f_0(x,v) = g(0,x,v)\quad  \text{for} \quad (x,v) \in \gamma_- .\]
Let $p \in [1, \infty )$ and $0 < \theta < 1/4$. $|\nu(t,x,v) | \lesssim \langle v \rangle $. $\| E \|_{L^\infty_{t,x}} + \| \nabla_x E \|_{L^\infty_{t,x}} < \infty$.

Assume
\[ \begin{split}
  \nabla_x f_0 , \nabla_v f_0 , 
  \in L^p (\Omega \times \mathbb R^3 ),
\\ \nabla_v g, \partial_{\tau_i } g \in L^p ( [0, T] \times \gamma_- ) ,
\\ \frac{ n(x) }{ n(x) \cdot v } \Big \{ \partial_t g + \sum_{i =1}^2 ( v \cdot \tau_i ) \partial_{\tau_i } g  + \nu g - H +  E \cdot \nabla_v g \Big \} \in L^p ([0, T ] \times \gamma_- ),
\\ 
 \frac{ n(x) \cdot \iint \partial_x E }{ n(x) \cdot v } \Big \{ \partial_t g + \sum_{i =1}^2 ( v \cdot \tau_i) \partial_{ \tau_i } g   - \nu g + H \Big \} \in L^p ([0, T ] \times \gamma_- ),
\\ \nabla_x H ,\nabla_v H \in L^p  ([0, T ] \times \Omega \times \mathbb R^3 ),
\\ e^{- \theta |v|^2 } \nabla_x \nu, e^{-\theta |v|^2 } \nabla_v \nu \in L^p ([ 0, T ] \times \Omega \times \mathbb R^3 ),
\\ e^{\theta |v|^2 } f_0 \in L^\infty ( \Omega \times \mathbb R^3 ) , e^{\theta |v|^2 } g \in L^\infty ( [0, T] \times \gamma_- ),
\\ e^{\theta |v|^2 } H \in L^\infty ([0, T] \times \Omega \times \mathbb R^3 ).
\end{split} \]
Then for any $T > 0$, there exists a unique solution $f$ to (\ref{traninflowfixE}), such that $f, \partial_t, \nabla_x f ,\nabla_v f \in C^0( [ 0, T ] ; L^p (\Omega \times \mathbb R^3) ) $ and their traces satisfiey
\Be \label{traceofinflow} \begin{split}
\nabla_v f |_{\gamma_-} = \nabla_v g, \nabla_x f|_{\gamma_-} = \nabla_x g, \quad \text{on} \quad \gamma_- ,
\\ \nabla_x f(0,x,v) = \nabla_x f_0, \nabla_v f(0,x,v) = \nabla_v f_0, \quad \text{in} \quad \Omega \times \mathbb R^3,  
\\ \partial_t f(0,x,v) = \partial_t f_0, \quad \text{in} \quad \Omega \times \mathbb R^3. 
\end{split} \Ee
where $\nabla_x g $ is given by (\ref{defofgradientxg}).
\begin{proof}
See \cite{CaoSIAM}.
\end{proof}

\end{proposition}

\section{Velocity lemma and the nonlocal to local estimate}
Recall the definition of specular trajectories in \eqref{cycles}. In this section we prove some properties of the specular trajectories which are crucial in order to establish the main result.

Let's give the precise definition for the weight function $\alpha$. We first need a cutoff function: for any $\epsilon >0$, let $\chi_\epsilon: [0,\infty) \to [0,\infty) $ be a smooth function satisfying:
%

\begin{equation} \label{chicondition} \begin{split}
& \chi_\epsilon(x) =x \,\, \text{for} \,\, 0 \le x \le \frac{\epsilon}{4}, 
\\ &\chi_\epsilon(x) = C_\epsilon  \,\, \text{for}\,\, x \ge \frac{\epsilon}{2},
\\& \chi_\epsilon(x) \text{ is increasing}  \,\, \text{for} \,\, \frac{\epsilon}{4} < x < \frac{\epsilon}{2},
\\ & \chi_\epsilon'(x) \le 1.
\end{split} \end{equation}

Let $d(x,\partial \Omega) := \inf_{y \in \partial \Omega} \| x - y \| $. And for any $\delta > 0$, let 
\[
\Omega ^\delta : = \{ x \in \Omega : d(x, \partial \Omega ) < \delta \}.
\]
Since $\partial \Omega$ is $C^2$, we claim that if $\delta \ll 1$ is small enough we have:
\begin{equation} \label{distfunctionunique}
\text{for any} \, x \in \Omega ^\delta \, \text{ there exists a unique} \, \bar x \in \partial \Omega \, \text{such that} \, d(x,\bar x ) = d(x, \partial \Omega), \, \text{ moreover } \sup_{x\in \Omega^\delta } | \nabla_x \bar x | < \infty.
\end{equation}
To prove the claim, we have by (\ref{gradientxinot0}) WLOG locally we can assume $\eta$ takes the form $\eta ( x_\parallel)  = (x_{\parallel,1} , x_{\parallel,2} , \bar \eta(x_{\parallel,1}, x_{\parallel,2}))$, and $\bar x = \eta(\bar x_\parallel) =(\bar x_{\parallel,1} , \bar x_{\parallel,2} , \bar \eta(\bar x_{\parallel,1}, \bar x_{\parallel,2}))$. Denote $\partial_i \bar \eta = \frac {\p }{ \partial x_{\parallel, i } } \bar \eta (x_{\parallel,1} , x_{\parallel,2} ) $, and $\partial_{i,j} \bar \eta = \frac {\p^2 }{ \partial x_{\parallel, i } \partial x_{\parallel, j } } \bar \eta (x_{\parallel,1} , x_{\parallel,2} )$.
Now since $|\eta(\bar x_\parallel) - x|^2 = \inf_{y \in \partial \Omega } | y -x |^2$, $\bar x_\parallel$ satisfies
\[ 
\omega (x_1,x_2,x_3, \bar x_{\parallel,1},\bar x_{\parallel,2} )= 
\begin{bmatrix}
(\bar x_{\parallel,1} - x_1) + (\bar \eta(\bar x_{\parallel,1}, \bar x_{\parallel,2}) - x_3) \p_1 \bar \eta ( \bar x_{\parallel,1}, \bar x_{\parallel,2})  \\
(\bar x_{\parallel,2} - x_2) + (\bar \eta(\bar x_{\parallel,1}, \bar x_{\parallel,2}) - x_3) \p_2 \bar \eta ( \bar x_{\parallel,1}, \bar x_{\parallel,2})
\end{bmatrix} 
=0.
\]
Since
\[ \begin{split}
\det (\frac{\partial \omega}{ \partial x_\parallel} )= & \det
\begin{bmatrix}
1 + (\p_1 \bar \eta)^2 + (\bar \eta - x_3 ) \p_{1,1} \bar \eta_{} & \p_2 \bar \eta_{}\p_1\bar \eta_{} + (\bar \eta - x_3 ) \p_{1,2} \bar \eta \\
\p_1 \bar \eta_{}\p_2 \bar \eta_{} +(\bar \eta -x_3) \p_{1,2} \bar \eta_{} & 1 + (\p_2 \bar \eta_{}) ^2 + (\bar \eta -x_3 ) \p_{1,2} \bar \eta_{}
\end{bmatrix}
\\ & = ( 1 + (\p_1 \bar \eta_{})^2 ) ( 1 + (\p_2 \bar \eta_{}) ^2 ) - (\p_1 \bar \eta_{} \p_2 \bar \eta_{ } ) ^2 +O(|\bar \eta - x_3 |)
\\ & = 1 + (\p_1 \bar \eta_{})^2 + (\p_2 \bar \eta_{}) ^2+O(|\bar \eta - x_3 |) > 0,
\end{split} \]
if $|\bar \eta (x_\parallel) - x_3 | $ is small enough. By the implicit function theorem $(\bar x_{\parallel,1}, \bar x_{\parallel,2 } )$ are functions of $x_1, x_2, x_3$ if $x$ is close enough to $\partial \Omega$.

Moreover,
\[ \begin{split}
\frac{\partial \bar x_{\parallel} }{ \partial x_j }&  = - ( \frac{\partial \omega}{ \partial \bar x_{\parallel} } ) ^{-1}  \cdot \frac{\partial \omega}{ \partial x_j }
\\ & = \frac{1}{ \det ( \frac{\partial \omega}{ \partial \bar x_{\parallel} }) }
\begin{bmatrix}
1 +  (\p_2 \bar \eta) ^2 + (\bar \eta -x_3 )  \partial_{1,2} \bar \eta & - \p_2  \bar \eta \p_1 \bar \eta - (\bar \eta - x_3 ) \p_{1,2} \bar \eta \\
- \p_1  \bar \eta \p_2 \bar \eta -(\bar \eta -x_3) \p_{1,2} \bar \eta & 1 + (\p_1 \bar \eta)^2 + (\bar \eta - x_3 ) \p_{1,1} \bar \eta 
\end{bmatrix}
\cdot \frac{\partial \omega }{ \partial x_j }
\end{split}\]
is bounded as $\frac{\partial \omega }{\partial x_j } $ is bounded and $\det ( \frac{\partial \omega}{ \partial \overline x })$ is bounded from below if $x$ is close enough to the boundary. Therefore $| \nabla_x \bar x |$ is bounded. This proves (\ref{distfunctionunique}).

Now define
\[
\beta(t,x,v) = \bigg[ |v \cdot \nabla \xi (x)| ^2 + \xi (x)^2 - 2 (v \cdot \nabla^2 \xi(x) \cdot v ) \xi(x) - 2(E(t,\overline x ) \cdot \nabla \xi (\overline x ) )\xi(x) \bigg]^{1/2},
\]
for all $(x, v ) \in  \Omega ^{\delta} \times \mathbb R^3 $. Let $\delta': = \min \{| \xi (x)| : x\in \Omega, d(x, \partial \Omega) = \delta \} $, and let $\chi_{\delta'}$ be a smooth cutoff function satisfies (\ref{chicondition}), then define
\begin{equation} \label{alphadef}
\alpha(t,x,v) : = 
\begin{cases}
 (\chi_{\delta'}  ( \beta(t,x,v) ) )^{} & x \in  \Omega^\delta, \\
 C_{\delta'} ^{} & x \in \Omega \setminus  \Omega^\delta.
\end{cases} \end{equation}
The following lemmas about $\alpha$ is important for our estimate:

\begin{lemma}[Velocity lemma near boundary] \label{velocitylemma} 
Suppose $E(t,x)$ satisfies $\| E \|_{C^1} < \infty$ and the sign condition (\ref{signEonbdry}). Then $\alpha$ is continuous, and for $\delta \ll 1 $ small enough, we have for any $(t,x,v) \in [0,T ] \times \O \times \mathbb R^3$, and $0 \le s<t $, $\alpha$ satisfies
\begin{equation}  \label{velocitylemmaintform}
e^{ - C \int_s ^ t ( |V_{\mathbf{cl}}(\tau')| + 1 ) d \tau'} \alpha ( s,X_{\mathbf{cl}}(s),V_{\mathbf{cl}}(s) ) \le \alpha (t,x,v) \le  e^{C \int_s ^ t ( |V_{\mathbf{cl}}(\tau')| + 1 ) d \tau'} \alpha (s,X_{\mathbf{cl}}(s),V_{\mathbf{cl}}(s)),
\end{equation}
for any $C \ge \frac{C_{\xi}(\| E \|_{L^\infty_{t,x}} + \| \nabla E \|_{L^\infty_{t,x}} +\| \p_t E \|_{L^\infty_{t,x}} + 1 )}{C_E} $, where $C_\xi > 0$ is a large constant depending only on $\xi$. Here $( X_{\mathbf{cl}}(s),V_{\mathbf{cl}}(s)) = ( X_{\mathbf{cl}}(s;t,x,v), V_{\mathbf{cl}}(s;t,x,v)) $ as defined in \eqref{cycles}.
\end{lemma}
Similar estimates have been used in \cite{Guo_V} and then in \cite{Hwang,GKTT1}.
\begin{proof}
See \cite{CaoSIAM}.
\end{proof}

\begin{lemma} \label{int1overalphabetainu}
Suppose $E$ satisfies \eqref{signEonbdry}, then
for any y $\in \bar \Omega$, $1 < \beta < 3$, $0 < \kappa \le 1$, and $\theta > 0$ we have
\Be \label{kernelbddfor1overalpha}
\int_{\mathbb R^3} \frac{ e^{-\theta |v- u|^2}}{ |v - u | ^{2 - \kappa} [\alpha(s, y, u ) ]^\beta } du \le C \left( \frac{ 1}{( |v|^2 \xi(y) + c (y) )^{\frac{\beta - 1}{2} } } +1 \right) , 
\Ee
where $c(y) =  \xi(y) ^2 - C_E  \xi (y)$.
\end{lemma}
\begin{proof}
See \cite{CaoSIAM}.
\end{proof}

\begin{lemma}\label{keylemma}
(1)
Let $(t,x,v) \in [0,T] \times  \Omega \times \mathbb R^3$, $ 1 < \beta < 3$, $0 < \kappa \le 1 $. 
Suppose $E$ satisfies (\ref{signEonbdry}) and (\ref{c1bddforthepotentail}), 
then for $\varpi \gg 1$ large enough, we have for any $0 < \delta \ll 1$,
\begin{equation} \label{genlemma1}
 \begin{split}
 & \int_{\max\{ 0, t - \tb \}} ^t  \int_{\mathbb R^3} e^{ -\int_s^t\frac{\varpi}{2} \langle V(\tau;t,x,v) \rangle d\tau }  \frac{e^{-\frac{C_\theta}{2} |V(s)-u|^2 }}{|V(s) -u |^{2 - \kappa}  }  \frac{1}{(\alpha(s,X(s) ,u))^\beta} du  ds
\\  \lesssim &   e^{ 2C_\xi \frac{\| \nabla E\|_\infty  + \| E \|_{L^\infty_{t,x}}^2 + \| E \|_{L^\infty_{t,x}}}{C_E}}  \frac{  \delta ^{\frac{3 -\beta}{2} }}{\langle v\rangle^2(C_E+1)^ {\frac{\beta-1}{2}} (\alpha(t,x,v))^{\beta -2 }  (   \| E \|_{L^\infty_{t,x}}^2 + 1 )^{\frac{3 -\beta}{2}}}  
\\ & +  \frac{ ( \| E \|_{L^\infty_{t,x}}^2+1)^{\beta -1} } { C_E^{\beta-1} \delta ^{\beta -1 } ( \alpha ( t,x,v) )^{\beta - 1 } } \frac{2}{ \varpi },
\end{split} 
\end{equation}
where $(X(s), V(s)) = (X(s;t,x,v), V(s;t,x,v)) $ as in \eqref{hamilton_ODE}.

   \vspace{4pt}
   
(2)    Let $[X_{\mathbf{cl}}(s;t,x,v), V_{\mathbf{cl%
}}(s;t,x,v)]$ be the specular backward trajectory as in \eqref{cycles}. Let $Z(s,x,v) \geq 0$ be any bounded non-negative function in the phase space.

For any $\varepsilon>0$, there exists $l \gg_{}1$ such that for any $r >0$,
\begin{equation}  \label{specular_nonlocal}
\begin{split}
& \int_{0}^{t} \int_{\mathbb{R}^{3}} e^{- l\langle v\rangle (t-s)} \frac{%
e^{-\theta|V_{\mathbf{cl}}(s;t,x,v)-u|^{2}}}{|V_{\mathbf{cl}}(s;t,x,v)-u|^{2-\kappa}}
\frac{\langle u\rangle^{r}}{\langle v\rangle^{r}} \frac{Z(s,x,v)}{ \big[ %
\alpha(s,X_{\mathbf{cl}}(s;t,x,v),u) \big]^{\beta}} \mathrm{d}u\mathrm{d}s \\
& \lesssim \ \frac{ O(\varepsilon)}{\langle v\rangle \big[\alpha(t,x,v)\big]^{\beta-1} }
\sup_{0 \leq s\leq t} \big\{ e^{-  \frac{l}{2}  \langle v\rangle
(t-s)} Z(s,x,v) \big\}.
\end{split}%
\end{equation}
\end{lemma}

\begin{proof} [\textbf{Proof of (1) Lemma \protect\ref{keylemma}}] The proof is similar to the proof of Lemma 11 in \cite{CaoSIAM}, but with some modifications been made in order to achieve \eqref{specular_nonlocal} later. We separate the proof into several cases. 

In \textit{Step 1, Step 2, Step 3} we prove (\ref{genlemma1}) for the case when $x \in \partial \Omega$ and $t \le \tb $.

In \textit{Step 4} we prove (\ref{genlemma1}) for the case when $x \in \partial \Omega$ and $t > \tb $.

In \textit{Step 5} we prove (\ref{genlemma1}) for the case when $x \in  \Omega$ and $t \le \tb $.

In \textit{Step 6} we prove (\ref{genlemma1}) for the case when $x \in  \Omega$ and $t > \tb $.


\bigskip

\textit{Step 1} \qquad
Let's first start with the case $t \ge \tb$ and prove (\ref{genlemma1}), 
%
Let's shift the time variable: $s \mapsto t - \tb + s$, and let $\tilde X(s) = X(t -\tb +s )$, $\tilde V(s) = V(t -\tb +s ) $. Then $s \in [0, \tb] $ and from (\ref{kernelbddfor1overalpha}) we only need to bound the integral
\begin{equation} \label{trajint}
\int_0^{\tb} e^{ -\int_{t -\tb +s}^t\frac{\varpi}{2} \langle V(\tau;t,x,v) \rangle d\tau }  \frac{ 1}{\left[ |\tilde V(s)|^2 \xi(\tilde X( s) ) + \xi^2(\tilde X( s)) - C_E \xi (  \tilde X(s) ) \right]^{\frac{\beta - 1}{2} } }  ds.
\end{equation}

Let's assume $x \in \partial \Omega$ and $ v \cdot \nabla \xi (x) > 0$. Then by the velocity lemma (Lemma \ref{velocitylemma}) we have $\vb \cdot \nabla \xi(\xb) < 0$.

Claim: for any $0< \delta \ll 1$ small enough, if we let 
\begin{equation} \label{sigma12}
\sigma_1 = \delta \frac{ \vb \cdot \nabla \xi (\xb) } { |v | ^2 + \| E \|_{L^\infty_{t,x}}^2 +  \| E \|_{L^\infty_{t,x}}+1 },  \text{ and } \sigma_2 = \delta \frac{ v \cdot \nabla \xi (x) } { |v|^2 + \| E \|_{L^\infty_{t,x}}^2 +  \| E \|_{L^\infty_{t,x}}+1 },
\end{equation}
then $ |\xi(\tilde X(s) |$ is monotonically increasing on $[0, \sigma_1]$, and monotonically decreasing on $[\tb - \sigma_2 , \tb]$. Moreover, we have the following bounds:
\begin{equation} \label{sigma12lowerbd}
|\xi ( \tilde X(\sigma_1) ) | \ge  \frac{ \delta ( \vb \cdot \nabla \xi (\xb) ) ^2 }{ 2 ( |v| ^2 +  \| E \|_{L^\infty_{t,x}}^2 +  \| E \|_{L^\infty_{t,x}} +1) }, \, |\xi ( \tilde X(\sigma_2 ) ) | \ge \frac{\delta ( v \cdot \nabla \xi (x) ) ^2 }{ 2 ( |v| ^2 +  \| E \|_{L^\infty_{t,x}}^2 +  \| E \|_{L^\infty_{t,x}}+1 ) },
\end{equation}
\begin{equation}
\begin{split} \label{xiupperbd}
& |\xi (\tilde X(s) |  \le  \frac{ 3 \delta ( \vb \cdot \nabla \xi (\xb) ) ^2 }{ 2 ( |v| ^2 +  \| E \|_{L^\infty_{t,x}}^2 +  \| E \|_{L^\infty_{t,x}}+1 ) }, \, s \in [ 0, \sigma_1],
\\ &  |\xi (\tilde X(s) | \le \frac{ 3 \delta ( v \cdot \nabla \xi (x) ) ^2 }{ 2 ( |v| ^2 +  \| E \|_{L^\infty_{t,x}}^2 +  \| E \|_{L^\infty_{t,x}}+1 ) }, \, s \in [ \tb - \sigma_2, \tb],
\end{split}
\end{equation}
and
\begin{equation} \label{vdotnlowerbd}
\begin{split}
& | \tilde V(s) \cdot \nabla \xi (\tilde X(s))|  \ge \frac{ |\vb \cdot \nabla \xi (\xb ) | }{2}, \, s \in [0, \sigma_1],
\\ & | \tilde V(s) \cdot \nabla \xi (\tilde X(s))|  \ge \frac{ |v \cdot \nabla \xi (x ) | }{2}, \, s \in [\tb - \sigma_2, \tb].
\end{split}
\end{equation}

To prove the claim we first note that $\frac{d}{ds } \xi ( \tilde X(s) ) |_{s= 0 } = \vb \cdot \nabla \xi (\xb) < 0$, and
\begin{equation} \label{accelerationbd}
 \begin{split}
\frac{ d^2}{d^2 s } \xi (\tilde X(s) ) ) = & \frac{d}{ds} (\tilde V(s) \cdot \nabla \xi (\tilde X(s))) = \tilde V(s) \cdot \nabla^2 \xi (\tilde X(s)) \cdot \tilde V(s) + E(s ,\tilde X(s)) \cdot \nabla \xi (\tilde X(s))
\\  \le & C(|\tilde V(s) | ^2 + \| E \|_{L^\infty_{t,x}} ) \le C(2 |v|^2 + 2 (\tb \| E \|_{L^\infty_{t,x}})^2 +  \| E \|_{L^\infty_{t,x}} ) \le C_1 ( |v|^2 +  \| E \|_{L^\infty_{t,x}}^2 +  \| E \|_{L^\infty_{t,x}}+1),
\end{split} 
\end{equation}
for some $C_1 >0$. Thus if $\delta $ small enough, we have $\frac{d}{ds } \xi ( \tilde X(s) ) < 0 $ for all $s \in [0, \delta \frac{| \vb \cdot \nabla \xi (\xb) |}{ |v| ^2 + \| E \|_{L^\infty_{t,x}}^2 +  \| E \|_{L^\infty_{t,x}}+1  }]$. Therefore $\xi ( \tilde X(s) ) $ is decreasing on $[0, \sigma_1]$.

Similarly $ \frac{d}{ds } \xi ( \tilde X(s) ) |_{s= \tb } = v \cdot \nabla \xi(x) > 0$, and since $ | \frac{ d^2}{d^2 s } \xi (\tilde X(s) ) ) | \lesssim (|v|^2 +\| E \|_{L^\infty_{t,x}}^2 +  \| E \|_{L^\infty_{t,x}} +1) $ we have that $ \frac{d}{ds } \xi ( \tilde X(s) )  > 0 $ for all $s \in [ \tb - \delta \frac{ | v \cdot \nabla \xi (v) | }{ |v| ^2 +\| E \|_{L^\infty_{t,x}}^2 +  \| E \|_{L^\infty_{t,x}}  +1}, \tb ] $ if $\delta $ small enough. Therefore $\xi ( \tilde X(s) ) $ is increasing on $[ \tb - \sigma_2, \tb]$.

Next we establish the bounds (\ref{sigma12lowerbd}), (\ref{xiupperbd}), and (\ref{vdotnlowerbd}). By (\ref{accelerationbd}), we have
\[ \begin{split}
|\xi ( \tilde X(\sigma_1 ) ) | = & \int_0 ^{\sigma_1} -  \tilde V(s) \cdot \nabla \xi ( \tilde X(s) ) ds  
\\ = & \int_0^{\sigma_1}  \left( \int_0 ^s   - \frac{d}{d\tau} (\tilde V(\tau) \cdot \nabla \xi (\tilde X(\tau ) ) )  d\tau -  \vb \cdot \nabla \xi (\xb )  \right)   ds 
\\ \ge & \int_0^{\sigma_1}  \left(|\vb \cdot \nabla \xi (\xb) | - C_1(|v|^2 + \| E \|_{L^\infty_{t,x}}^2 +  \| E \|_{L^\infty_{t,x}} +1 ) s  \right) ds 
\\ = & \sigma_1 |\vb \cdot \nabla \xi (\xb) | - \frac{\sigma_1^2 }{2 } C_1 (|v|^2 + \| E \|_{L^\infty_{t,x}}^2 +  \| E \|_{L^\infty_{t,x}}  +1 )  
\\ = & \sigma_1 \left( |\vb \cdot \nabla \xi (\xb) |- \frac{ \delta C_1}{2} |\vb \cdot \nabla \xi (\xb ) | \right) 
\\ \ge & \frac{ \sigma_1}{2}  |\vb \cdot \nabla \xi (\xb) | = \frac{ \delta ( \vb \cdot \nabla \xi (\xb) ) ^2 }{ 2 ( |v| ^2 + \| E \|_{L^\infty_{t,x}}^2 +  \| E \|_{L^\infty_{t,x}} +1 ) }.
\end{split} \]

And by the same argument we have $|\xi ( \tilde X(\sigma_2 ) ) | \ge \frac{\delta ( v \cdot \nabla \xi (x) ) ^2 }{ 2 ( |v| ^2 + \| E \|_{L^\infty_{t,x}}^2 +  \| E \|_{L^\infty_{t,x}} +1  ) } $ for $\delta \ll 1$. This proves (\ref{sigma12lowerbd}).

To prove (\ref{xiupperbd}), we have from (\ref{accelerationbd}), for $s \in [0,\sigma_1]$,
\[ \begin{split}
|\xi (\tilde X(s) | \le & s \left( |\vb \cdot \nabla \xi (\xb) |+ \frac{ \delta C_1}{2} |\vb \cdot \nabla \xi (\xb ) | \right)  
\\ \le & \frac{3 s}{2} | \vb \cdot \nabla \xi(\xb) | \le  \frac{ 3 \delta ( \vb \cdot \nabla \xi (\xb) ) ^2 }{ 2 ( |v| ^2 +  \| E \|_{L^\infty_{t,x}}^2 +  \| E \|_{L^\infty_{t,x}} +1 ) },
\end{split} \]
and $ |\xi (\tilde X(s) | \le \frac{ 3 \delta ( v \cdot \nabla \xi (x) ) ^2 }{ 2 ( |v| ^2 +  \| E \|_{L^\infty_{t,x}}^2 +  \| E \|_{L^\infty_{t,x}}+1 ) }$ for $s \in [ \tb -\sigma_2, \tb] $. This proves (\ref{xiupperbd}).

Finally for (\ref{vdotnlowerbd}), again from (\ref{accelerationbd}),
\[ 
\begin{split}
| \tilde V(s) \cdot \nabla \xi (\tilde X(s))| \ge & |\vb \cdot \nabla \xi (\xb)  | - \int_0 ^{\sigma_1} C_1 (|v|^2 +  \| E \|_{L^\infty_{t,x}}^2 +  \| E \|_{L^\infty_{t,x}} +1 ) ds  
\\ \ge & |\vb \cdot \nabla \xi (\xb)  |  - C_1 \delta |\vb \cdot \nabla \xi (\xb ) |  \ge \frac{ |\vb \cdot \nabla \xi (\xb ) | }{2}.
\end{split}
\]
And similarly $| \tilde V(s) \cdot \nabla \xi (\tilde X(s))| \ge \frac{| v \cdot \nabla \xi ( x ) |}{2} $ for $s \in [ \tb - \delta_2, \tb] $. This proves the claim.

\bigskip

\textit{Step 2} \qquad
Recall the definition of $\sigma_1, \sigma_2$ in (\ref{sigma12}), and $C_E$ in (\ref{signEonbdry}).
In this step we establish the lower bound:
\begin{equation} \label{xilowerbd}
 | \xi( \tilde X(s)  ) | > \frac{C_E}{10} (\sigma_2)^2, \, \text{for all } s \in [\sigma_1, \tb - \sigma_2 ].
\end{equation}

Suppose towards contradiction that $I := \{ s \in [ \sigma_1, \tb - \sigma_2]:  | \xi( \tilde X(s)  ) | \le \frac{C_E}{10}  (\sigma_2)^2 \} \neq \emptyset$.


Then from (\ref{velocitylemmaintform}) and (\ref{sigma12lowerbd}) we have
\[ \begin{split}
 \frac{C_E}{10} (\sigma_2)^2 \le &  \delta^2  \frac{C_E}{10} \frac{ (v \cdot \nabla \xi (x))^2 } { |v|^2 +   \| E \|_{L^\infty_{t,x}}^2 +  \| E \|_{L^\infty_{t,x}} +1 } 
 \\  \le & \delta^2 \frac{C_E}{10} e^{ C_\xi \frac{\|  \nabla E \|_{L^\infty_{t,x}}  + \| E \|_{L^\infty_{t,x}}^2 + \| E \|_{L^\infty_{t,x}}}{C_E}}  \frac{  (\vb \cdot \nabla \xi (\xb))^2 } { |v|^2 +  \| E \|_{L^\infty_{t,x}}^2 +  \| E \|_{L^\infty_{t,x}}  +1 } 
 \\ \le & 2 \delta \frac{C_E}{10} e^{ C_\xi \frac{\|  \nabla E \|_{L^\infty_{t,x}}  + \| E \|_{L^\infty_{t,x}}^2 + \| E \|_{L^\infty_{t,x}}}{C_E}}  |\xi (\tilde X(\sigma_1 ) ) |
 \\  < &  |\xi (\tilde X(\sigma_1 ) ) |,
 \end{split} \]
  if $\delta \ll 1$. So $\sigma_1 \notin I$. Let $s^*:= \min\{s \in  I \}$ be the minimum of such $s$. Then clearly 
  \[
  \frac{d}{ds} \xi (\tilde X(s) ) |_{s = s^*} = \tilde V(s^* ) \cdot \nabla \xi ( \tilde X(s^*) )  \ge 0.
  \]
Now expanding around $\tilde X(s)$, we have
\begin{equation} 
\begin{split}
E(s ,\tilde X(s)) \cdot \nabla \xi (\tilde X(s)) = E(s , \overline {\tilde X(s)} ) \cdot \nabla \xi ( \overline {\tilde X(s) }) + c (\tilde X(s) ) \cdot \xi (\tilde X(s) ),
\end{split} 
\end{equation}
with $|c(\tilde X(s) ) | < \frac{C_{\xi}(\| E \|_{L^\infty_{t,x}} + \| \nabla E \|_{L^\infty_{t,x}} )}{C_E}$.
Thus
\Be \label{accnearbdy}
\begin{split}
   \frac{d}{ds} (\tilde V(s) \cdot \nabla \xi (\tilde X(s))) = & \tilde V(s) \cdot \nabla^2 \xi (\tilde X(s)) \cdot \tilde V(s) + E(s ,\tilde X(s)) \cdot \nabla \xi (\tilde X(s))
   \\ = & \tilde V(s) \cdot \nabla^2 \xi (\tilde X(s)) \cdot \tilde V(s) + E(s , \overline {\tilde X(s)} ) \cdot \nabla \xi ( \overline {\tilde X(s) }) + c (\tilde X(s) ) \cdot \xi (\tilde X(s) ) 
     \\ \ge & C_E - \frac{C_{\xi}(\| E \|_{L^\infty_{t,x}} + \| \nabla E \|_{L^\infty_{t,x}} )}{C_E} | \xi (\tilde X(s) ) |,
   \end{split}
\Ee
so 
\[
\frac{d}{ds} (\tilde V(s) \cdot \nabla \xi (\tilde X(s))) |_{s = s^*}   \ge  C_E - \delta^2  \frac{C_{\xi}(\| E \|_{L^\infty_{t,x}} + \|  \nabla E \|_{L^\infty_{t,x}} )}{C_E}  \frac{C_E}{10} \frac{ (v \cdot \nabla \xi (x))^2 } { |v|^2 +   \| E \|_{L^\infty_{t,x}}^2 +  \| E \|_{L^\infty_{t,x}}  + 1}    \ge \frac{C_E}{2},
\]
for $\delta \ll 1 $ small enough. Then we have $ \frac{d}{ds} (\tilde V(s) \cdot \nabla \xi (\tilde X(s)))$ is increasing on the interval $[s^*,\tb]$ as $|\xi (\tilde X(s)) |$ is decreasing. So
\[
\frac{d}{ds} (\tilde V(s) \cdot \nabla \xi (\tilde X(s))) \ge \frac{C_E}{2}, \quad s \in [s^* , \tb].
\]
%
%
 And therefore
 \[ \begin{split}
 |\xi(\tilde X(s^* ) ) | = & \int_{s^*}^ {\tb } \tilde V(s) \cdot \nabla \xi ( \tilde X(s) ) ds  
 \\ = & \int_{s^*}^ {\tb } \left( \int_{s^*}^{s} \frac{d}{d\tau }  ( \tilde V(\tau) \cdot \nabla \xi ( \tilde X(\tau) ) ) d\tau + \tilde V(s^* ) \cdot \nabla \xi( \tilde X(s^*) ) \right) ds
 \\ \ge & \int_{s^*}^ {\tb } (s - s^* ) \frac{C_E}{2} ds  = \frac{C_E }{4} ( \tb - s^* ) ^2 \ge \frac{C_E}{4} (\sigma_2)^2,
\end{split}  \]
which is a contradiction. Therefore we conclude \eqref{xilowerbd}.

\bigskip

\textit{Step 3} \qquad
Let's split the time integration (\ref{trajint}) as
\begin{equation} 
\begin{split}
\int_0^{\tb}  & e^{ -\int_{t -\tb +s}^t\frac{\varpi}{2} \langle V(\tau;t,x,v) \rangle d\tau }  \frac{ 1}{\left[ |\tilde V(s)|^2 \xi(\tilde X( s) ) + \xi^2(\tilde X( s)  - C_E \xi ( \tilde X(s) ) \right]^{\frac{\beta - 1}{2} } }  ds
\\ = & \int_0^{\sigma_1} + \int_{\sigma_1}^{\tb - \sigma_2} + \int_{\tb -\sigma_2}^{\tb }  = (\textbf{I})+ (\textbf{II})+ (\textbf{III}).
\end{split} 
\end{equation}

Let's first estimate $ (\textbf{I}),  (\textbf{III})$:

From \textit{Step 2} we have that $| \xi (\tilde X(s) | $ is monotonically increasing on $[0, \sigma_1] $ and $[\tb - \sigma_2, \tb] $, so we have the change of variables:
\[
ds = \frac{d |\xi | }{ | \tilde V(s) \cdot \nabla \xi (\tilde X(s))|  }.  
\]
Using this change of variable and the bounds (\ref{xiupperbd}), (\ref{vdotnlowerbd}), and the $| \tilde V(s) |^2+1 \gtrsim | v|^2 +1$, $ (\textbf{I})$ is bounded by 
\begin{equation} \label{Ibd}
\begin{split}
 (\textbf{I}) \le & \int_0^{\sigma_1}   \frac{ 1 }{\left[ |\tilde V(s)|^2 \xi(\tilde X( s) ) + \xi^2(\tilde X( s)- C_E \xi (\tilde X(s) ) \right]^{\frac{\beta - 1}{2} } }  ds
\\ \le & \int_0^{\frac{ 3 \delta ( \vb \cdot \nabla \xi (\xb) ) ^2 }{ 2 ( |v| ^2 +   \| E \|_{L^\infty_{t,x}}^2 +  \| E \|_{L^\infty_{t,x}} +1 ) }} \frac{ 1 } {  | \tilde V(s) \cdot \nabla \xi (\tilde X(s))|  ( (C_E+ |\tilde V(s) | )  |\xi | ) ^{\frac{\beta-1}{2}}} d |\xi |
\\ \lesssim &   \int_0^{\frac{ 3 \delta ( \vb \cdot \nabla \xi (\xb) ) ^2 }{ 2 ( |v| ^2 +   \| E \|_{L^\infty_{t,x}}^2 +  \| E \|_{L^\infty_{t,x}} +1 ) }} \frac{ 2  } {  | \vb \cdot \nabla \xi (\xb)|  ( (C_E +|v|^2)   |\xi | )^ {\frac{\beta-1}{2}}}d |\xi | 
\\ = &  \frac{2}{| \vb \cdot \nabla \xi (\xb)| (C_E+|v|^2)^ {\frac{\beta-1}{2}}} \left[ |\xi |^{\frac{3 - \beta} {2}} \right] _0^{\frac{ 3 \delta ( \vb \cdot \nabla \xi (\xb) ) ^2}{ 2 ( |v| ^2 +  \| E \|_{L^\infty_{t,x}}^2 +  \| E \|_{L^\infty_{t,x}} +1 ) }} 
\\ = &  \frac{2^{\frac{\beta-1}{2}} \delta ^{\frac{3 -\beta}{2} }} { (C_E+|v|^2)^ {\frac{\beta-1}{2}} | \vb \cdot \nabla \xi (\xb) | ^{\beta -2 } (|v|^2 +  \| E \|_{L^\infty_{t,x}}^2 +  \| E \|_{L^\infty_{t,x}} +1)^{\frac{3 -\beta}{2}}}
\\ \lesssim & e^{ 2C_\xi \frac{\|  \nabla E \|_{L^\infty_{t,x}}  + \| E \|_{L^\infty_{t,x}}^2 + \| E \|_{L^\infty_{t,x}}}{C_E}}  \frac{  \delta ^{\frac{3 -\beta}{2} }}{(C_E+|v|^2)^ {\frac{\beta-1}{2}} (\alpha(t,x,v))^{\beta -2 }  (|v|^2 +  \| E \|_{L^\infty_{t,x}}^2 +  \| E \|_{L^\infty_{t,x}} +1 )^{\frac{3 -\beta}{2}}}
\\ \lesssim  & e^{ 2C_\xi \frac{\|  \nabla E \|_{L^\infty_{t,x}}  + \| E \|_{L^\infty_{t,x}}^2 + \| E \|_{L^\infty_{t,x}}}{C_E}}  \frac{  \delta ^{\frac{3 -\beta}{2} }}{\langle v\rangle^2(C_E+1)^ {\frac{\beta-1}{2}} (\alpha(t,x,v))^{\beta -2 }  (   \| E \|_{L^\infty_{t,x}}^2 + 1 )^{\frac{3 -\beta}{2}}}
.
\end{split} 
\end{equation}
And by the same computation we get 
\begin{equation} \label{IIIbd}
 (\textbf{III}) \lesssim  e^{ 2C_\xi \frac{\|  \nabla E \|_{L^\infty_{t,x}}  + \| E \|_{L^\infty_{t,x}}^2 + \| E \|_{L^\infty_{t,x}}}{C_E}}  \frac{  \delta ^{\frac{3 -\beta}{2} }}{\langle v\rangle^2(C_E+1)^ {\frac{\beta-1}{2}} (\alpha(t,x,v))^{\beta -2 }  (   \| E \|_{L^\infty_{t,x}}^2 + 1 )^{\frac{3 -\beta}{2}}}  .
 \end{equation}


Finally for (\textbf{II}), using the lower bound for $|\xi (\tilde X(s )) |$ in (\ref{xilowerbd}), we have
\begin{equation} \label{IIbd}
\begin{split}
(\textbf{II}) = & \int_{\sigma_1}^{\sigma_2}  e^{ -\int_{t -\tb +s}^t\frac{\varpi}{2} \langle V(\tau;t,x,v) \rangle d\tau }   \frac{ 1 }{\left[ |\tilde V(s)|^2 \xi(\tilde X( s) ) + \xi^2(\tilde X( s) -C_E \xi (\tilde X(s) ) \right]^{\frac{\beta - 1}{2} } }  ds
\\ \le & \int_0^{\tb} e^{ -\int_{t -\tb +s}^t\frac{\varpi}{2} \langle V(\tau;t,x,v) \rangle d\tau } \frac{1}{ | ( |\tilde V(s)|^2 +C_E) \xi (\tilde X(s)) |^{\frac{\beta -1}{2}} } ds
\\ \lesssim &  \frac{1}{C_E^{\beta-1}( \langle v \rangle \sigma_2)^{\beta - 1 } }\int_0^{\tb} e^{ \int_{t -\tb +s}^t \frac{\varpi}{2}    d\tau } ds 
\\ \lesssim & \frac{ (|v| +\| E \|_{L^\infty_{t,x}} + \| E \|_{L^\infty_{t,x}}^2 +1)^{\beta -1} } { C_E^{\beta-1}  \langle v \rangle^{\beta -1 } \delta ^{\beta -1 } ( \alpha ( t,x,v) )^{\beta - 1 } }\int_0^{\tb} e^{ ( s - \tb )  \frac{\varpi}{2}   } ds 
  \lesssim   \frac{ (|v| +\| E \|_{L^\infty_{t,x}} + \| E \|_{L^\infty_{t,x}}^2 +1)^{\beta -1} } { C_E^{\beta-1}  \langle v \rangle^{\beta -1 } \delta ^{\beta -1 } ( \alpha ( t,x,v) )^{\beta - 1 } } \frac{2}{ \varpi }.
\end{split} 
\end{equation}

This proves (\ref{genlemma1}) for the case $ x \in \partial \Omega$ and $t \le \tb$. 

%
%

\textit{Step 4} \qquad
Now suppose $x \in \partial \Omega$ and  $ \tb >  t$. It suffices to bound the integral:
\begin{equation} \label{trajintt}
\int_0^{t} e^{ -\int_{s}^t\frac{\varpi}{2} \langle V(\tau;t,x,v) \rangle d\tau }  \frac{ 1}{\left[ | V(s)|^2 \xi( X( s) ) + \xi^2( X( s) - C_E \xi (  X(s) ) \right]^{\frac{\beta - 1}{2} } }  ds.
\end{equation}
Denote
\[
X(0;t,x,v ) = x_0, V(0;t,x,v) = v_0.
\]
Let  
\[
\sigma_2 = \delta \frac{ v \cdot \nabla \xi (x) } { |v|^2 + \| E \|_{L^\infty_{t,x}} + \| E \|_{L^\infty_{t,x}}^2 +1 }
\]
as defined in (\ref{sigma12}). If 
\[
\sigma_2 \ge t,
\]
then from \textit{Step 2} $|\xi (X(s) ) | $ is decreasing on $[0,t]$, and by
(\ref{xiupperbd}), (\ref{vdotnlowerbd}), and the bound for (\textbf{III}) (\ref{IIIbd}), we get the desired estimate.
Now we assume 
\[
\sigma_2 < t.
\]
So from (\ref{sigma12lowerbd}) we have
\begin{equation} \label{sigma2lowerbd}
|\xi (  X(\sigma_2 ) ) | \ge \frac{\delta ( v \cdot \nabla \xi (x) ) ^2 }{ 2 ( |v| ^2 +  \| E \|_{L^\infty_{t,x}} + \| E \|_{L^\infty_{t,x}}^2+1 ) }.
\end{equation}

Now if $|\xi ( x_0 ) | \le \delta \frac{ \alpha^2 ( t,x,v) }{10(|v|^2 + \| E \|_{L^\infty_{t,x}} + \| E \|_{L^\infty_{t,x}}^2+1)} $, 
\begin{equation} 
\begin{split}
\alpha^2(t,x,v) \lesssim & e^{ C_\xi \frac{\|  \nabla E \|_{L^\infty_{t,x}}  + \| E \|_{L^\infty_{t,x}}^2 + \| E \|_{L^\infty_{t,x}}}{C_E}}   \alpha^2 (0,x_0, v_0 )    \\ \lesssim & e^{ C_\xi \frac{\|  \nabla E \|_{L^\infty_{t,x}}  + \| E \|_{L^\infty_{t,x}}^2 + \| E \|_{L^\infty_{t,x}}}{C_E}}  ( (\nabla \xi (x_0 ) \cdot v_0 ) ^2 + (|v_0|^2 + |\xi(x_0)| + \| E \|_{L^\infty_{t,x}} ) |\xi (x_0 ) |)
\\ \lesssim & e^{ C_\xi \frac{\|  \nabla E \|_{L^\infty_{t,x}}  + \| E \|_{L^\infty_{t,x}}^2 + \| E \|_{L^\infty_{t,x}}}{C_E}} (\nabla \xi (x_0 ) \cdot v_0 ) ^2 + \delta \alpha^2 (t,x,v),
\end{split}
\end{equation}
 So 
\begin{equation} \label{alphaxi0v0}
\frac{1}{2} \alpha(t,x,v) \lesssim  e^{ C_\xi \frac{\|  \nabla E \|_{L^\infty_{t,x}}  + \| E \|_{L^\infty_{t,x}}^2 + \| E \|_{L^\infty_{t,x}}}{C_E}} |\nabla \xi(x_0) \cdot v_0 | ,
\end{equation}
if $ \delta \ll 1 $ is small enough.

Claim: 
\[
 \nabla \xi (x_0 ) \cdot v_0 <0 .
 \]
Since otherwise by (\ref{accnearbdy}) we have
\[
\frac{d}{ds} | \xi (X(s) ) | < 0,
\]
for all $s \in [0,t]$, so $|\xi(X(s))| $ is always decreasing, which contradicts (\ref{sigma2lowerbd}).

Therefore $\nabla \xi (x_0 ) \cdot v_0 <0$, and we can run the same argument from \textit{Step 1}, \textit{Step 2}, \textit{Step 3} with $\nabla \xi (\xb) \cdot \vb$ replaced by $\nabla \xi (x_0) \cdot v_0$, and by (\ref{alphaxi0v0}) we get the same estimate.
 
 If $|\xi ( x_0 ) | > \delta \frac{  \alpha^2 ( t,x,v) }{10(|v|^2 +\| E \|_{L^\infty_{t,x}}^2 + \| E \|_{L^\infty_{t,x}} +1)} $, then  we have
 \begin{equation}
  \frac{C_E \sigma_2 ^2}{10} =   \delta^2  \frac{C_E}{10} \frac{ (v \cdot \nabla \xi (x))^2 } { |v|^2 +\| E \|_{L^\infty_{t,x}}^2 + \| E \|_{L^\infty_{t,x}} + 1 } < C_E \delta | \xi (x_0)|< |\xi (x_0 )|,
   \end{equation}
for $\delta \ll 1 $ small enough. Therefore by (\ref{sigma2lowerbd}) and the same argument in \textit{Step 3} we get the same lower bound
\begin{equation} \label{xilowerbdx0}
 | \xi( s  ) | > \frac{C_E}{10} (\sigma_2)^2, \, \text{for all } s \in [0, t - \sigma_2 ].
\end{equation}
And therefore we get the desired estimate.

\bigskip

\textit{Step 5} \qquad
We now consider the case when $x \in \Omega$ and $ t \ge \tb$. We need to bound the integral (\ref{trajint}).
Let
\[
\sigma_1 = \delta \frac{ \vb \cdot \nabla (\xb) }{ |v|^2 + \| E \|_{L^\infty_{t,x}}^2 + \| E \|_{L^\infty_{t,x}} +1 },
\]
as defined in (\ref{sigma12lowerbd}). If
\[
\sigma_1 \ge t,
\]
then from \textit{Step 2} $|\xi (\tilde X(s) ) | $ is increasing on $[0, \tb]$,and by
(\ref{xiupperbd}), (\ref{vdotnlowerbd}), and the bound for (\textbf{I}) in (\ref{Ibd}), we get the desired estimate.

Now we assume 
\[
\sigma_1 < t.
\]
So from (\ref{sigma12lowerbd}) we have
\begin{equation} \label{sigma1lowerbd}
|\xi (  \tilde X(\sigma_1 ) ) | \ge \frac{\delta ( \vb \cdot \nabla \xi (\xb) ) ^2 }{ 2 ( |v| ^2 +\| E \|_{L^\infty_{t,x}}^2 + \| E \|_{L^\infty_{t,x}}+ 1 ) }.
\end{equation}

Now if 
\begin{equation} \label{xixlowerbd1}
|\xi ( x ) | \le \delta \frac{ \alpha^2 ( t,x,v) }{10(|v|^2 +\| E \|_{L^\infty_{t,x}}^2 + \| E \|_{L^\infty_{t,x}} +1)} ,
\end{equation}
we have
\begin{equation} 
\begin{split}
\alpha^2(t,x,v)    \le &  (\nabla \xi (x ) \cdot v ) ^2 + C(|v|^2  + \| E \|_{L^\infty_{t,x}}+1 ) |\xi (x )| 
\\ \le & (\nabla \xi (x ) \cdot v ) ^2 + \delta \alpha^2 (t,x,v) \le (\nabla \xi (x ) \cdot v ) ^2 + \frac{1}{10} \alpha^2(t,x,v),
\end{split}
\end{equation}
if $ \delta \ll 1 $ is small enough. So 
\begin{equation} \label{alphaxiv}
\frac{1}{2} \alpha(t,x,v) \le |\nabla \xi(x) \cdot v | .
\end{equation}

Claim: 
\[
 \nabla \xi (x ) \cdot v >0 .
 \]
Since otherwise by (\ref{accnearbdy}) we have
\[
\frac{d}{ds} | \xi (\tilde X(s) ) | > 0,
\]
for all $s \in [0,\tb]$, so $|\xi(\tilde X(s))| $ is always increasing, thus 
\[
|\xi(\tilde X(s))| \le \delta \frac{ \alpha^2 ( t,x,v) }{10(|v|^2 +\| E \|_{L^\infty_{t,x}}^2 + \| E \|_{L^\infty_{t,x}} +1)},
\]
for all $s \in [0,\tb]$, which contradicts (\ref{sigma1lowerbd}).

Therefore $\nabla \xi (x ) \cdot v > 0$, and we can run the same argument from \textit{Step 2}, \textit{Step 3}, \textit{Step 4} , and by (\ref{alphaxiv}) we get the same estimate.
 
 If 
 \Be \label{xixupperbd1}
 |\xi ( x ) | > \delta \frac{  \alpha^2 ( t,x,v) }{10(|v|^2 +\| E \|_{L^\infty_{t,x}}^2 + \| E \|_{L^\infty_{t,x}} +1)},
 \Ee
 we claim:
\begin{equation} \label{xilowerbd2}
|\xi ( \tilde X(s)  ) | \ge \delta^2 \frac{ \alpha^2(t,x,v) }{ |v|^2 +\| E \|_{L^\infty_{t,x}}^2 + \| E \|_{L^\infty_{t,x}} + 1 },
\end{equation}
for all $s \in [ \sigma_1, \tb ] $. Since otherwise let 
\[
s^* : = \min \{ s  \in [ \sigma_1, t ] :  |\xi ( \tilde X(s)  ) | < \delta^2 \frac{ \alpha^2(t,x,v) }{ |v|^2 + \| E \|_{L^\infty_{t,x}}^2 + \| E \|_{L^\infty_{t,x}} +1 } \}.
\]
From (\ref{sigma1lowerbd}) we have $s^* > \sigma_1$, and 
\[
\frac{d}{ds} | \xi (\tilde X(s^* ) ) | < 0.
\]
And from (\ref{accnearbdy}) we have
\[
\frac{d^2}{ds^2} | \xi (\tilde X(s ) ) | < 0,
\]
for all $s \in [s^*,t] $. So $| \xi (\tilde X(s ) ) | $ is always decreasing on $[s^*,\tb]$. Therefore
\[
|\xi(x) | = | \xi (\tilde X(\tb ) ) | <| \xi (\tilde X(s^* ) ) |< \delta^2 \frac{ \alpha^2(t,x,v) }{ |v|^2 + \| E \|_{L^\infty_{t,x}}^2 + \| E \|_{L^\infty_{t,x}} +1 },
\]
which contradicts (\ref{xixupperbd1}). Therefore the lower bound (\ref{xilowerbd2}) and the estimates (\ref{IIbd}), (\ref{Ibd}) gives the desired bound.

\bigskip

\textit{Step 6} \qquad
Finally we consider the case $x \in \Omega$ and $t < \tb$.
First suppose 
\[
| \xi(x) | \le \delta \frac{ \alpha^2(t,x,v) }{ 10 ( |v|^2 + \| E \|_{L^\infty_{t,x}}^2 + \| E \|_{L^\infty_{t,x}} +1 ) }.
\]
From (\ref{alphaxiv}) we have
\[
\frac{\alpha(t,x,v) }{2} \le | v \cdot \nabla \xi (x) |.
\]

If $v \cdot \nabla \xi (x) > 0 $, then by (\ref{accnearbdy}) we have $\xi ( X(t+ t')  ) = 0 $ for some $t' \lesssim \frac{\delta}{C_E ^2 } < 1$. Therefore we can extend the trajectory until it hits the boundary and conclude the desired bound from \textit{Step 3}.

If $v \cdot \nabla \xi (x) < 0 $, again by (\ref{accnearbdy}) we have $ | \xi (X(s) ) | $ is increasing on $[0,t ] $ and $|V(s) \cdot \nabla \xi (X(s) ) | $ is decreasing on $[0,t]$. Therefore using the change of variable $s \mapsto |\xi |$:
\begin{equation} \begin{split}
\int_0^{t} & e^{ -\int_{s}^t\frac{\varpi}{2} \langle V(\tau;t,x,v) \rangle d\tau }  \frac{ 1}{\left[ | V(s)|^2 \xi( X( s) ) + \xi^2( X( s) -C_E \xi (  X(s) ) \right]^{\frac{\beta - 1}{2} } }  ds
\\ \lesssim & \int_0^{\delta \frac{ \alpha^2(t,x,v) }{ 10 ( |v|^2 +\| E \|_{L^\infty_{t,x}}^2 + \| E \|_{L^\infty_{t,x}} + 1 ) }} \frac{1}{ | V(s) \cdot \nabla \xi (X(s)) |  (C_E |\xi |)^{\frac{\beta -1}{2}} }  d|\xi | 
\lesssim  \int_0^{\delta \frac{ \alpha^2(t,x,v) }{ 10 ( |v|^2 + \| E \|_{L^\infty_{t,x}}^2 + \| E \|_{L^\infty_{t,x}} +1 ) }} \frac{1}{ | v \cdot \nabla \xi(x) | (C_E |\xi |)^{\frac{\beta -1}{2}} }  d|\xi | 
\\ \lesssim  &\int_0^{\delta \frac{ \alpha^2(t,x,v) }{ 10 ( |v|^2 + \| E \|_{L^\infty_{t,x}}^2 + \| E \|_{L^\infty_{t,x}} +1 ) }} \frac{1}{ | \alpha(t,x,v)  (C_E|\xi |)^{\frac{\beta -1}{2}} }  d|\xi | 
\lesssim \frac{  \delta ^{\frac{3 -\beta}{2} }}{ C_E^{\frac{\beta -1}{2} } (\alpha(t,x,v))^{\beta -2 }  (|v|^2 + \| E \|_{L^\infty_{t,x}}^2 + \| E \|_{L^\infty_{t,x}} +1 )^{\frac{3 -\beta}{2}}},
\end{split}
\end{equation}
which is the desired estimate.

Now suppose 
\begin{equation} \label{xilowerbd3}
| \xi(x) | > \delta \frac{ \alpha^2(t,x,v) }{ 10 ( |v|^2 + \| E \|_{L^\infty_{t,x}}^2 + \| E \|_{L^\infty_{t,x}} +1 ) },
\end{equation}
and
\[
| \xi(x_0) | \le \delta \frac{ \alpha^2(t,x,v) }{ 10 ( |v|^2 + \| E \|_{L^\infty_{t,x}}^2 + \| E \|_{L^\infty_{t,x}} +1 ) }.
\]
Then by (\ref{alphaxi0v0}) we have
\begin{equation} \label{alphaxi0v02}
\frac{\alpha(t,x,v) }{2} \lesssim e^{ C_\xi \frac{\|  \nabla E \|_{L^\infty_{t,x}}  + \| E \|_{L^\infty_{t,x}}^2 + \| E \|_{L^\infty_{t,x}}}{C_E}} |\nabla \xi(x_0) \cdot v_0 |.
\end{equation}
Now if $v_0 \cdot \nabla \xi (x_0) > 0 $, then from (\ref{accnearbdy}) we have $|\xi (X(s) ) |$ is decreasing for all $s \in [0, t] $. And this contradicts with (\ref{xilowerbd3}).
So we must have
\[
v_0 \cdot \nabla \xi (x_0) < 0.
\]
Then we can define $\sigma_1 = \delta \frac{ |v_0 \cdot \nabla \xi (x_0 ) | }{ |v|^2 + \| E \|_{L^\infty_{t,x}}^2 + \| E \|_{L^\infty_{t,x}} + 1 }$ as before. 
Now if $\sigma_1 \ge t$ then $|\xi (X(s) ) | $ is increasing on $[0,t]$, using the change of variable $x \mapsto |\xi | $ and the estimate (\ref{Ibd}) and (\ref{alphaxi0v02}) we get the desired bound.

If $\sigma_1 < t$, then from (\ref{sigma12lowerbd}) we have
\[
|\xi (X(\sigma_1) ) | \ge \delta \frac{ (v_0 \cdot \nabla \xi (x_0 ))^2 }{ 2 (|v|^2 + \| E \|_{L^\infty_{t,x}}^2 + \| E \|_{L^\infty_{t,x}} + 1 )}.
\]
And then from the argument for (\ref{xilowerbd2}) we get
\[
|\xi (  X(s)  ) | \ge \delta^2 \frac{ \alpha^2(t,x,v) }{ |v|^2 +  \| E \|_{L^\infty_{t,x}}^2 + \| E \|_{L^\infty_{t,x}} +1 },
\]
for all $s \in [\sigma_1, t]$. This lower bound combined with the estimate (\ref{IIbd}), (\ref{Ibd}) gives the desired bound.

Finally we left with the case
\[
| \xi(x_0) | > \delta \frac{ \alpha^2(t,x,v) }{ 10 ( |v|^2 + \| E \|_{L^\infty_{t,x}}^2 + \| E \|_{L^\infty_{t,x}} + 1 ) }.
\]
Then again, from the argument for (\ref{xilowerbd2}) we get
\[
|\xi (  X(s)  ) | \ge \delta^2 \frac{ \alpha^2(t,x,v) }{ |v|^2 +  \| E \|_{L^\infty_{t,x}}^2 + \| E \|_{L^\infty_{t,x}}+ 1 },
\]
for all $s \in [0, t ] $. This lower bound combined with the estimate (\ref{IIbd}) gives the desired bound.
\end{proof}

\begin{proof}[\textbf{Proof of (2) Lemma \protect\ref{keylemma}}]
Since $\frac{\langle u\rangle^{r}}{\langle v\rangle^{r}} \lesssim \frac{ \langle u \rangle ^r}{\langle V_{\mathbf{cl}}(s) \rangle ^r }  \lesssim \{ 1+
|V_{\mathbf{cl}}(s) -u|^{2}\}^{\frac{r}{2}}$ and $\langle V_{\mathbf{cl}}(s)-u\rangle^{r}
e^{-\theta |V_{\mathbf{cl}}(s)-u|^{2}}\lesssim e^{-C_{\theta,r} |V_{\mathbf{%
cl}}(s)-u|^{2}}$, it suffices to consider the case $r=0$. It is important to control the \textit{number of bounces}, 
\[
\ell_{*}(s)= \ell_{*}(s;t,x,v)\in\mathbb{N} \ \ \ \text{such that} \ \ \   t^{\ell_{*}+1} \leq s< t^{\ell_{*}}.
\]
An important consequence of Velocity lemma is that for the specular cycles
$$\alpha(s,X_{\mathbf{cl}}(s;t,x,v), V_{\mathbf{cl}}(s;t,x,v)  ) \gtrsim e^{-\mathcal{C}\langle v \rangle ||t-s|} \alpha(t,x,v),$$
and therefore for the specular cycles
\begin{equation}\label{number_control}
\begin{split}
\ell_{*}(s;t,x,v) &\leq \frac{|t-s|}{ \min_{0 \leq \ell \leq \ell_{*}(s;t,x,v)}  |t^{\ell}-t^{\ell+1}| } \lesssim
 \frac{|t-s|}{ \min_{0 \leq \ell \leq \ell_{*}(s;t,x,v)}  \frac{ {\alpha(t^\ell, x^{\ell}, v^{\ell})}}{|v^{\ell}|^{2}} } \\
 &\lesssim \frac{|t-s|(|v|^{2}+1)}{{\alpha(t,x,v)}}e^{\mathcal{C}\langle v \rangle (t-s)}.
\end{split}
\end{equation}
For fixed $(t,x,v)$ we use the following notation $\alpha(s):= \alpha(s;t,x,v) := \alpha(X_{\mathbf{cl}}(s;t,x,v), V_{\mathbf{cl%
}}(s;t,x,v)).$

Now we consider the estimate (\ref{specular_nonlocal}). From \eqref{Ibd}, \eqref{IIIbd}, and \eqref{IIbd} we have
\begin{equation} \label{specntl0}
\begin{split}
&\int_{0}^{t} \int_{\mathbb{R}%
^{3}} e^{-l \langle v\rangle (t-s)} \frac{e^{-\theta |V_{\mathbf{cl}} (s)
-u|^{2}}}{|V_{\mathbf{cl}} (s) -u|^{2-\kappa}} \frac{Z(s,x,v)}{\big[ %
\alpha(s,X_{\mathbf{cl}}(s;t,x,v),u) \big]^{\beta}} \mathrm{d}u \mathrm{d}s \\
&\lesssim \sum_{\ell=0}^{\ell_{*}(0;t,x,v)}   \int^{t^{\ell}}_{t^{\ell+1}}
\int_{\mathbb{R}^{3}} e^{- l \langle v\rangle (t-s)} \frac{e^{-\theta |
v^{\ell}-u |^{2}}}{|v^{\ell} -u|^{2-\kappa}} \frac{Z(s,x,v)}{\big[ %
\alpha(s,X_{\mathbf{cl}}(s;t,x,v),u)  \big]^{\beta}} \mathrm{d}u
\mathrm{d}s \\
&\lesssim \sup_{0 \leq s\leq t} \big\{ e^{-  \frac{l}{2}  \langle v\rangle
(t-s)} Z(s,x,v) \big\}  
\\ & \ \ \ \ \ \times \sum_{\ell=0}^{\ell_{*}(0;t,x,v)} \left( e^{- \frac{l}{2} \langle v\rangle (t-t^\ell)} \frac{  \delta ^{\frac{3 -\beta}{2} }}{\langle v\rangle^2 (\alpha(t^\ell,x^\ell,v^\ell))^{\beta -2 }  } +   \frac{1 } {  \delta ^{\beta -1 } ( \alpha ( t^\ell,x^\ell,v^\ell) )^{\beta - 1 } }\int_{t^{\ell+1}}^{t^\ell}  e^{- \frac{l}{2} \langle v\rangle (t-s)} ds  \right)
\\ &\lesssim \sup_{0 \leq s\leq t} \big\{ e^{-  \frac{l}{2}  \langle v\rangle
(t-s)} Z(s,x,v) \big\}  
\\ & \ \ \ \ \ \times \sum_{\ell=0}^{\ell_{*}(0;t,x,v)} \left( e^{- \frac{l}{4} \langle v\rangle (t-t^\ell)} \frac{  \delta ^{\frac{3 -\beta}{2} }}{\langle v\rangle^2 (\alpha(t,x,v))^{\beta -2 }  } +   \frac{e^{C\langle v \rangle |t-t^\ell|} } {  \delta ^{\beta -1 } ( \alpha ( t,x,v) )^{\beta - 1 } }\int_{t^{\ell+1}}^{t^\ell}  e^{- \frac{l}{2} \langle v\rangle (t-s)} ds  \right).
\end{split}
\end{equation}
Clearly
\Be \label{specntl1} \begin{split}
& \sum_{\ell=0}^{\ell_{*}(0;t,x,v)}  \frac{e^{C\langle v \rangle |t-t^\ell|} } {  \delta ^{\beta -1 } ( \alpha ( t,x,v) )^{\beta - 1 } }\int_{t^{\ell+1}}^{t^\ell}  e^{- \frac{l}{2} \langle v\rangle (t-s)} ds 
\\ & \lesssim  \frac{1 } {  \delta ^{\beta -1 } ( \alpha ( t,x,v) )^{\beta - 1 } }\int_{0}^{t}  e^{- \frac{l}{4} \langle v\rangle (t-s)} ds
\\ & \lesssim  \frac{1 } { l \langle v \rangle \delta ^{\beta -1 } ( \alpha ( t,x,v) )^{\beta - 1 } }.
\end{split} \Ee
And for $ \sum_{\ell=0}^{\ell_{*}(0;t,x,v)} e^{- \frac{l}{4} \langle v\rangle (t-t^\ell)} \frac{  \delta ^{\frac{3 -\beta}{2} }}{\langle v\rangle^2 (\alpha(t,x,v))^{\beta -2 }  } $, we let $\tilde \ell $ be the bounce that $t^{\tilde \ell } \ge t-\frac{1}{\langle v \rangle}$ and $t^{\tilde \ell +1 } < t-\frac{1}{\langle v \rangle}$, and decompose $ \sum_{\ell=0}^{\ell_{*}(0;t,x,v)}  =  \sum_{\ell=0}^{\tilde \ell} +  \sum_{\ell = \tilde \ell +1}^{\ell_{*}(0;t,x,v)} $. Then from \eqref{number_control}
\[
 \sum_{\ell=0}^{\tilde \ell} e^{- \frac{l}{4} \langle v\rangle (t-t^\ell)} \le |  \tilde \ell | \lesssim \frac{1/\langle v \rangle }{\alpha(t,x,v) / |v|^2} \lesssim \frac{|v|}{\alpha(t,x,v) }.
\]
For $\ell \ge \tilde \ell +1 $, we have 
\[
| t -t^{\ell+1} | \le |t -t^\ell | + |t^\ell -t^{\ell +1} | \le  |t -t^\ell | + C\frac{1}{\langle v \rangle }  \le  |t -t^\ell | + C|t-t^\ell | = (1+C)|t-t^\ell |.
\]
Thus
\[ \begin{split}
 \sum_{\ell = \tilde \ell +1}^{\ell_{*}(0;t,x,v)} e^{- \frac{l}{4} \langle v\rangle (t-t^\ell)} & \le  \sum_{\ell = \tilde \ell +1}^{\ell_{*}(0;t,x,v)} e^{-\frac{l}{8} \langle v \rangle (t-t^\ell ) } e^{- \frac{l}{8(1+C)} \langle v\rangle (t-t^{\ell+1})}
 \\ & \le \max_{\ell} \left\{ \frac{e^{-\frac{l}{8} \langle v \rangle (t-t^\ell ) }} {  |t^\ell -t^{\ell+1} |}  \right\}\sum_{\ell = 0}^{\ell_*} |t^\ell -t^{\ell+1} |e^{- \frac{l}{8(1+C)} \langle v\rangle (t-t^{\ell+1})} 
 \\ & \lesssim \frac{\langle v \rangle^2e^{-\frac{l}{8} \langle v \rangle (t-t^\ell ) } e^{C \langle v \rangle (t -t^\ell ) }}{\alpha(t,x,v) } \int_0^t e^{- \frac{l}{8(1+C)} \langle v\rangle (t-s)}  ds
 \\ & \lesssim  \frac{\langle v \rangle (1+C)}{l \alpha(t,x,v) }.
\end{split} \]
Therefore
\Be \label{specntl2}
 \sum_{\ell=0}^{\ell_{*}(0;t,x,v)} e^{- \frac{l}{4} \langle v\rangle (t-t^\ell)} \frac{  \delta ^{\frac{3 -\beta}{2} }}{\langle v\rangle^2 (\alpha(t,x,v))^{\beta -2 }  } \lesssim  \frac{  \delta ^{\frac{3 -\beta}{2} }}{\langle v\rangle (\alpha(t,x,v))^{\beta -1 }  }.
\Ee
Combing \eqref{specntl0}, \eqref{specntl1} and \eqref{specntl2} we prove \eqref{specular_nonlocal}.
\end{proof}

\section{Moving frame for specular cycles}

We use the moving frame for the specular cycles introduced in \cite{GKTT1}. We denote the standard spherical coordinate $\mathbf{x}_{\parallel} = \mathbf{x}_{\parallel} (\omega) =(\mathbf{x}_{\parallel,1}, \mathbf{x}_{\parallel,2})$ for $\omega \in\mathbb{S}^{2}$
\[
\omega = (\cos \mathbf{x}_{\parallel,1}(\omega) \sin \mathbf{x}_{\parallel,2}(\omega), \sin  \mathbf{x}_{\parallel,1}(\omega) \sin \mathbf{x}_{\parallel,2}(\omega), \cos \mathbf{x}_{\parallel,2}(\omega)),
\]
where $\mathbf{x}_{\parallel,1}(\omega) \in [0,2\pi)$ is the azimuth and $ \mathbf{x}_{\parallel,2} (\omega)\in [0,\pi)$ is the inclination. 

%

We define an orthonormal basis of $\mathbb{R}^{3}$, $\{\hat{r}(\omega), \hat{\phi} (\omega), \hat{\theta } (\omega)  \}$, with $\hat{r}(\omega):= \omega$ and 
\begin{equation}\notag
\begin{split}
 {\hat{{\phi}}}(\omega) & \ : = \ (\cos  \mathbf{x}_{\parallel,1}(\omega) \cos \mathbf{x}_{\parallel,2}(\omega), \sin \mathbf{x}_{\parallel,1}(\omega) \cos \mathbf{x}_{\parallel,2}(\omega), -\sin \mathbf{x}_{\parallel,2}(\omega)),\\
 {\hat{{\theta}}} (\omega) & \ := \ (-\sin  \mathbf{x}_{\parallel,1}(\omega), \cos  \mathbf{x}_{\parallel,1}(\omega),0).
\end{split}
\end{equation}
Moreover, ${\hat{ {r}}}\times  {\hat{\phi}} = {\hat{\theta}}, \     {\hat{\phi}}\times  {\hat{\theta}} =  {\hat{r}}, \  {\hat{\theta}}\times  {\hat{r}} = {\hat{\phi}},$ and
\begin{equation}\label{Dr}
 \partial_{ \mathbf{x}_{\parallel,1}}  { \hat{r}} = \sin \mathbf{x}_{\parallel,2} \  {\hat{\theta}},  \ \ \ \partial_{\mathbf{x}_{\parallel,2}}  {\hat{r}} =  {\hat{\phi}},
\end{equation}
where $ \partial_{ \mathbf{x}_{\parallel,1}}  { \hat{r}}$ does not vanish (non-degenerate) away from $\mathbf{x}_{\parallel,2} =0$ or $\pi.$

Without loss of generality we assume $\mathbf{0}=(0,0,0)\in\Omega$. For
\[
\mathbf{p}=(z, w)\in \partial\Omega\times \mathbb{S}^{2} \ \text{ with } \ n(z)\cdot w=0,
\]
we define the north pole $\mathcal{N}_{\mathbf{p}} \in\partial\Omega$ and the south pole $\mathcal{S}_{\mathbf{p}} \in\partial\Omega$ as
\[
\mathcal{N}_{\mathbf{p}} := |\mathcal{N}_{\mathbf{p}}| ( n(z)\times w) \in\partial\Omega, \ \ \ \mathcal{S}_{\mathbf{p}} := - |\mathcal{S}_{\mathbf{p}}| (  n(z)\times w) \in\partial\Omega,
\]
where $\partial_{\mathbf{x}_{\parallel,1}} \hat{r}$ is degenerate. We define the straight-line $\mathcal{L}_{\mathbf{p}}$ passing both poles
\[
\mathcal{L}_{\mathbf{p}} := \{ \tau \mathcal{N}_{\mathbf{p}} + (1-\tau) \mathcal{S}_{\mathbf{p}} : \tau\in \mathbb{R}\}.
\]

\begin{lemma}\label{chart_lemma} Assume $\Omega$ is convex  (\ref{convex}). Fix $\mathbf{p}=(z, w)\in \partial\Omega\times \mathbb{S}^{2}$ with $n(z)\cdot w=0$.

(i) There exists a smooth map (spherical-type coordinate)
\begin{equation}\label{x_parallel}
\begin{split}
 \mathbf{\eta}_{\mathbf{p}} \ : \ \ &  \ \ \ [0,2\pi)\times (0,\pi) \ \ \ \ \rightarrow \ \ \partial\Omega \backslash\{\mathcal{N}_{\mathbf{p}} , \mathcal{S}_{\mathbf{p}}  \}, \\
  &\mathbf{x}_{\parallel_{\mathbf{p}}}:= (\mathbf{x}_{\parallel_{\mathbf{p}},1  }, \mathbf{x}_{\parallel_{\mathbf{p}},2  })     \mapsto \ \ \ \  \mathbf{\eta}_{\mathbf{p}}( \mathbf{x}_{\parallel_{\mathbf{p}}}) ,
\end{split}
\end{equation}
which is one-to-one and onto. Here on $[0,2\pi)\times (0,\pi)$ we have $\partial_{i} \mathbf{\eta}_{\mathbf{p}}:= \frac{\partial \mathbf{\eta}_{\mathbf{p}}}{\partial \mathbf{x}_{\parallel_{\mathbf{p}} ,i  }}\neq 0$ and
 \begin{equation}\label{nondegenerate_eta}
 \frac{\partial   \mathbf{\eta}_{ \mathbf{p}} }{\partial\mathbf{x}_{\parallel_{ \mathbf{p}},1}}  ( \mathbf{x}_{\parallel_{\mathbf{p}} }  ) \times \frac{\partial   \mathbf{\eta}_{ \mathbf{p}} }{\partial\mathbf{x}_{\parallel_{\mathbf{p}},2}}  ( \mathbf{x}_{\parallel_{\mathbf{p}} }  )\neq 0.
 \end{equation}
We define
$$\mathbf{n}_{ \mathbf{p}}:= n\circ\mathbf{\eta}_{ \mathbf{p}} : [ 0 ,2\pi) \times (0,\pi) \rightarrow  \mathbb{S}^{2}.$$

\vspace{8pt}

(ii) We define the $\mathbf{p}-$spherical coordinate in the tubular neighborhood of the boundary:

\vspace{4pt}

For $\delta>0, \ \delta_{1}>0 , \ C>0,$ we have a smooth one-to-one and onto map

\vspace{4pt}

\begin{equation}\notag
\begin{split} 
\Phi_{\mathbf{p}} :  [0,C\delta) \times [0, 2\pi) \times (\delta_{1},\pi-\delta_{1}) \times \mathbb{R}\times \mathbb{R}^{2} & \   \rightarrow \  \{x\in \bar{\Omega} : |\xi(x)|< \delta\} \backslash B_{C\delta_{1}}(\mathcal{L}_{\mathbf{p}})\times \mathbb{R}^{3},\\
(\mathbf{x}_{\perp_{\mathbf{p}}}, \mathbf{x}_{\parallel_{\mathbf{p}},1 }, \mathbf{x}_{\parallel_{\mathbf{p}},2 }, \mathbf{v}_{\perp_{\mathbf{p}}}, \mathbf{v}_{\parallel_{\mathbf{p}},1},  \mathbf{v}_{\parallel_{\mathbf{p}},2}) & \ \mapsto \ 
\Phi_{\mathbf{p}} (\mathbf{x}_{\perp_{\mathbf{p}}}, \mathbf{x}_{\parallel_{\mathbf{p}},1 }, \mathbf{x}_{\parallel_{\mathbf{p}},2 }, \mathbf{v}_{\perp_{\mathbf{p}}}, \mathbf{v}_{\parallel_{\mathbf{p}},1},  \mathbf{v}_{\parallel_{\mathbf{p}},2}),
\\
\end{split}
\end{equation}
\vspace{4pt}

where $B_{C\delta_{1}}(\mathcal{L}_{\mathbf{p}}): =\{ x\in \mathbb{R}^{3} : |x - y |< C\delta_{1} \text{ for some } y \in \mathcal{L}_{\mathbf{p}}  \}.$

Explicitly,
\begin{equation}\label{polar}
\Phi_{\mathbf{p}} ( \mathbf{x}_{\perp_{\mathbf{p}}},\mathbf{x}_{\parallel_{\mathbf{p}}} ,  \mathbf{v}_{\perp_{\mathbf{p}}},\mathbf{v}_{\parallel_{\mathbf{p}}}):=
\left[\begin{array}{c}
\mathbf{x}_{\perp_{\mathbf{p}}}[- \mathbf{n}_{\mathbf{p}}   (\mathbf{x}_{\parallel_{\mathbf{p}}})     ]  + \mathbf{\eta}_{ \mathbf{p}}(\mathbf{x}_{\parallel_{ \mathbf{p}}}) \\
  \mathbf{v}_{\perp_{ \mathbf{p}}} [-    \mathbf{n}_{\mathbf{p}}   (\mathbf{x}_{\parallel_{ \mathbf{p}}}) ]  + \mathbf{v}_{\parallel_{ \mathbf{p}}} \cdot \nabla \mathbf{\eta}_{ \mathbf{p}} (\mathbf{x}_{\parallel_{ \mathbf{p}}}) + \mathbf{x}_{\perp_{ \mathbf{p}}} \mathbf{v}_{\parallel_{ \mathbf{p}}} \cdot \nabla [-\mathbf{n}_{ \mathbf{p}} (\mathbf{x}_{\parallel_{ \mathbf{p}}})]
  \end{array} \right],
\end{equation}
where $\nabla \mathbf{\eta}_{\mathbf{p}}= (\partial_{1} \mathbf{\eta}_{\mathbf{p}}, \partial_{2} \mathbf{\eta}_{\mathbf{p}})= ( \frac{\partial \mathbf{\eta}_{\mathbf{p}}   }{\partial \mathbf{x}_{\parallel_{   \mathbf{p } ,1  }   }} ,\frac{\partial \mathbf{\eta}_{\mathbf{p}}   }{\partial \mathbf{x}_{\parallel_{   \mathbf{p } ,2  }   }}  )$ and $\nabla \mathbf{n}_{\mathbf{p}}= (\partial_{1} \mathbf{n}_{\mathbf{p}}, \partial_{2} \mathbf{n}_{\mathbf{p}}  )= (\frac{\partial \mathbf{ n}_{\mathbf{p}}   }{\partial \mathbf{x}_{\parallel_{   \mathbf{p } ,1  }   }},\frac{\partial \mathbf{n}_{\mathbf{p}}   }{\partial \mathbf{x}_{\parallel_{   \mathbf{p } ,2  }   }}  ).$

The Jacobian matrix is
\begin{equation}\begin{split}\label{jac_Phi}
&\frac{\partial \Phi( \mathbf{x}_{\perp}, \mathbf{x}_{\parallel}, \mathbf{v}_{\perp}, \mathbf{v}_{\parallel})}{\partial ( \mathbf{x}_{\perp}, \mathbf{x}_{\parallel}, \mathbf{v}_{\perp}, \mathbf{v}_{\parallel})}\\
&{\tiny = \left[\begin{array}{ccc|ccc}
n(\mathbf{x}_{\parallel})
 &
 \substack{ \frac{\partial \eta}{\partial \mathbf{x}_{\parallel,1}} (\mathbf{x}_{\parallel}) \\ + \mathbf{x}_{\perp} \frac{\partial n}{\partial \mathbf{x}_{\parallel,1}} (\mathbf{x}_{\parallel})}
  \frac{\partial \Phi_{1}(\mathbf{x}_{\perp}, \mathbf{x}_{\parallel})}{\partial (\mathbf{x}_{\perp}, \mathbf{x}_{\parallel})}
  &
  \substack{ \frac{\partial \eta}{\partial \mathbf{x}_{\parallel,2}} (\mathbf{x}_{\parallel})  \\  + \mathbf{x}_{\perp} \frac{\partial n}{\partial \mathbf{x}_{\parallel,2}} (\mathbf{x}_{\parallel}) }
  & &\mathbf{0}_{3,3} & \\ \hline
 -\mathbf{v}_{\parallel}  \cdot \nabla_{\mathbf{x}_{\parallel}} \mathbf{n}(\mathbf{x}_{\parallel})&
  \substack{ -\mathbf{v}_{\perp} \frac{ \partial \mathbf{n}}{\partial \mathbf{x}_{\parallel,1}}(\mathbf{x}_{\parallel}) \\
  + \mathbf{v}_{\parallel} \cdot \nabla_{\mathbf{x}_{\parallel}} \frac{\partial \mathbf{\eta}}{\partial \mathbf{x}_{\parallel,1}} (\mathbf{x}_{\parallel})  \\
   - \mathbf{x}_{\perp} \mathbf{v}_{\parallel} \cdot \nabla_{\mathbf{x}_{\parallel}} \frac{\partial \mathbf{n}}{\partial \mathbf{x}_{\parallel,1}}(\mathbf{x}_{\parallel})  }&
 \substack{ -\mathbf{v}_{\perp} \frac{\partial \mathbf{n}}{\partial \mathbf{x}_{\parallel,2}}(\mathbf{x}_{\parallel}) \\ + \mathbf{v}_{\parallel} \cdot \nabla_{\mathbf{x}_{\parallel}} \frac{\partial \mathbf{\eta}}{\partial \mathbf{x}_{\parallel,2}} (\mathbf{x}_{\parallel})  \\ - \mathbf{x}_{\perp} \mathbf{v}_{\parallel} \cdot \nabla_{\mathbf{x}_{\parallel}} \frac{\partial \mathbf{n}}{\partial \mathbf{x}_{\parallel,2}}(\mathbf{x}_{\parallel})  }
 &
  n(\mathbf{x}_{\parallel})
 &
 \substack{ \frac{\partial \eta}{\partial \mathbf{x}_{\parallel,1}} (\mathbf{x}_{\parallel}) \\ + \mathbf{x}_{\perp} \frac{\partial n}{\partial \mathbf{x}_{\parallel,1}} (\mathbf{x}_{\parallel})}
  \frac{\partial \Phi_{1}(\mathbf{x}_{\perp}, \mathbf{x}_{\parallel})}{\partial (\mathbf{x}_{\perp}, \mathbf{x}_{\parallel})}
 &
  \substack{ \frac{\partial \eta}{\partial \mathbf{x}_{\parallel,2}} (\mathbf{x}_{\parallel})  \\  + \mathbf{x}_{\perp} \frac{\partial n}{\partial \mathbf{x}_{\parallel,2}} (\mathbf{x}_{\parallel}) }
 \end{array}\right].}\end{split}
 \end{equation}

We fix an inverse map
\begin{equation}\notag
\begin{split}
\Phi_{\mathbf{p}}^{-1 }   :    \{x\in \bar{\Omega} : |\xi(x)|< \delta\}\backslash  B_{C\delta^{\prime}}(\mathcal{L}_{\mathbf{p}})    \times \mathbb{R}^{3}    & \ \rightarrow \  [0,C\delta) \times  [0, 2\pi) \times (\delta_{1},\pi-\delta_{1}) \times \mathbb{R}\times \mathbb{R}^{2}.
\end{split}
\end{equation}
In general this choice is not unique but once we fix the range as above then an inverse map is uniquely determined.

We denote, for $(x,v) \in   \{x\in \bar{\Omega} : |\xi(x)|< \delta\}\backslash  B_{C\delta^{\prime}}(\mathcal{L}_{\mathbf{p}})    \times \mathbb{R}^{3}$
\[
(\mathbf{x}_{\perp_{\mathbf{p}} }, \mathbf{x}_{\parallel_{\mathbf{p}},1}, \mathbf{x}_{\parallel_{\mathbf{p}},2}, \mathbf{v}_{\perp_{\mathbf{p}}}, \mathbf{v}_{\parallel_{\mathbf{p}},1}, \mathbf{v}_{\parallel_{\mathbf{p}},2})=  \Phi_{\mathbf{p}}^{-1} (x,v).
\]


\vspace{8pt}

(iii) Let $\mathbf{q} = (y,u) \in \partial\Omega \times \mathbb{S}^{2}$ with $n(y)\cdot u=0$ and $|\mathbf{p}- \mathbf{q}| \ll 1$ and
\[
  {\Phi}_{\mathbf{p}} ( \mathbf{x}_{\perp_{  \mathbf{p}   }}, \mathbf{x}_{\parallel_{ \mathbf{p}}}, \mathbf{v}_{\perp_{\mathbf{p}}}, \mathbf{v}_{\parallel_{ \mathbf{p}}}) =(x,v)
= {\Phi}_{\mathbf{q}} ( \mathbf{x}_{\perp_{ \mathbf{q}   }}, \mathbf{x}_{\parallel_{ \mathbf{q}}}, \mathbf{v}_{\perp_{ \mathbf{q}}}, \mathbf{v}_{\parallel_{ \mathbf{q}}}) .
\]
Then
\begin{equation}\label{chart_changing}
\begin{split}
  \frac{\partial (\mathbf{x}_{\perp_{   \mathbf{p} }}, \mathbf{x}_{\parallel_{   \mathbf{p} }}, \mathbf{v}_{\perp_{   \mathbf{p} }}, \mathbf{v}_{\parallel_{   \mathbf{p} }})}{\partial ( \mathbf{x}_{\perp_{ \mathbf{q} }}, \mathbf{x}_{\parallel_{ \mathbf{q} }}, \mathbf{v}_{\perp_{ \mathbf{q} }}, \mathbf{v}_{\parallel_{ \mathbf{q} }})} = \nabla \Phi_{ \mathbf{q} }^{-1} \nabla \Phi_{  \mathbf{p} } 
=   
\mathbf{Id}_{6,6} + O_{\xi}(|\mathbf{p}-\mathbf{q}|)
\left[
\begin{array}{ccc|ccc}
0 & 0 & 0 &  &  &  \\
0 & 1 & 1 &  & \mathbf{0}_{3,3} & \\
0 & 1 & 1 &  & &   \\ \hline
0 & 0 & 0 & 0 & 0 & 0 \\
0 & |v| &  |v|  & 0 & 1 & 1 \\
0 &  |v|  &  |v|  & 0 & 1 & 1
\end{array}
\right].
\end{split}
\end{equation}\end{lemma}

\begin{proof}
See \cite{GKTT1}.
\end{proof}

\begin{lemma}\label{lemma_flow} 
(i) For $|\xi(X_{\mathbf{cl}}(s;t,x,v)) |< \delta$ and $|X_{\mathbf{cl}}(s;t,x,v)-\mathcal{L}_{\mathbf{p}}|> C\delta_{1}$ we define
\begin{equation}\notag
\begin{split}
 (\mathbf{X}_{\mathbf{p}}(s;t,x,v), \mathbf{V}_{\mathbf{p}}(s;t,x,v))& := \Phi^{-1}_{\mathbf{p}} ( {X}_{\mathbf{cl}}(s;t,x,v), {V}_{\mathbf{cl}}(s;t,x,v))\\
& := (\mathbf{x}_{\perp_{\mathbf{p}}}(s;t,x,v), \mathbf{x}_{\parallel_{\mathbf{p}}}(s;t,x,v),\mathbf{v}_{\perp_{\mathbf{p}}}(s;t,x,v), \mathbf{v}_{\parallel_{\mathbf{p}}}(s;t,x,v) )
.
\end{split}
\end{equation}
Then $|v|\simeq |\mathbf{V}_{\mathbf{p}}| $ and
\begin{equation}\label{ODE_ell}
\left[\begin{array}{c}  
\dot{\mathbf{x}} _{\perp_{\mathbf{p}}}\\
\dot{\mathbf{x}} _{\parallel_{ \mathbf{p}}} \\
\dot{\mathbf{v}} _{\perp_{ \mathbf{p}}} \\
\dot{\mathbf{v}} _{\parallel_{ \mathbf{p}}} 
\end{array} \right](s;t,x,v)
= \left[\begin{array}{c}  
\mathbf{v}_{\perp_{ \mathbf{p}}}  \\
\mathbf{v}_{\parallel_{ \mathbf{p}}}  \\
F_{\perp_{ \mathbf{p}}}(\mathbf{x}_{ \mathbf{p}} , \mathbf{v}_{ \mathbf{p}} )\\
F_{\parallel_{ \mathbf{p}}}(\mathbf{x}_{ \mathbf{p}} , \mathbf{v}_{ \mathbf{p}} )
\end{array} \right](s;t,x,v). 
\end{equation}
Here
\begin{equation}\label{F_perp}
\begin{split}
F_{\perp_{ \mathbf{p}}} = & F_{\perp_{\mathbf{p}}}(\mathbf{x}_{\perp_{ \mathbf{p}}}, \mathbf{x}_{\parallel_{ \mathbf{p}}},   \mathbf{v}_{\parallel_{ \mathbf{p}}})\\
= & \sum_{j,k=1}^{2} \mathbf{v}_{\parallel_{ \mathbf{p}},k} \mathbf{v}_{\parallel_{ \mathbf{p}},j} \ \partial_{j} \partial_{k} \mathbf{\eta}_{ \mathbf{p}}(\mathbf{x}_{\parallel_{ \mathbf{p}}}) \cdot \mathbf{n}_{ \mathbf{p}}(\mathbf{x}_{\parallel_{ \mathbf{p}}})  
  - \mathbf{x}_{\perp_{ \mathbf{p}}} \sum_{k=1}^{2} \mathbf{v}_{\parallel_{ \mathbf{p}},k} (\mathbf{v}_{\parallel_{ \mathbf{p}}}\cdot \nabla) \partial_{k} \mathbf{n}_{ \mathbf{p}}(\mathbf{x}_{\parallel_{ \mathbf{p}}}) \cdot \mathbf{n}_{ \mathbf{p}}(\mathbf{x}_{\parallel_{ \mathbf{p}}})
  \\ &  -E(s, - \mathbf{x}_\perp \mathbf n (\mathbf x_\parallel ) + \mathbf \eta (\mathbf x_\parallel ) ) \cdot \mathbf n (\mathbf x_\parallel ),
\end{split}
\end{equation}
where
$$\sum_{j,k=1}^{2} \mathbf{v}_{\parallel_{ \mathbf{p}},k} \mathbf{v}_{\parallel_{ \mathbf{p}},j} \ \partial_{j} \partial_{k} \mathbf{\eta}_{ \mathbf{p}}(\mathbf{x}_{\parallel_{ \mathbf{p}}}) \cdot \mathbf{n}_{ \mathbf{p}}(\mathbf{x}_{\parallel_{ \mathbf{p}}})\lesssim_{\xi}- |\mathbf{v}_{\parallel}|^{2},$$
and
\begin{equation}\label{F||}
\begin{split}
F_{\parallel_{ \mathbf{p}}}&=F_{\parallel_{ \mathbf{p}}}(\mathbf{x}_{\perp_{ \mathbf{p}}}, \mathbf{x}_{\parallel_{ \mathbf{p}}}, \mathbf{v}_{\perp_{ \mathbf{p}}}, \mathbf{v}_{\parallel_{ \mathbf{p}}})\\
& =  \sum_{i=1,2} G_{ \mathbf{p},ij} (\mathbf{x}_{\perp_{ \mathbf{p}}}, \mathbf{x}_{\parallel_{ \mathbf{p}}})\frac{(-1)^{i }}{\mathbf{n}_{ \mathbf{p}}(\mathbf{x}_{\parallel_{ \mathbf{p}}}) \cdot (\partial_{1} \mathbf{\eta}_{ \mathbf{p}}(\mathbf{x}_{\parallel_{ \mathbf{p}}}) \times \partial_{2} \mathbf{\eta}_{ \mathbf{p}}(\mathbf{x}_{\parallel_{ \mathbf{p}}}))} \\
& \ \ \  \times \big\{
2 \mathbf{v}_{\perp_{ \mathbf{p}}} \mathbf{v}_{\parallel_{ \mathbf{p}}} \cdot \nabla \mathbf{n}_{ \mathbf{p}} (\mathbf{x}_{\parallel_{ \mathbf{p}}}) - \mathbf{v}_{\parallel_{ \mathbf{p}}} \cdot  \nabla^{2} \mathbf{\eta}_{ \mathbf{p}}(\mathbf{x}_{\parallel_{ \mathbf{p}}}) \cdot \mathbf{v}_{\parallel_{ \mathbf{p}}} + \mathbf{x}_{\perp_{ \mathbf{p}}} \mathbf{v}_{\parallel_{ \mathbf{p}}} \cdot \nabla^{2} \mathbf{n}_{ \mathbf{p}}(\mathbf{x}_{\parallel_{ \mathbf{p}}}) \cdot \mathbf{v}_{\parallel_{ \mathbf{p}}} 
\\ & \quad \quad \quad  - E(s, - \mathbf{x}_\perp \mathbf n (\mathbf x_\parallel ) + \mathbf \eta (\mathbf x_\parallel )) 
\big\} \cdot \big\{ \mathbf{n}_{ \mathbf{p}} (\mathbf{x}_{\parallel_{ \mathbf{p}}}) \times \partial_{i+1} \mathbf{\eta}_{ \mathbf{p}}(\mathbf{x}_{\parallel_{ \mathbf{p}}})  \big\},
\end{split}
\end{equation}
where a smooth bounded function $G_{ \mathbf{p},ij} (\mathbf{x}_{\perp_{ \mathbf{p}}}, \mathbf{x}_{\parallel_{\mathbf{p}}})$ is specified in (\ref{G}).

\vspace{4pt}

(ii) For $\tau \in ( t^{\ell+1}, t^{\ell})$, if the $\mathbf{p}^{ \ell}-$spherical coordinate is well-defined in $[\tau , t^{\ell})$ then 
\[
[\mathbf{X}_{ \ell}(\tau;t,x,v), \mathbf{V}_{ \ell}(\tau;t,x,v)] \equiv [\mathbf{X}_{ \ell}(\tau;t^{\ell} ,0, \mathbf{x}_{\parallel_{\ell}}^{\ell} , \mathbf{v}_{\perp_{\ell}}^{\ell}, \mathbf{v}_{\parallel_{\ell}}^{\ell} ), \mathbf{V}_{ \ell} (\tau;t^{\ell} ,0, \mathbf{x}_{\parallel_{\ell}}^{\ell} , \mathbf{v}_{\perp_{\ell}}^{\ell}, \mathbf{v}_{\parallel_{\ell}}^{\ell} )]
\]
and, for $\partial_{\mathbf{v}_{ {\ell}}^{\ell}  } = [\partial_{\mathbf{v}_{\perp_{\ell}}^{\ell}  }, \partial_{\mathbf{v}_{\parallel_{\ell}}^{\ell}  }]$,
\begin{equation}\label{Dxv_free}
\left[\begin{array}{cc}
|\partial_{\mathbf{x}^{\ell}_{\parallel_{\ell}}   } \mathbf{X}_{ \ell   }(\tau)| & 
|\partial_{\mathbf{v}_{ {\ell}}^{\ell}  } \mathbf{X}_{ \ell}(\tau)| \\
|\partial_{\mathbf{x}^{\ell}_{\parallel_{\ell}}  } \mathbf{V}_{ \ell}(\tau)| & 
|\partial_{\mathbf{v}^{\ell}_{ {\ell}}  } \mathbf{V}_{ \ell}(\tau)| 
\end{array}\right]
\lesssim 
\left[\begin{array}{cc}
1 & |\tau- t^{\ell}|\\
(O_{\xi, \|  \nabla E \|_{L^\infty_{t,x}}}(1) + | {v} |^{2})|\tau- t^{\ell}| & 
1
\end{array}
\right].
\end{equation}
For $t^{\ell+1} < \tau < s < t^{\ell}$ then
\begin{equation}\notag
\begin{split}
&[\mathbf{X}_{\ell}(\tau;t,x,v), \mathbf{V}_{\ell}(\tau; t,x,v)  ] \\
\equiv & \ [ \mathbf{X}_{\ell}(\tau; s, \mathbf{X}_{\ell} (s;t,x,v),\mathbf{V}_{\ell} (s;t,x,v) ),  
\mathbf{V}_{\ell}(\tau; s, \mathbf{X}_{\ell} (s;t,x,v),\mathbf{V}_{\ell} (s;t,x,v) )  ]
,
\end{split}
\end{equation}
and 
\begin{equation}\label{Dxv_free_s}
\left[\begin{array}{cc}
|\partial_{ \mathbf{X}_{\ell}(s)  } \mathbf{X}_{ \ell   }(\tau)|  & |\partial_{ \mathbf{V}_{\ell}(s) } \mathbf{X}_{ \ell}(\tau)|\\
|\partial_{ \mathbf{X}_{\ell}(s)  }  \mathbf{V}_{ \ell}(\tau)|  & 
|\partial_ { \mathbf{V}_{\ell}(s) }  \mathbf{V}_{ \ell}(\tau)| 
\end{array} \right]\lesssim 
\left[\begin{array}{cc} 
1 & |\tau -s|\\
(O_{\xi, \|  \nabla E \|_{L^\infty_{t,x}}}(1) + | {v} |^{2}) |\tau-s| & 1
\end{array} \right].
\end{equation}
Moreover, for either $[\p_{\mathbf{X}}, \p_{\mathbf{V}}]
= [ \p_{\mathbf{x}^{\ell}_{\parallel^{\ell}}}, \p_{\mathbf{v}^{\ell}_{\perp^{\ell}}},\p_{\mathbf{v}^{\ell}_{\parallel^{\ell}}}]$ or $[\p_{\mathbf{X}}, \p_{\mathbf{V}}]
= [ \p_{\mathbf{X}_{\ell} (s)  }, \p_{\mathbf{V}_{\ell} (s)  } ]$
\begin{equation}\label{Dxv_F}
\left[\begin{array}{cc}
|\p_{\mathbf{X}} F(\tau)| & |\p_{\mathbf{V}} F(\tau)|  \\
| \frac{d}{d\tau }\p_{\mathbf{X}} F(\tau)| & | \frac{d}{d\tau }\p_{\mathbf{V}} F(\tau)|  
\end{array}\right]
\lesssim \left[\begin{array}{cc}
O_{\xi, \|  \nabla E \|_{L^\infty_{t,x}}}(1) + |v|^{2} &O_{\xi, \|  \nabla E \|_{L^\infty_{t,x}}}(1) + |v| \\
O_{\xi, \|  \nabla E \|_{L^\infty_{t,x}}}(1) +|v|^{3} &O_{\xi, \|  \nabla E \|_{L^\infty_{t,x}}}(1) + |v|^{2}
\end{array} \right].
\end{equation}
%
\end{lemma}
\begin{proof} From $\dot{v}=0$ and the second equation of (\ref{polar}) equals
\begin{equation}\label{dotv0}
\begin{split}
E(s, - \mathbf{x}_\perp \mathbf n (\mathbf x_\parallel ) + \mathbf \eta (\mathbf x_\parallel ) )  = & \dot{\mathbf{v}}_{\perp} (s) [-\mathbf{n}(\mathbf{x}_{\parallel}(s))] -2 \mathbf{v}_{\perp} (s)\mathbf{v}_{\parallel} \cdot \nabla \mathbf{n}(\mathbf{x}_{\parallel}(s))+ \dot{\mathbf{v}}_{\parallel} (s)\cdot \nabla \mathbf{\eta}(\mathbf{x}_{\parallel}(s))\\
& + \mathbf{v}_{\parallel} \cdot \nabla^{2} \mathbf{\eta}(\mathbf{x}_{\parallel} ) \cdot \mathbf{v}_{\parallel}  -\mathbf{x}_{\perp} \dot{\mathbf{v}}_{\parallel} \cdot \nabla \mathbf{n}(\mathbf{x}_{\parallel}) - \mathbf{x}_{\perp} \mathbf{v}_{\parallel} \cdot \nabla^{2} \mathbf{n}(\mathbf{x}_{\parallel}) \cdot \mathbf{v}_{\parallel}.
\end{split}
\end{equation}
We take the inner product with $\mathbf{n}(\mathbf{x}_{\parallel}(s))$ to the above equation to have
\begin{equation}\label{dot_v_perp}
\begin{split}
\dot{\mathbf{v}}_{\perp}(s) =& -E(s, - \mathbf{x}_\perp \mathbf n (\mathbf x_\parallel ) + \mathbf \eta (\mathbf x_\parallel ) ) \cdot \mathbf n (\mathbf x_\parallel ) + [ \mathbf{v}_{\parallel} \cdot \nabla^{2} \mathbf{ \eta}(\mathbf{x}_{\parallel}) \cdot \mathbf{v}_{\parallel}] \cdot \mathbf{n}(\mathbf{x}_{\parallel})
\\ &-\mathbf{x}_{\perp} [ \mathbf{v}_{\parallel} \cdot \nabla^{2} \mathbf{n}(\mathbf{x}_{\parallel})\cdot \mathbf{v}_{\parallel}] \cdot \mathbf{n}( \mathbf{x}_{\parallel})
\\  :=& F_{\perp} ( \mathbf{v}_{\perp}, \mathbf{v}_{\parallel}, \mathbf{x}_{\parallel} ),
\end{split}
 \end{equation}
where we have used the fact $\nabla \mathbf{n} \perp \mathbf{n}$ and $\nabla \mathbf{\eta} \perp \mathbf{n}$.

Since $0=\xi(\mathbf{\eta}(\mathbf{x}_{\parallel}))$ we take $\mathbf{x}_{\parallel,i}$ and $\mathbf{x}_{\parallel,j}$ derivatives to have
\[
0=\partial_{\mathbf{x}_{\parallel,j}} \big[  \sum_{k} \partial_{k}\xi \partial_{\mathbf{x}_{\parallel},i} \mathbf{\eta}_{k} \big]
 = \sum_{k,m} \partial_{k}\partial_{m} \xi \partial_{\mathbf{x}_{\parallel, j}} \mathbf{\eta}_{m} \partial_{\mathbf{x}_{\parallel,i}} \mathbf{\eta}_{k} + \sum_{k}
 \partial_{k} \xi \partial_{\mathbf{x}_{\parallel,i}}\partial_{\mathbf{x}_{\parallel,j}} \mathbf{\eta}_{k}   ,
\]
and from the convexity (\ref{convex}) and $\mathbf{n}= \nabla \xi/ |\nabla \xi|$,
\[
\big[\mathbf{v}_{\parallel} \cdot \nabla^{2} \mathbf{\eta} \cdot \mathbf{v}_{\parallel}\big] \cdot \mathbf{n} =
\sum_{i,j,k}\frac{ \mathbf{v}_{\parallel,{i}} \partial_{k}\xi \partial_{i  } \partial_{j} \mathbf{\eta}_{k} \mathbf{v}_{\parallel,j}  }{| \nabla \xi|} = - \sum_{i,j,k,m} \frac{  \{ \mathbf{v}_{\parallel,i} \partial_{i} \mathbf{\eta}_{m} \} \partial_{k} \partial_{m} \xi \{  \partial_{j} \mathbf{\eta}_{m} \mathbf{v}_{\parallel,j} \} }{|\nabla \xi|}\lesssim_{\xi} -|\mathbf{v}_{\parallel}|^{2}.
\]

Define $a_{ij}(\mathbf{x}_{\parallel})$ via
\begin{equation}\notag
\left[\begin{array}{cc} a_{11} & a_{12} \\ a_{21} & a_{22} \end{array} \right] = \left[\begin{array}{cc}  \partial_{1} \mathbf{n} \cdot \partial_{1} \mathbf{n} & \partial_{1 } \mathbf{n} \cdot \partial _{2} \mathbf{n} \\ \partial_{2}   \mathbf{n}  \cdot \partial_{1}   \mathbf{n}  & \partial_{2 }   \mathbf{n}  \cdot \partial_{2}   \mathbf{n}  \end{array}\right] \left[\begin{array}{cc}  \partial_{1} \mathbf{\eta} \cdot \partial_{1}   \mathbf{\eta}  & \partial_{1 }  \mathbf{\eta}  \cdot \partial _{2}   \mathbf{\eta}  \\ \partial_{2}   \mathbf{\eta}  \cdot \partial_{1}   \mathbf{\eta} & \partial_{2 }   \mathbf{\eta} \cdot \partial_{2}  \mathbf{\eta}  \end{array}\right]^{-1},
\end{equation}
where $\text{det} (\partial_{i} \mathbf{\eta} \cdot \partial_{j} \mathbf{\eta}) = |\partial_{1} \mathbf{\eta}\times \partial_{2} \mathbf{\eta}|^{2} \neq 0$ due to (\ref{nondegenerate_eta}). Then $\nabla \mathbf{n}$ is generated by $\nabla \mathbf{\eta}$ :
\[
-\partial_{i} \mathbf{n} (\mathbf{x}_{\parallel}) = \sum_{k} a_{ik}(\mathbf{x}_{\parallel}) \partial_{k} \mathbf{\eta}(\mathbf{x}_{\parallel}).
\]
We take the inner product (\ref{dotv0}) with $(-1)^{i+1} ( \mathbf{n}(\mathbf{x}_{\parallel}) \times \partial_{i} \mathbf{n}(\mathbf{x}_{\parallel}))$ to have
\begin{equation}\notag
\begin{split}
&\sum_{k} ( \delta_{ki} + \mathbf{x}_{\perp} a_{ki}) \dot{\mathbf{v}}_{\parallel,k} \\
&=\frac{(-1)^{i+1}}{-\mathbf{n}(\mathbf{x}_{\parallel}) \cdot (\partial_{1} \mathbf{\eta}(\mathbf{x}_{\parallel}) \times \partial_{2} \mathbf{\eta}(\mathbf{x}_{\parallel}))} \\
& \    \times \Big\{-
2 \mathbf{v}_{\perp} \mathbf{v}_{\parallel} \cdot \nabla \mathbf{n} (\mathbf{x}_{\parallel}) + \mathbf{v}_{\parallel} \cdot  \nabla^{2} \mathbf{\eta}(\mathbf{x}_{\parallel}) \cdot \mathbf{v}_{\parallel} - \mathbf{x}_{\perp} \mathbf{v}_{\parallel} \cdot \nabla^{2} \mathbf{n}(\mathbf{x}_{\parallel}) \cdot \mathbf{v}_{\parallel} 
\\ & \quad \quad - E(s, - \mathbf{x}_\perp \mathbf n (\mathbf x_\parallel ) + \mathbf \eta (\mathbf x_\parallel ) ) 
\Big\} \cdot ( -\mathbf{n} (\mathbf{x}_{\parallel}) \times \partial_{i+1} \mathbf{\eta}(\mathbf{x}_{\parallel})),
\end{split}
\end{equation}
where we used the notational convention for $\partial_{i+1} \mathbf{\eta}$, the index $i+1 \ \text{mod } 2$ . For $|\xi(x)| \ll 1$(and therefore $|\mathbf{x}_{\perp}|\ll 1$) the matrix $\delta_{ki} + \mathbf{x}_{\perp} a_{ki}$ is invertible: there exists the inverse matrix $G_{ij}$ such that $\sum_{i} (\delta_{ki} + \mathbf{x}_{\perp} a_{ki}(\mathbf{x}_{\parallel})) G_{ij} (\mathbf{x}_{\perp}, \mathbf{x}_{\parallel}) = \delta_{kj}.$ Therefore we have
\begin{equation}\label{dot_v_||}
\begin{split}
\dot{\mathbf{v}}_{\parallel,j} &= \sum_{i} G_{ij} (\mathbf{x}_{\perp}, \mathbf{x}_{\parallel})\frac{(-1)^{i+1}}{-\mathbf{n}(\mathbf{x}_{\parallel}) \cdot (\partial_{1} \mathbf{\eta}(\mathbf{x}_{\parallel}) \times \partial_{2} \mathbf{\eta}(\mathbf{x}_{\parallel}))}\\
& \ \ \ \times  \Big\{-
2 \mathbf{v}_{\perp} \mathbf{v}_{\parallel} \cdot \nabla \mathbf{n} (\mathbf{x}_{\parallel}) + \mathbf{v}_{\parallel} \cdot  \nabla^{2}\mathbf{ \eta}(\mathbf{x}_{\parallel}) \cdot \mathbf{v}_{\parallel} - \mathbf{x}_{\perp} \mathbf{v}_{\parallel} \cdot \nabla^{2} \mathbf{n}(\mathbf{x}_{\parallel}) \cdot \mathbf{v}_{\parallel} 
\\ & \quad \quad- E(s, - \mathbf{x}_\perp \mathbf n (\mathbf x_\parallel ) + \mathbf \eta (\mathbf x_\parallel ) ) 
\Big\}\\
& \ \ \ \ \  \cdot (-\mathbf{n} (\mathbf{x}_{\parallel}) \times \partial_{i+1} \mathbf{\eta}(\mathbf{x}_{\parallel}))\\
&:= F_{\parallel,j}(\mathbf{x}_{\perp}, \mathbf{x}_{\parallel}, \mathbf{v}_{\perp}, \mathbf{v}_{\parallel}).
\end{split}
\end{equation}
Here
\begin{equation}\label{G}
\begin{split}
&\left[\begin{array}{cc}G_{11} & G_{12} \\ G_{21} & G_{22} \end{array} \right] \\
& = \frac{1}{1 + \mathbf{x}_{\perp} (a_{11} + a_{22}) + (\mathbf{x}_{\perp})^{2} (a_{11} a_{22} - a_{12} a_{21})} \left[\begin{array}{cc}  1+ \mathbf{x}_{\perp} a_{22} & -\mathbf{x}_{\perp} a_{12} \\ - \mathbf{x}_{\perp} a_{21} & 1+ \mathbf{x}_{\perp} a_{11}\end{array} \right],
\\
& \left[\begin{array}{cc} a_{11} & a_{12} \\ a_{21} & a_{22} \end{array} \right] \\ &= \frac{1}{|\partial_{1} \mathbf{\eta}|^{2} |\partial_{2} \mathbf{\eta}|^{2} - (\partial_{1}\mathbf{ \eta} \cdot \partial_{2} \mathbf{\eta})^{2}}
 \\
& \ \ \times
\left[\begin{array}{cc}
|\partial_{1} \mathbf{n}|^{2} |\partial_{2} \mathbf{\eta}|^{2} -(\partial_{1}\mathbf{n} \cdot \partial_{2} \mathbf{n}) (\partial_{1} \mathbf{\eta} \cdot \partial_{2} \mathbf{\eta})  &- |\partial_{1}\mathbf{n}|^{2} (\partial_{1} \mathbf{\eta} \cdot \partial_{2} \mathbf{\eta}) + (\partial_{1} \mathbf{n} \cdot \partial_{2} \mathbf{n}) |\partial_{1}\mathbf{ \eta}|^{2} \\
(\partial_{1} \mathbf{n} \cdot \partial_{2} \mathbf{n}) |\partial_{2} \mathbf{\eta}|^{2} - |\partial_{2} \mathbf{n}|^{2} (\partial_{1} \mathbf{\eta} \cdot \partial_{2}\mathbf{ \eta}) & - (\partial_{1} \mathbf{n} \cdot \partial_{2} \mathbf{n}) (\partial_{1}\mathbf{ \eta}\cdot \partial_{2} \mathbf{\eta}) + |\partial_{2} \mathbf{n}|^{2}| \partial_{1} \mathbf{\eta}|^{2}
 \end{array} \right].
\end{split}
\end{equation}  
To complete the proof of (\ref{ODE_ell}), from $\dot{x}=v$ and $%
\dot{v}=E,$ we have 
\begin{eqnarray*}
v &=&-\mathbf{v}_{\perp } \mathbf{n} +\mathbf{v}_{||}\cdot \nabla \eta +\mathbf{x}_{\perp }[-\nabla n(\mathbf{x}_{|| })]%
\dot{\mathbf{x}}_{||} \\
&=&\dot{\mathbf{x}}_{\perp }(-\mathbf{n}(\mathbf{x}_{||}))+\mathbf{x}_{\perp }[-\nabla \mathbf{n}(\mathbf{x}_{||})]\dot{\mathbf{x}}%
_{||}+\nabla \eta \dot{\mathbf{x}}_{||} \\
E(s, - \mathbf{x}_\perp \mathbf n (\mathbf x_\parallel ) + \mathbf \eta (\mathbf x_\parallel ) )  &=&\dot{\mathbf{v}}_{\perp }(-\mathbf{n}(\mathbf{x}_{||}))-\mathbf{v}_{\perp }\nabla \mathbf{n}\dot{\mathbf{x}}_{||}+\dot{\mathbf{v}}%
_{||}\nabla \eta +\mathbf{v}_{||}\nabla ^{2}\eta \dot{\mathbf{x}}_{||} \\
&&+\dot{\mathbf{x}}_{\perp }\mathbf{v}_{||}[-\nabla \mathbf{n}(\mathbf{x}_{||})]+\mathbf{x}_{\perp }\dot{\mathbf{v}}_{||}[-\nabla
\mathbf{n}(\mathbf{x}_{||})]+\mathbf{x}_{\perp }\mathbf{v}_{||}[-\nabla ^{2}n]\dot{\mathbf{x}}_{||}.
\end{eqnarray*}%
We therefore conclude that $\dot{\mathbf{x}}_{\perp }=\mathbf{v}_{\perp },$ and $\dot{\mathbf{x}}%
_{||}=\mathbf{v}_{||}$ from $\Phi _{\mathbf{p}}^{-1}.$ We then solve $\dot{\mathbf{v}}_{\perp }$ and $%
\dot{\mathbf{v}}_{||}$ to obtain (\ref{ODE_ell}).

Now we prove (\ref{Dxv_free}) and (\ref{Dxv_free_s}). From (\ref{F_perp}) and (\ref{F||}), $\dot{\mathbf{x}}_{\parallel_{ \ell}} = \mathbf{v}_{\parallel_{ \ell}}, \ \dot{\mathbf{x}}_{\perp_{ \ell}} = \mathbf{v}_{\perp_{ \ell}}$ and $\dot{\mathbf{v}}_{\perp_{ \ell}} =  {F}_{\perp_{ \ell}}$ and $\dot{\mathbf{v}}_{\parallel_{ \ell}} =  {F}_{\parallel_{ \ell}}$. Denote $\partial = [  \frac{\partial }{ \partial {\mathbf{x}_{\parallel_{ \ell}}^{\ell} } },  \frac{\partial }{\partial { \mathbf{v}_{\perp_{ \ell}}^{\ell} } },  \frac{\partial }{ \partial { \mathbf{v}_{\parallel_{ \ell}}^{\ell} }} ].$ From (\ref{F_perp}) and (\ref{F||}),
\begin{equation}\label{D_F}
\left[\begin{array}{c}
|\partial F_{\perp}| \\
|\partial F_{\parallel}| 
\end{array}\right]
\lesssim 
\left[\begin{array}{c}
(O_{\xi, \|  \nabla E \|_{L^\infty_{t,x}}}(1) + | V(\tau) |^{2}) \{ |\partial \mathbf{x}_{\perp}|  +  | \partial\mathbf{x}_{\parallel}| \}  +  | V(\tau)| |\partial \mathbf{v}_{\parallel}|\\
  (O_{\xi, \|  \nabla E \|_{L^\infty_{t,x}}}(1) + | { V(\tau)} |^{2}) \{  |\partial \mathbf{x}_{\perp}| + |\partial \mathbf{x}_{\parallel}| \}   
+ | V(\tau)| \{ |\partial \mathbf{v}_{\perp}| + |\partial \mathbf{v}_{\parallel}|\}
\end{array}\right].
\end{equation}
%
Now we use a single (rough) bound of $|\partial F_{\perp}|  + |\partial F_{\parallel}| \lesssim  (O_{\xi, \|  \nabla E \|_{L^\infty_{t,x}}}(1) + | { V(\tau)} |^{2}) \{  |\partial \mathbf{x}_{\perp}| + |\partial \mathbf{x}_{\parallel}| \}   
+ | V(\tau)| \{ |\partial \mathbf{v}_{\perp}| + |\partial \mathbf{v}_{\parallel}|\}$ to have
\begin{equation}\notag
\begin{split}
& \frac{d}{d\tau} \{  |\partial  \mathbf{v}_{\perp_{ \ell   } }(\tau)|  +   |\partial  \mathbf{v}_{\parallel_{ \ell   } }(\tau)|   \} 
\\ &\lesssim \  |\partial F_{\perp_{\ell}}(\tau)|   +  |\partial F_{\parallel_{\ell}}(\tau)|\\
&\lesssim \ (O_{\xi, \|  \nabla E \|_{L^\infty_{t,x}}}(1) + | { V(\tau)} |^{2}) \big\{     |\partial \mathbf{x}_{\perp_{\ell   }}(\tau)|   + |\partial \mathbf{x}_{\parallel_{\ell   }}(\tau)| \big\}+ | V(\tau)|\big\{    |\partial \mathbf{v}_{\perp_{\ell   }}(\tau)| + |\partial \mathbf{v}_{\parallel _{\ell   }}(\tau)|  \big\}.
\end{split}
\end{equation}
Combining with $ \frac{d}{d\tau}[\mathbf{x}_{\perp_{\ell}}(\tau), \mathbf{x}_{\parallel_{\ell}}(\tau) ]= [\mathbf{v}_{\perp_{\ell}}(\tau), \mathbf{v}_{\parallel_{\ell}}(\tau) ]$ yields  
\begin{equation} \label{Dxvtdbd}
\begin{split}
&\frac{d}{d\tau}\left[\begin{array}{ccccc}|\partial \mathbf{x}_{\perp_{\ell   }}(\tau)| + |\partial \mathbf{x}_{\parallel_{\ell   }}(\tau)| \\ |\partial \mathbf{v}_{\perp_{\ell   }}(\tau)| + |\partial \mathbf{v}_{\parallel_{\ell   }}(\tau)|   \end{array}\right]  
 \\ & \lesssim_{\xi}  \ \left[\begin{array}{cc} 0 & 1 \\ (O_{\xi, \|  \nabla E \|_{L^\infty_{t,x}}}(1) + | { V(\tau)} |^{2}) & | V(\tau)|\end{array}\right]  \left[\begin{array}{ccccc}|\partial \mathbf{x}_{\perp_{\ell   }}(\tau)| + |\partial \mathbf{x}_{\parallel_{\ell   }}(\tau)| \\ |\partial \mathbf{v}_{\perp_{\ell   }}(\tau)| + |\partial \mathbf{v}_{\parallel_{\ell   }}(\tau)|   \end{array}\right].
 \end{split}
\end{equation}

Now for $M \gg 1 $, lets first prove \eqref{Dxv_free} for $|v| < M$. From \eqref{Dxvtdbd} we have
\[ \begin{split}
|\partial \mathbf{X}_{ \ell   }(\tau)| + |\partial \mathbf{V}_{ \ell   }(\tau)| \lesssim & 1 + \int_\tau^{t^\ell} \left(1 + O_{\xi, \|  \nabla E \|_{L^\infty_{t,x}}}(1) + |V(\tau' ) |  + |V(\tau' ) |^2 \right) |\partial \mathbf{X}_{ \ell   }(\tau')| + |\partial \mathbf{V}_{ \ell   }(\tau')| d\tau'
\\ \lesssim & 1 + \int_\tau^{t^\ell} \left( 1 + O_{\xi, \|  \nabla E \|_{L^\infty_{t,x}}}(1)  + M^2 \right) |\partial \mathbf{X}_{ \ell   }(\tau')| + |\partial \mathbf{V}_{ \ell   }(\tau')| d\tau'.
\end{split} \]
From Gronwall we have
\Be \label{pxvrbd}
|\partial \mathbf{X}_{ \ell   }(\tau)| + |\partial \mathbf{V}_{ \ell   }(\tau)| \lesssim_{\xi, \|  \nabla E \|_{L^\infty_{t,x}}, M } 1.
\Ee
For $ \partial_{\mathbf{v}} = [ \frac{\partial }{\partial { \mathbf{v}_{\perp_{ \ell}}^{\ell} } },  \frac{\partial }{ \partial { \mathbf{v}_{\parallel_{ \ell}}^{\ell} }} ]$, from \eqref{pxvrbd} we have 
\Be \label{pvxrbd}
|\partial_{\mathbf{v}} \mathbf{X}_{ \ell   }(\tau)| \le \int_{\tau}^{t^{\ell }} |\partial_{\mathbf{v}}   \mathbf{V}_{ \ell   }(\tau') |  d\tau' \lesssim_{\xi, \|  \nabla E \|_{L^\infty_{t,x}}, M} | \tau - t^{\ell} |.
\Ee
And for $\partial_{\mathbf{x}} =  \frac{\partial }{ \partial {\mathbf{x}_{\parallel_{ \ell}}^{\ell} } }$, from \eqref{Dxvtdbd}, \eqref{pxvrbd} we have
\[
\begin{split}
|\partial_{\mathbf{x}}  \mathbf{V}_{ \ell   }(\tau)| \le & \int_\tau^{t^{\ell } } \left(O_{\xi, \|  \nabla E \|_{L^\infty_{t,x}}}(1)  + |V(\tau' ) |^2 \right) |\partial_{\mathbf{x}}  \mathbf{X}_{ \ell   }(\tau')| + |V(\tau' ) | |\partial_{\mathbf{x}}  \mathbf{V}_{ \ell   }(\tau')|  \, d\tau'
\\ \lesssim & \left( O_{\xi, \|  \nabla E \|_{L^\infty_{t,x}}}(1)  + |v |^2 \right) |\tau - t^{\ell } | + M \int_\tau^{t^\ell}   |\partial_{\mathbf{x}}  \mathbf{V}_{ \ell   }(\tau')| \, d \tau'.
\end{split}
\]
From Gronwall we have
\Be \label{pxofvbd}
|\partial_{\mathbf{x}}  \mathbf{V}_{ \ell   }(\tau)| \lesssim_{\xi, \|  \nabla E \|_{L^\infty_{t,x}}, M } \left( O_{\xi, \|  \nabla E \|_{L^\infty_{t,x}}}(1)  + |v |^2 \right) | \tau - t^\ell |.
\Ee
Combining \eqref{pxvrbd}, \eqref{pvxrbd}, and \eqref{pxofvbd} we prove \eqref{Dxv_free} for $|v| <  M$.

%
For the case $|v| \ge M \gg 1$, we have $|V(\tau) | < 2 |v|$, so
\[
\begin{split}
&\frac{d}{d\tau}\left[\begin{array}{ccccc}|\partial \mathbf{x}_{\perp_{\ell   }}(\tau)| + |\partial \mathbf{x}_{\parallel_{\ell   }}(\tau)| \\ |\partial \mathbf{v}_{\perp_{\ell   }}(\tau)| + |\partial \mathbf{v}_{\parallel_{\ell   }}(\tau)|   \end{array}\right]  \ \lesssim_{\xi}  \ \left[\begin{array}{cc} 0 & 1 \\ (O_{\xi, \|  \nabla E \|_{L^\infty_{t,x}}}(1) + | v |^{2}) & | v|\end{array}\right]  \left[\begin{array}{ccccc}|\partial \mathbf{x}_{\perp_{\ell   }}(\tau)| + |\partial \mathbf{x}_{\parallel_{\ell   }}(\tau)| \\ |\partial \mathbf{v}_{\perp_{\ell   }}(\tau)| + |\partial \mathbf{v}_{\parallel_{\ell   }}(\tau)|   \end{array}\right].
 \end{split}
\]
By Lemma \ref{mtx_Grwn} we prove our claim (\ref{Dxv_free}) for the case $|v| \ge M$. The proof of (\ref{Dxv_free_s}) is exactly same but we use $\partial = [ \partial_{\mathbf{X}_{\ell}}(s), \partial_{\mathbf{V}_{\ell}}(s) ]$ to conclude the proof.

We prove the first row of (\ref{Dxv_F}) by (\ref{D_F}). By taking the time derivative to (\ref{F_perp}), (\ref{F||}) and applying (\ref{ODE_ell}) we prove the second row of row of (\ref{Dxv_F}).

\end{proof}

\section{Derivative estimate for the generalized characteristics}
The main goal of this section is to prove the following key estimate for the derivatives of the generalized characteristics $(X_{\mathbf{cl}}(s;t,x,v), V_{\mathbf{cl}}(s;t,x,v))$ defined in \eqref{cycles}.
 \begin{theorem}\label{theorem_Dxv}
There exists $C=C(\Omega, E)>0$ such that for all $(t,x,v)\in [0,T] \times \bar{\Omega}\times \mathbb{R}^{3}$, $0 \le s \le t$, with $s\neq t^{\ell}$ for $\ell = 1,2,\cdots, \ell_{*}$ 
\begin{equation}\label{lemma_Dxv}
\begin{split}
|\nabla _{x}X_{\mathbf{cl}}(s;t,x,v)| & \ \lesssim \ e^{C|v|(t-s)}\frac{|v|+1}{%
\alpha (t,x,v)} , \\
|\nabla _{v}X_{\mathbf{cl}}(s;t,x,v)| &  \ \lesssim  \ e^{C|v|(t-s)}\frac{1}{|v|+1}%
, \\
|\nabla_{x}V_{\mathbf{cl}}(s;t,x,v)| & \ \lesssim \ e^{C|v|(t-s)}\frac{|v|^{3} +1%
}{\alpha^2 (t,x,v)} , \\
|\nabla_{v}V_{\mathbf{cl}}(s;t,x,v)| &  \ \lesssim \ e^{C|v|(t-s)}\frac{|v|+1}{%
\alpha (t,x,v)} .
\end{split}
\end{equation}
\end{theorem}
In order to achieve this, we need a crucial bound on the backward exit time:
\begin{lemma}
Suppose $E(t,x) \cdot n(x) > c_E$ for all $x \in \p \Omega$, then there exists $ C = C(\Omega, E ) \gg 1 $ and $0 < T \ll 1 $ such that for any $(t,x,v) \in [0, T ] \times \overline \Omega \times \mathbb R^3$, $t^1(t,x,v)  > 0$,
\Be \label{tbest}
\frac{|t -t^1 | }{|\mathbf{v}_\perp^1 |} + \frac{|t -t^1 | |v|}{|\mathbf{v}_\perp^1 |} + \frac{|t -t^1 | |v|^2}{|\mathbf{v}_\perp^1 |} < C.
\Ee
And for $(t,x,v) \in [0, T ] \times \gamma_+ \times \mathbb R^3$, $t^1(t,x,v) < 0 $,
\Be \label{tbest2}
\frac{|t | }{|\mathbf{v}_\perp |} + \frac{|t | |v|}{|\mathbf{v}_\perp |} + \frac{|t  | |v|^2}{|\mathbf{v}_\perp |} < C.
\Ee
\end{lemma}
\begin{proof}
Let $ N > 10 ( \| E \|_{L^\infty_{t,x}}  +1 ) $ be fixed. Let's first consider the case  $(t,x,v) \in [0, T ] \times \overline \Omega \times \mathbb R^3$, $t^1(t,x,v)  > 0$, and prove
\Be \label{tbsimvperp} 
\frac{|t - t^1|}{ |\mathbf{v}_\perp^1|} \lesssim 1 \text{ for all } |v| < N.
\Ee
From \eqref{F_perp} we have
\[
F_{\perp} (s) < - c_\xi | \mathbf v_\parallel |^2 - c_E + C_\xi \mathbf{x}_{\perp } | \mathbf v_\parallel |^2.
\]
By choosing $T < \frac{c_\xi}{4 N C_\xi } $, we have $\mathbf x_\perp < 2N T < \frac{ c_\xi }{2 C_\xi} $, thus
\Be \label{Fperpneg}
F_{\perp} (s) < -c_E  - c_\xi   | \mathbf v_\parallel |^2 + \frac{ c_\xi}{2}  | \mathbf v_\parallel |^2 < - c_E  , \, \text{ for all } t^1 < s < t.
\Ee
Therefore
\Be \label{xperpexpan} \begin{split}
0 < \mathbf{x}_\perp (t )  &= \int_{t^1}^t \mathbf{v}_\perp(s) ds  
\\ &=  \int_{t^1}^t  \left( - \mathbf{v}_\perp^1 + \int_{t^1}^s F_\perp (\tau ) d\tau \right) ds 
\\ &= (t - t^1 )(- \mathbf{v}_\perp^1 ) + \int_{t^1}^t \int_{t^1}^t F_\perp(\tau) d \tau d s.
\end{split} \Ee

So from \eqref{Fperpneg} and \eqref{xperpexpan},
\Be \label{FperpEbdd}
\frac{c_E}{2}(t -t^1)^2 < -  \int_{t^1}^t \int_{t^1}^t F_\perp(\tau) d \tau d s < |t - t^1| |\mathbf{v}_\perp^1|.
\Ee
Therefore $\frac{c_E}{2}(t -t^1) < |\mathbf{v}_\perp^1|$, and this proves \eqref{tbsimvperp}.

Next, for $|v| \ge N$, let $d = \max_{ x,y \in \overline \Omega } |x- y|$, then $\xi (X(t + t' )) = 0 $ for some $t' < \frac{2d}{N}$ by extending the field as $E(s,x) = E(T,x) $ for $s>T$ if necessary. So we can without loss of generality assume $x \in \p \Omega$. We claim
\Be \label{tbv2est}
\frac{|t -t^1 | |v|^2}{|\mathbf{v}_\perp^1 |} \lesssim 1 \text{ for all } |v| \ge N.
\Ee
Since $(x,v) \in \gamma_+$ we have
\Be \label{xiexp1}
\begin{split}
 0 = \xi(x^1) =& \xi(x) - \int_{t^1}^t  \nabla \xi (X(s)) \cdot V(s) ds
 \\ = & - (t - t^1) v \cdot \nabla \xi (x) + \int_{t^1}^t \int_s^t \left( V(\tau) \cdot \nabla^2 \xi(X(\tau)) \cdot V(\tau) + E(\tau, X(\tau))\cdot \nabla \xi (X(\tau)) \right) d\tau ds.
\end{split}
\Ee
Note that for $T < \frac{N}{4\| E \|_{L^\infty_{t,x}}}$, $\frac{|v|}{2} < |V(\tau) | < 2 |v| $ for all $\tau \in [t^1,t]$. Thus from \eqref{xiexp1}
\Be \label{xiexp2} \begin{split}
 &  |t - t^1| ( v \cdot \nabla \xi (x) )
\\ \ge &  \frac{C}{8}|t -t^1|^2 |v|^2 + \int_{t^1}^t \int_s^t E(\tau, X(\tau))\cdot \nabla \xi (X(\tau))  d\tau ds
\\ \ge & \frac{C}{8}|t -t^1|^2 |v|^2 + \frac{|t-t^1|^2}{2} E(t,x) \cdot \nabla \xi (x)  - \int_{t^1}^t \int_s^t \int_\tau^t \frac{d}{d\tau' }\left( E(\tau', X(\tau')) \cdot \nabla \xi (X(\tau')) \right) d\tau' d\tau ds
\\ \ge & \frac{C}{8}|t -t^1|^2 |v|^2 - |t - t^1|^3 C_{E,\xi}(1 + |v| )
\\ \ge & |t -t^1|^2 \left(  \frac{C}{8}|v|^2 - |t-t^1| C_{E,\xi}(1 + |v| ) \right).
\end{split} \Ee
Since $|v| \ge N$, we have $ \frac{C}{8}|v|^2 - |t-t^1| C_{E,\xi}(1 + |v| ) > \frac{C}{20} |v|^2 $. Therefore \eqref{xiexp2} gives
\Be \label{vdotnxiest}
| v \cdot \nabla \xi (x) | > \frac{C}{20} |t-t^1 | |v|^2.
\Ee
Then using the velocity lemma we have $ |t -t^1||v|^2 \lesssim |v \cdot \nabla \xi (x) | \lesssim |\mathbf{v}_\perp^1 | $, and we conclude \eqref{tbv2est}.

Now combining \eqref{tbsimvperp} and \eqref{tbv2est} we actually have for all $(x,v) \in \gamma_+$, 
\[
\frac{|t - t^1 |}{|\mathbf{v}_\perp^1 |}+\frac{|t -t^1 | |v|^2}{|\mathbf{v}_\perp^1 |} \lesssim 1.
\]
Therefore
\[
\frac{|t -t^1 | |v|}{|\mathbf{v}_\perp^1 |} \le  \max \{ \frac{|t - t^1 |}{|\mathbf{v}_\perp^1 |} , \frac{|t -t^1 | |v|^2}{|\mathbf{v}_\perp^1 |} \} \lesssim 1,
\]
and we conclude \eqref{tbest}.

The proof of \eqref{tbest2} is similar. If $|v| <N$, we have
\Be
0 < \mathbf x_\perp(0) = - \int_{0}^t \mathbf v_\perp(s) ds = -   \int_0^t \left( \mathbf v_\perp - \int_s^t F_\perp(\tau) d\tau \right) ds   = - t \mathbf v_\perp + \int_0^t \int_s^t F_\perp(\tau) d\tau ds,
\Ee
So same as \eqref{FperpEbdd} we have
\[
\frac{c_E}{2}t ^2 < -  \int_{0}^t \int_{s}^t F_\perp(\tau) d \tau d s < t |\mathbf{v}_\perp|.
\]
Therefore $\frac{c_E}{2} t < |\mathbf{v}_\perp|$. And if $|v| >N$, similarly we get
\Be \label{xiexp12}
\begin{split}
 0 > & \xi(X(0)) 
 \\ =& \xi(x) - \int_{0}^t  \nabla \xi (X(s)) \cdot V(s) ds
 \\ = & -| t| ( v \cdot \nabla \xi (x)) + \int_{0}^t \int_s^t \left( V(\tau) \cdot \nabla^2 \xi(X(\tau)) \cdot V(\tau) + E(\tau, X(\tau))\cdot \nabla \xi (X(\tau)) \right) d\tau ds.
\end{split}
\Ee
Then by the same argument as lines between \eqref{xiexp1} and \eqref{vdotnxiest} we get $| v \cdot \nabla \xi (x) | > \frac{C}{20} t |v|^2$, and this proves \eqref{tbest2}.
\end{proof}

We need a version of Gronwall's inequality for matrices:
\begin{lemma}\label{mtx_Grwn}
Let $m > 0$, $a(\tau),b(\tau),f(\tau),g(\tau) \geq 0$ for all $0 \leq \tau \leq t$, and satisfy $|v| > M \gg 1 $, and 
\begin{equation}\notag
\begin{split}
\left[\begin{array}{c} a( {\tau})  \\ b( {\tau})   \end{array} \right] \lesssim 
\left[\begin{array}{cc} 0 & 1 \\ m + |v|^{2} & |v| \end{array}\right]
 \left[\begin{array}{c} \int^{t}_{\tau } a(\tau^{\prime}) \mathrm{d}\tau^{\prime} \\ \int^{t}_{\tau} b(\tau^{\prime}) \mathrm{d}\tau^{\prime} \end{array}\right] + 
 \left[\begin{array}{c} g(t-\tau) \\  h(t-\tau) \end{array} \right]
\end{split}
\end{equation}
then
\begin{equation}\label{matrix_gronwall}
\begin{split}
\left[\begin{array}{c} a(\tau) \\ b(\tau) \end{array}\right]  \lesssim  & \ 
 e^{C  (\tau-t)} 
 \left[\begin{array}{cc}
 1 & |\tau-t|\\
 |v|^{2}|\tau-t|& 1
 \end{array} \right]
 \left[\begin{array}{c}
 g(0)   \\
h(0) 
  \end{array} \right]\\
  &
  + \int^{\tau}_{t}  e^{C  (\tau-\tau^{\prime})}
   \left[\begin{array}{cc}
 1 &|\tau-\tau^{\prime}|\\
 |v|^{2}|\tau-\tau^{\prime}|& 1
 \end{array} \right]
   \left[\begin{array}{c}
 |g'(t-\tau^{\prime}) |  \\
|h'(t-\tau^{\prime})  |
  \end{array} \right]
  \dd \tau^{\prime}.
\end{split}
\end{equation}
\end{lemma}
\begin{proof}
 First we consider $A^{\e}, B^{\e}$ solving, for $\e>0$,
\begin{equation}\label{sup_eqnt}
\left[\begin{array}{c} A^{\e}( {\tau})  \\ B^{\e}( {\tau})   \end{array} \right] = 
C 
\left[\begin{array}{cc} 0 & 1 \\m+ |v|^{2} &  |v| \end{array}\right]
 \left[\begin{array}{c} \int^{t}_{\tau } A^{\e}(\tau^{\prime}) \mathrm{d}\tau^{\prime} \\ \int^{t}_{\tau}B^{\e} (\tau^{\prime}) \mathrm{d}\tau^{\prime} \end{array}\right] +  \left[\begin{array}{c} g(t-\tau) \\  h(t-\tau) \end{array} \right] +  \left[\begin{array}{c} \e \\ \e \end{array}\right]
 .
\end{equation}
We claim that 
\begin{equation}\label{super_bound}
\begin{split}
\left[\begin{array}{c}
A^{\e} (\tau) \\
B^{\e} (\tau)
\end{array} \right]\lesssim & \ 
 e^{C (\tau-t)} 
 \left[\begin{array}{cc}
 1 & |\tau-t|\\
 |v|^{2}|\tau-t| & 1
 \end{array} \right]
 \left[\begin{array}{c}
 g(0) + \e \\
h(0) + \e
  \end{array} \right]\\
  &
  + \int^{\tau}_{t}  e^{C (\tau-\tau^{\prime})}
   \left[\begin{array}{cc}
 1 & |\tau-\tau^{\prime}|\\
 |v|^{2}|\tau-\tau^{\prime}| & 1
 \end{array} \right]
   \left[\begin{array}{c}
 |g'(t-\tau^{\prime})| \\
|h'(t-\tau^{\prime}) |
  \end{array} \right]
  \dd \tau^{\prime}.\end{split}
\end{equation}

We consider the matrix
$\left[\begin{array}{cc} 1 & 0 \\ 0 & 1 \end{array} \right]\left[\begin{array}{cc} 0 & 1 \\ m+|v|^{2} & |v|\end{array} \right]  = \left[\begin{array}{cc} 0 & 1 \\ m+ |v|^{2} &   |v|  \end{array} \right]$. Denote 
\begin{eqnarray*}
r_{1} : = \frac{ 1  + \sqrt{5 + \frac{4m}{|v|^2}}}{2} , \ r_{2} : = \frac{1 - \sqrt{5 + \frac{4m}{|v|^2}}}{2}, \ r_{3} : = \frac{1}{\sqrt{5 + \frac{4m}{|v|^2}}  }.
\end{eqnarray*} Then we diagonalize this matrix as 
\begin{equation*} 
 \left[\begin{array}{cc} 0 & 1 \\  m + |v|^{2} &  |v|  \end{array} \right] 
 =  \left[\begin{array}{cc} 1 & 1 \\ r_{1}|v| &r_{2}|v| \end{array} \right]
 \left[\begin{array}{cc}
 r_{1}|v| & 0 \\ 0 &  r_{2}|v| \end{array}\right] 
 \left[\begin{array}{cc}
- r_{2} r_{3}  &r_{3} \frac{1}{|v|} \\
r_{1} r_{3}
&- r_{3} \frac{1}{|v|}
  \end{array} \right]. 
\end{equation*}
Denote $ \left[\begin{array}{c}  { {\mathcal{A}^{\e}  }}( {\tau}) \\  {\mathcal{B}^{\e}}( {\tau}) \end{array} \right]:= 
   \left[\begin{array}{cc}
- r_{2} r_{3}  &r_{3} \frac{1}{|v|} \\
r_{1} r_{3}
&- r_{3} \frac{1}{|v|}
  \end{array} \right]
  \left[\begin{array}{c} { {A}}^{\e}( {\tau}) \\  { {B}}^{\e}( {\tau})\end{array} \right] $ and rewrite the equations as 
\begin{equation*}
\begin{split}
\frac{d}{d {\tau}}
\left[\begin{array}{c}  { {\mathcal{A}^{\e}}}( {\tau}) \\  {\mathcal{B}^{\e}} ({\tau}) \end{array} \right]
  & =C
  \left[\begin{array}{cc}
  r_{1} |v| & 0 \\
  0 & r_{2} |v|
  \end{array}
  \right] 
 \left[\begin{array}{c}  { {\mathcal{A}^{\e}}}( {\tau}) \\  {\mathcal{B}^{\e}}( {\tau}) \end{array} \right]  
 +  \left[\begin{array}{cc}
- r_{2} r_{3}  &r_{3} \frac{1}{|v|} \\
r_{1} r_{3}
&- r_{3} \frac{1}{|v|}
  \end{array} \right]  \left[\begin{array}{c} {g'}(t-\tau)  \\     {h'}(t-\tau)   \end{array} \right].
\end{split}
\end{equation*}
Directly we compute
\begin{equation}\notag
\begin{split}
 \left[\begin{array}{c}  { {\mathcal{A}^{\e}}}( {\tau})
  \\  {\mathcal{B}^{\e}}( {\tau}) \end{array} \right] &
   =  \left[\begin{array}{c}
 e^{C   r_{1}|v| { (\tau-t)}}
  { {\mathcal{A}^{\e}}}(t) \\   
  e^{C_{\xi }  r_{2}|v|  {  (\tau-t)} }  {\mathcal{B}^{\e}}(t) \end{array} \right]\\
   & \   \ \  +  \int ^{\tau}_{t}   \left[\begin{array}{cc}
 e^{ C r_{2}|v| (\tau-\tau^{\prime})}& 0 \\ 0 &  e^{Cr_{2}|v| (\tau-\tau^{\prime})}\end{array}\right]    \left[\begin{array}{cc}
- r_{2} r_{3}  &r_{3} \frac{1}{|v|} \\
r_{1} r_{3}
&- r_{3} \frac{1}{|v|}
  \end{array} \right]
  \left[\begin{array}{c}  {g'}(t-\tau^{\prime})  \\     {h'}(t-\tau^{\prime})  \end{array} \right] \mathrm{d}\tau^{\prime}.
  \end{split}
  \end{equation}
  Then 
  \begin{eqnarray*}
  \left[\begin{array}{c}
  A^{\e} (\tau) \\
  B^{\e} (\tau)
  \end{array} \right] &=&  
  \left[\begin{array}{cc}
  1 & 1\\
r_{1} |v| & 
r_{2} |v|  
\end{array}  \right]
    \left[\begin{array}{c}
  \mathcal{A}^{\e} (\tau) \\
  \mathcal{B}^{\e} (\tau)
  \end{array} \right] \\
  &=&    \left[\begin{array}{cc}
  1 & 1\\
r_{1} |v| & 
r_{2} |v|  
\end{array}  \right] 
\left[\begin{array}{cc}
e^{C r_{1} |v|( \tau-t) } & 0 \\
0 & e^{C r_{2} |v| (\tau-t) } 
\end{array} \right]   \left[\begin{array}{cc}
- r_{2} r_{3}  &r_{3} \frac{1}{|v|} \\
r_{1} r_{3}
&- r_{3} \frac{1}{|v|}
  \end{array} \right]
  \left[\begin{array}{c} { {A}}^{\e}( {t}) \\  { {B}}^{\e}( {t})\end{array} \right] \\
  &&+ 
  \int^{\tau}_{t}  \left[\begin{array}{cc}
  1 & 1\\
r_{1} |v| & 
r_{2} |v|  
\end{array}  \right] 
\left[\begin{array}{cc}
e^{C r_{1} |v| (\tau - \tau^{\prime} )} & 0 \\
0 & e^{C r_{2} |v| (\tau - \tau^{\prime} )} 
\end{array} \right] \\
&& \ \ \ \ \ \ \ \ \ \ \ \ \ \  \ \ \  \ \ \  \ \ \ \ \ \  \times   \left[\begin{array}{cc}
- r_{2} r_{3}  &r_{3} \frac{1}{|v|} \\
r_{1} r_{3}
&- r_{3} \frac{1}{|v|}
  \end{array} \right]  \left[\begin{array}{c}  {g'}(t-\tau^{\prime})  \\     {h'}(t-\tau^{\prime})  \end{array} \right] \mathrm{d}\tau^{\prime} .
    \end{eqnarray*}
  Directly, the RHS equals
    \begin{eqnarray*}
  & & 
  \left[\begin{array}{cc}
  r_{3} \big( r_{1} e^{C r_{2} |v| (\tau-t)} 
  - r_{2} e^{C r_{1} |v| (\tau-t)} 
  \big)
  & \frac{r_{3}}{|v|} \big(  e^{C r_{1} |v|  (\tau-t)} -e^{C r_{2} |v|  (\tau-t)}   \big)\\
-  r_{1} r_{2} r_{3} |v| \big(  e^{C r_{1} |v| (\tau-t)} -e^{C r_{2} |v|  (\tau-t)}   \big) & r_{3 } \big( r_{1} e^{C r_{2} |v| (\tau-t)} 
  - r_{2} e^{C r_{1} |v| (\tau-t)} 
  \big)
\end{array}  \right] \left[\begin{array}{c}
  A^{\e} (t) \\
  B^{\e} (t)
  \end{array} \right]\\
  &&+  \int^{\tau}_{t} 
   \left[\begin{array}{cc}
  r_{3} \big( r_{1} e^{C r_{2} |v| (\tau-\tau^{\prime})} 
  - r_{2} e^{C r_{1} |v| (\tau- \tau^{\prime})} 
  \big)
  & \frac{r_{3}}{|v|} \big(  e^{C r_{1} |v|  (\tau- \tau^{\prime})} -e^{C r_{2} |v|  (\tau-\tau^{\prime})}   \big)\\
-  r_{1} r_{2} r_{3} |v| \big(  e^{C r_{1} |v| (\tau-\tau^{\prime})} -e^{C r_{2} |v|  (\tau-\tau^{\prime})}   \big) & r_{3 } \big( r_{1} e^{C r_{2} |v| (\tau-\tau^{\prime})} 
  - r_{2} e^{C r_{1} |v| (\tau-\tau^{\prime})} 
  \big)
\end{array}  \right] \\
&&   \ \ \ \ \ \ \ \ \times  \left[\begin{array}{c}  {g'}(t-\tau^{\prime})  \\     {h'}(t-\tau^{\prime})  \end{array} \right] \mathrm{d}\tau^{\prime}.
\end{eqnarray*}
Since $|v| > M$, we have $|r_1 - r_2 | \lesssim 1$, so by expansion we have $|e^{C r_{1} |v|  (\tau-t)} -e^{C r_{2} |v|  (\tau-t)}| \lesssim_{C^{\xi}, \delta } |v||\tau-t|  e^{C_{\xi, \delta}|v|(\tau-t)}$. Therefore we conclude (\ref{super_bound}).

  
  Now we claim 
  \begin{equation}\label{aAbB}
  a(\tau) \leq  A (\tau), \ \ \ b(\tau) \leq 
B  (\tau), \ \ \ \text{for all} \ \ \tau \leq t.
  \end{equation}
First we claim that $a(\tau) \leq A^{\e}(\tau)$ and $b(\tau) \leq B^{\e}(\tau)$ for all $\tau$. Otherwise, we should have at least for some time $\tau_{0}$ such that $a(\tau) \leq A^{\e}(\tau)$ and $b(\tau) \leq B^{\e}(\tau)$ for $ \tau_{0} \leq \tau \leq t$ but either $a(\tau)> A^{\e}(\tau)$ or $b(\tau)> B^{\e}(\tau)$ for a small neighborhood of $\tau> \tau_{0}$. Especially either $a(\tau_{0})= A^{\e} (\tau_{0})$ or $b(\tau_{0})= B^{\e} (\tau_{0})$. But this is impossible. Since
\begin{eqnarray*}
\left[\begin{array}{c}
A^{\e}(\tau) - a(\tau) \\
B^{\e}(\tau) - b(\tau)
\end{array} \right]
\geq C 
\left[\begin{array}{cc}
0 & 1 \\
m+ |v|^{2}  & |v|
\end{array} \right]
\left[\begin{array}{c}
\int^{t}_{\tau} (A^{\e} (\tau^{\prime}) - a(\tau^{\prime})) \dd \tau^{\prime}\\
\int^{t}_{\tau} (B^{\e} (\tau^{\prime}) - b(\tau^{\prime})) \dd \tau^{\prime}
\end{array} \right] +   \left[\begin{array}{c} \e \\ \e
\end{array} \right],
\end{eqnarray*}
we have $\left[\begin{array}{c}
A^{\e}(\tau) - a(\tau) \\
B^{\e}(\tau) - b(\tau)
\end{array} \right] \geq   \left[\begin{array}{c} \e \\ \e
\end{array} \right]>0$ as $\tau \rightarrow \tau_{0}^{+}$. Then we prove the inequalities (\ref{aAbB}) by letting $\e \rightarrow 0$.  
Finally we prove the claim (\ref{matrix_gronwall}) from (\ref{super_bound}) and (\ref{aAbB}) and letting $\e \rightarrow0 $.\end{proof}

\begin{proof}[\textbf{Proof of Theorem \ref{theorem_Dxv}}]

First we consider the case of $t<t_{\mathbf{b}}(t,x,v).$ Directly
\begin{equation}\notag\label{Dxv_interior}
\begin{split}
\left| \frac{\partial(  {X}_{\mathbf{cl}}(s;t,x,v), V_{\mathbf{cl}}(s,t,x,v))}{\partial (t,x,v)}  \right|
 \lesssim  
 \left[\begin{array}{ccc}
 |v| +(t-s) & 1 & (t-s)
 \\ \| E \|_{L^\infty_{t,x}} +(t-s) & |v|+(t-s) & 1 
 \end{array}\right].
\end{split}
\end{equation}
The computation will be the same as we will get for \eqref{ss1}.

 \vspace{8pt}

Now we consider the case of $t\geq t_{\mathbf{b}}(t,x,v)$. We split our proof into 10 steps.

 \vspace{8pt}

\noindent{\textit{Step 1. Moving frames and grouping with respect to the scaling $t|v|=L_{\xi}$, with fixed $0< L_{\xi}\ll 1.$   }  }

\vspace{4pt}

Fix $(t,x,v)\in [0,\infty)\times \bar{\Omega}\times \mathbb{R}^{3}.$ Also we fix small constant $\delta$ such that $\delta \ll \| E \|_{L^\infty_{t,x}}$. We define, at the boundary,
\begin{equation}
\mathbf{r}^{\ell} : = \frac{|\mathbf{v}_{\perp}^{\ell}|}{|v^{\ell}|}.\label{r}
\end{equation}

Bounces $\ell$ (and $(t^{\ell}, x^{\ell},v^{\ell})$) are categorized as \textit{Type I}, \textit{Type II}, or  \textit{Type III}:
\begin{equation}\label{type_r}
\begin{split}
 \text{all the bounces } \ell \text{ are \textit{Type I} if and only if }  & |v| \le \delta, \\
 \text{a bounce } \ell \text{ is \textit{Type II} if and only if } & |v| > \delta, \mathbf{r}^{\ell} \leq   \sqrt{\delta}, \\
 \text{a bounce } \ell \text{ is \textit{Type III} if and only if }  & |v| > \delta, \mathbf{r}^{\ell} >  \sqrt{\delta} .
\end{split}
\end{equation}
Now we choose $T < \frac{ \sqrt \delta}{\| E \|_{L^\infty_{t,x}}^2 +1}$. Then if $|v|\le \delta$, we have
\begin{equation}\notag
\begin{split}
\max_{t^{\ell+1} \le s \le t^\ell }|\xi(X_{\mathbf{cl}}(s ;t^{\ell}, x^{\ell},v^{\ell})  )| \le |v| T + \| E \|_{L^\infty_{t,x}} T^2 \le 2 \delta.
\end{split}
\end{equation}
And if $ |v| > \delta, \mathbf{r}^{\ell} \leq   \sqrt{\delta}$, we have from \eqref{tbest}
\[
\max_{t^{\ell+1} \le s \le t^\ell }|\xi(X_{\mathbf{cl}}(s ;t^{\ell}, x^{\ell},v^{\ell})  )| \lesssim |t^\ell - t^{\ell +1 } |^2 |v^\ell |^2 + (\| E \|_{L^\infty_{t,x}}^2 + 1 )T^2 \lesssim   \left( \frac{|\mathbf{v}_{\perp}^{\ell}|}{|v^{\ell}|} \right)^2 + \delta \lesssim \delta.
\]
Therefore if a bounce $\ell$ is \textit{Type I} or \textit{Type II} then $\max_{t^{\ell+1} \leq \tau \leq t^{\ell}} |\xi(X_{\mathbf{cl}}(\tau;t,x,v))| \leq C\delta$.

\vspace{4pt}

Now we assign a coordinate chart for each bounce $\ell$ (moving frames). For \textit{Type I} bounces $\ell$ in (\ref{type_r}) we let $\mathbf{p}^{\ell}= (z^{\ell}, w^{\ell})$ with $z^{\ell} =x^{\ell}$ and $w^{\ell} = \tau_1(x^\ell)$. We choose $\mathbf{p}^{\ell}-$spherical coordinate in Lemma \ref{chart_lemma} and (\ref{polar}) with this $\mathbf{p}^{\ell}$.

For \textit{Type II} bounce $\ell$, we choose $ {\mathbf{p}}^{\ell}:=(z^{\ell}, w^{\ell})$ on $\partial\Omega\times \mathbb{S}^{2}$ with $n(z^{\ell} ) \cdot w^{\ell}=0$
 \begin{equation}\label{pl}
 z^{\ell} = x^{\ell} , \ \ \  w^{\ell}= \frac{v^{\ell} - {(v^{\ell}\cdot n(z^{\ell}))}n(z^{\ell})}{|  v^{\ell} - {(v^{\ell}\cdot n(z^{\ell}))}n(z^{\ell})  |} .
\end{equation}
Note that, by the definition of \textit{Type I } bounce, $|v^{\ell} - (v^{\ell} \cdot n(z^{\ell}) n(z^{\ell}))|^{2}=|v|^{2} - |\mathbf{v}_{\perp}^{\ell}|^{2}\gtrsim |v|^{2}(1-\delta) \gtrsim_{\delta} |v|^{2}$ and hence $w^{\ell}$ is well-defined.

Moreover for \textit{Type I } and \textit{Type II} bounce
\begin{equation}\label{L}
|X_{\mathbf{cl}}(s;t,x,v)- \mathcal{L}_{\mathbf{p}^{\ell}}| \gtrsim C_\delta>0,
\end{equation}
for $|v||t^{\ell}-s| \leq  \frac{1}{100} \min_{x\in\partial\Omega} |x|.$ This is due to the fact that the projection of $V_{\mathbf{cl}}(s)$ on the plane passing $z^{\ell}$ and perpendicular to $n(z^{\ell}) \times w^{\ell}$ is at most $|v|$ magnitude but the distance from $z^{\ell}$ to the origin(the projection of poles $\mathcal{N}_{\mathbf{p}^{\ell}}$ and $\mathcal{S}_{\mathbf{p}^{\ell}}$) has lower bound $\frac{1}{10} \min_{x\in \partial\Omega}|x|,$ $|s -t^{\ell}|\ll1.$

For \textit{Type III} bounce $\ell$$(t^{\ell},x^{\ell},v^{\ell})$, we choose $\mathbf{p}^{\ell}= (z^{\ell}, w^{\ell})$ with $|z^{\ell}-x^{\ell}|\leq \sqrt{\delta}$ and we choose arbitrary $w^{\ell} \in\mathbb{S}^{2}$ satisfying $n(z^{\ell})\cdot w^{\ell}=0$. Note that unlike \textit{Type I}, this $\mathbf{p}^{\ell}-$spherical coordinate might not be defined for $s\in [t^{\ell+1}, t^{\ell}]$ but only defined near the boundary.

Whenever the moving frame is defined (for all $\tau \in (t^{\ell+1}, t^{\ell}]$ when $\ell$ is \textit{Type I} or \textit{Type II}, and $|\tau -t^{\ell}| \ll1$ when $\ell$ is \textit{Type III}) we denote, by (\ref{polar}),
\[
(\mathbf{X}_{\mathbf{\ell}}(\tau), \mathbf{V}_{\mathbf{\ell}}(\tau))=
( \mathbf{x}_{\perp_\mathbf{\ell}}(\tau), \mathbf{x}_{\parallel_\mathbf{\ell}}(\tau), \mathbf{v}_{\perp_\mathbf{\ell}}(\tau), \mathbf{v}_{\parallel_{\mathbf{\ell}}}(\tau)) 
: = \Phi^{-1}_{\mathbf{p}^{\ell}}(X_{\mathbf{cl}}(\tau), V_{\mathbf{cl}}(\tau)).
\]
Especially at the boundary we denote
\[
(\mathbf{x}_{\perp_{\ell}}^{\ell},\mathbf{x}_{\parallel_{\ell}}^{\ell} , \mathbf{v}_{\perp_{\ell}}^{\ell},\mathbf{v}_{\parallel_{\ell}}^{\ell} )
:= \lim_{\tau \uparrow t^{\ell}} (\mathbf{X}_{\mathbf{\ell}}(\tau), \mathbf{V}_{\mathbf{\ell}}(\tau))
, \ \ \ \text{with  } \mathbf{x}_{\perp_{\ell}}^{\ell}=0, \ \mathbf{v}_{\perp_{\ell}}^{\ell}\geq 0.
\]
Then we define
\[
( \mathbf{x}_{\perp_{\ell}}^{\ell+1}, \mathbf{x}_{\parallel_{\ell}}^{\ell+1},  \mathbf{v}_{\parallel_{\ell}}^{\ell+1} ) = 
\lim_{\tau \downarrow t^{\ell+1}} (   \mathbf{x}_{\perp_{\ell}}(\tau), \mathbf{x}_{\parallel_{\ell}}(\tau),  \mathbf{v}_{\parallel_{\ell}}(\tau) ),
\]
and
\begin{equation}\label{v_perp}
  \mathbf{v}_{\perp_{\ell}}^{\ell+1}:= - \lim_{\tau \downarrow t^{\ell+1}} \mathbf{v}_{\perp_{\ell}}(\tau).
\end{equation}


Now we regroup the indices of the specular cycles, without order changing, as
\begin{equation}\notag
\begin{split}
\{0,1,2,\cdots,   \ell_{*}-1, \ell_{*}\}  =\{0\} \cup  \mathcal{G}_{1} \cup \mathcal{G}_{2} \cup \cdots \cup\mathcal{G}_{[\frac{|t-s||v|}{L_{\xi}}]} \cup \mathcal{G}_{[\frac{|t-s||v|}{L_{\xi}}]+1} 
,
\end{split}
\end{equation}
where $\big[ a \big]\in\mathbb{N}$ is the greatest integer less than or equal to $a$. Each group is
\begin{equation}\label{group}
\begin{split}
\mathcal{G}_{1} &= \{ 1, \cdots, \ell_{1}-1, \ell_{1}\},\\
 \mathcal{G}_{2} &= \{ \ell_{1}, \ell_{1}+1,\cdots , \ell_{2}-1,\ell_{2}\},    \\
 & \ \   \vdots\\
\mathcal{G}_{[\frac{|t-s||v|}{L_{\xi}}]} &= \{\ell_{[\frac{|t-s||v|}{L_{\xi}}]-1}  ,\ell_{[ \frac{|t-s||v|}{{L_{\xi}}}]-1 }+1,\cdots,\ell_{[ \frac{|t-s||v|}{L_{\xi}}]}-1,\ell_{[ \frac{|t-s||v|}{L_{\xi}}]}  \},\\
 \mathcal{G}_{[ \frac{|t-s||v|}{L_{\xi}}]+1}  &= \{\ell_{[ \frac{|t-s||v|}{L_{\xi}}] }  ,\ell_{[ \frac{|t-s||v|}{L_{\xi}}] } +1,\cdots,\ell_{*}
  \},
\end{split}
\end{equation}
where $\ell_{1} = \inf\{ \ell \in\mathbb{N} : |v|\times |t^{0} - t^{\ell_{1}}| \geq L_{\xi} \}$ and inductively
\begin{equation}\label{L_xi}
\ell_{i} = \inf\{ \ell \in\mathbb{N} :    |v|\times |t^{\ell_{i}} - t^{\ell_{i+1}}| \geq L_{\xi}\},
\end{equation}
and we have denoted $\ell_{*} = \ell_{[ \frac{|t-s||v|}{L_{\xi}}]+1}$.

Our analysis is carried out in each group $G_{i}$. We note
that within each $G_{i},$ $|t^{\ell_{i}}-t^{\ell_{i+1}}||v|<L_{\xi }$ by our
design, so from the velocity lemma, $r_{\ell_{i}}$ is comparable to each other,
so is $|v^{\ell}|.$ We can also cover the entire $G_{i}$ via a single chart in
Section 8. By the chain rule, with the assigned $\mathbf{p}^{\ell}-$spherical coordinate (moving frame), we have for fixed $0 \leq s \leq t$ and $s \in (t^{\ell_{*}+1}, t^{\ell_{*}})$
\begin{equation}\label{chain}
\begin{split}
&\frac{\partial ( X_{\mathbf{cl}}(s;t,x,v),V_{\mathbf{cl}}(s;t,x,v))}{\partial (t,x,v)}\\
 = &\underbrace{\frac{\partial ( X_{\mathbf{cl}}(s), V_{\mathbf{cl}}(s))}{\partial (t^{\ell_{*}},  \mathbf{x}_{\parallel_{\ell_{*}}}^{\ell_{*}}, \mathbf{v}_{\perp_{\ell_{*}}}^{\ell_{*}}, \mathbf{v}_{\parallel_{\ell_{*}}}^{\ell_{*}})}  }_{\text{from the last bounce to the }s-\text{plane}} \\
\times& \underbrace{\prod_{i=1}^{ [\frac{|t-s||v|}{L_{*}}]} \underbrace{\frac{\partial (t^{\ell_{i+1}},  \mathbf{x}_{\parallel_{\ell_{i+1}}}^{\ell_{i+1}}, \mathbf{v}_{\perp_{\ell_{i+1}}}^{\ell_{i+1}}, \mathbf{v}_{\parallel_{\ell_{i+1}}}^{\ell_{i+1}})}{\partial (t^{\ell_{i+1}-1},  \mathbf{x}_{\parallel_{{\ell_{i+1}-1}} }^{\ell_{i+1}-1}, \mathbf{v}_{\perp_{{\ell_{i+1}-1}}}^{\ell_{i+1}-1}, \mathbf{v}_{\parallel_{{\ell_{i+1}-1}}}^{\ell_{i+1}-1})}
 \times \cdots \times
 \frac{\partial (t^{\ell_{i}+1},  \mathbf{x}_{\parallel_{{\ell_{i} +1}}}^{\ell_{i} +1}, \mathbf{v}_{\perp_{{\ell_{i} +1}}}^{\ell_{i}+1}, \mathbf{v}_{\parallel_{{\ell_{i} +1}}}^{\ell_{i}+1})}{\partial (t^{\ell_{i} },  \mathbf{x}_{\parallel_{ {\ell_{i} }}}^{\ell_{i}  }, \mathbf{v}_{\perp_{ {\ell_{i} }}}^{\ell_{i} }, \mathbf{v}_{\parallel_{ {\ell_{i} }}}^{\ell_{i} })}  }_{i-\text{th intermediate group}} }_{\text{whole intermediate groups}}\\
 \times & \underbrace{ \frac{\partial (t^{1},  \mathbf{x}_{\parallel_{1}}^{1}, \mathbf{v}_{\perp_{1}}^{1}, \mathbf{v}_{\parallel_{1}}^{1})}{\partial (t,x,v)}}_{\text{from the } t-\text{plane to the first bounce} }.
\end{split}
\end{equation}

Before we start to calculate the matrix for any bounces, we first prove an claim that will be used later: there exists a constant $C = C(\xi)$ such that for any bounce $\ell$ and any $t^{\ell+1} < s < t^\ell$ we have
\Be \label{Ftimedest}
 \left| \frac{ \p F_\perp(s) }{\p \tau}  + \frac{ \p F_\perp(s) }{\p t^\ell}  \right| +  \left| \frac{ \p F_\parallel(s) }{\p \tau}  + \frac{ \p F_\parallel(s) }{\p t^\ell}  \right|
< C \| \p_t E \|_{L^{\infty}_{t,x}}
\Ee
By direct computation we have
\Be \label{dtellall} \begin{split}
& \frac{ \p \mathbf x_\perp(s) } {\p t^\ell} = - \mathbf v_\perp(s) + \int_s^{t^\ell} \int_\tau^{t^\ell}   \left( \p_{\tau'} F_\perp (\tau') + \p_{t^\ell} F_\perp(\tau') \right)  d\tau' d\tau, 
\\ & \frac{ \p \mathbf x_\parallel(s) } {\p t^\ell} = - \mathbf v_\parallel(s) + \int_s^{t^\ell} \int_\tau^{t^\ell}  \left( \p_{\tau'} F_\parallel (\tau') + \p_{t^\ell} F_\parallel(\tau') \right)d\tau' d\tau, 
\\ & \frac{ \p \mathbf v_\parallel(s) } {\p t^\ell} = - F_\parallel(s) - \int_s^{t^\ell} \left( \p_\tau F_\parallel (\tau) + \p_{t^\ell} F_\parallel(\tau) \right) d\tau,
 \\ & \frac{ \p \mathbf v_\parallel(s) } {\p t^\ell} = - F_\perp(s) - \int_s^{t^\ell} \left( \p_\tau F_\perp (\tau) + \p_{t^\ell} F_\perp(\tau) \right) d\tau,
\end{split} \Ee
and
\Be \label{diftderv1} \begin{split}
& \left| \frac{ \p F_\perp(s) }{\p \tau}  + \frac{ \p F_\perp(s) }{\p t^\ell}  \right| +  \left| \frac{ \p F_\parallel(s) }{\p \tau}  + \frac{ \p F_\parallel(s) }{\p t^\ell}  \right|
 \\  =&  \left| \nabla_{\mathbf{x}_\perp} F_\perp \cdot \left( \mathbf{v}_\perp(s) + \frac{ \p \mathbf x_\perp(s) } {\p t^\ell} \right) + \nabla_{\mathbf{x}_\parallel} F_\perp \cdot  \left( \mathbf{v}_\parallel(s) +  \frac{ \p \mathbf x_\parallel(s) } {\p t^\ell}  \right)+ \nabla_{\mathbf{v}_\parallel} F_\perp \cdot  \left( F_\parallel(s) + \frac{ \p \mathbf v_\parallel(s) } {\p t^\ell} \right) - \p_s E \cdot \mathbf n(\mathbf{x}_\parallel) \right|
 \\ & +  \left| \nabla_{\mathbf{x}_\perp} F_\parallel \cdot \left( \mathbf{v}_\perp(s) + \frac{ \p \mathbf x_\perp(s) } {\p t^\ell} \right) + \nabla_{\mathbf{x}_\parallel} F_\parallel \cdot  \left( \mathbf{v}_\parallel(s) +  \frac{ \p \mathbf x_\parallel(s) } {\p t^\ell}  \right)+ \nabla_{\mathbf{v}_\perp} F_\parallel \cdot  \left( F_\perp(s) + \frac{ \p \mathbf v_\perp(s) } {\p t^\ell} \right)  \right.
 \\ &  \quad \left. + \nabla_{\mathbf{v}_\parallel} F_\parallel \cdot  \left( F_\parallel(s) + \frac{ \p \mathbf v_\parallel(s) } {\p t^\ell} \right) - \sum_{i=1,2} G_{ \mathbf{p},ij} (\mathbf{x}_{\perp_{ \mathbf{p}}}, \mathbf{x}_{\parallel_{ \mathbf{p}}})\frac{(-1)^{i } \p_s E(s, - \mathbf{x}_\perp \mathbf n (\mathbf x_\parallel ) + \mathbf \eta (\mathbf x_\parallel ) )  \cdot \big\{ \mathbf{n}_{ \mathbf{p}} (\mathbf{x}_{\parallel_{ \mathbf{p}}}) \times \partial_{i+1} \mathbf{\eta}_{ \mathbf{p}}(\mathbf{x}_{\parallel_{ \mathbf{p}}})  \big\}
}{\mathbf{n}_{ \mathbf{p}}(\mathbf{x}_{\parallel_{ \mathbf{p}}}) \cdot (\partial_{1} \mathbf{\eta}_{ \mathbf{p}}(\mathbf{x}_{\parallel_{ \mathbf{p}}}) \times \partial_{2} \mathbf{\eta}_{ \mathbf{p}}(\mathbf{x}_{\parallel_{ \mathbf{p}}}))}  \right|.
 \end{split} \Ee
Then from \eqref{dtellall}, \eqref{diftderv1}, and using the fact that $\| \nabla_{\mathbf x_\parallel, \mathbf x_\perp} F_{\parallel} \|_\infty + \| \nabla_{\mathbf x_\parallel, \mathbf x_\perp} F_{\perp} \|_\infty \lesssim |v|^2 +1$, $\| \nabla_{\mathbf v_\parallel} F_{\perp} \|_\infty + \| \nabla_{\mathbf v_\perp, \mathbf v_\parallel} F_{\parallel} \|_\infty \lesssim |v| + 1$, we have
\Be \label{diftderv1} \begin{split}
& \left| \frac{ \p F_\perp(s) }{\p \tau}  + \frac{ \p F_\perp(s) }{\p t^\ell}  \right| +  \left| \frac{ \p F_\parallel(s) }{\p \tau}  + \frac{ \p F_\parallel(s) }{\p t^\ell}  \right|
 \\  \lesssim & (|v|^2 +1 )  \int_s^{t^\ell} \int_\tau^{t^\ell} \left(  \left|  \p_{\tau'} F_\perp (\tau') + \p_{t^\ell} F_\perp(\tau') \right| + \left| \p_{\tau'} F_\parallel (\tau') + \p_{t^\ell} F_\parallel(\tau') \right| \right)  d\tau' d\tau
 \\ & + (|v|+1) \int_s^{t^\ell}   \left(  \left|  \p_{\tau'} F_\perp (\tau') + \p_{t^\ell} F_\perp(\tau') \right| + \left| \p_{\tau'} F_\parallel (\tau') + \p_{t^\ell} F_\parallel(\tau') \right| \right)  d\tau' + \|\p_s E \|_\infty
  \\  \lesssim &  (|v|^2 +1 )   \int_s^{t^\ell}  (\tau' - s ) \left(  \left|  \p_{\tau'} F_\perp (\tau') + \p_{t^\ell} F_\perp(\tau') \right| + \left| \p_{\tau'} F_\parallel (\tau') + \p_{t^\ell} F_\parallel(\tau') \right| \right)  d\tau'
 \\ & + (|v|+1) \int_s^{t^\ell}   \left(  \left|  \p_{\tau'} F_\perp (\tau') + \p_{t^\ell} F_\perp(\tau') \right| + \left| \p_{\tau'} F_\parallel (\tau') + \p_{t^\ell} F_\parallel(\tau') \right| \right)  d\tau' + \|\p_s E \|_\infty
 \\ \lesssim & (|v|+1) \int_s^{t^\ell}   \left(  \left|  \p_{\tau'} F_\perp (\tau') + \p_{t^\ell} F_\perp(\tau') \right| + \left| \p_{\tau'} F_\parallel (\tau') + \p_{t^\ell} F_\parallel(\tau') \right| \right)  d\tau' + \|\p_s E \|_\infty,
\end{split} \Ee
where for the second inequality we switch the order of integration $\int_s^{\ell} \int_\tau^{t^\ell } 1  d\tau' d\tau =\int_s^{\ell} \int_s^{\tau' } 1  d\tau d\tau' =  \int_s^{\ell } (\tau' -s ) d\tau'$, and for the third inequality we use $|v|(t^\ell -s ) \lesssim 1 $.

Therefore from \eqref{diftderv1} and Gronwall's inequality we get
\[
\left| \frac{ \p F_\perp(s) }{\p \tau}  + \frac{ \p F_\perp(s) }{\p t^\ell}  \right| +  \left| \frac{ \p F_\parallel(s) }{\p \tau}  + \frac{ \p F_\parallel(s) }{\p t^\ell}  \right| \lesssim   \|\p_s E \|_{L^{\infty}_{t,x}} e^{ \int_s^{t^\ell } (|v| + 1 ) d\tau' }  \lesssim   \|\p_s E \|_{L^{\infty}_{t,x}},
\]
and this proves \eqref{Ftimedest}.

\vspace{4pt}
\noindent\textit{Step 2. From the last bounce $\ell_{*}$ to the $s-$plane}

We choose $s^{ \ell_{*}}\in (\frac{t^{\ell_{*}}+s}{2},t^{\ell_{*}})\subset(s,t^{\ell_{*}})$ such that $|v||t^{\ell_{*}} - s^{\ell_{*}}|\ll 1$ and the $\ell_{*}-$spherical coordinate $(\mathbf{X}_{\ell_{*}}(s^{\ell_{*}}), \mathbf{V}_{\ell_{*}}(s^{\ell_{*}}))$ is well-defined regardless of {type}s of $\ell_{*}$ in (\ref{type_r}). Notice that $s^{\ell_{*}}$ is independent of $t^{\ell_{*}}$ and $s$ so that $\frac{\partial s^{\ell_{*}}}{\partial t^{\ell_{*}}}=0= \frac{\partial s^{\ell_{*}} }{\partial s  }.$

We first follow the flow in $(x,v)$
co-ordinate to near the boundary at $t=s^{\ell_{* }}$, change to the chart
to $(X,V),$ then follow the flow in $(X,V).$ Regarding $s^{\ell_{*}}$ as a free variable, by the chain rule,
 \begin{equation}\notag
 \begin{split}
&\frac{\partial (
X_{\mathbf{cl}}(s), V_{\mathbf{cl}}(s))}
{\partial (t^{\ell_{*}},
\mathbf{x}_{\parallel_{\ell_{*}}}^{\ell_{*}}, \mathbf{v}_{\perp_{\ell_{*}}}^{\ell_{*}}, \mathbf{v}_{\parallel_{\ell_{*}}}^{\ell_{*}})}\\
&= \frac{\partial ( X_{\mathbf{cl}}(s), V_{\mathbf{cl}}(s))}
{\partial (s^{\ell_{*}}, \mathbf{X}_{\ell_{*}}(s^{\ell_{*}}),   \mathbf{V}_{\ell_{*}}(s^{\ell_{*}})  )}
   \frac{\partial (s^{\ell_{*}}, \mathbf{x}_{\perp_{\ell_{*}}}(s^{\ell_{*}}),  \mathbf{x}_{\parallel_{\ell_{*}}}(s^{\ell_{*}}),   \mathbf{v}_{\perp_{\ell_{*}}}(s^{\ell_{*}}), \mathbf{v}_{\parallel_{\ell_{*}}}(s^{\ell_{*}})
    )}{\partial (t^{\ell_{*}},
     \mathbf{x}_{\parallel_{\ell_{*}}}^{\ell_{*}}, \mathbf{v}_{\perp_{\ell_{*}}}^{\ell_{*}}, \mathbf{v}_{\parallel_{\ell_{*}}}^{\ell_{*}})}
\\
& = \frac{\partial (
X_{\mathbf{cl}}(s), V_{\mathbf{cl}}(s))}{\partial (s^{\ell_{*}},  {X}_{ \mathbf{cl} }(s^{\ell_{*}}),   {V}_{ \mathbf{cl}   }(s^{\ell_{*}})  )}
\frac{\partial (s^{\ell_{*}},  {X}_{  \mathbf{cl}   } (s^{\ell_{*}}),   {V}_{ \mathbf{cl}   }(s^{\ell_{*}})  )}{\partial (s^{\ell_{*}}, \mathbf{X}_{ {\ell_{*}}}(s^{\ell_{*}}),   \mathbf{V}_{ {\ell_{*}}}(s^{\ell_{*}})  )}\\
& \ \ \ \ \ \ \  \times 
   \frac{\partial (s^{\ell_{*}}, \mathbf{x}_{\perp_{\ell_{*}}}(s^{\ell_{*}}),  \mathbf{x}_{\parallel_{\ell_{*}}}(s^{\ell_{*}}),   \mathbf{v}_{\perp_{\ell_{*}}}(s^{\ell_{*}}), \mathbf{v}_{\parallel}(s^{\ell_{*}})
    )}{\partial (t^{\ell_{*}},
     \mathbf{x}_{\parallel_{\ell_{*}}}^{\ell_{*}}, \mathbf{v}_{\perp_{\ell_{*}}}^{\ell_{*}}, \mathbf{v}_{\parallel_{\ell_{*}}}^{\ell_{*}})}.
 \end{split}
 \end{equation}
 Firstly, we claim
 \begin{equation}\label{ss1}
 \frac{\partial ( X_{   \mathbf{cl}}(s), V_{ \mathbf{cl}  }(s))}{\partial (s^{\ell_{*}}, \mathbf{X}_{ {\ell_{*}}}(s^{\ell_{*}}),   \mathbf{V}_{ {\ell_{*}}}(s^{\ell_{*}})  )}
 =
 \left[\begin{array}{ccc}
 -V_{\mathbf{cl}}(s^{\ell_{*} }) +O(1)|s^{\ell_*} - s |& O_{\xi}(1) (1+|v||s^{\ell_{*}}-s|) & O_{\xi}(1)  |s^{\ell_{*}}-s|\\
-E -O(1)(s^{\ell_*} - s)|v| & O_{\xi}(1)  (|v| + |s^{\ell_*} - s | ) & O_{\xi}(1)(1 + |s^{\ell_*} -s | )
 \end{array} \right].
 \end{equation}


Since
\Be \label{slstarderivatives}
\begin{split}
X_{\mathbf{cl}}(s) = & X_{\mathbf{cl}}(s^{\ell_{*} }) -\int_s^{s^{\ell_*}}  V_{\mathbf{cl}}(\tau) \, d \tau = X_{\mathbf{cl}}(s^{\ell_{*} }) - (s^{\ell_*} - s ) V_{\mathbf{cl}}(s^{\ell_{*}}) +\int_s^{s^{\ell_*}} \int_\tau^{s^{\ell_* } } E(\tau', X_{\mathbf{cl} }(\tau' ) ) \, d\tau'd\tau , 
\\ V_{\mathbf{cl}}(s) = & V_{\mathbf{cl}}(s^{\ell_{*}}) - \int_s^{s^{\ell_* } } E(\tau, X_{\mathbf{cl} }(\tau ) ) \, d\tau,
\end{split}
\Ee
we have
\Be \label{pXclpsellstar} \begin{split}
\frac{ \p X_{\mathbf{cl}}(s) }{\p s^{\ell_*}} = & - V_{\mathbf{cl}} (s^{\ell_*}) + \int_s^{s^{\ell_*}} \left[ E(s^{\ell_*}, X_{\mathbf{cl}}(s^{\ell_*})) + \int_\tau^{s^{\ell_* }} \nabla_x E(\tau', X_{\mathbf{cl}}(\tau' ) )\frac{ \p X_{\mathbf{cl}}( \tau') }{\p s^{\ell_*}} \, d\tau' \right] d\tau
\\ = &  - V_{\mathbf{cl}} (s^{\ell_*}) + ( s^{\ell_*} -s ) E(s^{\ell_*}, X_{\mathbf{cl}}(s^{\ell_*})) + \int_s^{ s^{\ell_*}} \int_s^{\tau' } \nabla_x E(\tau', X_{\mathbf{cl}}(\tau' ) )\frac{ \p X_{\mathbf{cl}}( \tau') }{\p s^{\ell_*}} \, d\tau d\tau'
\\ = &  - V_{\mathbf{cl}} (s^{\ell_*}) + ( s^{\ell_*} -s ) E(s^{\ell_*}, X_{\mathbf{cl}}(s^{\ell_*})) + \int_s^{ s^{\ell_*}}  (\tau' -s )  \nabla_x E(\tau', X_{\mathbf{cl}}(\tau' ) )\frac{ \p X_{\mathbf{cl}}( \tau') }{\p s^{\ell_*}} \, d\tau'.
\end{split} \Ee
By Gronwall we have
\Be \label{pXclpsellstarbdd}
|\frac{ \p X_{\mathbf{cl}}(s) }{\p s^{\ell_*}} | \lesssim ( |V_{\mathbf{cl}} (s^{\ell_*}) | + (s^{\ell_*} - s ) |E|) e^{\int_s^{s^{\ell_*}}   (s^{\ell_*} - s ) | \nabla_x E | \, d\tau'}  \lesssim  |V_{\mathbf{cl}} (s^{\ell_*}) | + |s^{\ell_*} - s|.
\Ee
Plug \eqref{pXclpsellstarbdd} into \eqref{pXclpsellstar} we get
\Be \label{Xclsellstardif}
\frac{ \p X_{\mathbf{cl}}(s) }{\p s^{\ell_*}} = - V_{\mathbf{cl}} (s^{\ell_*}) + O_{\xi, \|  \nabla E \|_{L^\infty_{t,x}}}(1)|s^{\ell_*} - s|.
\Ee
Similarly we have
\Be \label{pVClpsellstarbd}
 \begin{split}
\frac{ \p V_{\mathbf{cl}}(s) }{\p s^{\ell_*}} & = - E(s^{\ell_*}, X_{\mathbf{cl}}(s^{\ell_*})) - \int_s^{s^{\ell_*}} \nabla_x E(\tau, X_{\mathbf{cl}} (\tau ) ) \frac{ \p X_{\mathbf{cl}}(\tau) }{\p s^{\ell_*}} \, d \tau 
\\ &=  - E(s^{\ell_*}, X_{\mathbf{cl}}(s^{\ell_*})) - O_{\xi, \|  \nabla E \|_{L^\infty_{t,x}}}|(1) (s^{\ell_*} - s) | v|.
\end{split} \Ee

Also, using the fact that for $\p = [ \frac{\p}{X_{\mathbf{cl}}(s^{\ell_*})}, \frac{\p}{V_{\mathbf{cl}}(s^{\ell_*})}]$, $|\p X_{\mathbf{cl}}(s) | + |\p V_{\mathbf{cl}}(s) | \lesssim 1 $, we can combine \eqref{slstarderivatives}, \eqref{pXclpsellstarbdd}, and \eqref{pVClpsellstarbd} to get

\begin{equation}\notag
\begin{split}
\frac{\partial ( X_{\mathbf{cl}}(s), V_{\mathbf{cl}}(s))}{\partial (s^{\ell_{*}} ,  {X}_{\mathbf{cl}}(s^{\ell_{*}}),   {V}_{\mathbf{cl}}(s^{\ell_{*}})  )}&= \left[\begin{array}{ccc}
 -V_{\mathbf{cl}}(s^{\ell_{*}}) +O(1)|s^{\ell_*} - s |
 & \mathbf{Id}_{3,3} +O(1)|s^{\ell_*} - s |& -(s^{\ell_{*}}-s) \mathbf{Id}_{3,3}+O(1)|s^{\ell_*} - s |\\
-E -O(1)(s^{\ell_*} - s)|v| & \mathbf{0}_{3,3}+O(1)|s^{\ell_*} - s | & \mathbf{Id}_{3,3}+O(1)|s^{\ell_*} - s |
\end{array} \right].
\end{split}
\end{equation}

 Furthermore due to Lemma \ref{chart_lemma}, we conclude
\begin{equation}\notag
\begin{split}
&\frac{\partial (s^{\ell_{*}},  {X}_{\mathbf{cl}}(s^{\ell_{*}}),   {V}_{\mathbf{cl}}(s^{\ell_{*}})  )}{\partial (s^{\ell_{*}}, \mathbf{X}_{ {\ell_{*}}}(s^{\ell_{*}}),   \mathbf{V}_{ {\ell_{*}}}(s^{\ell_{*}})  )}=\\
&  {\tiny\left[ \begin{array}{c|ccc|ccc}
1 &   & \mathbf{0}_{1,3} &  &  & \mathbf{0}_{1,3} &   \\ \hline
 \mathbf{0}_{3,1}
  & -\mathbf{n}_{{\ell_{*}}}  &
   \substack{  \partial_{ 1} \mathbf{\eta}_{{\ell_{*}}}  \\
- \mathbf{x}_{\perp_{{\ell_{*}}}}  \partial_{ 1} \mathbf{n}_{{\ell_{*}}}  }&
 \substack{   \partial_{2} \mathbf{\eta}_{{\ell_{*}}}\\
-\mathbf{x}_{\perp_{{\ell_{*}}}}  \partial_{2} \mathbf{n}_{{\ell_{*}}}   }
 & & \mathbf{0}_{3,3} & \\ \hline
\mathbf{0}_{3,1}&
-\mathbf{v}_{\parallel_{{\ell_{*}}}} \cdot \nabla_{ \mathbf{x}_{\parallel_{{\ell_{*}}}} } \mathbf{n}_{{\ell_{*}}}
&  \substack{ \mathbf{v}_{\parallel_{{\ell_{*}}}} \cdot \nabla  \partial_{  1  } \mathbf{\eta}_{{\ell_{*}}} \\
-  \mathbf{v}_{\perp_{{\ell_{*}}}}\partial_{1}  \mathbf{n}_{{\ell_{*}}}   \\
-  \mathbf{x}_{\perp_{{\ell_{*}}}}  \mathbf{v}_{\parallel_{{\ell_{*}}}}\cdot \nabla \partial_{ 1} \mathbf{n}_{{\ell_{*}}}   }
 & \substack{ \mathbf{v}_{\parallel_{{\ell_{*}}}} \cdot \nabla \partial_{2} \mathbf{\eta}_{{\ell_{*}}} \\
- \mathbf{v}_{\perp_{{\ell_{*}}}} \partial_{ 2 } \mathbf{n}{\ell_{*}} \\
-  \mathbf{x}_{\perp_{{\ell_{*}}}}  \mathbf{v}_{\parallel_{{\ell_{*}}}}\cdot \nabla  \partial_{2} \mathbf{n}_{{\ell_{*}}}  }
    & -\mathbf{n}_{{\ell_{*}}}  
    & \substack{\partial_{1} \mathbf{\eta}_{{\ell_{*}}} \\
   - \mathbf{x}_{\perp_{{\ell_{*}}}}\partial_{1} \mathbf{n}_{{\ell_{*}}} }
     & \substack{\partial_{2} \mathbf{\eta}_{{\ell_{*}}}  \\ - \mathbf{x}_{\perp_{{\ell_{*}}}} \partial_{2} \mathbf{n}}_{{\ell_{*}}}
  \end{array}\right],}
\end{split}
\end{equation}
 where all entries are evaluated at $(\mathbf{X}_{ {\ell_{*}}}(s^{\ell_{*}}), \mathbf{V}_{ {\ell_{*}}}(s^{\ell_{*}})).$ The multiplication of above two matrices gives (\ref{ss1}).

Secondly, we claim that whenever $\mathbf{p}^{\ell}-$spherical coordinate is defined for all $\tau \in [s^{\ell}, t^{\ell}]$, we have following $7 \times 6$ matrix
\begin{equation}\label{s1tstar}
\begin{split}
&\frac{\partial (s^{\ell }, \mathbf{x}_{\perp_{{\ell }}}(s^{\ell }),  \mathbf{x}_{\parallel_{{\ell }}}(s^{\ell }),   \mathbf{v}_{\perp_{{\ell }}}(s^{\ell }), \mathbf{v}_{\parallel_{{\ell }}}(s^{\ell })
    )}{\partial (t^{\ell },  \mathbf{x}_{\parallel_{{\ell }}}^{\ell }, \mathbf{v}_{\perp_{{\ell }}}^{\ell }, \mathbf{v}_{\parallel_{{\ell }}}^{\ell })}\\
=& \tiny{ \left[\begin{array}{c|c|cc}
0 &  \mathbf{0}_{1,2} & 0 &   \mathbf{0}_{1,2}  \\ \hline
- \mathbf v_\perp (s^\ell) +O(1) |t^\ell - s^\ell |^2   & O_{\xi}(1) |v|^{2} |t^{\ell}-s^{\ell}|^{2} & O_{\xi}(1)   |t^{\ell}-s^{\ell}|  &  O_{\xi}(1)     |v| |t^{\ell}-s^{\ell}|^{2 }   \\
- \mathbf v_\parallel (s^\ell) +O(1) |t^\ell - s^\ell |^2   &\mathbf{Id}_{2,2}+ O_{\xi}(1) |v|^{2}|t^{\ell}-s^{\ell}|^{2} &  O_{\xi}(1)   |v| |t^{\ell}-s^{\ell}|^{2} &
O_{\xi}(1)|t^{\ell}-s^{\ell}| ( \mathbf{Id}_{2,2}+  |v||t^{\ell}-s^{\ell}|)   \\ \hline
O_{}(1) (|v|^{2} + 1 )    &(O_{}(1) + |v|^{2})|t^{\ell}-s^{\ell}|  &1+O_{\xi}(1)|v||t^{\ell}-s^{\ell}| &    O_{\xi}(1)|v||t^{\ell}-s^{\ell}| \\
O_{}(1) (|v|^{2} + 1 )  &(O_{}(1) + |v|^{2})|t^{\ell}-s^{\ell}| & O_{\xi}(1)|v||t^{\ell}-s^{\ell}| & \mathbf{Id}_{2,2}+  O_{\xi}(1)|v||t^{\ell}-s^{\ell}|
\end{array}\right] .}
\end{split}
\end{equation}
In this step we just need (\ref{s1tstar}) for $\ell=\ell_{*}$ but we need (\ref{s1tstar}) for general $\ell$ in Step 8.

Clearly the first raw is identically zero since $s^{\ell}$ is chosen to be independent of $(t^{\ell}, \mathbf{x}_{\parallel_{\ell}}^{\ell}, \mathbf{v}_{\perp_{\ell}}^{\ell}, \mathbf{v}_{\parallel_{\ell}}^{\ell})$. By directly taking $\frac{\p}{\p t^\ell }$ derivative to $\mathbf v_{\perp,\parallel}(s^\ell ) = \mathbf v_{\perp,\parallel}^\ell - \int_{s^\ell }^{t^\ell } F_{\perp,\parallel} (\tau) d\tau $ and $\mathbf x_{\perp,\parallel}(s^\ell ) =\mathbf  x_{\perp,\parallel}^\ell - \int_{s^\ell}^{t^\ell } \mathbf v_{\perp,\parallel} (\tau) d\tau$ we have
\Be
\frac{ \p \mathbf v_{\perp,\parallel}(s) }{\p t^\ell } = - F_{\perp,\parallel}(t^\ell ) - \int_{s^\ell}^{t^\ell} \p_{t^\ell } F_{\perp,\parallel} (\tau) d\tau = - F_{\perp,\parallel}(s) - \int_{s^\ell}^{t^\ell}  \left( \p_{t^\ell } F_{\perp,\parallel}(\tau) + \p_\tau F_{\perp,\parallel}(\tau ) \right)  d\tau,
\Ee
and
\Be \begin{split}
\frac{ \p \mathbf x_{\perp,\parallel}(s^\ell ) }{\p t^\ell }  &= - \mathbf v_{\perp,\parallel}(t^\ell ) - \int_{s^\ell }^{t^\ell} \p_{t^\ell } \mathbf v_{\perp,\parallel} (\tau) d\tau 
\\ &=  - \mathbf v_{\perp,\parallel}(s^\ell )  - \int_{s^\ell }^{t^\ell } F_{\perp,\parallel}(s) ds - \int_{s^\ell }^{t^\ell} \p_{t^\ell } \mathbf v_{\perp,\parallel} (\tau)
\\ & =  - \mathbf v_{\perp,\parallel}(s^\ell ) + \int_{s^\ell}^{t^\ell } \int_\tau^{t^\ell } \left( \p_{t^\ell } F_{\perp,\parallel}(\tau') + \p_\tau F_{\perp,\parallel}(\tau' ) \right) d\tau' d\tau.
\end{split} \Ee
Then from \eqref{Ftimedest} we get the desired estimate for the first column of \eqref{s1tstar}.

Now we turn to other entries in (\ref{s1tstar}). From the characteristics ODE, (\ref{ODE_ell}) in the $\mathbf{p}^{\ell}-$spherical coordinate, (\ref{Dxv_free}), (\ref{Dxv_free_s}), and (\ref{Dxv_F}),
we deduce (\ref{s1tstar}) for $|v||s^{\ell}-t^{\ell}|\lesssim 1$.

\vspace{8pt}

\noindent\textit{Step 3. From $t-$plane to the first bounce}

\vspace{4pt}

We choose $s^{1} \in (t^{1}, \frac{t^{1}+t}{2})\subset (t^{1},t)$ such that $|v||t^{1}-s^{1}|\ll 1$ and the polar coordinate $(\mathbf{X}_{1}(s^{1}), \mathbf{V}_{1}(s^{1}))$ is well-defined. More precisely we choose $0 < \Delta  $ such that $|v||t-\Delta -t^{1}| \ll 1$ and define
\begin{equation}
s^{1}:= t - \Delta.\label{Delta_step3}
\end{equation}
We first follow the flow in the cartesian coordinate to near the boundary at $s^{1}$, change to the chart to $\mathbf{p}^{\ell}-$spherical coordinate, then follow the flow in that coordinate. 

Then, by the chain rule,
 \begin{equation}\notag
 \begin{split}
& \frac{\partial (t^{1}  , \mathbf{x}_{\parallel_{1}}^{1} , \mathbf{v}_{\perp_{1}}^{1} , \mathbf{v}_{\parallel_{1}}^{1}  )}{\partial (t,x,v)}\\
 & = \frac{\partial (t^{1}  , \mathbf{x}_{\parallel_{1}}^{1}, \mathbf{v}_{\perp_{1}}^{1}, \mathbf{v}_{\parallel_{1}}^{1})}{\partial (s^{1}, X_{\mathbf{cl}}(s^{1} ), V_{\mathbf{cl}}(s^{1} ))}
\frac{\partial (s^{1}, X_{\mathbf{cl}}(s^{1}), V_{\mathbf{cl}}(s^{1}))}{\partial (t,x,v)}
\\
&= \frac{\partial (t^{1}  , \mathbf{x}_{\parallel_{1}}^{1}, \mathbf{v}_{\perp_{1}}^{1}, \mathbf{v}_{\parallel_{1}}^{1})}{ \partial (s^{1}, \mathbf{x}_{\perp_{1}}(s^{1}), \mathbf{x}_{\parallel_{1}}(s^{1}), \mathbf{v}_{\perp_{1}}(s^{1}), \mathbf{v}_{\parallel_{1}}(s^{1}) )}
\frac{\partial (s^{1}, \mathbf{X}_{1}(s^{1}),  \mathbf{V}_{ 1}(s^{1})    )}{\partial (s^{1}, X_{ \mathbf{cl}}(s^{1} ), V_{ \mathbf{cl}}(s^{1} ))}
\frac{\partial (s^{1}, X_{\mathbf{cl}}(s^{1}), V_{ \mathbf{cl}}(s^{1}))}{\partial (t,x,v)}.
 \end{split}
 \end{equation}
We fix $\mathbf{p}^{1}-$spherical coordinate and drop the index of the chart.

 Firstly, we claim 
%
  {\tiny	 
  \Be \label{t1s1}
  \begin{split}
 &  \frac{\partial (t^{1} , \mathbf{x}_{\parallel}^{1}, \mathbf{v}_{\perp}^{1}, \mathbf{v}_{\parallel}^{1})}{ \partial (s^{1}, \mathbf{x}_{\perp}(s^{1}), \mathbf{x}_{\parallel}(s^{1}), \mathbf{v}_{\perp}(s^{1}), \mathbf{v}_{\parallel}(s^{1}) )}
   \\ \lesssim_{\Omega} &
   {\left[\begin{array}{c|cc|cc}
\frac{|v| + O(1) |s^1 -t^1|}{|v_\perp^1|} &  \frac{1}{|\mathbf{v}_{\perp}^{1}|}  &   \frac{(|v|^{2}+O(1))|s^{1}-t^{1}|^{2}}{|\mathbf{v}_{\perp}^{1}|} & \frac{|s^{1}-t^{1}|}{|\mathbf{v}_{\perp}^{1}|} & \frac{(|v|^{}+O(1))|s^{1}-t^{1}|^{2}}{|\mathbf{v}_{\perp}^{1}|}  \\ \hline
|v| + \frac{|v|^2 + O(1)}{|v_\perp^1|} &     \frac{|v|}{|\mathbf{v}_{\perp}^{1}|}  +(|v|^{2}+O(1))|s^{1}-t^{1}|^{2} & \mathbf{Id}_{2,2} + (|v|^{}+O(1))|s^{1}-t^{1}| &   \frac{|s^{1}-t^{1}||v|}{|\mathbf{v}_{\perp}^{1}|} + |s^{1}-t^{1}|^{2}|v| & |s^{1}-t^{1}|\\ \hline
 |v|^2 +O(1) + \frac{|v|^3 +O(1)}{|v_\perp^1|}& \frac{|v|^{2}+O(1)}{|\mathbf{v}_{\perp}^{1}|} + (|v|^{2}+O(1))|s^{1}-t^{1}|  & (|v|^{2}+O(1))|s^{1}-t^{1}| & 1+ |v||s^{1}-t^{1}| & (|v| +O(1))|s^{1}-t^{1}| \\
 |v|^2 +O(1) + \frac{|v|^3 +O(1)}{|v_\perp^1|}& \frac{|v|^{2}+O(1)}{|\mathbf{v}_{\perp}^{1}|} + (|v|^{2}+O(1))|s^{1}-t^{1}|  &  ( |v|^{2}+O(1))|s^{1}-t^{1}|   & 1+ |v||s^{1}-t^{1}|  &  \mathbf{Id}_{2}+ (|v| +O(1)) |s^{1}-t^{1}| \end{array}\right]}.
 \end{split}
 \end{equation}
  }

The $t^{1}$ is determined via $\mathbf{x}_{\perp}(t^{1})=0$, i.e.
\begin{equation}\label{equation_t1}
0= \mathbf{x}_{\perp}(s^{1}) - \mathbf{v}_{\perp}(s^{1}) (s^{1}-t^{1}) + \int^{s^{1}}_{t^{1}} \int^{s^{1}}_{s} F_{\perp}(\mathbf{X}(\tau), \mathbf{V}_{\mathbf{cl}}(\tau)) \mathrm{d}\tau \mathrm{d}s,
\end{equation}
where $\mathbf{X}(\tau)  = \mathbf{X}(\tau; s^{1}, \mathbf{X}(s^{1};t,x,v) , \mathbf{V}(s^{1};t,x,v) ), \mathbf{V}(\tau)   =\mathbf{V} (\tau; s^{1}, \mathbf{X}(s^{1};t,x,v) , \mathbf{V}(s^{1};t,x,v) ).$ For $\p \in \{\p_{\mathbf{x}_{\perp} (s^{1})}, \p_{\mathbf{x}_{\parallel} (s^{1})} , \p_{\mathbf{v}_{\perp} (s^{1})}, \p_{\mathbf{v}_{\parallel} (s^{1})} \}$,
\begin{equation}\label{p_xperp_vperp}\begin{split}
&\mathbf{v}_{\perp}(s^{1}) \p t^{1}
- \p t^{1} \int^{s^{1}}_{t^{1}} F_{\perp} (\mathbf{X} (\tau), \mathbf{V} (\tau)) \dd \tau+
\p \mathbf{x}_{\perp}(s^{1}) - \p\mathbf{v}_{\perp}(s^{1}) (s^{1}-t^{1})  \\
&+ \int^{s^{1}}_{t^{1}} \int^{s^{1}}_{s} 
\{\p \mathbf{X}(\tau) \cdot \nabla_{\mathbf{X}}
F_{\perp} + \p \mathbf{V}(\tau) \cdot \nabla_{\mathbf{V}}
F_{\perp}
\}
(\mathbf{X}(\tau), \mathbf{V}(\tau)) \mathrm{d}\tau \mathrm{d}s=0.\end{split}
\end{equation}
But $\mathbf{v}_{\perp}^{1}  = - \lim_{s \downarrow t^{1}} \mathbf{v}_{\perp}(s) = - \mathbf{v}_{\perp}(s^{1}) + \int^{s^{1}}_{t^{1}} F_{\perp}(\mathbf{X}(\tau), \mathbf{V}(\tau)) \mathrm{d}\tau$, we apply Lemma \ref{lemma_flow} and
 $|s^{1}-t^{1}| \lesssim_{\xi}\min\{  \frac{| \mathbf{v}_{\perp}^{1}|}{|v|^{2}} , t \}$ and (\ref{tbest}), 
 \[
 \left[\begin{array}{c}
 \frac{\partial t^{1}}{\partial \mathbf{x}_{\perp }(s^{1})}\\
 \frac{\partial t^{1}}{\partial \mathbf{x}_{\parallel }(s^{1})}\\
 \frac{\partial t^{1}}{\partial \mathbf{v}_{\perp }(s^{1})}\\
 \frac{\partial t^{1}}{\partial \mathbf{v}_{\parallel }(s^{1})}
 \end{array} \right]
 = 
  \left[\begin{array}{c}
  \frac{1}{%
 \mathbf{v}_{\perp }^{1} }\left\{
1+\int^{s^{1}}_{t^{1}}\int^{s^{1}}_{s}\frac{\partial }{\partial \mathbf{x}_{\perp
}(s^{1})}F_{\perp }( \mathbf{X}(\tau), \mathbf{V}(\tau)   )\mathrm{d}\tau \mathrm{d}s\right\}\\
\frac{1}{%
 \mathbf{v}^{1}_{\perp } }\int^{s^{1}}_{t^{1}}%
\int^{s^{1}}_{s}\frac{\partial }{\partial \mathbf{x}_{\parallel }(s^{1})}F_{\perp
}(    \mathbf{X}(\tau), \mathbf{V}(\tau)     )\mathrm{d}\tau \mathrm{d}s\\
\frac{1}{\mathbf{v}_{\perp }^{1}}%
\left\{ (t^{1}-s^{1})+\int^{s^{1}}_{t^{1}}\int^{s^{1}}_{s}\frac{\partial }{%
\partial \mathbf{v}_{\perp }(s^{1})}F_{\perp }( \mathbf{X}(\tau), \mathbf{V}(\tau)  )\mathrm{d}\tau \mathrm{d}%
s\right\} \\
\frac{1}{
 \mathbf{v}_{\perp }^{1} }\int^{s_{1}}_{t^{1}}%
\int^{s_{1}}_{s}\frac{\partial }{\partial \mathbf{v}_{\parallel }(s_{1})}F_{\perp
}( \mathbf{X}(\tau), \mathbf{V}(\tau) )\mathrm{d}\tau \mathrm{d}s
   \end{array} \right]
 \lesssim_{\xi, t}
 \left[
 \begin{array}{c}
 \frac{1}{|\mathbf{v}_{\perp}^{1}|}\\
 \frac{(|v|^{2} +O(1))|s^{1}-t^{1}|^{2}}{|\mathbf{v}_{\perp}^{1}|}\\
  \frac{|s^{1}-t^{1}|}{|\mathbf{v}_{\perp}^{1}|}\\
  \frac{(|v|+O(1))|s^{1}-t^{1}|^{2}   }{|\mathbf{v}_{\perp}^{1}|}
 \end{array}\right],
 \] 
 Taking $(\mathbf{x}(s^1),\mathbf{v}(s^1))$ derivatives of the characteristic equations
 \begin{equation}\notag
\begin{split}
\mathbf{v}_{\perp}^{1} &= - \lim_{s \downarrow t^{1}} \mathbf{v}_{\perp}(s) = - \mathbf{v}_{\perp}(s^{1}) + \int^{s^{1}}_{t^{1}} F_{\perp}(\mathbf{X}_{\mathbf{cl}}(\tau), \mathbf{V}_{\mathbf{cl}}(\tau)) \mathrm{d}\tau,\\
\mathbf{x}_{\parallel}^{1} & = \mathbf{x}_{\parallel}(s^{1})
- \int^{s^{1}}_{t^{1}} \mathbf{v}_{\parallel}(s) \dd s,\\
\mathbf{v}_{\parallel}^{1} &= \mathbf{v}_{\parallel}(s^{1}) - \int^{s^{1}}_{t^{1}} F_{\parallel}(\mathbf{X}_{\mathbf{cl}}(\tau), \mathbf{V}_{\mathbf{cl}}(\tau)) \mathrm{d}\tau.
\end{split}
\end{equation}
and using the above estimates and (\ref{p_xperp_vperp}) and Lemma \ref{lemma_flow} yields
%
%
\begin{equation*}\begin{split}
\left[\begin{array}{c}
\frac{\partial \mathbf{x}^{1}_{\parallel}}{\partial \mathbf{x}_{\perp}(s^{1})}
 \\
\frac{\partial \mathbf{x}_{\parallel}^{1}}{\partial  \mathbf{x}_{\parallel}(s^{1})} \\
\frac{\partial \mathbf{x}_{\parallel}^{1}}{\partial  \mathbf{v}_{\perp}(s^{1})} \\
\frac{\partial \mathbf{x}_{\parallel}^{1}}{\partial \mathbf{v}_{\parallel}(s^{1})}
\end{array}
\right]
 &\lesssim_{\xi,t} 
\left[\begin{array}{c} 
\frac{|v|}{|\mathbf{v}_{\perp}^{1}|} + (|v|^{2}+O(1))|s^{1}-t^{1}|^{2}  \\
\mathbf{Id}_{2,2} + (|v|+O(1))|s^{1}-t^{1}|\\
\frac{|s^{1}-t^{1}||v|}{|\mathbf{v}_{\perp}^{1}|} + |s^{1}-t^{1}|^{2} |v|\\
 |s^{1}-t^{1}|
\end{array}\right],\ 
\left[\begin{array}{c}
\frac{\partial \mathbf{v}^{1}_{\perp}}{\partial \mathbf{x}_{\perp}(s^{1})}
 \\
\frac{\partial \mathbf{v}_{\perp}^{1}}{\partial  \mathbf{x}_{\parallel}(s^{1})} \\
\frac{\partial \mathbf{v}_{\perp}^{1}}{\partial  \mathbf{v}_{\perp}(s^{1})} \\
\frac{\partial \mathbf{v}_{\perp}^{1}}{\partial \mathbf{v}_{\parallel}(s^{1})}
\end{array}
\right]  \lesssim_{\xi,t} 
\left[\begin{array}{c} 
 \frac{|v|^{2}+O(1)}{|\mathbf{v}_{\perp}^{1}|} + (|v|^{2}+O(1))|s^{1}-t^{1}|\\
 (|v|^{2}+O(1))|s^{1}-t^{1}|\\
 1+ |v||s^{1}-t^{1}|\\
  (|v|+O(1))|s^{1}-t^{1}|
\end{array}\right],
\end{split}
\end{equation*}
and
\[
\left[\begin{array}{c}
\frac{\partial \mathbf{v}^{1}_{\parallel}}{\partial \mathbf{x}_{\perp}(s^{1})}
 \\
\frac{\partial \mathbf{v}_{\parallel}^{1}}{\partial  \mathbf{x}_{\parallel}(s^{1})} \\
\frac{\partial \mathbf{v}_{\parallel}^{1}}{\partial  \mathbf{v}_{\perp}(s^{1})} \\
\frac{\partial \mathbf{v}_{\parallel}^{1}}{\partial \mathbf{v}_{\parallel}(s^{1})}
\end{array}
\right]  \lesssim_{\xi,t} 
\left[\begin{array}{c} 
 \frac{(|v|^{2}+O(1))}{|\mathbf{v}_{\perp}^{1}|} + (|v|^{2}+O(1))|s^{1}-t^{1}|\\
 ( |v|^{2}+O(1))|s^{1}-t^{1}|\\
 1+ |v||s^{1}-t^{1}|\\
 \mathbf{Id}_{2,2} +  (|v|+O(1))|s^{1}-t^{1}|
\end{array}\right].
\] 

Secondly, we claim
\begin{equation}\label{s1t}
\begin{split}
 &\frac{\partial ( \mathbf{X}_{1}(s^{1}),  \mathbf{V}_{1}(s^{1})    )} {\partial (t, x,v)}
 =\frac{\partial ( \mathbf{X}_{1}(s^{1}),  \mathbf{V}_{1}(s^{1})    )}{\partial ( X_{\mathbf{cl}}(s^{1} ), V_{\mathbf{cl}}(s^{1} ))}
\frac{\partial ( X_{\mathbf{cl}}(s^{1}), V_{\mathbf{cl}}(s^{1}))}{\partial (t,x,v)}
\\
&=
 \left[\begin{array}{c|c|ccccccccccccccc}
 & O(1) + O_{\xi}(|v||t^{1}-s^{1}|^2) &  O_{\xi} ( |t-s^{1}|)\\
|t -s^1|^2&O(1) +  O_{\xi}(|v||t^{1}-s^{1}|^2) &    O_{\xi}( |t-s^{1}|)\\
 &O(1)+  O_{\xi}(|v||t^{1}-s^{1}|^2)& O_{\xi}( |t-s^{1}|)\\ \hline
  & O_{\xi}(  |v|) &O(1) +  O_{\xi}(|v| |t-s^{1}| ) \\
|t -s^1|&O_{\xi}(  |v|) & O(1) +  O_{\xi}(|v| |t-s^{1}| )\\
 &O_{\xi}( |v|) & O(1) + O_{\xi}(|v| |t-s^{1}| )
\end{array} \right]_{6 \times 7},
\end{split}
\end{equation}
where the entries are evaluated at $(\mathbf{X}_{1}(s^{1}), \mathbf{V}_{1}(s^{1}))$. Note that $|v||t^{1}-s^{1}|\lesssim_{\xi } 1$.

From (\ref{jac_Phi})
\begin{equation}\notag
\begin{split}
& \frac{\partial ( {X}_{\mathbf{cl}}(s^{1}),  {V}_{\mathbf{cl}}(s^{1}))}{\partial (\mathbf{X} (s^{1}), \mathbf{V}(s^{1}))}
= \frac{\partial \Phi (  \mathbf{X} (s^{1}), \mathbf{V} (s)  )}{\partial ( \mathbf{X} (s^{1}), \mathbf{V} (s) )}
 : =  \left[\begin{array}{c|c} A & \mathbf{0}_{3,3} \\ \hline B & A   \end{array}\right]  + \mathbf{x}_{\perp} \left[\begin{array}{c|c}   \mathbf{0}_{3,3}  & \mathbf{0}_{3,3} \\ \hline D &  \mathbf{0}_{3,3}   \end{array}\right].
\end{split}
\end{equation}
From direct computation and (\ref{nondegenerate_eta}),
\begin{equation}\notag
\begin{split}
\text{det}(A)&=\text{det} \bigg[\begin{array}{ccc} && \\  -\mathbf{n}(\mathbf{x}_{\parallel}) & \substack{ \frac{\partial \mathbf{\eta}}{\partial \mathbf{x}_{\parallel,1}} (\mathbf{x}_{\parallel}) \\ + \mathbf{x}_{\perp} [- \frac{\partial \mathbf{n}}{\partial \mathbf{x}_{\parallel,1}} (\mathbf{x}_{\parallel}) ] } &  \substack{ \frac{\partial \mathbf{\eta}}{\partial \mathbf{x}_{\parallel,2}} (\mathbf{x}_{\parallel})  \\  + \mathbf{x}_{\perp} [ -\frac{\partial \mathbf{n}}{\partial \mathbf{x}_{\parallel,2}} (\mathbf{x}_{\parallel}) ] }\\ &&\end{array}\bigg]
=
- \mathbf{n}(\mathbf{x}_{\parallel}) \cdot \left(  \frac{\partial \mathbf{\eta}}{\partial \mathbf{x}_{\parallel,1}} (\mathbf{x}_{\parallel})\times   \frac{\partial \mathbf{\eta}}{\partial \mathbf{x}_{\parallel,2}} (\mathbf{x}_{\parallel}) \right) + O_{\xi}(|\mathbf{x}_{\perp}|) \neq 0,\\
A^{-1} &= \frac{1}{[-\mathbf{n}]\cdot (\partial_{\mathbf{x}_{\parallel,1}} \mathbf{\eta} \times \partial_{\mathbf{x}_{\parallel,2}} \mathbf{\eta}) +O(1)|\mathbf x_\perp|}
\\ & \times  \bigg[ (1-\mathbf x_\perp)^2( \partial_{\mathbf{x}_{\parallel,1}} \mathbf{\eta} \times \partial_{\mathbf{x}_{\parallel,2}} \mathbf{\eta})^{T},  (1-\mathbf x_\perp)( \partial_{\mathbf{x}_{\parallel,2}} \mathbf{\eta}  \times [-\mathbf{n}] )^{T},   (1-\mathbf x_\perp)([-\mathbf{n}]\times \partial_{\mathbf{x}_{\parallel,1}} \mathbf{\eta} )^{T}\bigg].
\end{split}
\end{equation}
From basic linear algebra
\begin{equation}\notag
\begin{split}
&\text{det}\left( \frac{\partial ( {X}_{\mathbf{cl}}(s^{1}),  {V}_{\mathbf{cl}}(s^{1}))}{\partial (\mathbf{X}_{\mathbf{cl}}(s^{1}), \mathbf{V}_{\mathbf{cl}}(s^{1}))}  \right)= \text{det} \left[\begin{array}{c|c} A  & \mathbf{0}_{3,3} \\ \hline B+ \mathbf{x}_{\perp} D & A  \end{array}   \right]
= \{\det(A )\}^{2}    = \big\{[-\mathbf{n}] \cdot (\partial_{1} \mathbf{\eta} \times \partial_{2} \mathbf{\eta}) +O(1) | \mathbf x_\perp | \big\}^{2} ,
\end{split}
\end{equation}
and $\left( \frac{\partial ( {X}_{\mathbf{cl}}(s^{1}),  {V}_{\mathbf{cl}}(s^{1}))}{\partial (\mathbf{X}_{\mathbf{cl}}(s^{1}), \mathbf{V}_{\mathbf{cl}}(s^{1}))}  \right)$ is invertible. By the basic linear algebra
\begin{equation}\label{XVchange}
 \begin{split}
 &\frac{\partial (\mathbf{X}_{\mathbf{cl}}(s^{1}), \mathbf{V}_{\mathbf{cl}}(s^{1})   )}{   \partial (  X_{\mathbf{cl}} (s^{1}), V_{\mathbf{cl}}(s^{1}) )} =
\left[ \frac{   \partial (  X_{\mathbf{cl}} (s^{1}), V_{\mathbf{cl}}(s^{1}) )}{\partial (\mathbf{X}_{\mathbf{cl}}(s^{1}), \mathbf{V}_{\mathbf{cl}}(s^{1})   )}\right]^{-1} = \left[ \begin{array}{c|c} A   & \mathbf{0}_{3,3} \\ \hline B+ \mathbf{x}_{\perp} D & A  \end{array}\right]^{-1} \\
 &= \left[\begin{array}{c|c} A ^{-1} & \mathbf{0}_{3,3} \\ \hline - A ^{-1} (B+\mathbf{x}_{\perp} D)  A ^{-1} &  A ^{-1}   \end{array}\right] = \left[\begin{array}{cc} A^{-1}(\mathbf{x}_{\parallel})  & \mathbf{0}_{3,3} \\ |v|  + O_{\xi}(\mathbf{x}_{\perp}) & A^{-1}(\mathbf{x}_{\parallel})      \end{array}\right],
\end{split}
\end{equation}
and we obtain 
 \begin{equation}\notag
 \begin{split}
 \frac{\partial (  \mathbf{X}_{\mathbf{cl}}(s^{1}) ,   \mathbf{V}_{\mathbf{cl}}(s^{1}) )}{\partial (  X_{\mathbf{cl}}(s^{1}) ,   {V}_{\mathbf{cl}}(s^{1}) )}  =  \left[\begin{array}{c|c }
  \frac{(1-\mathbf x_\perp)^2(\partial_{1} \mathbf{\eta} \times \partial_{2} \mathbf{\eta})^{T}}{ [-\mathbf{n}] \cdot (\partial_{1} \mathbf{\eta} \times \partial_{2} \mathbf{\eta})+O(1) |\mathbf x_\perp |}   &  \\
  \frac{(1-\mathbf x_\perp)(\partial_{2} \mathbf{\eta} \times [-\mathbf{n}])^{T}}{ [-\mathbf{n}] \cdot (\partial_{1} \mathbf{\eta} \times \partial_{2} \mathbf{\eta}) +O(1) |\mathbf x_\perp |}    &  \mathbf{0}_{3,3} \\
  \frac{(1-\mathbf x_\perp)( [-\mathbf{n}] \times \partial_{1} \mathbf{\eta})^{T}}{ [-\mathbf{n}] \cdot (\partial_{1} \mathbf{\eta} \times \partial_{2} \mathbf{\eta})+O(1) |\mathbf x_\perp |}   \\ \hline
  O_{\xi}(1)(  |v|) &  \frac{(1-\mathbf x_\perp)^2(\partial_{1} \mathbf{\eta} \times \partial_{2} \mathbf{\eta})^{T}}{ [-\mathbf{n}] \cdot (\partial_{1} \mathbf{\eta} \times \partial_{2} \mathbf{\eta}) +O(1) |\mathbf x_\perp |}     \\
   O_{\xi}(1)(|v|)&   \frac{(1-\mathbf x_\perp)(\partial_{2} \mathbf{\eta} \times [-\mathbf{n}])^{T}}{ [-\mathbf{n}] \cdot (\partial_{1} \mathbf{\eta} \times \partial_{2} \mathbf{\eta}) +O(1) |\mathbf x_\perp |}    \\
  O_{\xi}(1)(|v|) &  \frac{(1-\mathbf x_\perp)( [-\mathbf{n}] \times \partial_{1} \mathbf{\eta})^{T}}{ [-\mathbf{n}] \cdot (\partial_{1} \mathbf{\eta} \times \partial_{2} \mathbf{\eta}) +O(1) |\mathbf x_\perp |}     \end{array}\right].
 \end{split}
 \end{equation}

From $X_{\mathbf{cl}}(s^1;t,x,v) = x-(t-s^1)v + \int_{s^1}^t \int_s^t E(\tau) d\tau d s= x-\Delta \times  v + \int_{t -\Delta}^t \int_s^t E(\tau) d\tau ds,$ and $V_{\mathbf{cl}}(s^1;t,x,v)=v - \int_{t -\Delta}^t E(s) ds,$ we have
\[ \begin{split}
\frac{X_{\mathbf{cl} }(s^1) }{ \p t } = & - \int_{t - \Delta}^t E(s) ds + \int_{t - \Delta}^t E(t) ds + \int_{t- \Delta}^t \int_s^t \p_t E(\tau) d\tau ds
\\ = & - \int_{t - \Delta}^t \int_s^t \left( \frac{ \p E(\tau)}{\p \tau}  +  \frac{ \p E(\tau)}{\p t}  \right) d\tau ds
\\ = & - \int_{t - \Delta}^t \int_s^t  \left( \p_\tau E + \nabla E \cdot \nabla_x X  \right) d\tau ds = O(1) |t -s^1|^2
\\ \frac{V_{\mathbf{cl} }(s^1) }{ \p t } = & - E(t) + E(t -\Delta) - \int_{t-\Delta}^t \frac{ \p E(s) }{\p t } ds  
\\ = & - \int_{t-\Delta}^t \left( \frac{\p  E(s) }{ \p s } + \frac{ \p E(s) }{\p t } \right) ds
\\ = &  - \int_{t-\Delta}^t   \left( \p_s E + \nabla E \cdot \nabla_x X  \right) ds = O(1) |t -s^1|.
\end{split} \]
And using  $|\nabla_{x,v} X_{\mathbf{cl}}(s^1) | + |\nabla_{x,v} V_{\mathbf{cl}}(s^1) | \lesssim 1 $,
\begin{equation}\notag
\begin{split}
&\frac{\partial ( X_{\mathbf{cl}}(s_{1}), V_{\mathbf{cl}}(s_{1}))}{\partial (t,x,v)}
=
\left[\begin{array}{ccc} 
O(1) |t -s^1|^2 & \mathbf{Id}_{3,3} + O(1) |t -s^1|^2 & -(t-s^{1})\mathbf{Id}_{3,3} + O(1) |t -s^1|^2 \\ O(1) |t -s^1| & O(1) |t -s^1|  & \mathbf{Id}_{3,3} + O(1) |t -s^1|  \end{array}\right].
\end{split}
\end{equation} 

Finally we multiply above two matrices and use $|\mathbf{x}_{\perp}(s^{1})|\lesssim  |v||t^{1}-s^{1}|$ to conclude the second claim (\ref{s1t}).

\vspace{8pt}

\noindent \textit{Step 4. Estimate of  ${\partial (t^{\ell+1} , \mathbf{x}_{\parallel_{{\ell+1}}}^{\ell+1}, \mathbf{v}_{\perp_{{\ell+1}}}^{\ell+1}, \mathbf{v}_{\parallel_{{\ell+1}}}^{\ell+1})}/{\partial (t^{\ell}  , \mathbf{x}_{\parallel_{{\ell}}}^{\ell}, \mathbf{v}_{\perp_{{\ell}}}^{\ell}, \mathbf{v}_{\parallel_{{\ell}}}^{\ell})}$}

\vspace{4pt}

Recall $
\mathbf{r}^{\ell}$ from (\ref{r}). We show that for $0<T \ll 1$ small enough, there exists $0< \delta_1 \ll 1$, $M =M_{\xi,t}\gg 1$, such that for all $\ell \in \mathbb{N}$ and $0 \leq t^{\ell+1} \leq t^{\ell} \leq t$, if $\ell$ is \textit{Type II} or \textit{Type III},
\begin{equation}\label{l_to_l+1}
\begin{split}
J^{\ell+1}_{\ell}&:= \frac{\partial (t^{\ell+1}, \mathbf{x}_{\parallel_{\ell+1}
}^{\ell+1}, \mathbf{v}_{\perp_{\ell+1} }^{\ell+1},\mathbf{v}_{\parallel_{\ell+1} }^{\ell+1})}{\partial (t^{\ell}, \mathbf{x}_{\parallel_{\ell}
}^{\ell},\mathbf{v}_{\perp_{\ell} }^{\ell},\mathbf{v}_{\parallel_{\ell} }^{\ell})}  \\
& \leq \left[\begin{array}{c|ccc|ccc}
 1+ M |t^\ell - t^{\ell + 1 } |  & \frac{M}{|v|} \mathbf{r}^{\ell+1}&  \frac{M}{|v|} \mathbf{r}^{\ell+1}& \frac{M}{|v |^{2}} &  \frac{M}{|v|^{2}} \mathbf{r}^{\ell+1} &  \frac{M}{|v|^{2}}  \mathbf{r}^{\ell+1}\\ \hline
  M |t^\ell - t^{\ell + 1 } |   & 1 + M \mathbf{r}^{\ell+1} & M \mathbf{r}^{\ell+1} & \frac{M}{|v|}    & \frac{M}{|v|} \mathbf{r}^{\ell+1}   &  \frac{M}{|v|} \mathbf{r}^{\ell+1}  \\
   M |t^\ell - t^{\ell + 1 } |   &  M \mathbf{r}^{\ell+1} & 1 +M \mathbf{r}^{\ell+1} & \frac{M}{|v|}    & \frac{M}{|v|} \mathbf{r}^{\ell+1}   &  \frac{M}{|v|} \mathbf{r}^{\ell+1}  \\ \hline
   M |t^\ell - t^{\ell + 1 } |^2|v|  & M|v|( \mathbf{r}^{\ell+1})^{2} & M|v|( \mathbf{r}^{\ell+1})^{2} & 1+M\mathbf{r}^{\ell+1} & M (\mathbf{r}^{\ell+1})^{2} & M (\mathbf{r}^{\ell+1})^{2}\\
   M |t^\ell - t^{\ell + 1 } |   & M|v|\mathbf{r}^{\ell+1} & M|v| \mathbf{r}^{\ell+1} & M  & 1+ M\mathbf{r}^{\ell+1} & M \mathbf{r}^{\ell+1}\\
      M |t^\ell - t^{\ell + 1 } |   & M|v|\mathbf{r}^{\ell+1} & M|v| \mathbf{r}^{\ell+1} & M &  M\mathbf{r}^{\ell+1} & 1+M \mathbf{r}^{\ell+1}
 \end{array}\right]
 \\ & \leq
 \left[\begin{array}{c|ccc|ccc}
 1 +5M \mathbf r^{\ell+1} & \frac{M}{|v|} \mathbf{r}^{\ell+1}&  \frac{M}{|v|} \mathbf{r}^{\ell+1}& \frac{M}{|v |^{2}} &  \frac{M}{|v|^{2}} \mathbf{r}^{\ell+1} &  \frac{M}{|v|^{2}}  \mathbf{r}^{\ell+1}\\ \hline
 5M \mathbf r^{\ell+1} |v|   & 1 + M \mathbf{r}^{\ell+1} & M \mathbf{r}^{\ell+1} & \frac{M}{|v|}    & \frac{M}{|v|} \mathbf{r}^{\ell+1}   &  \frac{M}{|v|} \mathbf{r}^{\ell+1}  \\
  5M \mathbf r^{\ell+1}|v|   &  M \mathbf{r}^{\ell+1} & 1 +M \mathbf{r}^{\ell+1} & \frac{M}{|v|}    & \frac{M}{|v|} \mathbf{r}^{\ell+1}   &  \frac{M}{|v|} \mathbf{r}^{\ell+1}  \\ \hline
  5M (\mathbf r^{\ell+1} )^2|v|^2  & M|v|( \mathbf{r}^{\ell+1})^{2} & M|v|( \mathbf{r}^{\ell+1})^{2} & 1+M\mathbf{r}^{\ell+1} & M (\mathbf{r}^{\ell+1})^{2} & M (\mathbf{r}^{\ell+1})^{2}\\
 5M \mathbf r^{\ell+1}|v|^2   & M|v|\mathbf{r}^{\ell+1} & M|v| \mathbf{r}^{\ell+1} & M  & 1+ M\mathbf{r}^{\ell+1} & M \mathbf{r}^{\ell+1}\\
    5M \mathbf r^{\ell+1}|v|^2   & M|v|\mathbf{r}^{\ell+1} & M|v| \mathbf{r}^{\ell+1} & M &  M\mathbf{r}^{\ell+1} & 1+M \mathbf{r}^{\ell+1}
 \end{array}\right]
 \\
 &:=     \underbrace{  J(\mathbf{r}^{\ell+1})   }_{\text{  Definition of } J(\mathbf{r}^{\ell+1})}.
\end{split}
\end{equation}
And if $\ell$ is \textit{Type I}, then
\begin{equation}\label{l_to_l+1delta}
\begin{split}
J^{\ell+1}_{\ell}&:= \frac{\partial (t^{\ell+1}, \mathbf{x}_{\parallel_{\ell+1}
}^{\ell+1}, \mathbf{v}_{\perp_{\ell+1} }^{\ell+1},\mathbf{v}_{\parallel_{\ell+1} }^{\ell+1})}{\partial (t^{\ell}, \mathbf{x}_{\parallel_{\ell}
}^{\ell},\mathbf{v}_{\perp_{\ell} }^{\ell},\mathbf{v}_{\parallel_{\ell} }^{\ell})}  \\
& \leq \left[\begin{array}{c|ccc|ccc}
 1+ M |t^\ell - t^{\ell + 1 } |  & M \mathbf{v}_{\perp }^{\ell+1}&  M \mathbf{v}_{\perp }^{\ell+1}& M &  M \mathbf{v}_{\perp }^{\ell+1} & M  \mathbf{v}_{\perp }^{\ell+1}\\ \hline
  M |t^\ell - t^{\ell + 1 } |   & 1 + M \mathbf{v}_{\perp }^{\ell+1} & M \mathbf{v}_{\perp }^{\ell+1} & M    & M \mathbf{v}_{\perp }^{\ell+1}   & M \mathbf{v}_{\perp }^{\ell+1}  \\
   M |t^\ell - t^{\ell + 1 } |   &  M \mathbf{v}_{\perp }^{\ell+1} & 1 +M \mathbf{v}_{\perp }^{\ell+1} & M    & M \mathbf{v}_{\perp }^{\ell+1}   &  M \mathbf{v}_{\perp }^{\ell+1}  \\ \hline
   M |t^\ell - t^{\ell + 1 } |^2  & M( \mathbf{v}_{\perp }^{\ell+1})^{2} & M( \mathbf{v}_{\perp }^{\ell+1})^{2} & 1+M\mathbf{v}_{\perp }^{\ell+1} & M (\mathbf{v}_{\perp }^{\ell+1})^{2} & M (\mathbf{v}_{\perp }^{\ell+1})^{2}\\
   M |t^\ell - t^{\ell + 1 } |   & M\mathbf{v}_{\perp }^{\ell+1} & M \mathbf{v}_{\perp }^{\ell+1} & M  & 1+ M\mathbf{v}_{\perp }^{\ell+1} & M \mathbf{v}_{\perp }^{\ell+1}\\
      M |t^\ell - t^{\ell + 1 } |   & M\mathbf{v}_{\perp }^{\ell+1} & M \mathbf{v}_{\perp }^{\ell+1} & M &  M\mathbf{v}_{\perp }^{\ell+1} & 1+M \mathbf{v}_{\perp }^{\ell+1}
 \end{array}\right]   \\
 \\ & \leq
 \left[\begin{array}{c|ccc|ccc}
 1 +5M \mathbf{v}_{\perp }^{\ell+1}  & M\mathbf{v}_{\perp }^{\ell+1} &  M\mathbf{v}_{\perp }^{\ell+1} & M & M\mathbf{v}_{\perp }^{\ell+1}  & M \mathbf{v}_{\perp }^{\ell+1} \\ \hline
 5M \mathbf{v}_{\perp }^{\ell+1}   & 1 + M\mathbf{v}_{\perp }^{\ell+1}  & M\mathbf{v}_{\perp }^{\ell+1}  &M    & M\mathbf{v}_{\perp }^{\ell+1}    & M\mathbf{v}_{\perp }^{\ell+1}   \\
  5M \mathbf{v}_{\perp }^{\ell+1}   &  M\mathbf{v}_{\perp }^{\ell+1}  & 1 +M\mathbf{v}_{\perp }^{\ell+1}  & M    &M\mathbf{v}_{\perp }^{\ell+1}    & M\mathbf{v}_{\perp }^{\ell+1}   \\ \hline
  5M (\mathbf{v}_{\perp }^{\ell+1}  )^2  & M(\mathbf{v}_{\perp }^{\ell+1} )^{2} & M(\mathbf{v}_{\perp }^{\ell+1} )^{2} & 1+M\mathbf{v}_{\perp }^{\ell+1}  & M (\mathbf{v}_{\perp }^{\ell+1} )^{2} & M (\mathbf{v}_{\perp }^{\ell+1} )^{2}\\
 5M \mathbf{v}_{\perp }^{\ell+1}   & M \mathbf{v}_{\perp }^{\ell+1}  & M\mathbf{v}_{\perp }^{\ell+1}  & M  & 1+ M\mathbf{v}_{\perp }^{\ell+1}  & M\mathbf{v}_{\perp }^{\ell+1} \\
    5M \mathbf{v}_{\perp }^{\ell+1}    & M\mathbf{v}_{\perp }^{\ell+1}  & M\mathbf{v}_{\perp }^{\ell+1}  & M &  M\mathbf{v}_{\perp }^{\ell+1}  & 1+M\mathbf{v}_{\perp }^{\ell+1} 
 \end{array}\right]
\\
 &:=     \underbrace{  J(\mathbf{v}_{\perp }^{\ell+1})   }_{\text{  Definition of } J(\mathbf{v}_{\perp }^{\ell+1})}.
\end{split}
\end{equation}

We also denote the Jacobian matrix within a single $\mathbf{p}^{\ell}-$ spherical coordinate:
\[
\tilde{J}_{\ell}^{\ell+1} : =  \frac{\partial (t^{\ell+1} ,\mathbf{x}_{\parallel_{\ell }
}^{\ell+1}, \mathbf{v}_{\perp_{\ell } }^{\ell+1},\mathbf{v}_{\parallel_{\ell } }^{\ell+1})}{\partial (t^{\ell} ,\mathbf{x}_{\parallel_{\ell}
}^{\ell},\mathbf{v}_{\perp_{\ell} }^{\ell},\mathbf{v}_{\parallel_{\ell} }^{\ell})}.
\]

We split the proof for each \textit{Type}:

\noindent\textit{Proof of (\ref{l_to_l+1delta}) (\textit{Type I}), and (\ref{l_to_l+1}) when $\ell$ is \textit{Type II}}: Note that $\mathbf{p}^{\ell}-$spherical coordinate is well-defined of all $\tau\in [t^{\ell+1}, t^{\ell}]$ for those cases. Due to the chart changing
\begin{equation*}
 { \frac{\partial (t^{\ell+1}, \mathbf{x}_{\parallel_{\ell+1}
}^{\ell+1}, \mathbf{v}_{\perp_{\ell+1} }^{\ell+1},\mathbf{v}_{\parallel_{\ell+1} }^{\ell+1})}{\partial (t^{\ell}, \mathbf{x}_{\parallel_{\ell}
}^{\ell},\mathbf{v}_{\perp_{\ell} }^{\ell},\mathbf{v}_{\parallel_{\ell} }^{\ell})} } =
  \left[\begin{array}{c|c} 1 & \mathbf{0}_{1,5}\\  \hline \mathbf{0}_{5,1} & 
 \frac{\partial ( \mathbf{x}_{\parallel_{\ell+1}
}^{\ell+1}, \mathbf{v}_{\perp_{\ell+1} }^{\ell+1},\mathbf{v}_{\parallel_{\ell+1} }^{\ell+1})}{\partial ( \mathbf{x}_{\parallel_{\ell}
}^{\ell},\mathbf{v}_{\perp_{\ell} }^{\ell},\mathbf{v}_{\parallel_{\ell} }^{\ell})}
  \end{array} \right]
 \underbrace{  \frac{\partial (t^{\ell+1},0,\mathbf{x}_{\parallel_{\ell }
}^{\ell+1}, \mathbf{v}_{\perp_{\ell } }^{\ell+1},\mathbf{v}_{\parallel_{\ell } }^{\ell+1})}{\partial (t^{\ell},0,\mathbf{x}_{\parallel_{\ell}
}^{\ell},\mathbf{v}_{\perp_{\ell} }^{\ell},\mathbf{v}_{\parallel_{\ell} }^{\ell})}}_{=\tilde{J}_{\ell}^{\ell+1}}.
\end{equation*}
where $ \frac{\partial ( \mathbf{x}_{\parallel_{\ell+1}
}^{\ell+1}, \mathbf{v}_{\perp_{\ell+1} }^{\ell+1},\mathbf{v}_{\parallel_{\ell+1} }^{\ell+1})}{\partial ( \mathbf{x}_{\parallel_{\ell}
}^{\ell},\mathbf{v}_{\perp_{\ell} }^{\ell},\mathbf{v}_{\parallel_{\ell} }^{\ell})}$ is the $5\times 5$ right lower submatrix of (\ref{chart_changing}).

Note that $|\mathbf{p}^{\ell}-\mathbf{p}^{\ell+1}|\lesssim_{\xi}   \sqrt \delta$ from 
$(\ref{pl})$. In order to show (\ref{l_to_l+1}) and \eqref{l_to_l+1delta} it suffices to show that ${\tilde{J}_{\ell}^{\ell+1}}$ is bounded:
\begin{equation}\label{tilde_J}
\begin{split}
\tilde{J}_{\ell}^{\ell+1}  \ \leq  \  J(\mathbf{r}^{\ell+1}), \text{ if } \ell \text{ is } \textit{Type II } \text{ or } \textit{Type III},
\\ \tilde{J}_{\ell}^{\ell+1}  \ \leq  \  J(\mathbf{v}_{\perp}^{\ell+1}), \text{ if } \ell \text{ is } \textit{ Type I}.
\end{split}
\end{equation}

This is due to the following matrix multiplication
\begin{equation}\notag
\begin{split}
  &\left[\begin{array}{c|c} 1 & \mathbf{0}_{1,5}\\  \hline \mathbf{0}_{5,1} & \frac{\partial ( \mathbf{x}_{\parallel_{\ell+1}
}^{\ell+1}, \mathbf{v}_{\perp_{\ell+1} }^{\ell+1},\mathbf{v}_{\parallel_{\ell+1} }^{\ell+1})}{\partial ( \mathbf{x}_{\parallel_{\ell}
}^{\ell},\mathbf{v}_{\perp_{\ell} }^{\ell},\mathbf{v}_{\parallel_{\ell} }^{\ell})}\end{array} \right]\tilde{J}_{\ell}^{\ell+1}\\
  &\leq
 \left[
\begin{array}{c|cc|ccc}
1 &   \mathbf{0}_{1,2}  &  &&  \mathbf{0}_{1,3} & \\ \hline
  \mathbf{0}_{2,1} 
   & 1+ C\mathbf{r}^{\ell+1} &  C\mathbf{r}^{\ell+1} &  & \mathbf{0}_{3,3} & \\
&   C\mathbf{r}^{\ell+1} &  1+C\mathbf{r}^{\ell+1} &  & &   \\ \hline
&  0 & 0 & 1 & 0 & 0 \\
  \mathbf{0}_{3,1}   &C\mathbf{r}^{\ell+1} |v| & C\mathbf{r}^{\ell+1} |v|  & 0 &  1+C\mathbf{r}^{\ell+1} &  C\mathbf{r}^{\ell+1} \\
&   C\mathbf{r}^{\ell+1}|v|  &  C\mathbf{r}^{\ell+1}|v|  & 0 &  C\mathbf{r}^{\ell+1} & 1+C\mathbf{r}^{\ell+1}
\end{array}
\right]J(\mathbf{r}^{\ell+1}) \leq J( C \mathbf{r}^{\ell+1}), \text{ if } |v|> \delta,
\\ &\left[\begin{array}{c|c} 1 & \mathbf{0}_{1,5}\\  \hline \mathbf{0}_{5,1} & \frac{\partial ( \mathbf{x}_{\parallel_{\ell+1}
}^{\ell+1}, \mathbf{v}_{\perp_{\ell+1} }^{\ell+1},\mathbf{v}_{\parallel_{\ell+1} }^{\ell+1})}{\partial ( \mathbf{x}_{\parallel_{\ell}
}^{\ell},\mathbf{v}_{\perp_{\ell} }^{\ell},\mathbf{v}_{\parallel_{\ell} }^{\ell})}\end{array} \right]\tilde{J}_{\ell}^{\ell+1}\\
  &\leq
 \left[
\begin{array}{c|cc|ccc}
1 &   \mathbf{0}_{1,2}  &  &&  \mathbf{0}_{1,3} & \\ \hline
  \mathbf{0}_{2,1} 
   & 1+ C\mathbf{r}^{\ell+1} &  C\mathbf{r}^{\ell+1} &  & \mathbf{0}_{3,3} & \\
&   C\mathbf{r}^{\ell+1} &  1+C\mathbf{r}^{\ell+1} &  & &   \\ \hline
&  0 & 0 & 1 & 0 & 0 \\
  \mathbf{0}_{3,1}   &C\mathbf{r}^{\ell+1} |v| & C\mathbf{r}^{\ell+1} |v|  & 0 &  1+C\mathbf{r}^{\ell+1} &  C\mathbf{r}^{\ell+1} \\
&   C\mathbf{r}^{\ell+1}|v|  &  C\mathbf{r}^{\ell+1}|v|  & 0 &  C\mathbf{r}^{\ell+1} & 1+C\mathbf{r}^{\ell+1}
\end{array}
\right]J(\mathbf{v}_\perp^{\ell+1}) \leq J( C \mathbf{v}_\perp^{\ell+1}), \text{ if } |v|\le \delta,
\end{split}
\end{equation}
where we used (\ref{chart_changing}) with an adjusted constant $C >0$.

Now we prove the claim (\ref{tilde_J}). We fix the $\mathbf{p}^{\ell}-$spherical coordinate and drop the index $\ell$ for the chart.


If $\mathbf{v}_{\perp}^{\ell}=0$ then $t^{\ell+1}=t^{\ell}$. Otherwise if $\mathbf{v}_{\perp}^{\ell} \neq 0$ then $t^{\ell+1}$ is determined through
\begin{equation}
0=\mathbf{v}_{\perp}^{\ell}(t^{\ell+1}-t^{\ell})+\int^{t^{\ell}  }_{ t^{\ell+1}}\int^{t^{\ell}}_{s}F_{\perp }(
\mathbf{X}_{\ell}(\tau; t^{\ell}, x^{\ell},v^{\ell}), \mathbf{V}_{\ell}(\tau;t^{\ell}, x^{\ell},v^{\ell} )
)\mathrm{d}\tau
\mathrm{d}s.\label{define}
\end{equation}

We first consider the $\frac{\p}{\p t^\ell}$ derivatives. 

Using the trajectory in the standard coordinates we have
\Be
0 = \xi(x^{\ell + 1 } ) = \xi\left(x^\ell - (t^\ell - t^{\ell +1 } ) v^\ell + \int_{t^{\ell+1}}^{t^\ell} \int_s^{t^\ell } E(\tau, X(\tau ) ) d\tau ds \right).
\Ee
Taking the $\frac{\p}{\p t^\ell}$ derivative we get
\Be \begin{split}
0 &= \nabla \xi(x^{\ell+1}) \cdot  \left[ -(1 - \frac{ \p t^{\ell+1}}{\p t^\ell } ) v^\ell  -  \frac{ \p t^{\ell+1}}{\p t^\ell } \int_{t^{\ell + 1}}^{t^\ell } E(\tau, X(\tau) ) d\tau + \int_{t^{\ell + 1}}^{t^\ell } E(t^\ell, x^\ell ) ds + \int_{t^{\ell+1}}^{t^\ell} \int_s^{t^\ell } \frac{ \p E(\tau, X(\tau ) ) }{\p t^\ell } d\tau ds \right]
\\ & =  \nabla \xi(x^{\ell+1}) \cdot  \left[ -v^\ell + \frac{ \p t^{\ell+1}}{\p t^\ell } v^{\ell + 1 } +   \int_{t^{\ell+1}}^{t^\ell} E(s,X(s)) ds +  \int_{t^{\ell+1}}^{t^\ell} \left( E(t^\ell, x^\ell ) - E(s,X(s))  \right) ds + \int_{t^{\ell+1}}^{t^\ell} \int_s^{t^\ell } \frac{ \p E(\tau, X(\tau ) ) }{\p t^\ell } d\tau ds \right]
\\ & = \nabla \xi(x^{\ell+1}) \cdot \left[ - v^{\ell + 1 }+ \frac{ \p t^{\ell+1}}{\p t^\ell } v^{\ell + 1 } +  \int_{t^{\ell+1}}^{t^\ell} \int_s^{t^\ell } \left( \frac{ \p E(\tau, X(\tau ) ) }{\p \tau} + \frac{ \p E(\tau, X(\tau ) ) }{\p t^\ell }  \right) d\tau ds  \right].
\end{split} \Ee
Thus
\Be \label{ptelltell1}
\frac{ \p t^{\ell+1}}{\p t^\ell } = 1 - \frac{\nabla \xi (x^{\ell+1}) }{\nabla \xi (x^{\ell + 1} ) \cdot v^{\ell + 1 } } \cdot  \int_{t^{\ell+1}}^{t^\ell} \int_s^{t^\ell } \left( \frac{ \p E(\tau, X(\tau ) ) }{\p \tau} + \frac{ \p E(\tau, X(\tau ) ) }{\p t^\ell }  \right) d\tau ds.
\Ee
By \eqref{Xclsellstardif} we have
\Be \label{ptauptlEdiff} 
 \begin{split}
\left| \frac{ \p E(\tau, X(\tau ) ) }{\p \tau} + \frac{ \p E(\tau, X(\tau ) ) }{\p t^\ell } \right| & = | \p_s E(\tau, X(\tau)) _\infty +  \nabla_x E \cdot \left( V(\tau) - V(\tau ) + O(1)|t^\ell - t^{\ell +1 } | \right) | 
\\ & =  | \p_s E(\tau, X(\tau)) _\infty +  O(1) \nabla_x E(\tau, X(\tau)) (t^\ell - t^{\ell +1 } ) | 
\\ &\lesssim   \| \p_t E \|_{L^\infty_{t,x}} + \| \nabla_x E \|_{L^\infty_{t,x}} |t^\ell - t^{\ell +1 } |
\end{split}. \Ee
Thus from \eqref{ptelltell1}, \eqref{ptauptlEdiff}, and \eqref{tbest} we have
\Be \label{ptelltell1est}
\frac{ \p t^{\ell+1}}{\p t^\ell } = 1 - O_{\xi, E }(1) \frac{|t^\ell - t^{\ell+1} |^2}{ |v_\perp^{\ell+1} | }  = 1 - O_{\xi,E}(1)|t^\ell - t^{\ell+1} |.
\Ee
Now by directly computing $\frac{\p \mathbf{v}_\perp^{\ell+1} }{\p t^\ell}$ we would have
\Be \label{tlvl1exp}
\frac{\p \mathbf{v}_\perp^{\ell+1} }{\p t^\ell} = \frac{- F_\perp(t^{\ell+1} )}{\mathbf{v}_\perp^{\ell+1} } \int_{t^{\ell+1}}^{t^\ell} \int_s^{t^\ell } \left( \frac{ \p F_\perp(\tau) }{\p \tau} + \frac{ \p F_\perp(\tau) }{\p t^\ell }  \right) d\tau ds +  \int_{t^{\ell+1}}^{t^\ell}  \left(  \frac{ \p F_\perp(s) }{\p s} + \frac{ \p F_\perp(s) }{\p t^\ell }  \right) ds.
\Ee

Recall
\begin{equation}
\begin{split}
F_{\perp_{ }} = & F_{\perp_{}}(\mathbf{x}_{\perp_{ }}, \mathbf{x}_{\parallel_{ }},   \mathbf{v}_{\parallel_{ }})\\
= & \sum_{j,k=1}^{2} \mathbf{v}_{\parallel_{ },k} \mathbf{v}_{\parallel_{ },j} \ \partial_{j} \partial_{k} \mathbf{\eta}_{ }(\mathbf{x}_{\parallel_{ }}) \cdot \mathbf{n}_{ }(\mathbf{x}_{\parallel_{ }})  
  - \mathbf{x}_{\perp_{ }} \sum_{k=1}^{2} \mathbf{v}_{\parallel_{ },k} (\mathbf{v}_{\parallel_{ }}\cdot \nabla) \partial_{k} \mathbf{n}_{ }(\mathbf{x}_{\parallel_{ }}) \cdot \mathbf{n}_{ }(\mathbf{x}_{\parallel_{ }})
  \\ &  -E(s, - \mathbf{x}_\perp \mathbf n (\mathbf x_\parallel ) + \mathbf \eta (\mathbf x_\parallel ) ) \cdot \mathbf n (\mathbf x_\parallel ),
\end{split}
\end{equation}
So by direct computation
\Be
\dot F_\perp(\tau) := \frac{\p F_\perp(\tau) }{\p \tau } = \mathbf v_\perp \nabla_{\mathbf x_\perp} F_\perp + \mathbf v_\parallel \nabla_{\mathbf x_\parallel} F_\perp + F_\parallel \nabla_{\mathbf v_\parallel } F_\perp - \p_s E \cdot n(\mathbf x_\parallel),
\Ee
so $\| \nabla_{\mathbf x_\perp} \dot F_\perp \|_\infty + \|  \nabla_{\mathbf x_\parallel} \dot F_\perp \|_\infty \lesssim |v|^3 +1  $, and $\| \nabla_{\mathbf v_\perp} \dot F_\perp \|_\infty + \|  \nabla_{\mathbf v_\parallel} \dot F_\perp \|_\infty \lesssim |v|^2 +1  $. Thus together with \eqref{dtellall} we have
\Be \label{Fperpdottimed} \begin{split}
& \left|  \frac{d}{d\tau} \left(  \frac{ \p F_\perp(\tau) }{\p \tau} + \frac{ \p F_\perp(\tau) }{\p t^\ell } \right) \right| 
\\ = & \left|   \frac{ \p \dot F_\perp(\tau) }{\p \tau} + \frac{ \p  \dot F_\perp(\tau) }{\p t^\ell }  \right|
\\  = &  \left| \nabla_{\mathbf{x}_\perp} \dot F_\perp \cdot \left( \mathbf{v}_\perp(s) + \frac{ \p \mathbf x_\perp(s) } {\p t^\ell} \right) + \nabla_{\mathbf{x}_\parallel} \dot F_\perp \cdot  \left( \mathbf{v}_\parallel(s) +  \frac{ \p \mathbf x_\parallel(s) } {\p t^\ell}  \right)+ \nabla_{\mathbf{v}_\parallel} \dot F_\perp \cdot  \left( F_\parallel(s) + \frac{ \p \mathbf v_\parallel(s) } {\p t^\ell} \right) \right.
\\ & \left.+ \nabla_{\mathbf{v}_\perp} \dot F_\perp \cdot  \left( F_\perp(s) + \frac{ \p \mathbf v_\perp(s) } {\p t^\ell} \right) - \p_s^2 E \cdot \mathbf n(\mathbf{x}_\parallel) \right|
\\  \lesssim & (|v|^3 +1 )  \int_s^{t^\ell} \int_\tau^{t^\ell} \left(  \left|  \p_{\tau'} F_\perp (\tau') + \p_{t^\ell} F_\perp(\tau') \right| + \left| \p_{\tau'} F_\parallel (\tau') + \p_{t^\ell} F_\parallel(\tau') \right| \right)  d\tau' d\tau
 \\ & + (|v|^2+1) \int_s^{t^\ell}   \left(  \left|  \p_{\tau'} F_\perp (\tau') + \p_{t^\ell} F_\perp(\tau') \right| + \left| \p_{\tau'} F_\parallel (\tau') + \p_{t^\ell} F_\parallel(\tau') \right| \right)  d\tau' + \| \p_t^2 E \|_{L^\infty_{t,x}}
 \\  \lesssim & \| \p_t^2 E \|_{L^\infty_{t,x}} + (|v|^3+1)(t^\ell - t^{\ell+1} )^2 + (|v|^2+1)(t^\ell - t^{\ell+1} )
 \\  \lesssim &  \| \p_t^2 E \|_{L^\infty_{t,x}} + |v| + 1.
\end{split} \Ee
Combing \eqref{Ftimedest}, \eqref{tlvl1exp}, \eqref{Fperpdottimed}, and expanding $ \frac{ \p F_\perp(\tau) }{\p \tau} + \frac{ \p F_\perp(\tau) }{\p t^\ell } $ at $t^\ell$ we get
\Be \label{pvelltlpartial}
\begin{split}
\frac{\p \mathbf{v}_\perp^{\ell+1} }{\p t^\ell} 
= \left(  \frac{ \p F_\perp(t^\ell) }{\p \tau} + \frac{ \p F_\perp(t^\ell) }{\p t^\ell } \right) \left( \frac{ F_\perp(t^{\ell+1} )}{\mathbf{v}_\perp^{\ell+1} } \frac{|t^\ell - t^{\ell+1} |^2}{2}  -  |t^\ell - t^{\ell+1} | \right) + O_{\| E\|_{C^2}, \O }(1) |t^\ell - t^{\ell+1} |^2 (|v|+1).
\end{split} \Ee
Now since we have
\Be \label{fperpenest}  \begin{split}
0 = \mathbf{x}_\perp^{\ell}  &=\mathbf{x}_\perp^{\ell+1}  + \int_{t^{\ell+1}}^{t^\ell} \mathbf{v}_\perp(s) ds  
\\ &=  \int_{t^{\ell+1}}^{t^\ell}  \left( - \mathbf{v}_\perp^{\ell+1} + \int_{t^{\ell+1}}^s F_\perp (\tau ) d\tau \right) ds 
\\ &= (t^\ell - t^{\ell+1} )(- \mathbf{v}_\perp^{\ell+1} ) + \int_{t^{\ell+1}}^{t^\ell} \int_{t^{\ell+1}}^s F_\perp(\tau) d \tau d s
\\ & =(t^\ell - t^{\ell+1} )(- \mathbf{v}_\perp^{\ell+1} ) + \frac{|t^\ell - t^{\ell+1} | ^2}{2}  F_\perp(t^{\ell+1} ) +  O(1)(\| \p_t E \|_{L^\infty_{t,x}} + |v|^3 ) |t^\ell - t^{\ell+1} |^3,
\end{split} \Ee
we get the following important cancellation identity:
\Be \label{Fperpcancel}
\frac{ F_\perp(t^{\ell+1} )}{\mathbf{v}_\perp^{\ell+1} } \frac{|t^\ell - t^{\ell+1} |^2}{2}  -  |t^\ell - t^{\ell+1} | = O(1) (\| \p_t E \|_{L^\infty_{t,x}} + |v|^3 )\frac{ |t^\ell - t^{\ell+1} |^3}{\mathbf{v}_\perp^{\ell+1} }.
\Ee
By \eqref{pvelltlpartial} and \eqref{Fperpcancel} we get
\Be
| \frac{\p \mathbf{v}_\perp^{\ell+1} }{\p t^\ell} |  \lesssim \left( \| \p_t E \|_{L^\infty_{t,x}}^2 + \| \p_t^2 E \|_{L^\infty_{t,x}} +1 \right) \left( |v| |t^\ell - t^{\ell+1} |^2 +  |t^\ell - t^{\ell+1} |^2 \right).
\Ee
Next, taking $\frac{\p}{\p t^\ell }$ derivative to $\mathbf v_\parallel^{\ell +1} = \mathbf v_\parallel^\ell  -\int_{t^{\ell+1}}^{t^\ell} F_\parallel (s) ds$, and $\mathbf{x}_\parallel^{\ell+1} = \mathbf{x}_\parallel^{\ell} - (t^\ell - t^{\ell+1} ) \mathbf v_\parallel^\ell + \int_{t^{\ell+1}}^{t^\ell} \int_s^{t^\ell}  F_\parallel (\tau) d\tau ds$ we get
\[ \begin{split}
\frac{\p \mathbf v_\parallel^{\ell+1}}{\p t^\ell } &= - F_\parallel(t^\ell) + \frac{\p t^{\ell+1}}{\p t^\ell} F_\parallel (t^\ell) - \int_{t^{\ell+1}}^{t^\ell} \p_{t^\ell} F_\parallel (s) ds
\\ & = F_\parallel (t^{\ell+1}) - F_\parallel (t^\ell ) + O(1) \frac{ |t^\ell -t^{\ell+1} |^2 }{| \mathbf v_\perp^{\ell+1} | } - \int_{t^{\ell+1}}^{t^\ell} \p_{t^\ell} F_\parallel(s) ds
\\ & = O(1) \frac{ |t^\ell -t^{\ell+1} |^2 }{| \mathbf v_\perp^{\ell+1} | } - \int_{t^{\ell+1}}^{t^\ell}  \left( \p_s F_\parallel (s) + \p_{t^\ell} F_\parallel(s)   \right) ds \lesssim  |t^\ell -t^{\ell+1} |,
\end{split}\]
and
\[ \begin{split}
\frac{\p \mathbf x_\parallel^{\ell+1}}{\p t^\ell } &= -\mathbf v_\parallel^\ell + \frac{\p t^{\ell+1}}{\p t^\ell} \mathbf v_\parallel^{\ell+1} + \int_{t^{\ell+1}}^{t^\ell} F_\parallel (t^\ell) ds + \int_{t^{\ell+1}}^{t^\ell} \int_s^{t^\ell}  \p_{t^\ell} F_\parallel (\tau) d\tau ds
\\ & = \mathbf v_\parallel^{\ell+1} - \mathbf v_\parallel^\ell - O(1) \frac{ |t^\ell -t^{\ell+1} |^2 }{| \mathbf v_\perp^{\ell+1} | } \mathbf v_\parallel^{\ell+1} + \int_{t^{\ell+1}}^{t^\ell} F_\parallel (t^\ell) ds + \int_{t^{\ell+1}}^{t^\ell} \int_s^{t^\ell}  \p_{t^\ell} F_\parallel (\tau) d\tau ds
\\ & = \int_{t^{\ell+1}}^{t^\ell} \left(  F_\parallel(t^\ell) - F_\parallel(s)  \right) ds+ \int_{t^{\ell+1}}^{t^\ell} \int_s^{t^\ell}  \p_{t^\ell} F_\parallel (\tau) d\tau ds - O(1) \frac{ |t^\ell -t^{\ell+1} |^2 }{| \mathbf v_\perp^{\ell+1} | } \mathbf v_\parallel^{\ell+1}
\\ & = \int_{t^{\ell+1}}^{t^\ell} \int_s^{t^\ell} \left( \p_s F_\parallel (\tau) + \p_{t^\ell} F_\parallel(\tau)   \right) d\tau ds + O(1) |t^\ell -t^{\ell+1} | \lesssim  |t^\ell -t^{\ell+1} |.
\end{split}\]
Where we've used \eqref{Ftimedest} and \eqref{ptelltell1est}. This proves the first column of \eqref{l_to_l+1} and \eqref{l_to_l+1delta}.

Taking derivatives of (\ref{define}) as before and using $|t^{\ell}-t^{\ell+1}| \lesssim_{\xi,t} \min \{\frac{|\mathbf{v}_{\perp}^{\ell+1}|}{|v|^{2}},1\}$ and Lemma \ref{lemma_flow},
\begin{equation}\label{tx}
\left[\begin{array}{c}
\frac{\partial t^{\ell+1}}{\partial \mathbf{x}_{\parallel }^{\ell}}\\
\frac{\partial t^{\ell+1}}{\partial \mathbf{v}_{\perp }^{\ell}}\\
\frac{\partial t^{\ell+1}}{\partial \mathbf{v}_{\parallel }^{\ell}}
\end{array}
\right]
= \left[\begin{array}{c}
\frac{1}{\mathbf{v}_{\perp
}^{\ell+1}}\int_{t^{\ell}}^{t^{\ell+1}}\int_{t^{\ell}}^{s}\frac{\partial }{\partial
\mathbf{x}_{\parallel }^{\ell}}F_{\perp }(   \mathbf{X}_{\ell}(\tau),   \mathbf{V}_{\ell}(\tau))\mathrm{d}\tau \mathrm{d}s\\
\frac{1}{\mathbf{v}_{\perp }^{\ell+1}}%
\left\{ (t^{\ell+1}-t^{\ell})+\int_{t^{\ell}}^{t^{\ell+1}}\int_{t^{\ell}}^{s}\frac{\partial
}{\partial \mathbf{v}_{\perp }^{\ell}}F_{\perp }(   \mathbf{X}_{\ell}(\tau), \mathbf{V}_{\ell}(\tau) )\mathrm{d}\tau \mathrm{d}%
s\right\}\\
\frac{1}{\mathbf{v}_{\perp
}^{\ell+1}}\int^{t^{\ell}}_{t^{\ell+1}}\int^{t^{\ell}}_{s}\frac{\partial }{\partial
\mathbf{v}_{\parallel }^{\ell}}F_{\perp }(  \mathbf{X}_{\ell}(\tau),   \mathbf{V}_{\ell}(\tau))\mathrm{d}\tau \mathrm{d}s
\end{array} \right]
\lesssim_{\xi, t}
\left[\begin{array}{c}
\frac{|t^\ell - t^{\ell+1}|^2}{|\mathbf{v}_\perp^{\ell+1}|}( |v|^2 +O(1))
\\
\frac{|t^\ell - t^{\ell+1}|}{|\mathbf{v}_\perp^{\ell+1}|} +\frac{|t^\ell - t^{\ell+1}|^2}{|\mathbf{v}_\perp^{\ell+1}|}( |v| +O(1))
\\
\frac{|t^\ell - t^{\ell+1}|^2}{|\mathbf{v}_\perp^{\ell+1}|}( |v| +O(1))
\end{array} \right].
\end{equation}
Thus from \eqref{tbest} we have
\Be \begin{split}
\left[\begin{array}{c}
\frac{\partial t^{\ell+1}}{\partial \mathbf{x}_{\parallel }^{\ell}}\\
\frac{\partial t^{\ell+1}}{\partial \mathbf{v}_{\perp }^{\ell}}\\
\frac{\partial t^{\ell+1}}{\partial \mathbf{v}_{\parallel }^{\ell}}
\end{array}
\right] \lesssim
\left[\begin{array}{c}
\frac{1}{|v|} \frac{|\mathbf{v}_{\perp}^{\ell+1}|}{|v|} \\
\frac{1}{|v|^{2}} \\
 \frac{1}{|v|^{2}} \frac{|\mathbf{v}_{\perp}^{\ell+1}|}{|v|} 
\end{array} \right], \text{ for } |v| > \delta. 
\quad
\left[\begin{array}{c}
\frac{\partial t^{\ell+1}}{\partial \mathbf{x}_{\parallel }^{\ell}}\\
\frac{\partial t^{\ell+1}}{\partial \mathbf{v}_{\perp }^{\ell}}\\
\frac{\partial t^{\ell+1}}{\partial \mathbf{v}_{\parallel }^{\ell}}
\end{array}
\right] \lesssim
\left[\begin{array}{c}
|\mathbf{v}_\perp^{\ell+1}|\\
O(1)\\
|\mathbf{v}_\perp^{\ell+1}|
\end{array} \right], \text{ for } |v| \le \delta.
\end{split} \Ee
 %
Taking $(\mathbf{x}(t^\ell),\mathbf{v}(t^\ell))$ derivatives of the characteristic equations 
\[
\mathbf{x}_\parallel^{\ell+1} = \mathbf{x}_\parallel^{\ell} - \int_{t^{\ell+1}}^{t^\ell} \mathbf{v}_\parallel(s;t^\ell x^\ell, v^\ell ) ds,
\]
 by Lemma \ref{lemma_flow} and (\ref{tx}), we estimate directly
%
%
%
\begin{equation}\notag
 \left[\begin{array}{c}
 \frac{\partial \mathbf{x}_{\parallel }^{\ell+1}}{\partial \mathbf{x}_{\parallel }^{\ell}}\\
 \frac{\partial \mathbf{x}_{\parallel }^{\ell+1}}{\partial \mathbf{v}_{\perp }^{\ell}}\\
 \frac{\partial \mathbf{x}_{\parallel }^{\ell+1}}{\partial \mathbf{v}_{\parallel }^{\ell}}
\end{array}\right] 
\lesssim_{\xi, t}
 \left[\begin{array}{c}
  \mathbf{Id}_{2,2} + \frac{|t^\ell - t^{\ell+1}|^2|v|^3}{|\mathbf{v}_\perp^{\ell+1}|} + O(1)\frac{|t^\ell - t^{\ell+1}|^2|v|}{|\mathbf{v}_\perp^{\ell+1}|}\\
O(1)\frac{|t^\ell - t^{\ell+1}||v|}{|\mathbf{v}_\perp^{\ell+1}|} + |t^\ell - t^{\ell+1}|
  \\
\frac{|t^\ell - t^{\ell+1}|^2}{|\mathbf{v}_\perp^{\ell+1}|}( |v|^2 +O(1)|v|) + |t^\ell - t^{\ell+1}|
 \end{array}\right] .
 \end{equation}
 Thus from \eqref{tbest} we have
 \begin{equation}\notag
 \left[\begin{array}{c}
 \frac{\partial \mathbf{x}_{\parallel }^{\ell+1}}{\partial \mathbf{x}_{\parallel }^{\ell}}\\
 \frac{\partial \mathbf{x}_{\parallel }^{\ell+1}}{\partial \mathbf{v}_{\perp }^{\ell}}\\
 \frac{\partial \mathbf{x}_{\parallel }^{\ell+1}}{\partial \mathbf{v}_{\parallel }^{\ell}}
\end{array}\right] 
\lesssim_{\xi, t}
 \left[\begin{array}{c}
  \mathbf{Id}_{2,2} + \frac{|\mathbf{v}_{\perp}^{\ell}|}{|v|}  \\
  \frac{1}{|v|}\\
   \frac{1}{|v|} \frac{|\mathbf{v}_{\perp}^{\ell+1}|}{|v|}
 \end{array}\right] , \text{ for } |v| > \delta.
 \quad
  \left[\begin{array}{c}
 \frac{\partial \mathbf{x}_{\parallel }^{\ell+1}}{\partial \mathbf{x}_{\parallel }^{\ell}}\\
 \frac{\partial \mathbf{x}_{\parallel }^{\ell+1}}{\partial \mathbf{v}_{\perp }^{\ell}}\\
 \frac{\partial \mathbf{x}_{\parallel }^{\ell+1}}{\partial \mathbf{v}_{\parallel }^{\ell}}
\end{array}\right] 
\lesssim_{\xi, t}
 \left[\begin{array}{c}
  \mathbf{Id}_{2,2} +|\mathbf{v}_{\perp}^{\ell}|   \\
 O(1) \\
 |\mathbf{v}_{\perp}^{\ell}|
 \end{array}\right] , \text{ for } |v| \le \delta.
 \end{equation}
Also,
 \begin{equation} \notag
  \left[\begin{array}{c}
  \frac{\partial  \mathbf{v}_{\parallel }^{\ell+1}}{\partial \mathbf{x}_{\parallel }^{\ell}}\\
  \frac{\partial  \mathbf{v}_{\parallel }^{\ell+1}}{\partial  \mathbf{v}_{\perp }^{\ell}}\\
  \frac{\partial  \mathbf{v}_{\parallel }^{\ell+1}}{\partial  \mathbf{v}_{\parallel }^{\ell}}
  \end{array}\right] \lesssim_{\xi, t}
    \left[\begin{array}{c}
\frac{|t^\ell - t^{\ell+1}|^2}{|\mathbf{v}_\perp^{\ell+1}|}( |v|^2 +O(1)|v|)^2 + |t^\ell - t^{\ell+1}|(|v|^2+O(1))
   \\
     ( |v|^2+O(1))\left( \frac{|t^\ell - t^{\ell+1}|}{|\mathbf{v}_\perp^{\ell+1}|} + \frac{|t^\ell - t^{\ell+1}|^2}{|\mathbf{v}_\perp^{\ell+1}|}( |v|^2 +O(1)|v|) \right) +|t^\ell - t^{\ell+1}|(|v|+O(1))
       \\
   \mathbf{Id}_{2,2}+  ( |v|^2 +O(1)|v|)  \frac{|t^\ell - t^{\ell+1}|^2}{|\mathbf{v}_\perp^{\ell+1}|}(|v|+O(1))+|t^\ell - t^{\ell+1}|(|v|+O(1))
      \end{array}\right].
\end{equation}
 Thus from \eqref{tbest} we have
 \begin{equation} \notag
  \left[\begin{array}{c}
  \frac{\partial  \mathbf{v}_{\parallel }^{\ell+1}}{\partial \mathbf{x}_{\parallel }^{\ell}}\\
  \frac{\partial  \mathbf{v}_{\parallel }^{\ell+1}}{\partial  \mathbf{v}_{\perp }^{\ell}}\\
  \frac{\partial  \mathbf{v}_{\parallel }^{\ell+1}}{\partial  \mathbf{v}_{\parallel }^{\ell}}
  \end{array}\right] \lesssim_{\xi, t}
    \left[\begin{array}{c}
   | \mathbf{v}_{\perp}^{\ell+1}| \\
      1+\frac{| \mathbf{v}_{\perp}^{\ell+1}|}{|v^{\ell}|} \\
       \mathbf{Id}_{2,2} +  \frac{| \mathbf{v}_{\perp}^{\ell+1}|}{|v^{\ell}|}
      \end{array}\right], \text{ for } |v| > \delta.
      \quad
       \left[\begin{array}{c}
  \frac{\partial  \mathbf{v}_{\parallel }^{\ell+1}}{\partial \mathbf{x}_{\parallel }^{\ell}}\\
  \frac{\partial  \mathbf{v}_{\parallel }^{\ell+1}}{\partial  \mathbf{v}_{\perp }^{\ell}}\\
  \frac{\partial  \mathbf{v}_{\parallel }^{\ell+1}}{\partial  \mathbf{v}_{\parallel }^{\ell}}
  \end{array}\right] \lesssim_{\xi, t}
    \left[\begin{array}{c}
   | \mathbf{v}_{\perp}^{\ell+1}| \\
      1+| \mathbf{v}_{\perp}^{\ell+1}| \\
       \mathbf{Id}_{2,2} +  | \mathbf{v}_{\perp}^{\ell+1}|
      \end{array}\right], \text{ for } |v| \le \delta.
\end{equation}

Now we move to $D \mathbf{v}_{\perp}^{\ell+1}$ estimates.
Taking derivatives in (\ref{v_ell}), from the extra cancellation in
terms of order of $t^{\ell}-t^{\ell+1}$ in (\ref{Fperpcancel}), by (\ref{tx}), and plugging the expansion
\[
\frac{\partial }{\partial \mathbf{x}_{\parallel }^{\ell}}F_{\perp }( \mathbf{X}_{\ell}(\tau), \mathbf{V}_{\ell}(\tau) ) = \frac{\partial }{\partial \mathbf{x}_{\parallel }^{\ell} } F_\perp(\mathbf x^\ell , \mathbf v^\ell ) - \int_\tau^{t^\ell } \frac{d}{d\tau '} \left(  \frac{\partial }{\partial \mathbf{x}_{\parallel }^{\ell}}F_{\perp }( \mathbf{X}_{\ell}(\tau'), \mathbf{V}_{\ell}(\tau') )  \right) d\tau'
\]
into
\[
\frac{\partial \mathbf{v}_{\perp }^{\ell+1}}{\partial \mathbf{x}_{\parallel }^{\ell}} =\frac{%
-F_{\perp }(\mathbf{x}^{\ell+1}, \mathbf{v}^{\ell+1})}{\mathbf{v}_{\perp }^{\ell+1}}\int^{t^{\ell}}_{t^{\ell+1}}\int^{t^{\ell}}_{s}%
\frac{\partial }{\partial \mathbf{x}_{\parallel }^{\ell}}F_{\perp }( \mathbf{X}_{\ell}(\tau), \mathbf{V}_{\ell}(\tau) )\mathrm{d}\tau
\mathrm{d}s+\int^{t^{\ell}}_{t^{\ell+1}}\frac{\partial }{\partial \mathbf{x}_{\parallel
}^{\ell}}F_{\perp }(\mathbf{X}_{\ell}(\tau), \mathbf{V}_{\ell}(\tau) )\mathrm{d}\tau,
\]
and using the cancellation \eqref{Fperpcancel} we obtain
\begin{equation}
\begin{split}
\frac{\partial \mathbf{v}_{\perp }^{\ell+1}}{\partial \mathbf{x}_{\parallel }^{\ell}} & =\Big\{\frac{(t^{\ell}-t^{\ell+1})F_{\perp }(\mathbf{x}^{\ell+1}, \mathbf{v}^{\ell+1})}{-2\mathbf{v}_{\perp }^{\ell+1}}+1%
\Big\}(t^{\ell}-t^{\ell+1})\frac{\partial }{\partial \mathbf{x}_{\parallel }^{\ell}}F_{\perp
}(\mathbf{x}^{\ell}, \mathbf{v}^{\ell})  \\
&   \ \ \ \ \ +\frac{%
F_{\perp }(\mathbf{x}^{\ell+1}, \mathbf{v}^{\ell+1})}{\mathbf{v}_{\perp }^{\ell+1}}\int^{t^{\ell}}_{t^{\ell+1}}\int^{t^{\ell}}_{s}%
\int_\tau^{t^\ell } \frac{d}{d\tau '} \left(  \frac{\partial }{\partial \mathbf{x}_{\parallel }^{\ell}}F_{\perp }( \mathbf{X}_{\ell}(\tau'), \mathbf{V}_{\ell}(\tau') )  \right) d\tau'\mathrm{d}\tau \mathrm{d}s
\\ & \ \ \ \ \ +\int^{t^{\ell}}_{t^{\ell+1}}\int_\tau^{t^\ell } \frac{d}{d\tau '} \left(  \frac{\partial }{\partial \mathbf{x}_{\parallel }^{\ell}}F_{\perp }( \mathbf{X}_{\ell}(\tau'), \mathbf{V}_{\ell}(\tau') )  \right) d\tau'\mathrm{d}\tau
\\ & \lesssim  \Big\{   -1 + O_{\xi}(1) \frac{|t^{\ell}-t^{\ell+1}|^{2 }  (|v^\ell |^{3}+1)}{|\mathbf{v}_{\perp}^{\ell+1}|} + 1\Big\} |t^{\ell}-t^{\ell+1}|  (|v^{\ell}|^{2}+1) \\
&   \ \ \ \ \ +\frac{%
F_{\perp }(\mathbf{x}^{\ell+1}, \mathbf{v}^{\ell+1})}{\mathbf{v}_{\perp }^{\ell+1}}\int^{t^{\ell}}_{t^{\ell+1}}\int^{t^{\ell}}_{s}%
\int_\tau^{t^\ell } \frac{d}{d\tau '} \left(  \frac{\partial }{\partial \mathbf{x}_{\parallel }^{\ell}}F_{\perp }( \mathbf{X}_{\ell}(\tau'), \mathbf{V}_{\ell}(\tau') )  \right) d\tau'\mathrm{d}\tau \mathrm{d}s
\\ & \ \ \ \ \ +\int^{t^{\ell}}_{t^{\ell+1}}\int_\tau^{t^\ell } \frac{d}{d\tau '} \left(  \frac{\partial }{\partial \mathbf{x}_{\parallel }^{\ell}}F_{\perp }( \mathbf{X}_{\ell}(\tau'), \mathbf{V}_{\ell}(\tau') )  \right) d\tau'\mathrm{d}\tau.
\end{split}
\end{equation}

Now since 
\Be \begin{split}
& \frac{d}{d\tau '} \left(  \frac{\partial }{\partial \mathbf{x}_{\parallel }^{\ell}}F_{\perp }( \mathbf{X}_{\ell}(\tau'), \mathbf{V}_{\ell}(\tau') )  \right)
\\ & \lesssim |v^\ell |^3  +  |  \frac{d}{d\tau '} \frac{\partial }{\partial \mathbf{x}_{\parallel }^{\ell}}  \left(  E(\tau',  \mathbf{X}_{\ell}(\tau') ) \cdot \mathbf n ( \mathbf{X}_{\ell}(\tau')) \right) | 
\\ & \lesssim |v^\ell |^3  +  |  \frac{d}{d\tau '}   \left(  \mathbf n ( \mathbf{X}_{\ell}(\tau')) \cdot \nabla_x E(\tau',  \mathbf{X}_{\ell}(\tau') ) \cdot   \frac{\partial \mathbf{X}_{\ell}(\tau')}{\partial \mathbf{x}_{\parallel }^{\ell}} +E(\tau',  \mathbf{X}_{\ell}(\tau') ) \cdot \nabla_x\mathbf n ( \mathbf{X}_{\ell}(\tau')) \cdot  \frac{\partial \mathbf{X}_{\ell}(\tau')}{\partial \mathbf{x}_{\parallel }^{\ell}}    \right) | 
\\ & \lesssim |v^\ell |^3 + \left|  \mathbf n ( \mathbf{X}_{\ell}(\tau'))\cdot \nabla_x E(\tau',  \mathbf{X}_{\ell}(\tau') ) \cdot  \left(  \frac{d}{d\tau '} \frac{\partial \mathbf{X}_{\ell}(\tau')}{\partial \mathbf{x}_{\parallel }^{\ell}} \right)  + \left(  \frac{d}{d\tau '} \mathbf n ( \mathbf{X}_{\ell}(\tau')) \right)  \cdot \nabla_x E(\tau',  \mathbf{X}_{\ell}(\tau') ) \cdot    \frac{\partial \mathbf{X}_{\ell}(\tau')}{\partial \mathbf{x}_{\parallel }^{\ell}} \right.
\\   & \left. \quad \quad    + \mathbf n ( \mathbf{X}_{\ell}(\tau')) \cdot  \p_t \nabla_x E(\tau',  \mathbf{X}_{\ell}(\tau') ) \cdot   \frac{\partial \mathbf{X}_{\ell}(\tau')}{\partial \mathbf{x}_{\parallel }^{\ell}} + \sum_{1 \le i,j,k \le 3 }  \mathbf n^i ( \mathbf{X}_{\ell}(\tau')) \p_{x_j} \p_{x_k } E^i(\tau',\mathbf{X}_{\ell}(\tau') ) \frac{\partial \mathbf{X}^j_{\ell}(\tau')}{\partial \mathbf{x}_{\parallel }^{\ell}} \mathbf V^k_\ell(\tau' ) \right.
\\ &\left.  \quad \quad + \left( \p_t   E(\tau',  \mathbf{X}_{\ell}(\tau') ) +  \nabla_x E(\tau',  \mathbf{X}_{\ell}(\tau') ) \cdot    \mathbf V_\ell(\tau' )  \right)  \cdot \nabla_x\mathbf n ( \mathbf{X}_{\ell}(\tau')) \cdot  \frac{\partial \mathbf{X}_{\ell}(\tau')}{\partial \mathbf{x}_{\parallel }^{\ell}} \right.
\\ & \left.  \quad \quad + E(\tau',  \mathbf{X}_{\ell}(\tau') ) \cdot \left( \frac{d}{d\tau' } \nabla_x\mathbf n ( \mathbf{X}_{\ell}(\tau')) \right) \cdot  \frac{\partial \mathbf{X}_{\ell}(\tau')}{\partial \mathbf{x}_{\parallel }^{\ell}} +  E(\tau',  \mathbf{X}_{\ell}(\tau') ) \cdot \nabla_x\mathbf n ( \mathbf{X}_{\ell}(\tau')) \cdot  \left( \frac{d}{d \tau' } \frac{\partial \mathbf{X}_{\ell}(\tau')}{\partial \mathbf{x}_{\parallel }^{\ell}}  \right) \right|
\\ & \lesssim   |v^\ell |^3 + \left|  \mathbf n ( \mathbf{X}_{\ell}(\tau'))\cdot \nabla_x E(\tau',  \mathbf{X}_{\ell}(\tau') ) \cdot  \frac{\partial \mathbf{V}_{\ell}(\tau')}{\partial \mathbf{x}_{\parallel }^{\ell}}  + \left(   \nabla_x \mathbf n ( \mathbf{X}_{\ell}(\tau')) \cdot \mathbf V_\ell (\tau' ) \right)  \cdot \nabla_x E(\tau',  \mathbf{X}_{\ell}(\tau') ) \cdot    \frac{\partial \mathbf{X}_{\ell}(\tau')}{\partial \mathbf{x}_{\parallel }^{\ell}} \right.
\\   & \left. \quad \quad    + \mathbf n ( \mathbf{X}_{\ell}(\tau')) \cdot  \p_t \nabla_x E(\tau',  \mathbf{X}_{\ell}(\tau') ) \cdot   \frac{\partial \mathbf{X}_{\ell}(\tau')}{\partial \mathbf{x}_{\parallel }^{\ell}} + \sum_{1 \le i,j,k \le 3 }  \mathbf n^i ( \mathbf{X}_{\ell}(\tau')) \p_{x_j} \p_{x_k } E^i(\tau',\mathbf{X}_{\ell}(\tau') ) \frac{\partial \mathbf{X}^j_{\ell}(\tau')}{\partial \mathbf{x}_{\parallel }^{\ell}} \mathbf V^k_\ell(\tau' ) \right.
\\ & \left. \quad \quad + \left( \p_t   E(\tau',  \mathbf{X}_{\ell}(\tau') ) +  \nabla_x E(\tau',  \mathbf{X}_{\ell}(\tau') ) \cdot    \mathbf V_\ell(\tau' )  \right)  \cdot \nabla_x\mathbf n ( \mathbf{X}_{\ell}(\tau')) \cdot  \frac{\partial \mathbf{X}_{\ell}(\tau')}{\partial \mathbf{x}_{\parallel }^{\ell}} \right.
\\ & \left. \quad \quad + E(\tau',  \mathbf{X}_{\ell}(\tau') ) \cdot \left(  \nabla_x^2 \mathbf n ( \mathbf{X}_{\ell}(\tau')) \cdot \mathbf V_\ell(\tau' )  \right) \cdot  \frac{\partial \mathbf{X}_{\ell}(\tau')}{\partial \mathbf{x}_{\parallel }^{\ell}} +  E(\tau',  \mathbf{X}_{\ell}(\tau') ) \cdot \nabla_x\mathbf n ( \mathbf{X}_{\ell}(\tau')) \cdot   \frac{\partial \mathbf{V}_{\ell}(\tau')}{\partial \mathbf{x}_{\parallel }^{\ell}}  \right|
\\ & \lesssim  |v^\ell |^3  + |v^\ell|  \|\nabla_x^2 E \|_{L^{\infty}_{t,x}} + \| \p_t \nabla_x E \|_{L^{\infty}_{t,x}},
\end{split} \Ee
where we use the bounds from \eqref{Dxv_free}. We have
\Be \label{estwant}\begin{split}
\frac{\partial \mathbf{v}_{\perp }^{\ell+1}}{\partial \mathbf{x}_{\parallel }^{\ell}}  \lesssim  &   \Big(\frac{|t^{\ell}-t^{\ell+1}|(|v^{\ell}|^{2}+1)}{|\mathbf{v}_{\perp}^{\ell+1}|} \Big) \left( |t^{\ell}-t^{\ell+1}|^{2} (|v^{\ell}|^{3}+1) + |v^\ell |^3  + |v^\ell|  \|\nabla_x^2 E \|_{L^{\infty}_{t,x}} + \| \p_t \nabla_x E \|_{L^{\infty}_{t,x}} \right)
\\ \lesssim_{\xi,t} & \min\{  \frac{|\mathbf{v}_{\perp}^{\ell+1}|^{2}}{|v^{\ell}|},  |\mathbf{v}_{\perp}^{\ell+1}|^{2} \}
\end{split} \Ee
as long as  $ \|\nabla_x^2 E \|_{L^{\infty}_{t,x}} + \| \p_t \nabla_x E \|_{L^{\infty}_{t,x}} < \infty$. Similarly,
\begin{equation}\notag
\begin{split}
\frac{\partial \mathbf{v}_{\perp }^{\ell+1}}{\partial \mathbf{v}_{\perp }^{\ell}}& =-1-\frac{%
\partial t^{\ell+1}}{\partial \mathbf{v}_{\perp }^{\ell}}F_{\perp
}(\mathbf{x}^{\ell+1 }, \mathbf{v}^{\ell+1})+\int^{t^{\ell}}_{t^{\ell+1}}\frac{\partial }{\partial \mathbf{v}_{\perp }^{\ell}}%
F_{\perp }( \mathbf{X}_{\ell}(\tau), \mathbf{V}_{\ell}(\tau) )\mathrm{d}\tau \\
&= -1+\frac{F_{\perp }(\mathbf{x}^{\ell+1}, \mathbf{v}^{\ell+1})}{\mathbf{v}_{\perp }^{\ell+1}}(t^{\ell}-t^{\ell+1})-\frac{%
F_{\perp } (\mathbf{x}^{\ell+1 }, \mathbf{v}^{\ell+1})}{\mathbf{v}_{\perp }^{\ell+1}}\int_{t^{\ell}}^{t^{\ell+1}}\int_{t^{\ell}}^{s}%
\frac{\partial }{\partial \mathbf{v}_{\perp }^{\ell}}F_{\perp }(\mathbf{X}_{\ell}(\tau), \mathbf{V}_{\ell}(\tau))\mathrm{d}\tau
\mathrm{d}s \\
& \ \ \ \ \ \ \ \ +\int^{t^{\ell}}_{t^{\ell+1}}\frac{\partial }{\partial \mathbf{v}_{\perp }^{\ell}}%
F_{\perp }( \mathbf{X}_{\ell}(\tau) , \mathbf{V}_{\ell}(\tau) )\mathrm{d}\tau \\
&=- 1+2+O_{\xi }(1)\frac{|t^{\ell}-t^{\ell+1}|^{2}(|v^{\ell}|^{3}+1)}{\mathbf{v}_{\perp }^{\ell+1}} \\
& \ \ \ \ \ -
\frac{F_{\perp } (\mathbf{x}^{\ell }, \mathbf{v}^{\ell})}{\mathbf{v}_{\perp }^{\ell+1}}\frac{(t^{\ell}-t^{\ell+1})^{2}}{2}%
\Big\{ \lim_{s \uparrow t^{\ell}}  \frac{\partial }{\partial \mathbf{v}_{\perp }^{\ell}}   F_{\perp }( \mathbf{X}_{\ell}(\tau), \mathbf{V}_{\ell}(\tau)  )+O_{\xi
}(1)|t^{\ell}-t^{\ell+1}|(|v^{\ell}|^{2}+1)\Big\} \\
 & \ \ \ \ \ +(t^{\ell}-t^{\ell+1})\Big\{  \lim_{s \uparrow t^{\ell}}\frac{\partial }{\partial \mathbf{v}_{\perp }^{\ell}}F_{\perp
}( \mathbf{X}_{\ell}(\tau), \mathbf{V}_{\ell}(\tau)   )+O_{\xi }(1)|t^{\ell}-t^{\ell+1}|(|v^{\ell}|^{2}+1)\Big\} \\
&  =1 + O_{\xi}(1) \Big\{\frac{|t^{\ell}-t^{\ell+1}|^{2}(|v^{\ell}|^{3}+1)}{|\mathbf{v}_{\perp}^{\ell+1}|}+ \frac{|t^{\ell}-t^{\ell+1}|^{3}}{|\mathbf{v}_{\perp}^{\ell+1}|}(|v^{\ell}|^{3} +1) \Big| \lim_{s \uparrow t^{\ell}}\frac{\partial }{\partial \mathbf{v}_{\perp }^{\ell}}F_{\perp
}( \mathbf{X}_{\ell}(\tau), \mathbf{V}_{\ell}(\tau)   )\Big| + |t^{\ell}-t^{\ell+1}|^{2} (|v^{\ell}|^{2} +1)\Big\}\\
&\lesssim 1 + |t^{\ell}-t^{\ell+1}|^{2} ( |v^{\ell}|^{2}+1) \Big\{ 1+ \frac{|v^{\ell}| +1}{|\mathbf{v}_{\perp}^{\ell+1}|} + \frac{|t^{\ell}-t^{\ell+1}| (|v^{\ell}|^{2} +1) }{|\mathbf{v}_{\perp}^{\ell+1}|}  \Big\} \lesssim_{\xi,t}  1+ \min \{ \frac{|\mathbf{v}_{\perp}^{\ell+1}|}{|v^{\ell}|}, |\mathbf{v}_\perp^{\ell+1} | \} , 
\end{split}
\end{equation}
\begin{equation}
\begin{split}
\frac{\partial \mathbf{v}_{\perp }^{\ell+1}}{\partial \mathbf{v}_{\parallel }^{\ell}}& =\frac{%
-F_{\perp }(\mathbf{x}^{\ell+1},  v^{\ell+1})}{ \mathbf{v}_{\perp }^{\ell+1}}\int^{t^{\ell}}_{t^{\ell+1}}\int^{t^{\ell}}_{s}%
\frac{\partial }{\partial  \mathbf{v}_{\parallel }^{\ell}}F_{\perp }( \mathbf{X}_{\ell}(\tau), \mathbf{V}_{\ell}(\tau) )\mathrm{d}\tau
\mathrm{d}s-\int^{t^{\ell}}_{t^{\ell+1}}\frac{\partial }{\partial  \mathbf{v}_{\parallel
}^{\ell}}F_{\perp }(  \mathbf{X}_{\ell}(\tau), \mathbf{V}_{\ell}(\tau) )\mathrm{d}\tau \\
& =\Big\{\frac{(t^{\ell}-t^{\ell+1})F_{\perp }(\mathbf{x}^{\ell+1}, \mathbf{v}^{\ell+1})}{-2 \mathbf{v}_{\perp }^{\ell+1}}+1%
\Big\}(t^{\ell}-t^{\ell+1})\frac{\partial }{\partial  \mathbf{v}_{\parallel }^{\ell}}F_{\perp
}(\mathbf{x}^{\ell}, \mathbf{v}^{\ell})  \\
& \ \ \ \  +O_{\xi }(1) |t^{\ell}-t^{\ell+1}|^{2}( |v^{\ell}|^{2}+1)
\Big\{\frac{|F_{\perp }(\mathbf{x}^{\ell+1}, \mathbf{v}^{\ell+1})|  |t^{\ell}-t^{\ell+1}| }{| \mathbf{v}_{\perp }^{\ell+1}|}%
+1\Big\}\\
& \lesssim_{\xi} |t^{\ell}-t^{\ell+1}|^{2}(|v^{\ell}|^{2}+1) \Big\{1+ \frac{|t^{\ell}-t^{\ell+1}|(|v^{\ell}|^{2}+1)}{| \mathbf{v}_{\perp}^{\ell+1}|}\Big\} \lesssim_{ \xi,t}  \min \{ \frac{| \mathbf{v}_{\perp}^{\ell+1}|^{2}}{|v^{\ell}|^{2}}, |\mathbf v_\perp^{\ell+1} |^2 \}.
\end{split}
\end{equation}
These estimates complete the proof of the claims (\ref{l_to_l+1delta}), and of (\ref{l_to_l+1})  when $\ell$ is \textit{Type II}.

 \vspace{4pt}

\noindent\textit{Proof of (\ref{l_to_l+1}) when $\ell$ is \textit{Type III}}: 
Recall that we chose a $\mathbf{p}^{\ell}-$spherical coordinate as $\mathbf{p}^{\ell}= (z^{\ell}, w^{\ell})$ with $|z^{\ell}-x^{\ell}| \leq \sqrt{\delta}$ and any $w^{\ell} \in \mathbb{S}^{2}$ with $n(z^{\ell})\cdot w^{\ell} =0.$

%
Fix $\ell$. Let us choose fixed numbers $\Delta_{1}, \Delta_{2}>0$ such that $|v|\Delta_{1}\ll 1$ and $|v||t^{\ell+1}- (t^{\ell}-\Delta_{1}-\Delta_{2})|\ll 1$ so that
$$
s^{\ell}\equiv t^{\ell}-\Delta_{1}, \ \  s^{\ell+1}\equiv s^{\ell}-\Delta_{2} = t^{\ell} - \Delta_{1} -\Delta_{2},
$$
 satisfying $|v||t^{\ell+1}-s^{\ell+1} |= |v||t^{\ell+1}-(t^{\ell}-\Delta_{1}-\Delta_{2})|\ll 1$ and $|v||t^{\ell}-s^{\ell} | =|v||\Delta_{1}| \ll 1$ so that the spherical coordinates are well-defined for $s\in [t^{\ell+1},s ^{\ell+1}]$ and $ s\in[s ^{\ell}, t^{\ell}]$. 
 
 Notice that
 \[
 \ \ \frac{\partial s^{\ell+1}}{\partial s^{\ell}}  = \frac{\partial ( s^{\ell} -\Delta_{1})}{\partial s ^{\ell}}=1,
 \ \ \frac{\partial s ^{\ell}}{\partial t^{\ell}} = \frac{\partial (t^{\ell}-\Delta_{1})}{\partial t^{\ell}}=1.
 \]

We first follow the flow in $\mathbf{p}^{\ell}-$spherical coordinate, then change to the Euclidian coordinate to near the boundary at $s^{\ell}$, follow the flow until $s^{\ell+1}$, and then change to the chart to $\mathbf{p}^{\ell+1}-$spherical coordinate. By the chain rule,
\begin{equation*}
\begin{split}
& \frac{\partial ( t^{\ell+1},   \mathbf{x}_{\parallel_{\ell+1}}^{\ell+1}, \mathbf{v}_{\perp_{\ell+1}}^{\ell+1}, \mathbf{v}_{\parallel_{\ell+1}}^{\ell+1}  )}{\partial (t^{\ell },   \mathbf{x}_{\parallel_{\ell}}^{\ell }, \mathbf{v}_{\perp_{\ell}}^{\ell }, \mathbf{v}_{\parallel_{\ell}}^{\ell } )} \\
=&
 \frac{\partial ( t^{\ell+1},   \mathbf{x}_{\parallel_{\ell+1}}^{\ell+1}, \mathbf{v}_{\perp_{\ell+1}}^{\ell+1}, \mathbf{v}_{\parallel_{\ell+1}}^{\ell+1}  )}
 {\partial( s^{\ell+1} , \mathbf{x}_{\perp_{\ell+1}}(s ^{\ell+1}), \mathbf{x}_{\parallel_{\ell+1}}(s ^{\ell+1}), \mathbf{v}_{\perp_{\ell+1}}(s ^{\ell+1}), \mathbf{v}_{\parallel_{\ell+1}}(s ^{\ell+1}) )}
\frac{\partial (s ^{\ell+1}, \mathbf{X}_{ {\mathbf{p}^{\ell+1}}}(s ^{\ell+1}), \mathbf{V}_{{\mathbf{p}^{\ell+1}}}(s ^{\ell+1}))}
{\partial (s ^{\ell+1}, X_{\mathbf{cl}}(s ^{\ell+1}),V_{\mathbf{cl}}(s ^{\ell+1}))}
\\
 &\times\frac{\partial (s^{\ell+1},  {X}_{\mathbf{cl}}(s ^{{\ell+1}}),  {V}_{\mathbf{cl}}(s ^{{\ell+1}}))}
{ \partial (s^{{\ell }},  {X}_{\mathbf{cl}}(s  ^{{\ell }}),  {V}_{\mathbf{cl}}(s ^{{\ell }}))}
 \frac{\partial (   s ^{{\ell }}, {X}_{\mathbf{cl}}(s ^{{\ell }}), {V}_{\mathbf{cl}}(s ^{{\ell }}))}
 {\partial (   s ^{{\ell }}, \mathbf{X}_{\mathbf{p}^{\ell}}(s^{{\ell }}), \mathbf{V}_{\mathbf{p}^{\ell}}(s ^{{\ell }}))}
 \frac{\partial(s ^{{\ell }}, \mathbf{x}_{\perp_{\ell}}(s ^{{\ell }}), \mathbf{x}_{\parallel_{\ell}}(s ^{{\ell }}), \mathbf{v}_{\perp_{\ell}}(s ^{{\ell }}), \mathbf{v}_{\parallel_{\ell}}(s ^{{\ell }}) )}
{\partial (t^{\ell },  \mathbf{x}_{\parallel_{\ell}}^{\ell }, \mathbf{v}_{\perp_{\ell}}^{\ell }, \mathbf{v}_{\parallel_{\ell}}^{\ell } )}.
\end{split}
\end{equation*}

We can express that $t^{\ell+1} = t^{\ell} - t_{\mathbf{b}}(x^{\ell}, v^{\ell}) = s^{\ell+1} + \Delta_{1} + \Delta_{2}- t_{\mathbf{b}}(x^{\ell}, v^{\ell}).$ Let us regard $t^{\ell+1}$ as $t^{1}$ and $s^{\ell+1}$ as $s^{1}$ and $\Delta_{1} + \Delta_{2}$ as $\Delta$ in (\ref{Delta_step3}). Then we use (\ref{t1s1}) and (\ref{tbest}) to have
\[ \begin{split}
  \frac{\partial ( t^{\ell+1},  \mathbf{x}_{\parallel}^{\ell+1}, \mathbf{v}_{\perp}^{\ell+1}, \mathbf{v}_{\parallel}^{\ell+1}  )}{\partial( s^{\ell+1}, \mathbf{x}_{\perp}(s^{\ell+1} ), \mathbf{x}_{\parallel}(s^{\ell+1} ), \mathbf{v}_{\perp}(s^{\ell+1} ), \mathbf{v}_{\parallel}(s^{\ell+1} ) )}
& \leq   \left[\begin{array}{c|c|c}
1 +O(1)|t^\ell - t^{\ell + 1 } | &  O_{\delta,\xi}(1)\frac{1}{|v|}  & O_{\delta,\xi}(1) \frac{1}{|v|^{2}} \\ \hline
O(1)|t^\ell - t^{\ell + 1 }| & O_{\delta,\xi}(1) & O_{\delta,\xi}(1) \frac{1}{|v|} \\ \hline
O(1)|t^\ell - t^{\ell + 1 }|  & O_{\delta,\xi}(1)(|v| + \frac{1}{ |\mathbf{v}_{\perp_{\ell+1}} |}) & O_{\delta,\xi}(1)
\end{array}\right]
\\ & \leq   \left[\begin{array}{c|c|c}
1 +O(1)|t^\ell - t^{\ell + 1 } | &  O_{\delta,\xi}(1)\frac{1}{|v|}  & O_{\delta,\xi}(1) \frac{1}{|v|^{2}} \\ \hline
O(1)|t^\ell - t^{\ell + 1 }| & O_{\delta,\xi}(1) & O_{\delta,\xi}(1) \frac{1}{|v|} \\ \hline
O(1)|t^\ell - t^{\ell + 1 }|  & O_{\delta,\delta' , \xi}(1)|v|  & O_{\delta,\xi}(1)
\end{array}\right].
\end{split}
  \] 
Where we have used   
From (\ref{XVchange})
\begin{equation}\notag
\begin{split}
 \frac{\partial (s^{\ell+1}, \mathbf{X}_{\mathbf{p}^{\ell+1}}(s ^{\ell+1}) ,   \mathbf{V}_{\mathbf{p}^{\ell+1}}(s ^{\ell+1}) )}{\partial (s ^{\ell+1}, X_{\mathbf{cl}}(s ^{\ell+1}) ,   {V}_{\mathbf{cl}}(s ^{\ell+1}) )}
 \lesssim_{\xi} \left[\begin{array}{c|c|c}
1 & \mathbf{0}_{1,3}  & \mathbf{0}_{1,3}\\ \hline
\mathbf{0}_{3,1} & O_{\xi}(1) & \mathbf{0}_{3,3}\\ \hline
\mathbf{0}_{3,1} & O_{\xi}(1)|v| & O_{\xi}(1)
\end{array} \right],
\end{split}
\end{equation}
and from $s^{\ell+1} = s^{\ell} -\Delta_{2}, \ X_{\mathbf{cl}}(s^{\ell+1} ) = X_{\mathbf{cl}}(s^{\ell} ) - (s^{\ell+1} -s^{\ell}) V_{\mathbf{cl}}(s^{\ell} ), \ V_{\mathbf{cl}}(s^{\ell +1} )=V_{\mathbf{cl}}(s^{\ell }),$
\begin{equation}\notag
\begin{split}
&\frac{\partial (s^{\ell+1}, X_{\mathbf{cl}}(s^{\ell+1}), V_{\mathbf{cl}}(s^{\ell+1} ))}{\partial (s^{\ell} ,X_{\mathbf{cl}}(s^{\ell} ), V_{\mathbf{cl}}(s^{\ell} ) )}
\lesssim_{\xi}  \left[\begin{array}{ccc}
1 & \mathbf{0}_{1,3} & \mathbf{0}_{1,3} \\
\mathbf{0}_{3,1}& \mathbf{Id}_{3,3} & |s_{1}-s_{2}|\mathbf{Id}_{3,3} \\
\mathbf{0}_{3,1} & \mathbf{0}_{3,3} & \mathbf{Id}_{3,3}
\end{array}\right],
\end{split}
\end{equation}
and from (\ref{jac_Phi})
\begin{equation}\notag
\begin{split}
& \frac{\partial (s^{\ell} , X_{\mathbf{cl}}(s^{\ell} ), V_{\mathbf{cl}}(s^{\ell} )  )}{\partial (s^{\ell} , \mathbf{X}_{\mathbf{p}^{\ell}}(s^{\ell} ), \mathbf{V}_{\mathbf{p}^{\ell}}(s^{\ell} )  )}
\lesssim_{\xi} \left[ \begin{array}{ccc}
1 &   \mathbf{0}_{1,3} &   \mathbf{0}_{1,3}   \\
\mathbf{0}_{3,1}&   O_{\xi}(1)
 &   \mathbf{0}_{3,3}   \\
\mathbf{0}_{3,1}   & |v|    & O_{\xi}(1)      \end{array}\right].
\end{split}
  \end{equation}
Recalig (\ref{s1tstar}), we have
$$
\frac{\partial (s^{\ell} , \mathbf{x}_{\perp_{\ell}}(s^{\ell}),  \mathbf{x}_{\parallel_{\ell}}(s^{\ell}),   \mathbf{v}_{\perp_{\ell}}(s^{\ell}), \mathbf{v}_{\parallel_{\ell}}(s^{\ell})
    )}{\partial (t^{\ell } , \mathbf{x}_{\parallel_{\ell}}^{\ell }, \mathbf{v}_{\perp_{\ell}}^{\ell }, \mathbf{v}_{\parallel_{\ell}}^{\ell })}
    \lesssim_{\xi}
     \left[\begin{array}{c|c|c}
1   & \mathbf{0}_{1,2} & \mathbf{0}_{1,3}\\ \hline
O_{\xi}(1)|v|    & O_{\xi}(1) & O_{\xi}(1)|t^{\ell}-s_{1}| \\ \hline
O_{\xi}(1)|v|^{2}   & O_{\xi}(1)|v| & O_{\xi}(1)
\end{array}\right].$$

By direct matrix multiplication
\begin{equation}\notag
\begin{split}
 \frac{\partial ( t^{\ell+1},   \mathbf{x}_{\parallel_{{\ell+1}}}^{\ell+1}, \mathbf{v}_{\perp_{{\ell+1}}}^{\ell+1}, \mathbf{v}_{\parallel_{{\ell+1}}}^{\ell+1}  )}{\partial (t^{\ell },  \mathbf{x}_{\parallel_{{\ell }}}^{\ell }, \mathbf{v}_{\perp_{{\ell }}}^{\ell }, \mathbf{v}_{\parallel_{{\ell }}}^{\ell } )}
   \lesssim_{ t,\xi} \left[\begin{array}{c|c|c}1   & \frac{1}{|v|} & \frac{1}{|v|^{2}} \\ \hline
\mathbf{0}_{2,1}     & 1 & \frac{1}{|v|}\\ \hline
 \mathbf{0}_{3,1}   & |v| & 1
\end{array}\right].
\end{split}
\end{equation}
Note that for \textit{Type III} we have $\mathbf{r}^{\ell+1} \gtrsim \sqrt{\delta}$ so that from (\ref{l_to_l+1})
\[
J(\mathbf{r}^{\ell+1}) \gtrsim \left[\begin{array}{c|c|c}
1   & \frac{M}{|v|}\sqrt{\delta} &  \frac{M}{|v|^{2}} \min\{1, \sqrt{\delta}\} \\ \hline
\mathbf{0}_{2,1}  & M \sqrt{\delta} & \frac{M}{|v|} \min \{  1, \sqrt{\delta} \} \\ \hline
\mathbf{0}_{3,1}   & M|v| \min \{ \delta, \sqrt{\delta}\} & M \min\{\delta,\sqrt{\delta}\}
\end{array} \right] \gtrsim_{\delta,t,\xi}  \frac{\partial ( t^{\ell+1},   \mathbf{x}_{\parallel_{{\ell+1}}}^{\ell+1}, \mathbf{v}_{\perp_{{\ell+1}}}^{\ell+1}, \mathbf{v}_{\parallel_{{\ell+1}}}^{\ell+1}  )}{\partial (t^{\ell }, \mathbf{x}_{\parallel_{{\ell }}}^{\ell }, \mathbf{v}_{\perp_{{\ell }}}^{\ell }, \mathbf{v}_{\parallel_{{\ell }}}^{\ell } )}
.\]
This proves our claim (\ref{l_to_l+1}) for \textit{Type III}.

\noindent{\textit{Step 5. Eigenvalues and diagonalization of (\ref{l_to_l+1})}}

\vspace{4pt}
We consider the case when $\ell$ is \textit{Type II} or \textit{Type III}. By a basic linear algebra (row and column operations), the characteristic polynomial of (\ref{l_to_l+1}) 
 equals, with $\mathbf{r}= \mathbf{r}^{\ell+1},$
\begin{equation}\notag
\begin{split}
&\text{det}\left[\begin{array}{cccccc} 1 +5M \mathbf r^{} - \lambda 
& \frac{M}{|v|} \mathbf{r} & \frac{M}{|v|}\mathbf{r} & \frac{M}{|v|^{2}} & \frac{M}{|v|^{2}} \mathbf{r} & \frac{M}{|v|^{2}} \mathbf{r} \\
5M \mathbf r^{} |v|  
 & 1+M\mathbf{r} -\lambda  & M \mathbf{r} & \frac{M}{|v|} & \frac{M}{|v|} \mathbf{r} &  \frac{M}{|v|} \mathbf{r}\\
5M \mathbf r^{} |v|   &  M \mathbf{r}& 1+M\mathbf{r} -\lambda  & \frac{M}{|v|} & \frac{M}{|v|} \mathbf{r} &  \frac{M}{|v|} \mathbf{r}\\
5M  \mathbf r^2|v|^2  & M|v| \mathbf{r}^{2} & M |v| \mathbf{r}^{2} & 1+ M\mathbf{r} -\lambda & M \mathbf{r}^{2} & M \mathbf{r}^{2} \\
5M \mathbf r^{}|v|^2   & M|v| \mathbf{r}  & M |v| \mathbf{r}  & M   & 1+ M\mathbf{r} -\lambda  & M \mathbf{r} \\
5M \mathbf r^{}|v|^2  & M|v| \mathbf{r}  & M |v| \mathbf{r}  & M     & M \mathbf{r} & 1+ M\mathbf{r} -\lambda\\
\end{array}\right]\\
&=  (\lambda - 1)^5(  \lambda - (10M\mathbf r +1)).
\end{split}
\end{equation}
Therefore eigenvalues are
\begin{equation}\label{eigenvalue}
\begin{split}
 \lambda_{1}&= \lambda_{2} = \lambda_{3} = \lambda_{4} =  \lambda_{5} =1, \lambda_6 =  1 + 10M\mathbf r.
\end{split}
\end{equation}
Corresponding eigenvectors are
\begin{eqnarray*}
 \left(\begin{array}{ccccccc} -\frac{1}{5|v|}   \\ 1 \\ 0 \\ 0 \\ 0 \\0   \end{array}\right),
 \left(\begin{array}{ccccccc}-\frac{1}{5|v|}   \\ 0 \\ 1 \\ 0 \\ 0 \\0   \end{array}\right),
 \left(\begin{array}{ccccccc}-\frac{1}{5|v|^2\mathbf r}  \\ 0 \\ 0 \\ 1 \\ 0 \\0   \end{array}\right),
 \left(\begin{array}{ccccccc}-\frac{1}{5|v|^2} \\ 0 \\ 0 \\ 0 \\ 1 \\ 0  \end{array}\right),
 \left(\begin{array}{ccccccc}-\frac{1}{5|v|^2}   \\ 0 \\ 0 \\ 0 \\ 0 \\1   \end{array}\right),
  \left(\begin{array}{ccccccc}\frac{1}{|v|^2}   \\ \frac{1}{|v|} \\ \frac{1}{|v|} \\ \mathbf r \\ 1 \\ 1   \end{array}\right).
\end{eqnarray*}

Write $P=P(\mathbf{r}^{\ell}) $ as a block matrix of above column eigenvectors. Then
\begin{equation}  \label{diagonal_matrix}
\mathcal{P}= \left[\begin{array}{ccccccc}
  -\frac{1}{5|v|} & -\frac{1}{5|v|} & -\frac{1}{5|v|^2\mathbf r} & -\frac{1}{5|v|^2} & -\frac{1}{5|v|^2} & \frac{1}{|v|^2} \\
  1 & 0 & 0 & 0 & 0 & \frac{1}{|v|} \\
  0 &1 & 0 & 0 & 0 & \frac{1}{|v|} \\
  0 & 0 &1 & 0 & 0 &  \mathbf{r}\\
  0& 0 & 0 & 1 & 0 & 1\\
  0 & 0 & 0 & 0 & 1 & 1
 \end{array}\right], \  
\mathcal{P}^{-1}  =  \left[\begin{array}{ccccccc}
 -\frac{|v|}{2}   & \frac{9}{10} &-\frac{1}{10} & -\frac{1}{10 |v| \mathbf{r}} & -\frac{1}{10 |v|} & -\frac{1}{10 |v|} \\
   -\frac{|v|}{2} & -\frac{1}{10} & \frac{9}{10} & -\frac{1}{10 |v| \mathbf{r}} & -\frac{1}{10 |v|} &-\frac{1}{10 |v|} \\
   -\frac{|v|^2 \mathbf r}{2} & -\frac{|v|\mathbf r}{10} & -\frac{|v|\mathbf r}{10}& \frac{9}{10} & - \frac{\mathbf r}{10} & - \frac{\mathbf r}{10} \\
   -\frac{|v|^2}{2} & - \frac{|v|}{10} & - \frac{|v|}{10} & -\frac{1}{10 \mathbf{r}} & \frac{9}{10} & - \frac{1}{10} \\
   -\frac{|v|^2}{2} & - \frac{|v|}{10}  &- \frac{|v|}{10}  & -\frac{1}{10 \mathbf{r}} & - \frac{1}{10} & \frac{9}{10} \\
   \frac{|v|^2}{2} &  \frac{|v|}{10}  & \frac{|v|}{10}  & \frac{1}{10 \mathbf{r}} & \frac{1}{10} & \frac{1}{10} \\
\end{array}\right].
\end{equation}
Therefore
\begin{equation}\notag
  {J}(\mathbf{r}) = \mathcal{P}(\mathbf{r}) \Lambda(\mathbf{r}) \mathcal{P}^{-1} (\mathbf{r}) ,
\end{equation}
and 
\[
 \Lambda(\mathbf{r}):= \text{diag}\Big[   1, 1, 1, 1, 1,  1 + 10M\mathbf r \Big],
 \]
where the notation $\text{diag}[a_{1},\cdots , a_{m}]$ is a $m\times m-$matrix with $a_{ii}=a_{i}$ and $a_{ij}=0$ for all $i\neq j.$

Similarly for the case when $\ell$ is \textit{Type I}, the eigenvalues of the matrix \eqref{l_to_l+1delta} are (with $\mathbf{v}_\perp = \mathbf{v}_{\perp }^{\ell+1} $)
\begin{equation}\label{eigenvalue}
\begin{split}
 \lambda_{1}&= \lambda_{2} = \lambda_{3} = \lambda_{4} =  \lambda_{5} =1, \lambda_6 =  1 + 10M\mathbf v_\perp.
\end{split}
\end{equation}
Corresponding eigenvectors are
\begin{eqnarray*}
 \left(\begin{array}{ccccccc} -\frac{1}{5}   \\ 1 \\ 0 \\ 0 \\ 0 \\0   \end{array}\right),
 \left(\begin{array}{ccccccc}-\frac{1}{5}   \\ 0 \\ 1 \\ 0 \\ 0 \\0   \end{array}\right),
 \left(\begin{array}{ccccccc}-\frac{1}{5\mathbf v_\perp}  \\ 0 \\ 0 \\ 1 \\ 0 \\0   \end{array}\right),
 \left(\begin{array}{ccccccc}-\frac{1}{5} \\ 0 \\ 0 \\ 0 \\ 1 \\ 0  \end{array}\right),
 \left(\begin{array}{ccccccc}-\frac{1}{5}   \\ 0 \\ 0 \\ 0 \\ 0 \\1   \end{array}\right),
  \left(\begin{array}{ccccccc} 1   \\ 1 \\ 1 \\ \mathbf v_\perp \\ 1 \\ 1   \end{array}\right).
\end{eqnarray*}

Write $P=P(\mathbf v_\perp^{\ell}) $ as a block matrix of above column eigenvectors. Then
\begin{equation}  \label{diagonal_matrix}
\mathcal{P}= \left[\begin{array}{ccccccc}
  -\frac{1}{5} & -\frac{1}{5} & -\frac{1}{5\mathbf v_\perp } & -\frac{1}{5} & -\frac{1}{5} &1 \\
  1 & 0 & 0 & 0 & 0 & 1 \\
  0 &1 & 0 & 0 & 0 & 1 \\
  0 & 0 &1 & 0 & 0 &  \mathbf v_\perp\\
  0& 0 & 0 & 1 & 0 & 1\\
  0 & 0 & 0 & 0 & 1 & 1
 \end{array}\right], \  
\mathcal{P}^{-1}  =  \left[\begin{array}{ccccccc}
 -\frac{1}{2}   & \frac{9}{10} &-\frac{1}{10} & -\frac{1}{10\mathbf v_\perp} & -\frac{1}{10} & -\frac{1}{10 } \\
   -\frac{1}{2} & -\frac{1}{10} & \frac{9}{10} & -\frac{1}{10 \mathbf v_\perp} & -\frac{1}{10} &-\frac{1}{10 } \\
   -\frac{\mathbf v_\perp}{2} & -\frac{\mathbf v_\perp}{10} & -\frac{\mathbf v_\perp}{10}& \frac{9}{10} & - \frac{\mathbf v_\perp}{10} & - \frac{\mathbf v_\perp}{10} \\
   -\frac{1}{2} & - \frac{1}{10} & - \frac{1}{10} & -\frac{1}{10\mathbf v_\perp} & \frac{9}{10} & - \frac{1}{10} \\
   -\frac{1}{2} & - \frac{1}{10}  &- \frac{1}{10}  & -\frac{1}{10\mathbf v_\perp} & - \frac{1}{10} & \frac{9}{10} \\
   \frac{1}{2} &  \frac{1}{10}  & \frac{1}{10}  & \frac{1}{10 \mathbf v_\perp} & \frac{1}{10} & \frac{1}{10} \\
\end{array}\right].
\end{equation}
Therefore
\begin{equation}\notag
  {J}(\mathbf{\mathbf v_\perp}) = \mathcal{P}(\mathbf v_\perp) \Lambda(\mathbf v_\perp) \mathcal{P}^{-1} (\mathbf v_\perp) ,
\end{equation}
and 
\[
 \Lambda(\mathbf v_\perp):= \text{diag}\Big[   1, 1, 1, 1, 1,  1 + 10M\mathbf v_\perp\Big],
 \]

\vspace{8pt}

 \noindent{\textit{Step 6. The $i-$th intermediate group}}

\vspace{4pt}

If $\ell$ is \textit{Type II} or \textit{Type III}, We claim that, for $i=1,2,\cdots, [\frac{|t-s| |v|}{L_{\xi}}]$,
\begin{equation}\label{one_group}
\begin{split}
&J^{\ell_{i+1}}_{\ell_{i+1}-1} \times \cdots \times J_{\ell_{i}}^{\ell_{i}+1}\\
 =& \ {\frac{\partial (t^{\ell_{i+1}} , \mathbf{x}_{\parallel_{{\ell_{i+1}}}}^{\ell_{i+1}}, \mathbf{v}_{\perp_{{\ell_{i+1}}}}^{\ell_{i+1}}, \mathbf{v}_{\parallel_{{\ell_{i+1}}}}^{\ell_{i+1}})}{\partial (t^{\ell_{i+1}-1} , \mathbf{x}_{\parallel_{{\ell_{i+1}}-1}}^{\ell_{i+1}-1}, \mathbf{v}_{\perp_{{\ell_{i+1}}-1}}^{\ell_{i+1}-1}, \mathbf{v}_{\parallel_{{\ell_{i+1}}-1}}^{\ell_{i+1}-1})}
 \times \cdots \times
 \frac{\partial (t^{\ell_{i}+1} , \mathbf{x}_{\parallel_{{\ell_{i}+1}}}^{\ell_{i} +1}, \mathbf{v}_{\perp_{{\ell_{i}+1}}}^{\ell_{i}+1}, \mathbf{v}_{\parallel_{{\ell_{i}+1}}}^{\ell_{i}+1})}{\partial (t^{\ell_{i} } , \mathbf{x}_{\parallel_{{\ell_{i}  }}}^{\ell_{i}  }, \mathbf{v}_{\perp_{{\ell_{i}  }}}^{\ell_{i} }, \mathbf{v}_{\parallel_{{\ell_{i}  }}}^{\ell_{i} })}  }\\
 \leq & \ 
 \mathcal{P}(\mathbf{r}_{i}) (\Lambda(\mathbf{r}_{i}))^{\frac{C_{\xi}}{\mathbf{r}_{i}}}  \mathcal{P}^{-1}( \mathbf{r}_{i}).
\end{split}
\end{equation}

By the definition of the group, $L_{\xi}\leq |v||t^{\ell_{i}}-t^{\ell_{i+1}}| \leq C_{1}<+\infty$ for all $i$. By the Velocity lemma(Lemma \ref{velocitylemma}),
  \[
\frac{1}{\mathcal{C}_{1}} e^{-\frac{\mathcal{C}}{2}C_{1}} \mathbf{r}^{\ell_{i}}  \leq  \mathbf{r}^{\ell_{i+1}} \equiv \frac{|\mathbf{v}_{\perp}^{\ell_{i+1}}|}{|v^{\ell_{i+1}}|},   \mathbf{r}^{\ell_{i+1}-1} \equiv \frac{|\mathbf{v}_{\perp}^{\ell_{i+1}-1}|}{|v^{\ell_{i+1}-1}|}, \cdots ,   \mathbf{r}^{\ell_{i }+1} \equiv \frac{|\mathbf{v}_{\perp}^{\ell_{i}+1}|}{|v^{\ell_{i }+1}|},   \mathbf{r}^{\ell_{i}} \equiv \frac{|\mathbf{v}_{\perp}^{\ell_{i}}|}{|v^{\ell_i}|}
  \leq \mathcal{C}_{1} e^{\frac{\mathcal{C}}{2}C_{1}} \mathbf{r}^{\ell_{i}},
  \]
and define
\[
\mathbf{r}_{i}\equiv \mathcal{C}_{1} e^{\frac{\mathcal{C}}{2}C_{1}} \mathbf{r}^{\ell_{i}}.
\]
Then we have
\begin{equation}\label{bound_ri}
\frac{1}{(\mathcal{C}_{1})^{2}} e^{-\mathcal{C} C_{1}} \mathbf{r}_{i}  \ \leq  \ \mathbf{r}^{j} \ \leq \ \mathbf{r}_{i} \ \ \ \text{for all} \ \  \ell_{i+1}\leq j \leq \ell_{i}.
\end{equation}
From (\ref{l_to_l+1}), we have a uniform bound for all $\ell_{i+1}\leq j \leq \ell_{i}$
\[
J^{j+1}_{j}  \lesssim  {J}(\mathbf{r}_{i})= \mathcal{P}(\mathbf{r}_{i}) \Lambda(\mathbf{r}_{i}) \mathcal{P}^{-1}(\mathbf{r}_{i}).
\]
Therefore
\[
J^{\ell_{i+1}}_{\ell_{i+1}-1} \times \cdots \times J^{\ell_{i}+1}_{\ell_{i}} \leq \mathcal{P}(\mathbf{r}_{i}) [ \Lambda(\mathbf{r}_{i})]^{|\ell_{i+1}-\ell_{i}|} \mathcal{P}^{-1}(\mathbf{r}_{i}).
\]

Now we only left to prove $|\ell_{i+1}-\ell_{i}| \lesssim_{\Omega} \frac{1}{\mathbf{r}_{i}}$: For any $\ell_{i+1} \leq j \leq \ell_{i}$, we have $\xi(x^{j})=0=\xi(x^{j+1})=\xi(x^{j}-(t^{j}-t^{j+1})v^{j}).$ We expand $\xi(x^{j}-(t^{j}-t^{j+1})v^{j})$ in time to have
\begin{equation}\notag
\begin{split}
\xi(x^{j+1}) &=  \xi(x^{j}) + \int^{t^{j+1}}_{t^{j}} \frac{d}{ds} \xi(X_{\mathbf{cl}}(s)) \mathrm{d}s\\
&= \xi(x^{j}) + (v^{j}\cdot \nabla \xi(x^{j}) ) (t^{j+1}-t^{j}) + \int^{t^{j+1}}_{t^{j}} \int^{s}_{t^{j}} \frac{d^{2}}{d\tau^{2}} \xi(X_{\mathbf{cl}}(\tau))  \mathrm{d}\tau\mathrm{d}s,\\
\end{split}
\end{equation}
and  
\begin{equation}\notag
0 = (v^{j}\cdot \nabla \xi(x^{j})) (t^{j+1}-t^{j}) + \frac{(t^{j}-t^{j+1})^{2}}{2} (V_{\mathbf{cl}} (\tau_* )  \cdot \nabla^{2}\xi(X_{\mathbf{cl}}(\tau_{*})) \cdot V_{\mathbf{cl}} (\tau_* ) + E(\tau, X_{\mathbf{cl}} (\tau_* ) ) \cdot \nabla\xi(X_{\mathbf{cl}}(\tau_{*})) ),
\end{equation}
for some $ \tau_{*} \in [t^{j+1},t^{j}]$. Therefore
\begin{equation}\notag
\begin{split}
\frac{v^{j}\cdot \nabla \xi(x^{j})}{|v|}& = (t^{j}-t^{j+1})|v| \frac{V_{\mathbf{cl}} (\tau_* )  \cdot \nabla^{2}\xi(X_{\mathbf{cl}}(\tau_{*})) \cdot V_{\mathbf{cl}} (\tau_* ) + E(\tau, X_{\mathbf{cl}} (\tau_* ) ) \cdot \nabla\xi(X_{\mathbf{cl}}(\tau_{*})) }{2|v|^{2}}.
\end{split}
\end{equation}
Thus there exists $C_{2} (\delta, \xi, E)\gg 1$
\begin{equation}\label{time_angle}
\frac{|v^{j}\cdot \nabla \xi(x^{j})|}{|v|} \leq C_{2}|t^{j}-t^{j+1}||v|.
\end{equation}
Therefore we have a lower bound of $|v||t^{j}-t^{j+1}|$: $
|v||t^{j}-t^{j+1}| \geq \frac{1}{C_{2}} |\mathbf{r}^{j}| \geq  \frac{1}{(\mathcal{C}_{1})^{2}C_{2}} e^{-\mathcal{C}C_{1}} \mathbf{r}_{i},$
where we have used (\ref{bound_ri}).
Finally, using the definition of one group($1 \leq  |v||t^{\ell_{i}}-t^{\ell_{i+1}}| \leq C_{1}$), we have the following upper bound of the number of bounces in this one group($i-$th intermediate group)
\[
|\ell_{i}-\ell_{i+1}|  \leq \frac{|v||t^{\ell_{i}}-t^{\ell_{i+1}}|}{ \min_{\ell_{i} \leq j \leq \ell_{i+1}} |v||t^{j}-t^{j+1}|}  \leq \frac{C_{1}}{ \frac{1}{(\mathcal{C}_{1})^{2} C_{2}  } e^{-\mathcal{C} C_{1}} \mathbf{r}_{i} } \lesssim_{\xi} \frac{1}{\mathbf{r}_{i}},
\]
and this completes our claim (\ref{one_group}).

\vspace{1cm}
Let's consider the whole intermediate groups
\Be
J^{\ell_{*}}_{\ell_{*}-1} \times \cdots \times J^{\ell+1}_{\ell} \times J^{\ell}_{\ell-1} \times \cdots \times J_{1}^{2} \le  J(\mathbf r^{\ell_* }) \times \cdots \times J(\mathbf r^{\ell+1}) \times J(\mathbf r^{\ell} ) \times \cdots J(\mathbf r^2 ).
\Ee
We have from \eqref{diagonal_matrix} that
\[
J(\mathbf r^{\ell+1}) \times J(\mathbf r^{\ell} ) = \mathcal{P}(\mathbf{r}^{\ell+1} ) \Lambda(\mathbf{r}^{\ell+1}) \mathcal{P}^{-1}(\mathbf{r}^{\ell+1} ) \mathcal{P}(\mathbf{r}^{\ell} ) \Lambda(\mathbf{r}^{\ell}) \mathcal{P}^{-1}(\mathbf{r}^{\ell} ),
\]
and by direct computation
\Be
\begin{split}
& \mathcal{P}^{-1}(\mathbf{r}^{\ell+1} )  \mathcal{P}(\mathbf{r}^{\ell} ) 
\\  =&  \left[\begin{array}{ccccccc}
 -\frac{|v|}{2}   & \frac{9}{10} &-\frac{1}{10} & -\frac{1}{10 |v| \mathbf{r}^{\ell+1}} & -\frac{1}{10 |v|} & -\frac{1}{10 |v|} \\
   -\frac{|v|}{2} & -\frac{1}{10} & \frac{9}{10} & -\frac{1}{10 |v| \mathbf{r}^{\ell+1}} & -\frac{1}{10 |v|} &-\frac{1}{10 |v|} \\
   -\frac{|v|^2 \mathbf r^{\ell+1}}{2} & -\frac{|v|\mathbf r^{\ell+1}}{10} & -\frac{|v|\mathbf r^{\ell+1}}{10}& \frac{9}{10} & - \frac{\mathbf r^{\ell+1}}{10} & - \frac{\mathbf r^{\ell+1}}{10} \\
   -\frac{|v|^2}{2} & - \frac{|v|}{10} & - \frac{|v|}{10} & -\frac{1}{10 \mathbf{r}^{\ell+1}} & \frac{9}{10} & - \frac{1}{10} \\
   -\frac{|v|^2}{2} & - \frac{|v|}{10}  &- \frac{|v|}{10}  & -\frac{1}{10 \mathbf{r}^{\ell+1}} & - \frac{1}{10} & \frac{9}{10} \\
   \frac{|v|^2}{2} &  \frac{|v|}{10}  & \frac{|v|}{10}  & \frac{1}{10 \mathbf{r}^{\ell+1}} & \frac{1}{10} & \frac{1}{10} \\
\end{array}\right]
\left[\begin{array}{ccccccc}
  -\frac{1}{5|v|} & -\frac{1}{5|v|} & -\frac{1}{5|v|^2\mathbf r^{\ell}} & -\frac{1}{5|v|^2} & -\frac{1}{5|v|^2} & \frac{1}{|v|^2} \\
  1 & 0 & 0 & 0 & 0 & \frac{1}{|v|} \\
  0 &1 & 0 & 0 & 0 & \frac{1}{|v|} \\
  0 & 0 &1 & 0 & 0 &  \mathbf{r}^{\ell}\\
  0& 0 & 0 & 1 & 0 & 1\\
  0 & 0 & 0 & 0 & 1 & 1
 \end{array}\right]
\\ = & \left[\begin{array}{ccccccc}
  1& 0 & \frac{1}{10|v|}(\frac{1}{\mathbf r^{\ell}} - \frac{1}{\mathbf r^{\ell+1} } ) &0 & 0 &\frac{1}{10|v|}(1 - \frac{\mathbf r^\ell}{\mathbf r^{\ell+1} } ) \\
  0& 1 & \frac{1}{10|v|}(\frac{1}{\mathbf r^{\ell}} - \frac{1}{\mathbf r^{\ell+1} } ) & 0 & 0 & \frac{1}{10|v|}(1 - \frac{\mathbf r^\ell}{\mathbf r^{\ell+1} } )\\
  0 &0 & 1+\frac{1}{10}(\frac{\mathbf r^{\ell+1}}{\mathbf r^{\ell} } -1 ) & 0 & 0 & \frac{9}{10}(\mathbf r^\ell - \mathbf r^{\ell+1} ) \\
  0 & 0 &\frac{1}{10}(\frac{1}{\mathbf r^{\ell}} - \frac{1}{\mathbf r^{\ell+1} } ) & 1 & 0 & \frac{1}{10}(1 - \frac{\mathbf r^\ell}{\mathbf r^{\ell+1} } )\\
  0& 0 & \frac{1}{10}(\frac{1}{\mathbf r^{\ell}} - \frac{1}{\mathbf r^{\ell+1} } ) & 0 & 1 & \frac{1}{10}(1 - \frac{\mathbf r^\ell}{\mathbf r^{\ell+1} } )\\
  0 & 0 & -\frac{1}{10}(\frac{1}{\mathbf r^{\ell}} - \frac{1}{\mathbf r^{\ell+1} } ) & 0 &0 &  1+\frac{1}{10}(\frac{\mathbf r^{\ell}}{\mathbf r^{\ell+1} } -1 )
 \end{array}\right]  .
 \end{split}
\Ee
Since from the definition of $\mathbf v_\perp^\ell$, and \eqref{Fperpcancel} we have
\begin{equation} \begin{split}
\label{v_ell}
\mathbf{v}_{\perp}^{\ell+1} = - \lim_{s\downarrow t^{\ell+1}} \mathbf{v}_{\perp}(s) & = -\mathbf{v}_{\perp}^{\ell} + \int^{t^{\ell}}_{t^{\ell+1}}  F_{\perp}(\mathbf{X} (\tau;t,x,v),\mathbf{V} (\tau;t,x,v) ) \mathrm{d}\tau 
\\ &= -\mathbf v_\perp^\ell + (t^\ell - t^{\ell+1} ) F_\perp(t^\ell) + O(1) |t^\ell - t^{\ell+1} |^2 ( |v|^3 + 1 )
\\ & =  -\mathbf v_\perp^\ell + 2 \mathbf v_\perp^{\ell+1} + O(1) |t^\ell - t^{\ell+1} |^2 ( |v|^3 + 1 ).
\end{split} \end{equation}
This implies $\mathbf v_\perp^\ell - \mathbf v_\perp^{\ell+1}  = O(1) |t^\ell - t^{\ell+1} |^2 ( |v|^3 + 1 )$. Similarly by plugging in
\[
(t^\ell - t^{\ell+1} ) F_\perp(t^\ell) = 2 \mathbf v_\perp^\ell + O(1) |t^\ell - t^{\ell+1} |^2 ( |v|^3 + 1 ),
\]
\eqref{v_ell} becomes
\[
\mathbf{v}_{\perp}^{\ell+1} = -\mathbf v_\perp^\ell + (t^\ell - t^{\ell+1} ) F_\perp(t^\ell) + O(1) |t^\ell - t^{\ell+1} |^2 ( |v|^3 + 1 ) = \mathbf v_\perp^\ell +O(1) |t^\ell - t^{\ell+1} |^2 ( |v|^3 + 1 ).
\]
Thus $\mathbf{v}_{\perp}^{\ell+1} -  \mathbf v_\perp^\ell = O(1) |t^\ell - t^{\ell+1} |^2 ( |v|^3 + 1 )$, therefore 
\Be \label{vlplus1vl}
| \mathbf{v}_{\perp}^{\ell+1} -  \mathbf v_\perp^\ell | = O(1) |t^\ell - t^{\ell+1} |^2 ( |v|^3 + 1 ).
\Ee
From \eqref{vlplus1vl} we have
\Be
|\frac{1}{\mathbf r^\ell} - \frac{1}{\mathbf r^{\ell+1}} |  = \frac{1}{|v|}  \frac{ |\mathbf v_\perp^{\ell+1} - \mathbf v_\perp^{\ell } |}{  |\mathbf v_\perp^{\ell+1} |  |\mathbf v_\perp^{\ell}| } \lesssim \frac{ |t^\ell - t^{\ell+1} |^2 ( |v|^2 + 1 )}{  |\mathbf v_\perp^{\ell+1} |  |\mathbf v_\perp^{\ell}| } \lesssim 1,
\Ee
and
\Be
| 1 - \frac{ \mathbf r^{\ell} }{\mathbf r^{\ell +1 } } | = \frac{ |\mathbf v_\perp^{\ell+1} - \mathbf v_\perp^{\ell } | }{\mathbf v_\perp^{\ell+1}} \lesssim \frac{ |t^\ell - t^{\ell+1} |^2 ( |v|^2 + 1 )}{  |\mathbf v_\perp^{\ell+1} | } \lesssim \mathbf r^\ell.
\Ee
Thus
\[
|  \mathcal{P}^{-1}(\mathbf{r}^{\ell+1} )  \mathcal{P}(\mathbf{r}^{\ell} ) |  
\le  \left[\begin{array}{ccccccc}
  1& 0 & \frac{M}{|v|} &0 & 0 &\frac{M}{|v|} \mathbf r^\ell \\
  0& 1 & \frac{M}{|v|} & 0 & 0 & \frac{M}{|v|} \mathbf r^\ell \\
  0 &0 & 1+M \mathbf r^\ell & 0 & 0 & M (\mathbf r^\ell)^2 \\
  0 & 0 & M & 1 & 0 &  M \mathbf r^\ell\\
  0& 0 &M  & 0 & 1 & M \mathbf r^\ell \\
  0 & 0 & M & 0 &0 &  1+M \mathbf r^\ell
 \end{array}\right] := \mathcal{Q}(\mathbf r^\ell ).
\]
Now we have
\Be \label{Jprodall} \begin{split} 
& J(\mathbf r^{\ell_* })  \times \cdots \times J(\mathbf r^{\ell+1}) \times J(\mathbf r^{\ell} ) \times \cdots J(\mathbf r^2 ) 
\\ \le & \widetilde{  \mathcal P(\mathbf r^{\ell_*} ) }\Lambda (\mathbf r^{\ell_*} )  \mathcal Q ( \mathbf r^{\ell_*-1}  ) \Lambda (\mathbf r^{\ell_* -1}  )  \cdots \mathcal Q (\mathbf r^\ell ) \Lambda ( \mathbf r^\ell ) \cdots \mathcal Q(\mathbf r^2 ) \Lambda (\mathbf r^2 ) \widetilde {\mathcal P^{-1} (\mathbf r^2 )}
\\ \le & \prod_{j = 2}^{\ell_*} (1 + 10 M \mathbf r^j )  \widetilde{  \mathcal P(\mathbf r^{\ell_*} ) }  \mathcal Q ( \mathbf r^{\ell_*-1}  )  \cdots  \mathcal Q(\mathbf r^2 )  \widetilde {\mathcal P^{-1} (\mathbf r^2 )}
\\ \le & C^{C(t-s)|v| } \widetilde{  \mathcal P(\mathbf r^{\ell_*} ) }  \mathcal Q ( \mathbf r^{\ell_*-1}  )  \cdots  \mathcal Q(\mathbf r^2 )  \widetilde {\mathcal P^{-1} (\mathbf r^2 )},
\end{split} \Ee
where we have used $\Lambda (\mathbf r^j ) \le  (1+ 10 M\mathbf r^j ) \mathbf{Id}_{6,6}$, and
\[
 \prod_{j = 2}^{\ell_*} (1 + 10 M \mathbf r^j ) \le  \prod_{i=1}^{ [ \frac{\tilde{t}|v|}{L_{\xi}} ]} \prod_{j = \ell_{i-1} }^{\ell_{i}} (1 + 10 M \mathbf r^j ) \lesssim  \prod_{i=1}^{ [ \frac{\tilde{t}|v|}{L_{\xi}} ]} (1 +10M \mathbf r_i)^{\frac{C_\xi}{\mathbf r_i } } \lesssim C^{C(t-s) |v| }.
\]
Next we estimate $ \mathcal Q ( \mathbf r^{\ell_*-1}  )  \cdots  \mathcal Q(\mathbf r^2 ) $. First by diagonalization we have
\Be \begin{split}
& \mathcal Q(\mathbf r)  = \mathcal R(\mathbf r) \mathcal B(\mathbf r) \mathcal R^{-1}(\mathbf r)
\\ = &    \left[\begin{array}{ccccccc}
  1& 0 & 0 &0 & 0 & \frac{1}{|v|} \\
  0& 1 & 0 & 0 & 0 & \frac{1}{|v|} \\
  0 &0 & 1  & 0 & -\mathbf r & \mathbf r \\
  0 & 0 & 0 & 1 & 0 &  1 \\
  0& 0 &0  & 0 & 0 & 1 \\
  0 & 0 & 0 & 0 & 1 &  1
 \end{array}\right]  
 \left[\begin{array}{ccccccc}
  1& 0 & 0 &0 & 0 & 0 \\
  0& 1 & 0 & 0 & 0 &0 \\
  0 &0 & 1  & 0 & 0 & 0\\
  0 & 0 & 0 & 1 & 0 &  0 \\
  0& 0 &0  & 0 & 1 & 0 \\
  0 & 0 & 0 & 0 & 0 &  1+2M\mathbf r
 \end{array}\right]
 \left[\begin{array}{ccccccc}
  1& 0 & -\frac{1}{2|v| \mathbf r} &0 & 0 &-\frac{1}{2 |v|} \\
  0& 1 & -\frac{1}{2|v| \mathbf r} & 0 & 0 & -\frac{1}{2 |v|} \\
  0 &0 & -\frac{1}{2 \mathbf r} & 0 & 0 & -\frac{1}{2 } \\
  0 & 0 & -\frac{1}{2 \mathbf r} & 1 & 0 &  -\frac{1}{2 }\\
  0& 0 &-\frac{1}{2 \mathbf r}  & 0 & 1 & \frac{1}{2 } \\
  0 & 0 & \frac{1}{2 \mathbf r} & 0 &0 &  \frac{1}{2 }
 \end{array}\right].
 \end{split}
\Ee
Thus
\[
 \prod_{j = 2}^{\ell_*}\mathcal Q ( \mathbf r^j ) \le  \prod_{i=1}^{ [ \frac{\tilde{t}|v|}{L_{\xi}} ]} \prod_{j = \ell_{i-1} }^{\ell_{i}} \mathcal Q ( \mathbf r^j ) \lesssim \prod_{i=1}^{ [ \frac{\tilde{t}|v|}{L_{\xi}} ]} \left[ \mathcal Q ( \mathbf r_i ) \right]^{\ell_i - \ell_{i-1} }  \le  \prod_{i=1}^{ [ \frac{\tilde{t}|v|}{L_{\xi}} ]} \mathcal R(\mathbf r_i )  [\mathcal B (\mathbf r_i ) ]^{\ell_i - \ell_{i-1} }  \mathcal R^{-1}(\mathbf r_i ),
 \]
note that  for some $C \gg 1$
\Be \label{Best}
[\mathcal B (\mathbf r_i ) ]^{\ell_i - \ell_{i-1} }  \lesssim [\mathcal B (\mathbf r_i ) ]^{\frac{C_\xi}{\mathbf r_i } }  \lesssim \text{diag}\Big[   1, 1, 1, 1, 1,  C \Big].
\Ee

Next we have again by explicit computation and using $|\frac{\mathbf r_i }{\mathbf r_{i+1} }  | \lesssim C_\xi$
\Be \begin{split}
 \mathcal R^{-1}(\mathbf r_{i+1}  ) \mathcal R (\mathbf r_i ) = & \left[\begin{array}{ccccccc}
  1& 0 & 0 &0 & \frac{1}{2|v|}( \frac{\mathbf r_i}{\mathbf r_{i+1} } - 1 )&  - \frac{1}{2|v|}( \frac{\mathbf r_i}{\mathbf r_{i+1} } - 1 )\\
  0& 1 & 0 & 0 & \frac{1}{2|v|}( \frac{\mathbf r_i}{\mathbf r_{i+1} } - 1 ) & - \frac{1}{2|v|}( \frac{\mathbf r_i}{\mathbf r_{i+1} } - 1 ) \\
  0 &0 & 1  & 0 & \frac{1}{2}( \frac{\mathbf r_i}{\mathbf r_{i+1} } - 1 ) & - \frac{1}{2}( \frac{\mathbf r_i}{\mathbf r_{i+1} } - 1 )  \\
  0 & 0 & 0 & 1 & \frac{1}{2}( \frac{\mathbf r_i}{\mathbf r_{i+1} } - 1 ) &  - \frac{1}{2}( \frac{\mathbf r_i}{\mathbf r_{i+1} } - 1 )  \\
  0& 0 &0  & 0 & \frac{1}{2}( \frac{\mathbf r_i}{\mathbf r_{i+1} } + 1 ) & - \frac{1}{2}( \frac{\mathbf r_i}{\mathbf r_{i+1} } - 1 )  \\
  0 & 0 & 0 & 0 & - \frac{1}{2}( \frac{\mathbf r_i}{\mathbf r_{i+1} } - 1 ) &   \frac{1}{2}( \frac{\mathbf r_i}{\mathbf r_{i+1} } + 1 )
 \end{array}\right] 
 \\ \lesssim & \left[\begin{array}{ccccccc}
  1& 0 & 0 &0 & \frac{C_\xi}{|v|} & \frac{C_\xi}{|v|} \\
  0& 1 & 0 & 0 &\frac{C_\xi}{|v|}  &\frac{C_\xi}{|v|} \\
  0 &0 & 1  & 0 &C_\xi & C_\xi  \\
  0 & 0 & 0 & 1 &C_\xi &  C_\xi  \\
  0& 0 &0  & 0 & C_\xi & C_\xi  \\
  0 & 0 & 0 & 0 & C_\xi &  C_\xi
 \end{array}\right]  := \mathcal S.
 \end{split}
\Ee
Again we diagonalize $\mathcal S$ as
\Be \begin{split}
\mathcal S  = & \mathcal F \mathcal A \mathcal F^{-1}
\\ := &  \left[\begin{array}{ccccccc}
  0& 1 & 0 &0 &0& \frac{2 C_\xi}{|v| (2 C_\xi -1 )} \\
  0& 0 & 1 & 0 &0  &\frac{2 C_\xi}{|v| (2 C_\xi -1 )} \\
  0 &0 & 0  & 1 & 0  & \frac{2 C_\xi}{ 2 C_\xi -1 }  \\
  0 & 0 & 0 & 0 & 1&  \frac{2 C_\xi}{2 C_\xi -1 } \\
  -1& 0 &0  & 0 & 0 & 1 \\
  1 & 0 & 0 & 0 & 0 &  1
 \end{array}\right] 
  \left[\begin{array}{ccccccc}
  0& 0 & 0 &0 & 0 & 0 \\
  0& 1 & 0 & 0 & 0 &0 \\
  0 &0 & 1  & 0 & 0 & 0\\
  0 & 0 & 0 & 1 & 0 &  0 \\
  0& 0 &0  & 0 & 1 & 0 \\
  0 & 0 & 0 & 0 & 0 &  2C_\xi
 \end{array}\right]
  \left[\begin{array}{ccccccc}
  0& 0 & 0 &0 & -\frac{1}{2} & \frac{1}{2} \\
  1& 0 & 0 & 0 & -\frac{ C_\xi}{|v| (2 C_\xi -1 )} & -\frac{ C_\xi}{|v| (2 C_\xi -1 )}  \\
  0 &1 & 0  & 0 &  -\frac{ C_\xi}{|v| (2 C_\xi -1 )}  &  -\frac{ C_\xi}{|v| (2 C_\xi -1 )}  \\
  0 & 0 & 1 & 0 &  -\frac{ C_\xi}{(2 C_\xi -1 )}  &  -\frac{ C_\xi}{(2 C_\xi -1 )} \\
  0& 0 &0  & 1 & -\frac{ C_\xi}{(2 C_\xi -1 )}  & -\frac{ C_\xi}{(2 C_\xi -1 )}  \\
  0 & 0 & 0 & 0 & \frac{1}{2} &  \frac{1}{2}
 \end{array}\right],
\end{split} \Ee
and directly
\Be \label{Sdiag}
 \begin{split}
\mathcal S^{ [ \frac{\tilde{t}|v|}{L_{\xi}} ]}  = &  \mathcal F \mathcal A^{ [ \frac{\tilde{t}|v|}{L_{\xi}} ]} \mathcal F^{-1}
\\ = & \mathcal F \text{diag}\Big[   0, 1, 1, 1, 1,   (2 C_\xi)^{ [ \frac{\tilde{t}|v|}{L_{\xi}} ]} \Big] \mathcal F^{-1}
\\ = & \left[\begin{array}{ccccccc}
  1& 0 & 0 &0 & \frac{1}{|v|} \frac{C_\xi}{2 C_\xi - 1 }((2C_\xi )^{ [ \frac{\tilde{t}|v|}{L_{\xi}} ]} - 1) & \frac{1}{|v|} \frac{C_\xi}{2 C_\xi - 1 }((2C_\xi )^{ [ \frac{\tilde{t}|v|}{L_{\xi}} ]} - 1) \\
  0& 1 & 0 & 0 &\frac{1}{|v|} \frac{C_\xi}{2 C_\xi - 1 }((2C_\xi )^{ [ \frac{\tilde{t}|v|}{L_{\xi}} ]} - 1)  &\frac{1}{|v|} \frac{C_\xi}{2 C_\xi - 1 }((2C_\xi )^{ [ \frac{\tilde{t}|v|}{L_{\xi}} ]} - 1) \\
  0 &0 & 1  & 0 & \frac{C_\xi}{2 C_\xi - 1 }((2C_\xi )^{ [ \frac{\tilde{t}|v|}{L_{\xi}} ]} - 1) &\frac{C_\xi}{2 C_\xi - 1 }((2C_\xi )^{ [ \frac{\tilde{t}|v|}{L_{\xi}} ]} - 1)  \\
  0 & 0 & 0 & 1 & \frac{C_\xi}{2 C_\xi - 1 }((2C_\xi )^{ [ \frac{\tilde{t}|v|}{L_{\xi}} ]} - 1) &   \frac{C_\xi}{2 C_\xi - 1 }((2C_\xi )^{ [ \frac{\tilde{t}|v|}{L_{\xi}} ]} - 1)  \\
  0& 0 &0  & 0 &  \frac{(2C_\xi )^{ [ \frac{\tilde{t}|v|}{L_{\xi}} ]} }{2 }& \frac{(2C_\xi )^{ [ \frac{\tilde{t}|v|}{L_{\xi}} ]} }{2 }  \\
  0 & 0 & 0 & 0 & \frac{(2C_\xi )^{ [ \frac{\tilde{t}|v|}{L_{\xi}} ]} }{2 } & \frac{(2C_\xi )^{ [ \frac{\tilde{t}|v|}{L_{\xi}} ]} }{2 }
 \end{array}\right]
\\  \le &    \left[\begin{array}{ccccccc}
  1& 0 & 0 &0 & \frac{1}{|v|} C^{[ \frac{\tilde{t}|v|}{L_{\xi}} ]} & \frac{1}{|v|} C^{[ \frac{\tilde{t}|v|}{L_{\xi}} ]} \\
  0& 1 & 0 & 0 & \frac{1}{|v|} C^{[ \frac{\tilde{t}|v|}{L_{\xi}} ]} & \frac{1}{|v|} C^{[ \frac{\tilde{t}|v|}{L_{\xi}} ]} \\
  0 &0 & 1  & 0 & C^{[ \frac{\tilde{t}|v|}{L_{\xi}} ]} &C^{[ \frac{\tilde{t}|v|}{L_{\xi}} ]}  \\
  0 & 0 & 0 & 1 &C^{[ \frac{\tilde{t}|v|}{L_{\xi}} ]}&  C^{[ \frac{\tilde{t}|v|}{L_{\xi}} ]} \\
  0& 0 &0  & 0 & C^{[ \frac{\tilde{t}|v|}{L_{\xi}} ]}&C^{[ \frac{\tilde{t}|v|}{L_{\xi}} ]} \\
  0 & 0 & 0 & 0 & C^{[ \frac{\tilde{t}|v|}{L_{\xi}} ]} & C^{[ \frac{\tilde{t}|v|}{L_{\xi}} ]}
 \end{array}\right] := \mathcal  D.
\end{split} \Ee
Therefore from \eqref{Best} and \eqref{Sdiag} we have for some $C_1 \gg 1$,
\Be
 \prod_{j = 2}^{\ell_*}\mathcal Q ( \mathbf r^j ) \le  C_1^{[ \frac{\tilde{t}|v|}{L_{\xi}} ]} \widetilde{\mathcal R( \mathbf r_{[ \frac{\tilde{t}|v|}{L_{\xi}} ]} )} \mathcal F \mathcal A^{ [ \frac{\tilde{t}|v|}{L_{\xi}} ]} \mathcal F^{-1} \widetilde{ \mathcal  R^{-1} (\mathbf r_1)} \le C_1^{[ \frac{\tilde{t}|v|}{L_{\xi}} ]} \widetilde{\mathcal R( \mathbf r_{[ \frac{\tilde{t}|v|}{L_{\xi}} ]} )} \mathcal D \widetilde{ \mathcal  R^{-1} (\mathbf r_1)}
\Ee

Finally using $\mathbf r_1 \sim \mathbf r^2 \sim \mathbf r^1$, and $\mathbf r^{\ell_*} \sim \mathbf r_{[ \frac{\tilde{t}|v|}{L_{\xi}} ]}$. Putting everything together we have from \eqref{Jprodall}, for $C_2 \gg 1 $,
\Be \begin{split}
& J(\mathbf r^{\ell_* })  \times \cdots \times J(\mathbf r^{\ell+1}) \times J(\mathbf r^{\ell} ) \times \cdots J(\mathbf r^2 ) 
\\ \le &C^{C(t-s)|v| } \widetilde{  \mathcal P(\mathbf r^{\ell_*} ) }  \prod_{j = 2}^{\ell_*}\mathcal Q ( \mathbf r^j ) \widetilde {\mathcal P^{-1} (\mathbf r^2 )}
\\ \le &  C_2^{C_2 (t-s)|v| }  \widetilde{  \mathcal P(\mathbf r^{\ell_*} ) }  \widetilde{\mathcal R( \mathbf r_{[ \frac{\tilde{t}|v|}{L_{\xi}} ]} )} \mathcal D \widetilde{ \mathcal  R^{-1} (\mathbf r_1)} \widetilde {\mathcal P^{-1} (\mathbf r^2 )}
\\ \lesssim & C_2^{C_2 (t-s)|v| }  \widetilde{  \mathcal P(\mathbf r^{\ell_*} ) }  \widetilde{\mathcal R( \mathbf r^{\ell_*}  )} \mathcal D \widetilde{ \mathcal  R^{-1} (\mathbf r^1)} \widetilde {\mathcal P^{-1} (\mathbf r^1 )}
\\ = & C_2^{C_2 (t-s)|v| } \times 
\\ &   \left[\begin{array}{cccccc}
\frac{1}{5|v|} &  \frac{1}{5|v|} &  \frac{1}{5|v|^2} &  \frac{1}{5|v|^2} &  \frac{6}{5|v|^2} &  \frac{2}{|v|^2}
\\ 1 & 0 & 0 & 0 & \frac{1}{|v|} & \frac{2}{|v|} 
\\ 0 & 1 & 0 & 0 &  \frac{1}{|v|} & \frac{2}{|v|}
\\ 0 & 0 & 0 & 0 & \frac{2}{\mathbf r^{\ell_*} } & \frac{2}{\mathbf r^{\ell_* } }
\\ 0 & 0 & 1 & 0 &1 &2
\\ 0 & 0 & 0 & 1 & 1 & 2
  \end{array} \right]   \left[\begin{array}{ccccccc}
  1& 0 & 0 &0 & \frac{1}{|v|} C^{[ \frac{\tilde{t}|v|}{L_{\xi}} ]} & \frac{1}{|v|} C^{[ \frac{\tilde{t}|v|}{L_{\xi}} ]} \\
  0& 1 & 0 & 0 & \frac{1}{|v|} C^{[ \frac{\tilde{t}|v|}{L_{\xi}} ]} & \frac{1}{|v|} C^{[ \frac{\tilde{t}|v|}{L_{\xi}} ]} \\
  0 &0 & 1  & 0 & C^{[ \frac{\tilde{t}|v|}{L_{\xi}} ]} &C^{[ \frac{\tilde{t}|v|}{L_{\xi}} ]}  \\
  0 & 0 & 0 & 1 &C^{[ \frac{\tilde{t}|v|}{L_{\xi}} ]}&  C^{[ \frac{\tilde{t}|v|}{L_{\xi}} ]} \\
  0& 0 &0  & 0 & C^{[ \frac{\tilde{t}|v|}{L_{\xi}} ]}&C^{[ \frac{\tilde{t}|v|}{L_{\xi}} ]} \\
  0 & 0 & 0 & 0 & C^{[ \frac{\tilde{t}|v|}{L_{\xi}} ]} & C^{[ \frac{\tilde{t}|v|}{L_{\xi}} ]}
 \end{array}\right] 
\left[  \begin{array}{cccccc}
|v| & 1 & \frac{1}{5} & \frac{3}{5 \mathbf r^1|v| } & \frac{1}{5|v|} & \frac{1}{5|v|}
\\ |v| & \frac{1}{5} & 1 & \frac{3}{5 \mathbf r^1|v| } & \frac{1}{5|v|} & \frac{1}{5|v|}
\\ |v|^2 & \frac{|v|}{5} & \frac{|v|}{5} & \frac{3}{5 \mathbf r^1 } & 1 & \frac{1}{5} 
\\ |v|^2 & \frac{|v|}{5} & \frac{|v|}{5} & \frac{3}{5 \mathbf r^1 } & \frac{1}{5} & 1
\\ \frac{|v|^2}{2} & \frac{|v|}{10} & \frac{|v|}{10} & \frac{1}{2 \mathbf r^1 } & \frac{1}{10} & \frac{1}{10}
\\ \frac{|v|^2}{2} & \frac{|v|}{10} & \frac{|v|}{10} & \frac{1}{2 \mathbf r^1 } & \frac{1}{10} & \frac{1}{10}
\end{array} \right]
\\ =  &C_2^{C_2 (t-s)|v| }   \left[\begin{array}{ccccccc}
  4 C^{[ \frac{\tilde{t}|v|}{L_{\xi}} ]} + \frac{4}{5}& \frac{20  C^{[ \frac{\tilde{t}|v|}{L_{\xi}} ]} +8 }{25|v|}  &  \frac{20  C^{[ \frac{\tilde{t}|v|}{L_{\xi}} ]} +8 }{25|v|}  &\frac{4(25 C^{[ \frac{\tilde{t}|v|}{L_{\xi}} ]} +3)}{25 \mathbf r^1 |v|^2} & \frac{4(5 C^{[ \frac{\tilde{t}|v|}{L_{\xi}} ]} +2)}{25 |v|^2}  & \frac{4(5 C^{[ \frac{\tilde{t}|v|}{L_{\xi}} ]} +2)}{25 |v|^2}  \\
  |v| (4  C^{[ \frac{\tilde{t}|v|}{L_{\xi}} ]} + 1)& \frac{4 C^{[ \frac{\tilde{t}|v|}{L_{\xi}} ]}}{5} +1 & \frac{4 C^{[ \frac{\tilde{t}|v|}{L_{\xi}} ]}}{5} +\frac{1}{5} & \frac{20 C^{[ \frac{\tilde{t}|v|}{L_{\xi}} ]} +3}{5\mathbf r^1 |v|} &\frac{4 C^{[ \frac{\tilde{t}|v|}{L_{\xi}} ]} +1}{5|v|} & \frac{4 C^{[ \frac{\tilde{t}|v|}{L_{\xi}} ]} +1}{5|v|} \\
   |v| (4  C^{[ \frac{\tilde{t}|v|}{L_{\xi}} ]} + 1) &\frac{4 C^{[ \frac{\tilde{t}|v|}{L_{\xi}} ]}}{5} +\frac{1}{5} & \frac{4 C^{[ \frac{\tilde{t}|v|}{L_{\xi}} ]}}{5} +1  & \frac{20 C^{[ \frac{\tilde{t}|v|}{L_{\xi}} ]} +3}{5\mathbf r^1 |v|} & \frac{4 C^{[ \frac{\tilde{t}|v|}{L_{\xi}} ]} +1}{5|v|}  & \frac{4 C^{[ \frac{\tilde{t}|v|}{L_{\xi}} ]} +1}{5|v|}  \\
  4 C^{[ \frac{\tilde{t}|v|}{L_{\xi}} ]} \mathbf r^{\ell_*} |v|^2 & \frac{4}{5} C^{[ \frac{\tilde{t}|v|}{L_{\xi}} ]} \mathbf r^{\ell_*} |v| &  \frac{4}{5} C^{[ \frac{\tilde{t}|v|}{L_{\xi}} ]} \mathbf r^{\ell_*} |v| & \frac{4 C^{[ \frac{\tilde{t}|v|}{L_{\xi}} ]} \mathbf r^{\ell_*} }{\mathbf r^1 } & \frac{4}{5}  C^{[ \frac{\tilde{t}|v|}{L_{\xi}} ]} \mathbf r^{\ell_* } &    \frac{4}{5}  C^{[ \frac{\tilde{t}|v|}{L_{\xi}} ]} \mathbf r^{\ell_* }\\
   |v|^2 (4  C^{[ \frac{\tilde{t}|v|}{L_{\xi}} ]} + 1)& \frac{|v|}{5} ({4 C^{[ \frac{\tilde{t}|v|}{L_{\xi}} ]}+1} )& \frac{|v|}{5} ({4 C^{[ \frac{\tilde{t}|v|}{L_{\xi}} ]}+1} )  & \frac{20 C^{[ \frac{\tilde{t}|v|}{L_{\xi}} ]}+3}{5 \mathbf r^1} & \frac{4 C^{[ \frac{\tilde{t}|v|}{L_{\xi}} ]}}{5} +1 & \frac{4 C^{[ \frac{\tilde{t}|v|}{L_{\xi}} ]}}{5} +\frac{1}{5} \\
   |v|^2 (4  C^{[ \frac{\tilde{t}|v|}{L_{\xi}} ]} + 1) &  \frac{|v|}{5} ({4 C^{[ \frac{\tilde{t}|v|}{L_{\xi}} ]}+1} ) &  \frac{|v|}{5} ({4 C^{[ \frac{\tilde{t}|v|}{L_{\xi}} ]}+1} )& \frac{20 C^{[ \frac{\tilde{t}|v|}{L_{\xi}} ]}+3}{5 \mathbf r^1} &\frac{4 C^{[ \frac{\tilde{t}|v|}{L_{\xi}} ]}}{5} +\frac{1}{5}  & \frac{4 C^{[ \frac{\tilde{t}|v|}{L_{\xi}} ]}}{5} +1
 \end{array}\right]
  \\
 & \lesssim  C^{C|t-s||v|}
 \left[\begin{array}{c|c|c|c}
O_\xi(1)   & \frac{1}{|v|} & \frac{1}{|v| |\mathbf{v}_{\perp}^{1}|} & \frac{1}{|v|^{2}}  \\ \hline
|v| &  O_{\xi}(1) & \frac{1}{|\mathbf{v}_{\perp}^{1}|} & \frac{1}{|v|} \\ \hline
 |\mathbf v_\perp^1||v|  & |\mathbf{v}_{\perp}^{1}| & O_{\xi}(1) & \frac{|\mathbf{v}_{\perp}^{1}|}{|v|}\\ \hline
|v|^2 & |v| & \frac{|v|}{|\mathbf{v}_{\perp}^{1}|} & O_{\xi}(1)
 \end{array} \right]_{6\times6}\label{whole_inter},
\end{split} \Ee
where we have used (\ref{time_angle}) and the Velocity lemma (Lemma \ref{velocitylemma}) and \eqref{tbest} and 
\[
\mathbf{r}_{i} = \mathcal{C}_{1} e^{\frac{\mathcal{C}}{2}C_{1}} \mathbf{r}^{i} \lesssim e^{C|t-s||v|} \frac{|\mathbf{v}_{\perp}^{1}|}{|v|}, \ \ \text{and} \ \frac{\mathbf{r}_{[ \frac{|t-s||v| }{L_{\xi}}]}}{\mathbf{r}_{1}}  = \frac{\mathbf{r}^{[  \frac{|t-s||v|}{L_{\xi}}]}}{\mathbf{r}^{1}}
= \frac{\Big|\mathbf{v}_{\perp}^{[ \frac{   |t-s||v| }{L_{\xi}}  ]}\Big|}{|\mathbf{v}_{\perp}^{1}|} \leq
\mathcal{C}_{1} e^{\frac{\mathcal{C}}{2} |v||t-s|}.
\]

The case for $\ell$ is \textit{Type I} is easier, we first claim
\begin{equation}\label{one_group2}
\begin{split}
&J^{\ell_{*}}_{\ell_{*}-1} \times \cdots \times J_{1}^{2}\\
 =& \ {\frac{\partial (t^{\ell_{*}} , \mathbf{x}_{\parallel_{{\ell_{*}}}}^{\ell_{*}}, \mathbf{v}_{\perp_{{\ell_{*}}}}^{\ell_{*}}, \mathbf{v}_{\parallel_{{\ell_{*}}}}^{\ell_{*}})}{\partial (t^{\ell_{*}-1} , \mathbf{x}_{\parallel_{{\ell_{*}}-1}}^{\ell_{*}-1}, \mathbf{v}_{\perp_{{\ell_{*}}-1}}^{\ell_{*}-1}, \mathbf{v}_{\parallel_{{\ell_{*}}-1}}^{\ell_{*}-1})}
 \times \cdots \times
 \frac{\partial (t^2 , \mathbf{x}_{\parallel_{2}}^{2}, \mathbf{v}_{\perp_{2}}^2, \mathbf{v}_{\parallel_{2}}^2)}{\partial (t^{1 } , \mathbf{x}_{\parallel_{{1  }}}^{1  }, \mathbf{v}_{\perp_{{1  }}}^{1 }, \mathbf{v}_{\parallel_{{1  }}}^{1 })}  }\\
 \leq & \ 
 \mathcal{P}(\mathbf{v}_{\perp}^1) (\Lambda(\mathbf{v}_{\perp}^1))^{\frac{C_{\xi}}{\mathbf{v}_{\perp}^1}}  \mathcal{P}^{-1}( \mathbf{v}_{\perp}^1).
\end{split}
\end{equation}
From the same arguments between \eqref{one_group} and \eqref{bound_ri}, we have
\begin{equation}\label{bound_ri2}
\frac{1}{(\mathcal{C}_{1})^{2}} e^{-\mathcal{C} C_{1}} \mathbf{v}_{\perp}^1  \ \leq  \ \mathbf{v}_{\perp}^{j} \ \leq \ \mathbf{v}_{\perp}^1 \ \ \ \text{for all} \ \  1 \leq j \leq \ell_{*}.
\end{equation}
Therefore
\[
J^{\ell_{*}}_{\ell_{*}-1} \times \cdots \times J^{2}_{1} \leq  \mathcal{P}(\mathbf{v}_{\perp}^1) (\Lambda(\mathbf{v}_{\perp}^1))^{|\ell_* |}  \mathcal{P}^{-1}( \mathbf{v}_{\perp}^1)
\]
Now we only left to prove $|\ell_*| \lesssim_{\Omega} \frac{1}{\mathbf{v}_{\perp}^1}$: For any $1 \leq j \leq \ell_{*}$, we have $\xi(x^{j})=0=\xi(x^{j+1})=\xi(x^{j}-(t^{j}-t^{j+1})v^{j}).$ We expand $\xi(x^{j}-(t^{j}-t^{j+1})v^{j})$ in time to have
\begin{equation}\notag
\begin{split}
\xi(x^{j+1}) &=  \xi(x^{j}) + \int^{t^{j+1}}_{t^{j}} \frac{d}{ds} \xi(X_{\mathbf{cl}}(s)) \mathrm{d}s\\
&= \xi(x^{j}) + (v^{j}\cdot \nabla \xi(x^{j}) ) (t^{j+1}-t^{j}) + \int^{t^{j+1}}_{t^{j}} \int^{s}_{t^{j}} \frac{d^{2}}{d\tau^{2}} \xi(X_{\mathbf{cl}}(\tau))  \mathrm{d}\tau\mathrm{d}s,\\
\end{split}
\end{equation}
and  
\begin{equation}\notag
0 = (v^{j}\cdot \nabla \xi(x^{j})) (t^{j+1}-t^{j}) + \frac{(t^{j}-t^{j+1})^{2}}{2} (V_{\mathbf{cl}} (\tau_* )  \cdot \nabla^{2}\xi(X_{\mathbf{cl}}(\tau_{*})) \cdot V_{\mathbf{cl}} (\tau_* ) + E(\tau, X_{\mathbf{cl}} (\tau_* ) ) \cdot \nabla\xi(X_{\mathbf{cl}}(\tau_{*})) ),
\end{equation}
for some $ \tau_{*} \in [t^{j+1},t^{j}]$. Therefore there exists $C_{2} (\delta, \xi, E)\gg 1$ such that
\begin{equation}\notag
\begin{split}
|t^{j} - t^{j+1} | = \frac{2 v^{j}\cdot \nabla \xi(x^{j})}{V_{\mathbf{cl}} (\tau_* )  \cdot \nabla^{2}\xi(X_{\mathbf{cl}}(\tau_{*})) \cdot V_{\mathbf{cl}} (\tau_* ) + E(\tau, X_{\mathbf{cl}} (\tau_* ) ) \cdot \nabla\xi(X_{\mathbf{cl}}(\tau_{*}))} \ge \frac{1}{C_{2}} \mathbf v_\perp^j \gtrsim \mathbf v_\perp^1
.
\end{split}
\end{equation}
Thus
\[
|\ell_{*}|  \leq \frac{T}{ \min_{ j }|t^{j}-t^{j+1}|}  \lesssim \frac{T}{\mathbf v_\perp^1},
\]
and this completes our claim (\ref{one_group2}).

Then directly from \eqref{one_group2} we have for some $C \gg 1$,
\Be \begin{split}
&J^{\ell_{*}}_{\ell_{*}-1} \times \cdots \times J_{1}^{2}
\\ \le  & \widetilde{ \mathcal{P}(\mathbf{v}_{\perp}^1) } (\Lambda(\mathbf{v}_{\perp}^1))^{\frac{C_{\xi}}{\mathbf{v}_{\perp}^1}}  \widetilde{ \mathcal{P}^{-1}( \mathbf{v}_{\perp}^1) }
\\ \le & \widetilde{ \mathcal{P}(\mathbf{v}_{\perp}^1) } ((1+M \mathbf v_\perp^1 ) ^{\frac{C_\xi}{\mathbf v_\perp^1 } }  \mathbf{Id}_{6,6} )  \widetilde{ \mathcal{P}^{-1}( \mathbf{v}_{\perp}^1) }
\\ \le & C  \widetilde{ \mathcal{P}(\mathbf{v}_{\perp}^1) }  \widetilde{ \mathcal{P}^{-1}( \mathbf{v}_{\perp}^1) }
\\  \le &C  \left[\begin{array}{c|cc|c|cc}
1 &\frac{9}{25} & \frac{9}{25} &  \frac{9}{25 | \mathbf v_\perp^1 |} & \frac{9}{25} & \frac{9}{25} \\ \hline
 1 &  1  & \frac{1}{5}  & \frac{1}{5 |\mathbf v_\perp^1 | } & \frac{1}{5} & \frac{1}{5}  \\
1 &\frac{1}{5}  & 1  & \frac{1}{5 |\mathbf v_\perp^1 | } & \frac{1}{5} & \frac{1}{5} \\ \hline
| \mathbf v_\perp^1 |& \frac{ | \mathbf v_\perp^1 |}{5} & \frac{ | \mathbf v_\perp^1 |}{5}  & 1   & \frac{| \mathbf v_\perp^1 | }{5} & \frac{| \mathbf v_\perp^1 | }{5} \\ \hline
1 & \frac{1}{5} & \frac{1}{5} &  \frac{1}{5 |\mathbf v_\perp^1 | }  & 1 & \frac{1}{5}
\\ 1 & \frac{1}{5} &\frac{1}{5} & \frac{1}{5 |\mathbf v_\perp^1 | }  & \frac{1}{5} & 1
 \end{array}
 \right].  
\end{split}  \label{wholeitvsmall}
\Ee
 
 \vspace{8pt}

\noindent\textit{Step 8. Intermediate summary for the matrix method and the final estimate for \textit{Type III}}

\vspace{4pt}

Recall from (\ref{chain}) and (\ref{s1tstar}), (\ref{whole_inter}), (\ref{t1s1}),
\begin{equation}\notag
\begin{split}
&\frac{\partial (s^{\ell_{*}},\mathbf{X}_{ \ell_{*}}(s^{\ell_{*}}),\mathbf{V}_{  \ell_{*}}(s^{\ell_{*}}))}{\partial (s^{1},\mathbf{X}_{1}(s^{1}),\mathbf{V}_{1}(s^{1}) )}\equiv
\frac{\partial (s^{\ell_{*}}, \mathbf{x}_{\perp_{\ell_{*}}}(s^{\ell_{*}}),\mathbf{x}_{\parallel_{\ell_{*}}}(s^{\ell_{*}}), \mathbf{v}_{\perp_{\ell_{*}}}(s^{\ell_{*}}), \mathbf{v}_{\parallel_{\ell_{*}}}(s^{\ell_{*}})   )}{\partial ( s^{1}, \mathbf{x}_{\perp_{1}}(s^{1}), \mathbf{x}_{\parallel_{1}}(s^{1}), \mathbf{v}_{\perp_{1}}(s^{1}), \mathbf{v}_{\parallel_{1}}(s^{1})   )}
\\
 = & 
\ \frac{ \partial (s^{\ell_{*}}, \mathbf{x}_{\perp_{\ell_{*}}}(s^{\ell_{*}}),\mathbf{x}_{\parallel_{\ell_{*}}}(s^{\ell_{*}}), \mathbf{v}_{\perp_{\ell_{*}}}(s^{\ell_{*}}), \mathbf{v}_{\parallel_{\ell_{*}}}(s^{\ell_{*}})   )}{\partial (t^{\ell_{*}} , \mathbf{x}_{\parallel_{\ell_{*}}}^{\ell_{*}}, \mathbf{v}_{\perp_{\ell_{*}}}^{\ell_{*}}, \mathbf{v}_{\parallel_{\ell_{*}}}^{\ell_{*}})}    \\
&  \times\prod_{i=1}^{ [\frac{|t-s||v|}{L_{\xi}}]}  \frac{\partial (t^{\ell_{i+1}} , \mathbf{x}_{\parallel_{ \ell_{i+1}}}^{\ell_{i+1}}, \mathbf{v}_{\perp_{ \ell_{i+1}}}^{\ell_{i+1}}, \mathbf{v}_{\parallel_{ \ell_{i+1}}}^{\ell_{i+1}})}
{\partial (t^{\ell_{i+1}-1} , \mathbf{x}_{\parallel_{{\ell_{i+1}-1}}}^{\ell_{i+1}-1}, \mathbf{v}_{\perp_{{\ell_{i+1}-1}}}^{\ell_{i+1}-1}, \mathbf{v}_{\parallel_{{\ell_{i+1}-1}}}^{\ell_{i+1}-1})}
 \times \cdots \times
 \frac{\partial (t^{\ell_{i}+1} , \mathbf{x}_{\parallel_{\ell_{i} +1}}^{\ell_{i} +1}, \mathbf{v}_{\perp_{\ell_{i} +1}}^{\ell_{i}+1}, \mathbf{v}_{\parallel_{\ell_{i} +1}}^{\ell_{i}+1})}{\partial (t^{\ell_{i} } , \mathbf{x}_{\parallel_{\ell_{i} }}^{\ell_{i}  }, \mathbf{v}_{\perp_{\ell_{i} }}^{\ell_{i} }, \mathbf{v}_{\parallel_{\ell_{i} }}^{\ell_{i} })}   \\
& \times    \frac{\partial (t^{1} , \mathbf{x}_{\parallel_{1}}^{1}, \mathbf{v}_{\perp_{1}}^{1}, \mathbf{v}_{\parallel_{1}}^{1})}{\partial (s^{1}, \mathbf{x}_{\perp_{1}}(s^{1}), \mathbf{x}_{\parallel_{1}}(s^{1}), \mathbf{v}_{\perp_{1}}(s^{1}), \mathbf{v}_{\parallel_{1}}(s^{1}) )}
 \\
 \leq&  \ (\ref{s1tstar}) \times (\ref{whole_inter}) \times (\ref{t1s1}) .
 \end{split}
 \end{equation}
 Then directly since $|v| > \delta$, we bound it by
 \begin{equation}\label{inter_1}
 \begin{split}
 \leq & \ (\ref{s1tstar}) \times C^{C|t-s||v|}\\
 &\times { \left[\begin{array}{c|cc|cc}
\frac{|v|^2}{|\mathbf v_\perp^1 |^2} & \frac{1}{|\mathbf{v}_{\perp}^{1}|} + \frac{|v|}{|\mathbf{v}_{\perp}^{1}|^{2}} + |t^{1}-s^{1}| & \frac{1}{|v|} & \frac{1}{|v||\mathbf{v}_{\perp}^{1}|}   + |s^{1}-t^{1}|^{2} & \frac{1}{|v|^{2}} + \frac{|s^{1}-t^{1}|}{|v|} \\ \hline
  \frac{|v|^3 }{|\mathbf v_\perp^1 | } + \frac{|v|^2}{|\mathbf v_\perp^1 |^2} &  \frac{|v|^{2}}{|\mathbf{v}_{\perp}^{1}|^{2}} + \frac{|v|}{|\mathbf{v}_{\perp}^{1}|} + |v||s^{1}-t^{1}| & O_\xi(1)  & \frac{1}{|\mathbf{v}_{\perp}^{1}|}   + |s^{1}-t^{1}| & \frac{1}{|v|}  
  \\ \hline
 \frac{|v|^3 }{|\mathbf v_\perp^1 | } &\frac{|v|^{2}}{|\mathbf{v}_{\perp}^{1}|} + |v| &|\mathbf v_\perp^1 | & O_{\xi}(1)  & \frac{|\mathbf{v}_{\perp}^{1}|}{|v|}\\
 \frac{|v|^4}{|\mathbf v_\perp^1 |^2} & \frac{|v|^{3}}{|\mathbf{v}_{\perp}^{1}|^{2}} + \frac{|v|^{2}}{|\mathbf{v}_{\perp}^{1}|} + |v|^{2} |s^1-t^{1}| & |v|  & \frac{|v|}{|\mathbf{v}_{\perp}^{1}|} + |v||s^1-t^{1}| & O_{\xi}(1)
 \end{array}
 \right] },
\end{split}
\end{equation}
 where we have used the Velocity lemma (Lemma \ref{velocitylemma}) and (\ref{time_angle}), \eqref{tbest}, and
 \[
 |v||t^{1}-s^{1}| \leq  \lesssim_{\Omega} \min \{\frac{|\mathbf{v}_{\perp}^{1}|}{|v|}, (t-s)|v|  \}
 \lesssim_{\Omega} C^{C|t-s||v|} \min \{\frac{|\mathbf{v}_{\perp}^{1}|}{|v|}, 1 \}.
 \]
 Again we use the Velocity lemma (Lemma \ref{velocitylemma}), \eqref{tbest}, and
\begin{eqnarray*}%
 |v||t^{\ell_{*}}-s^{\ell_{*}} | \leq \min\{ |v||t^{\ell_{*}} - t^{\ell_{*}+1}| ,| t-s| |v|\}\lesssim_{\Omega}  \min \{\frac{|\mathbf{v}_{\perp}^{\ell_{*}}|}{|v|}, |t-s||v|\}
 \lesssim_{\Omega} C^{C|t-s||v|}  \min \{\frac{|\mathbf{v}_{\perp}^{1}|}{|v|}, 1\}
 ,\\
\end{eqnarray*}
and
$|\mathbf{v}_{\perp}(s^{\ell_{*}})|   \lesssim_{\Omega} C^{C|v|(t-s)} |\mathbf{v}_{\perp}^{1}|$ to have, from (\ref{inter_1})
\begin{equation}\label{middle}
\begin{split}
\frac{\partial (s^{\ell_{*}},\mathbf{X}_{ \ell_{*}}(s^{\ell_{*}}),\mathbf{V}_{  \ell_{*}}(s^{\ell_{*}}))}{\partial (s^{1},\mathbf{X}_{1}(s^{1}),\mathbf{V}_{1}(s^{1}) )}
\lesssim C^{C|t-s||v|}  \left[\begin{array}{c|cc|cc}
 0 &  0 &  \mathbf{0}_{1,2} &  0 & \mathbf{0}_{1,2}  \\ \hline
 \frac{|v|^2}{| \mathbf v_\perp^1 |} &     \frac{|v|}{|\mathbf{v}_{\perp}^{1}|}  & \frac{|\mathbf v_\perp^1|}{|v|} & \frac{1}{|v|}& \frac{1}{|v|} \\
 \frac{|v|^3}{| \mathbf v_\perp^1 |^2} &     \frac{|v|^{2}}{|\mathbf{v}_{\perp}^{1}|^{2}} 
 & 1 &  \frac{1}{|\mathbf{v}_{\perp}^{1}|} &   \frac{1}{|v|}\\ \hline
&&&   \\
 \frac{|v|^4}{| \mathbf v_\perp^1 |^2}&     \frac{|v|^{3}}{|\mathbf{v}_{\perp}^{1}|^{2}} & |v| & 
   \frac{|v|}{|\mathbf{v}_{\perp}^{1}|} & O_{\xi}(1) \\
&&&
   \end{array}\right]_{7\times 7}.
\end{split}
\end{equation}

We consider the following case:
\begin{equation}\label{typeII}
\text{There exists } \ell \in [\ell_{*}(s;t,x,v), 0]  \text{ such that } \mathbf{r}^{\ell} \geq \sqrt{\delta}.
\end{equation}
Therefore $\ell$ is \textit{Type III} in (\ref{type_r}). Equivalently $\tau \in [t^{\ell+1}, t^{\ell}]$ for some $\ell_{*}\leq \ell \leq 0$ and $|\xi(X_{\mathbf{cl}}(\tau;t,x,v))|\geq C\delta$. By the Velocity lemma (Lemma \ref{velocitylemma}), for all $1 \leq i \leq  \ell_{*}(s;t,x,v) ,$
\[
|\mathbf{r}^{i}| \ =  \ \frac{|\mathbf{v}_{\perp}^{i}|}{|v|} \ \gtrsim_{\xi} \ e^{-C_{\xi}|v||t^{i}-t^{\ell}|} |\mathbf{r}^{\ell}| \ \gtrsim_{\xi} \ e^{-C_{\xi}|v|(t-s)} \sqrt{\delta}.
\]
Especially, for all $1 \leq i \leq  \ell_{*}(s;t,x,v) ,$
\[
|\mathbf{r}^{1}| \gtrsim_{\xi} e^{-C_{\xi}|v|(t-s)} \sqrt{\delta}, \ \ \ \frac{1}{|\mathbf{r}^{i}|} = \frac{|v|}{|\mathbf{v}_{\perp}^{i}|} \lesssim_{\xi} \frac{e^{C_{\xi}|v|(t-s)}}{\sqrt{\delta}}.
\]
Note that $\ell_{*}(s;t,x,v) \lesssim \max_{i}\frac{|v||t-s|}{ \mathbf{r}^{i}} \lesssim_{\delta} C^{C|v||t-s|}.$

Therefore in the case of (\ref{typeII}), from (\ref{middle}),
\begin{equation}\notag
 \begin{split}
  \frac{\partial (s^{\ell_{*}}, \mathbf{X}_{\ell_{*}} (s^{\ell_{*}}), \mathbf{V}_{\ell_{*} } (s^{\ell_{*}})   )}{\partial (s^{1}, \mathbf{X}_{1} (s^{1}), \mathbf{V}_{1} (s^{1})   )} 
 &\lesssim C^{C(t-s)|v|} \left[\begin{array}{c|cc|ccc}
 0 & 0 & \mathbf{0}_{1,2} & 0 & \mathbf{0}_{1,2} &\\ \hline
 |v| &  \frac{1}{\sqrt{\delta}}&     \frac{1}{\sqrt{\delta}}& \frac{1}{|v|}& \frac{1}{|v|}&\\
 |v| &  \frac{1}{ {\delta}}    &  \frac{1}{ {\delta}}    &  \frac{1}{|v|} \frac{1}{\sqrt{\delta}}&   \frac{1}{|v|}\\ \hline
&&&&& \\
|v|^{2}& |v|  \frac{1}{ {\delta}}   &     |v|     \frac{1}{ {\delta}}   & \frac{1}{\sqrt{\delta}} & 1 \\
&&&&&
   \end{array}\right]\\
   & \lesssim_{\delta} C^{C|v|(t-s)} \left[\begin{array}{c|c|c} 0 & \mathbf{0}_{1,3} & \mathbf{0}_{1,3} \\ \hline |v| & 1 & \frac{1}{|v|} \\ \hline |v|^{2} & |v|  & 1 \end{array}\right].
 \end{split}
 \end{equation}
From (\ref{ss1}) and (\ref{s1t}) we conclude
\begin{equation}\label{final_Dxv_nongrazing}
\begin{split}
 &\frac{\partial ( X_{\mathbf{cl}}(s;t,x,v),V_{\mathbf{cl}}(s;t,x,v))}{\partial (t,x,v)} \\
 &  \lesssim_{\delta,\xi} \ C^{C|v|(t-s)}  \ 
 \frac{\partial ( X_{\mathbf{cl}}(s), V_{\mathbf{cl}}(s))}{\partial (s^{\ell_{*}}, \mathbf{X}_{\ell_{*} }(s^{\ell_{*}}), \mathbf{V}_{ \ell_{*}}(s^{\ell_{*}}))}
\left[\begin{array}{c|c|c} 0 & \mathbf{0}_{1,3} & \mathbf{0}_{1,3} \\ \hline |v| & 1 & \frac{1}{|v|} \\ \hline |v|^{2} & |v|  & 1 \end{array}\right]  \frac{   {\partial (s^{1}, \mathbf{x}_{\perp_{1}}(s^{1}), \mathbf{x}_{\parallel_{1}}(s^{1}), \mathbf{v}_{\perp_{1}}(s^{1}), \mathbf{v}_{\parallel_{1}}(s^{1}) )}       }{\partial (t,x,v)}\\
&  \lesssim_{\delta,\xi}\ C^{C|v|(t-s)} 
\left[\begin{array}{ccc}
 |v| & 1 & |s^{\ell_{*}}-s| \\  1 & |v| & 1 \end{array}\right]
 \left[\begin{array}{ccc} 0 & \mathbf{0}_{1,3} & \mathbf{0}_{1,3} \\   |v| & 1 & \frac{1}{|v|} \\  |v|^{2} & |v|  & 1 \end{array}\right]
 \left[\begin{array}{c|c|c} 1 & \mathbf{0}_{1,3} & \mathbf{0}_{1,3} \\ |t -s^1|^2 & 1 & |t-s^{1}| \\ |t -s^1| & |v| &  1  \end{array}\right]\\
 & \lesssim_{\delta,\xi} \ C^{C|v|(t-s)}  \left[\begin{array}{ccc} 
 |v| +1 & 1 & \frac{1}{|v|} \\ |v|^{2} + 1& |v| & 1 \end{array} \right]_{6 \times 7}.
\end{split}
\end{equation}

Now for $|v| < \delta$, we have
\Be \label{middle2} \begin{split}
& \frac{\partial (s^{\ell_{*}},\mathbf{X}_{ \ell_{*}}(s^{\ell_{*}}),\mathbf{V}_{  \ell_{*}}(s^{\ell_{*}}))}{\partial (s^{1},\mathbf{X}_{1}(s^{1}),\mathbf{V}_{1}(s^{1}) )}
\\ \le  & \ (\ref{s1tstar}) \times (\ref{wholeitvsmall}) \times (\ref{t1s1})
 \\ \lesssim & \left[ \begin{array}{c|c|c|c}
 0 & \mathbf 0_{1,3} & 0 & \mathbf 0_{1,2} 
 \\ \hline |\mathbf v_\perp^1 | & |v|^2 |\mathbf v_\perp^1 | & | \mathbf v_\perp^1 | & |v| |\mathbf v_\perp^1 |^2
 \\  |v| & 1 & |v|  | \mathbf v_\perp^1 |^2 &  | \mathbf v_\perp^1 |
 \\ \hline 1 &  | \mathbf v_\perp^1 | & 1 & |v|  | \mathbf v_\perp^1 |
 \\  1 &  | \mathbf v_\perp^1 | & |v| | \mathbf v_\perp^1 | & 1
 \end{array} \right]_{7 \times 6}
 \times
 \left[ \begin{array}{c|c|c|c}
 1 & 1 & \frac{1}{ | \mathbf v_\perp^1 |} & 1 
 \\ \hline 1 & 1 & \frac{1}{ | \mathbf v_\perp^1 |} & 1 
 \\ \hline  | \mathbf v_\perp^1 | &  | \mathbf v_\perp^1 |& 1 &  | \mathbf v_\perp^1 |
 \\ \hline 1 & 1 & \frac{1}{ | \mathbf v_\perp^1 |} & 1
 \end{array} \right]_{6 \times 6}
 \times
 \left[ \begin{array}{c|cc|cc}
 \frac{|v|}{ | \mathbf v_\perp^1 |} & \frac{1}{ | \mathbf v_\perp^1 |} &  | \mathbf v_\perp^1 | & 1 &  | \mathbf v_\perp^1 | 
 \\ \hline \frac{1}{ | \mathbf v_\perp^1 |} & \frac{|v|}{ | \mathbf v_\perp^1 |} & 1 & |v| &  | \mathbf v_\perp^1 |
 \\ \hline \frac{1}{ | \mathbf v_\perp^1 |} & \frac{1}{ | \mathbf v_\perp^1 |} & | \mathbf v_\perp^1 | & 1 &  | \mathbf v_\perp^1 |
 \\  \frac{1}{ | \mathbf v_\perp^1 |} & \frac{1}{ | \mathbf v_\perp^1 |} &  | \mathbf v_\perp^1 | & 1 & 1
 \end{array} \right]_{6 \times 7}
 \\  \lesssim & \left[ \begin{array}{c|c|c|c}
 0 & \mathbf 0_{1,3} & 0 & \mathbf 0_{1,2} 
 \\ \hline |\mathbf v_\perp^1 | & |v|^2 |\mathbf v_\perp^1 | & | \mathbf v_\perp^1 | & |v| |\mathbf v_\perp^1 |^2
 \\  |v| & 1 & |v|  | \mathbf v_\perp^1 |^2 &  | \mathbf v_\perp^1 |
 \\ \hline 1 &  | \mathbf v_\perp^1 | & 1 & |v|  | \mathbf v_\perp^1 |
 \\  1 &  | \mathbf v_\perp^1 | & |v| | \mathbf v_\perp^1 | & 1
 \end{array} \right]_{7 \times 6}
 \times
 \left[ \begin{array}{c|cc|cc}
 \frac{1}{|\mathbf v_\perp^1 |^2 } & \frac{1}{|\mathbf v_\perp^1 |^2} & 1 & \frac{1}{|\mathbf v_\perp^1 | } & 1
 \\ \hline \frac{1}{|\mathbf v_\perp^1 |^2 } & \frac{1}{|\mathbf v_\perp^1 |^2} & 1 & \frac{1}{|\mathbf v_\perp^1 | } & 1
 \\ \hline \frac{1}{|\mathbf v_\perp^1 |} & \frac{1}{|\mathbf v_\perp^1 |} &|\mathbf v_\perp^1 | &1 & 
|\mathbf v_\perp^1|
\\  \frac{1}{|\mathbf v_\perp^1 |^2 } & \frac{1}{|\mathbf v_\perp^1 |^2} & 1 & \frac{1}{|\mathbf v_\perp^1 | } & 1
 \end{array} \right]_{6 \times 7} 
\\  \lesssim &
\left[ \begin{array}{c|c|c|c|c}
0 & 0 & \mathbf 0_{1,2} & 0 & \mathbf 0_{1,2}
\\ \hline  \frac{1}{|\mathbf v_\perp^1 | } & \frac{1}{|\mathbf v_\perp^1 |} &  |\mathbf v_\perp^1 |  &1 & |\mathbf v_\perp^1 | 
\\   \frac{|v|}{|\mathbf v_\perp^1 |^2 } & \frac{1}{|\mathbf v_\perp^1 |^2} & 1 & \frac{1}{|\mathbf v_\perp^1 | } & 1
\\ \hline  \frac{1}{|\mathbf v_\perp^1 |^2 } & \frac{1}{|\mathbf v_\perp^1 |^2} & 1 &  \frac{1}{|\mathbf v_\perp^1 | } & 1
\\   \frac{1}{|\mathbf v_\perp^1 |^2 } & \frac{1}{|\mathbf v_\perp^1 |^2} & 1 & \frac{1}{|\mathbf v_\perp^1 | } & 1
\end{array} \right]_{7\times 7}
\end{split} 
\Ee

Now let's address the derivatives $\p_x t^\ell$, and $\p_v t^\ell$ for any $1 \le \ell \le \ell^*$, as we will need it later. For $|v| > \delta$, we compute $[\text{the first row of } (\ref{whole_inter}) \times (\ref{t1s1}) ] \cdot  \frac{   {\partial (s^{1}, \mathbf{x}_{\perp}(s^{1}), \mathbf{x}_{\parallel}(s^{1}), \mathbf{v}_{\perp}(s^{1}), \mathbf{v}_{\parallel}(s^{1}) )}       }{\partial (t,x,v)}$ and use \eqref{tbest} to get
\Be \label{ptell}
 \begin{split}
\left[ \begin{array}{c}
\p_x t^\ell
\\ \p_v t^\ell
\end{array} \right]
\lesssim
\left[ \begin{array}{ccc}
\frac{|v|^2}{| \mathbf v_\perp^1| } & \frac{|v|}{| \mathbf v_\perp^1|^2 } & \frac{1}{ |v|| \mathbf v_\perp^1| }  
\end{array}   \right]
\left[\begin{array}{cc} \mathbf{0}_{1,3} & \mathbf{0}_{1,3} \\  1 & |t-s^{1}| \\  |v| &  1  \end{array}\right]
\lesssim \left[ \begin{array}{c}
\frac{|v|}{ | \mathbf v_\perp^1 |^2} \\  \frac{1}{ |v|| \mathbf v_\perp^1| }
\end{array} \right].
\end{split} 
\Ee
And similarly, for $|v| \le \delta$, we compute $[\text{the first row of } (\ref{wholeitvsmall}) \times (\ref{t1s1}) ] \cdot  \frac{   {\partial (s^{1}, \mathbf{x}_{\perp}(s^{1}), \mathbf{x}_{\parallel}(s^{1}), \mathbf{v}_{\perp}(s^{1}), \mathbf{v}_{\parallel}(s^{1}) )}       }{\partial (t,x,v)}$ and use \eqref{tbest} to get
\Be \label{ptell2}
 \begin{split}
\left[ \begin{array}{c}
\p_x t^\ell
\\ \p_v t^\ell
\end{array} \right]
\lesssim
\left[ \begin{array}{ccc}
\frac{1}{| \mathbf v_\perp^1| } & \frac{1}{| \mathbf v_\perp^1|^2 } & \frac{1}{ | \mathbf v_\perp^1| }  
\end{array}   \right]
\left[\begin{array}{cc} \mathbf{0}_{1,3} & \mathbf{0}_{1,3} \\  1 & |t-s^{1}| \\  |v| &  1  \end{array}\right]
\lesssim \left[ \begin{array}{c}
\frac{1}{ | \mathbf v_\perp^1 |^2} \\  \frac{1}{| \mathbf v_\perp^1| }
\end{array} \right].
\end{split} 
\Ee

We remark $\p \mathbf{x}_{\perp_{\ell_{*}}}$ and $\p \mathbf{v}_{\perp_{\ell_{*}}}$ have desired bounds but $\p \mathbf{x}_{\parallel_{\ell_{*}}}$ and $\p \mathbf{v}_{\parallel_{\ell_{*}}}$ still have undesired bounds in (\ref{middle}), \eqref{middle2}. 

We only need to consider the remaining cases, i.e. $\ell$ is \textit{Type I} or \textit{Type II}.
Note that in either case the moving frame ($\mathbf{p}^{\ell}-$spherical coordinate) is well-defined for all $\tau \in [s,t]$. In next two step we use the ODE method to refine the submatrix  of (\ref{middle}) and \eqref{middle2}: 
\begin{equation}\notag
\begin{split}
 \frac{\partial(\mathbf{x}_{\parallel_{\ell_{*}}}(s^{\ell_{*}}),   \mathbf{v}_{\parallel_{\ell_{*}}}(s^{\ell_{*}})   )}{\partial (\mathbf{x}_{\perp_{1}}(s^{1}), \mathbf{x}_{\parallel_{1}}(s^{1}), \mathbf{v}_{\perp_{1}}(s^{1}), \mathbf{v}_{\parallel_{1}}(s^{1})   )} 
 =
\left[\begin{array}{cccccccc}
 \frac{\partial\mathbf{x}_{\parallel_{\ell_{*}} }(s^{\ell_{*}})}{\partial \mathbf{x}_{\perp_{1}}(s^{1})  } & 
  \frac{\partial\mathbf{x}_{\parallel_{\ell_{*}} }(s^{\ell_{*}})}{\partial \mathbf{x}_{\parallel_{1} }(s^{1})  } & 
    \frac{\partial\mathbf{x}_{\parallel_{\ell_{*}} }(s^{\ell_{*}})}{\partial \mathbf{v}_{\perp_{1}}(s^{1})  } &  
    \frac{\partial\mathbf{x}_{\parallel_{\ell_{*}} }(s^{\ell_{*}})}{\partial \mathbf{v}_{\parallel_{1} }(s^{1})  } \\ 
     \frac{\partial\mathbf{v}_{\parallel_{\ell_{*}} }(s^{\ell_{*}})}{\partial \mathbf{x}_{\perp_{1}}(s^{1})  } &  
    \frac{\partial\mathbf{v}_{\parallel_{\ell_{*}} }(s^{\ell_{*}})}{\partial \mathbf{x}_{\parallel_{1} }(s^{1})  } &
    \frac{\partial\mathbf{v}_{\parallel_{\ell_{*}} }(s^{\ell_{*}})}{\partial \mathbf{v}_{\perp_{1}}(s^{1})  } &  
    \frac{\partial\mathbf{v}_{\parallel_{\ell_{*}} }(s^{\ell_{*}})}{\partial \mathbf{v}_{\parallel_{1} }(s^{1})  }  
\end{array}
\right]_{4\times 6}
.
\end{split}
\end{equation}

\vspace{8pt}

\noindent{\textit{Step 9. ODE method within the time scale $|t-s||v|\simeq L_{\xi}$}}

\vspace{4pt}
 
Recall the end points (time) of intermediate groups from (\ref{group}):  
\begin{equation}\notag
\begin{split}
s < \underbrace{ t^{\ell_{*}} < t^{\ell_{[ \frac{|t-s||v|}{L_{\xi}}]} +1 } }_{{[ \frac{|t-s||v|}{L_{\xi}}]  +1 }}  <  \underbrace{t^{\ell_{[ \frac{|t-s||v|}{L_{\xi}}]   }} < t^{\ell_{[ \frac{|t-s||v|}{L_{\xi}}]-1}+1}   }_{{[ \frac{|t-s||v|}{L_{\xi}}]   }}
<  \cdots <  \underbrace{t^{\ell_{i }} < t^{\ell_{i-1}+1  }}_{i }  < \cdots
<  \underbrace{t^{\ell_{2}}<t^{\ell_{1} +1} }_{2}< \underbrace{t^{\ell_{1}}< t^{1} }_{1}< t,
\end{split}
\end{equation}
where the underbraced numbering indicates the index of the intermediate group. We further choose points independently on $(t,x,v)$ for all $i=1,2, \cdots, [\frac{|t-s||v|}{L_{\xi}}]:$ 
\begin{eqnarray*} 
& t^{\ell_{1} + 1} < s^{2} < t^{\ell_{1}},&
\\
& t^{\ell_{2}+1} < s^{3}< t^{\ell_{2}},&\\
& \vdots&\\
&t^{\ell_{i}+1}<s^{i+1}  <\underbrace{ t^{\ell_{i }}< \cdots \cdots< t^{\ell_{i-1}+1} }_{i -\text{intermediate group}}< s^{i }<   t^{\ell_{i-1} }, & \\ 
&\vdots &\\ 
&t^{\ell_{[ \frac{|t-s||v|}{L_{\xi}}] }+1} <  s^{\ell_{[ \frac{|t-s||v|}{L_{\xi}}] }+1}  < t^{\ell_{[ \frac{|t-s||v|}{L_{\xi}}] } }
. & \end{eqnarray*} 

%
%

We claim the following estimate at $s^{i+1 }$ via $s^{i}$. Within the $i-$th intermediate group, we fix $\mathbf{p}^{\ell_{i}}-$spherical coordinate in \textit{Step 9}. The goal is to estimate derivatives with
respect to initial $(\mathbf{x}_{1},\mathbf{v}_{1})$ at $s^{i+1}$ in terms of $s^{i}$. This
is a different from previous steps. 
\begin{equation}\label{ODE_onegroup}
\begin{split}
&\left[\begin{array}{cc} |
  \frac{\partial \mathbf{x}_{\parallel_{\ell_{i}}} (s^{i+1})}{\partial {\mathbf{x}_{\perp_{1}}}(s^{1})  }|  &  |  \frac{\partial \mathbf{x}_{\parallel_{\ell_{i}}} (s^{i+1})}{\partial {\mathbf{x}_{\parallel_{1}}} (s^{1})  }|  \\
| \frac{\partial   \mathbf{v}_{\parallel _{\ell_{i}}}   (s^{i+1})}{  \partial {\mathbf{x}_{\perp _{1}}}(s^{1}) } |
& | \frac{\partial   \mathbf{v}_{\parallel _{\ell_{i}}}   ( s^{i+1})}{  \partial {\mathbf{x}_{\parallel _{1}}}(s^{1}) } |
\end{array} \right] \\
&\lesssim_{\delta,\xi}
\left[\begin{array}{cc} 1 & \frac{1}{|v|} \\ |v| & 1   \end{array} \right]
\left[\begin{array}{cc} |   \frac{\partial \mathbf{x}_{\parallel_{\ell_{i}}} ( s^{i })}{\partial {\mathbf{x}_{\perp_{1}}} (s^{1}) }|
& |   \frac{\partial \mathbf{x}_{\parallel_{\ell_{i}}}  ( s^{i })}{\partial {\mathbf{x}_{\parallel_{1}}} (s^{1})  }|
   \\
| \frac{\partial   \mathbf{v}_{\parallel _{\ell_{i}}}    ( s^{i })}{  \partial {\mathbf{x}_{\perp _{1}}} (s^{1})} |
&
| \frac{\partial   \mathbf{v}_{\parallel _{\ell_{i}}}    ( s^{i })}{  \partial {\mathbf{x}_{\parallel _{1}}} (s^{1})} |
 \end{array} \right]
+ e^{C|v||t-s^{i}|}
\left[\begin{array}{cc} 1 & \frac{1}{|v|} \\ |v| & 1   \end{array} \right]
\left[\begin{array}{cc}  0 & 0\\ |v| \Big(1+ \frac{|v|}{|\mathbf{v}_{\perp }^{1}|} \Big) &   |v| \Big(1+ \frac{|v|}{|\mathbf{v}_{\perp }^{1}|} \Big) \end{array} \right]  ,\\
&\left[\begin{array}{cc} |
  \frac{\partial \mathbf{x}_{\parallel_{\ell_{i}}}  (s^{i+1})}{\partial {\mathbf{v}_{\perp_{1}}}(s^{1})  }|  &  |  \frac{\partial \mathbf{x}_{\parallel_{\ell_{i}}}(s^{i+1})}{\partial {\mathbf{v}_{\parallel_{1}}}  (s^{1})}|  \\
| \frac{\partial   \mathbf{v}_{\parallel _{\ell_{i}}}   (s^{i+1})}{  \partial {\mathbf{v}_{\perp _{1}}} (s^{1})} |
& | \frac{\partial   \mathbf{v}_{\parallel _{\ell_{i}}}   (s^{i+1})}{  \partial {\mathbf{v}_{\parallel _{1}}} (s^{1})} |
\end{array} \right] \\
&\lesssim_{\delta,\xi}
\left[\begin{array}{cc} 1 & \frac{1}{|v|} \\ |v| & 1   \end{array} \right]
\left[\begin{array}{cc} |   \frac{\partial \mathbf{x}_{\parallel_{\ell_{i}}} ( s^{i})}{\partial {\mathbf{v}_{\perp_{1}}} (s^{1}) }|
& |   \frac{\partial \mathbf{x}_{\parallel_{\ell_{i}}}  ( s^{i})}{\partial {\mathbf{v}_{\parallel_{1}}}   (s^{1})}|
   \\
| \frac{\partial   \mathbf{v}_{\parallel _{\ell_{i}}}    ( s^{i})}{  \partial {\mathbf{v}_{\perp _{1}}}  (s^{1})} |
&
| \frac{\partial   \mathbf{v}_{\parallel _{\ell_{i}}}    ( s^{i})}{  \partial {\mathbf{v}_{\parallel _{1}}}  (s^{1})} |
 \end{array} \right]
+ e^{C|v||t-s^{i}|}
\left[\begin{array}{cc} 1 & \frac{1}{|v|} \\ |v| & 1   \end{array} \right]
\left[\begin{array}{cc}   0 &  0
\\ 1 & 1 \end{array} \right] .
\end{split}
\end{equation}
For the sake of simplicity we drop the index $\ell_{i}$.

Denote, from (\ref{F||}),
\begin{equation}\label{F||_decompose}
\begin{split}
F_{\parallel}(\mathbf{x}_{\perp}, \mathbf{x}_{\parallel}, \mathbf{v}_{\perp}, \mathbf{v}_{\parallel}) := D(\mathbf{x}_{\perp}, \mathbf{x}_{\parallel} ,\mathbf{v}_{\parallel})+H(\mathbf{x}_{\perp}, \mathbf{x}_{\parallel}, \mathbf{v}_{\parallel})\mathbf{v}_{\perp} ,
\end{split}
\end{equation}
where $D$ is a $\mathbf{r}^{3}$-vector-valued function and $H$ is a $3\times3$ matrix-valued function:
\begin{equation}\notag
\begin{split}
 D(\mathbf{x}_{\perp}, \mathbf{x}_{\parallel}, \mathbf{v}_{\parallel}) 
 =&\sum_{i} G_{ij} (\mathbf{x}_{\perp}, \mathbf{x}_{\parallel})\frac{(-1)^{i+1}}{-\mathbf{n}(\mathbf{x}_{\parallel}) \cdot (\partial_{1} \mathbf{\eta}(\mathbf{x}_{\parallel}) \times \partial_{2} \mathbf{\eta}(\mathbf{x}_{\parallel}))}\\
& \ \ \ \times  \Big\{  \mathbf{v}_{\parallel} \cdot  \nabla^{2}\mathbf{ \eta}(\mathbf{x}_{\parallel}) \cdot \mathbf{v}_{\parallel} - \mathbf{x}_{\perp} \mathbf{v}_{\parallel} \cdot \nabla^{2} \mathbf{n}(\mathbf{x}_{\parallel}) \cdot \mathbf{v}_{\parallel} - E(s, - \mathbf{x}_\perp \mathbf n (\mathbf x_\parallel ) + \mathbf \eta (\mathbf x_\parallel ) ) 
\Big\} \cdot (-\mathbf{n} (\mathbf{x}_{\parallel}) \times \partial_{i+1} \mathbf{\eta}(\mathbf{x}_{\parallel})),
\end{split}
\end{equation}
and 
\begin{equation}\notag
\begin{split}
 H(\mathbf{x}_{\perp}, \mathbf{x}_{\parallel}, \mathbf{v}_{\parallel}) 
 = \sum_{i} G_{ij} (\mathbf{x}_{\perp}, \mathbf{x}_{\parallel})\frac{(-1)^{i+1}}{-\mathbf{n}(\mathbf{x}_{\parallel}) \cdot (\partial_{1} \mathbf{\eta}(\mathbf{x}_{\parallel}) \times \partial_{2} \mathbf{\eta}(\mathbf{x}_{\parallel}))}   2  \mathbf{v}_{\parallel}\cdot \nabla \mathbf{n}(\mathbf{x}_{\parallel}) \cdot (-\mathbf{n} (\mathbf{x}_{\parallel}) \times \partial_{i+1} \mathbf{\eta}(\mathbf{x}_{\parallel})).
\end{split}
\end{equation}
Note that $H$ is linear in $\mathbf{v}_{\parallel}$. Here $G_{ij}(\cdot,\cdot)$ is a smooth bounded function defined in (\ref{G}) and we used the notational convention $i\equiv i \ \mathrm{mod } \ 2$.

From Lemma \ref{chart_lemma} we take the time integration of (\ref{ODE_ell}) along the characteristics to have
\begin{equation}\notag
\begin{split}
\mathbf{x}_{\parallel}( s^{i+1}) & = \mathbf{x}_{\parallel}(s^{i}) - \int^{s^{i}}_{s^{i+1}} \mathbf{v}_{\parallel}(\tau ) \mathrm{d}\tau ,\\
\mathbf{v}_{\parallel}(s^{i+1}) & =  \mathbf{v}_{\parallel}(s^{i}) - \int^{s^{i}}_{s^{i+1}} \big\{  H(\mathbf{x}_{\perp}(\tau )  , \mathbf{x}_{\parallel}(\tau ), \mathbf{v}_{\parallel}(\tau ))\mathbf{v}_{\perp}(\tau ) +  D(\mathbf{x}_{\perp}(\tau ), \mathbf{x}_{\parallel}(\tau ) ,\mathbf{v}_{\parallel}(\tau ))\big\} \ \mathrm{d}\tau .
\end{split}
\end{equation}
Note that $\mathbf{v}_{\perp}(\tau )$ is not continuous with respect to the time $\tau$. Using (\ref{ODE_ell}) we rewrite this time integration as
\begin{equation}\notag
\begin{split}
\int^{s^{i}}_{s^{i+1}}  H(\mathbf{x}_{\perp}(\tau ), \mathbf{x}_{\parallel}(\tau ), \mathbf{v}_{\parallel}(\tau ))   \mathbf{v}_{\perp}(\tau ) \mathrm{d}\tau 
 =  
 \int^{s^{i}}_{t^{\ell_{i-1}+1}  } + \sum_{\ell= \ell_{i}-1}^{\ell_{i-1}+1} \int^{t^{\ell}}_{t^{\ell+1}} + \int^{t^{\ell_{i}  }}_{ s^{i+1}} ,\\
\end{split}
\end{equation}
then we use $\mathbf{v}_{\perp}(\tau ) = \dot{\mathbf{x}}_{\perp}(\tau )$ and the integration by parts to have
\begin{equation}\notag
\begin{split}
 &  \int_{t^{\ell_{i-1} +1} }^{s^{i}}   H(\mathbf{x}_{\perp}(\tau ), \mathbf{x}_{\parallel}(\tau ), \mathbf{v}_{\parallel}(\tau ) ) \dot{\mathbf{x}}_{\perp}(\tau ) \mathrm{d}\tau 
  +  \sum_{\ell=\ell_{i}-1}^{\ell_{i-1}+1} \int^{t^{\ell}}_{t^{\ell+1}}  H(\mathbf{x}_{\perp}(\tau ), \mathbf{x}_{\parallel}(\tau ),\mathbf{v}_{\parallel}(\tau ) )\dot{\mathbf{x}}_{\perp}(\tau ) \mathrm{d}\tau  \\
 &+ \int^{t^{\ell_{i}   }}_{s^{i+1}}  H(\mathbf{x}_{\perp}(\tau ), \mathbf{x}_{\parallel}(\tau ), \mathbf{v}_{\parallel}(\tau )) \dot{\mathbf{x}}_{\perp}(\tau ) \mathrm{d}\tau \\
 =& \  H(s^{i})    {\mathbf{x}}_{\perp}(s^{i}) -H( t^{\ell_{i-1}+1} )  \underbrace{\mathbf{x}_{\perp}(  t^{\ell_{i-1} +1})}_{=0} - \int_{t^{\ell_{i-1}+1}}^{s^{i}}
\big[{\mathbf{v}}_{\perp}(\tau ),  {\mathbf{v}}_{\parallel}(\tau ), {F}_{\parallel}(\tau )  \big]\cdot \nabla H(\tau )
   {\mathbf{x}}_{\perp}(\tau ) \mathrm{d}\tau \\
 & + \sum_{\ell=\ell_{i}-1}^{\ell_{i-1}+1 } \Big\{
 H( t^{\ell} )  \underbrace{\mathbf{x}_{\perp}(t^{\ell})}_{=0} -  H(t^{\ell+1})   \underbrace{{\mathbf{x}}_{\perp}(t^{\ell+1})}_{=0}  
  - \int_{t^{\ell+1}}^{t^{\ell}}
\big[{\mathbf{v}}_{\perp}(\tau ),  {\mathbf{v}}_{\parallel}(\tau ), {F}_{\parallel}(\tau )  \big]  \cdot \nabla H(\tau )
   {\mathbf{x}}_{\perp}(\tau ) \mathrm{d}\tau 
 \Big\}\\
 & +
 H( t^{\ell_{i} } )  \underbrace{\mathbf{x}_{\perp}( t^{\ell_{i}  }) }_{=0}- H(s^{i+1})  { {\mathbf{x}}_{\perp}(s^{i+1}) }
    - \int^{ t^{\ell_{i} } }_{s^{i+1}}
\big[{\mathbf{v}}_{\perp}(\tau ),  {\mathbf{v}}_{\parallel}(\tau ), {F}_{\parallel}(\tau )  \big]   \cdot \nabla H(\tau )
  {\mathbf{x}}_{\perp}(\tau ) \mathrm{d}\tau \\
 =& \  H( \mathbf{x}_{\perp}, \mathbf{x}_{\parallel}, \mathbf{v}_{\parallel} ) (s^{i}) \mathbf{x}_{\perp}(s^{i})  -  H(s^{i+1}) {\mathbf{x}}_{\perp}(s^{i+1})
 - \int_{s^{i}}^{s^{i+1}}   \big[{\mathbf{v}}_{\perp}(\tau ),  {\mathbf{v}}_{\parallel}(\tau ), {F}_{\parallel}(\tau )  \big]    \cdot \nabla H(\tau )  {\mathbf{x}}_{\perp}(\tau ) \mathrm{d}\tau ,
\end{split}
\end{equation}
where we have used the fact $X_{\mathbf{cl}}(t^{\ell}) \in\partial\Omega$(therefore $\mathbf{x}_{\perp}(t^{\ell})=0$) and the notation $H(\tau ) = H(\mathbf{x}_{\perp}(\tau ), \mathbf{x}_{\parallel}(\tau ), \mathbf{v}_{\parallel}(\tau )), \ D(\tau ) = D(\mathbf{x}_{\perp}(\tau ), \mathbf{x}_{\parallel}(\tau ),\mathbf{v}_{\parallel}(\tau )), \ F_{\parallel}(\tau  ) = F_{\parallel}( \mathbf{x}_{\perp}(\tau ), \mathbf{x}_{\parallel}(\tau ), \mathbf{v}_{\perp}(\tau ), \mathbf{v}_{\parallel}(\tau ) ).$

Overall we have
\begin{equation}\label{xv_mildform}
\begin{split}
\mathbf{x}_{\parallel}(s^{i+1}) & = \mathbf{x}_{\parallel}(s^{i}) - \int_{s^{i+1}}^{s^{i}} \mathbf{v}_{\parallel}(\tau ) \mathrm{d}\tau ,\\
\mathbf{v}_{\parallel}(s^{i+1}) & =  \mathbf{v}_{\parallel}( s^{i }) - H( s^{i }) \mathbf{x}_{\perp} ( s^{i }) +
 H ( s^{i+1})  \mathbf{x}_{\perp}( s^{i+1}) \\
 & \ \  + \int^{ s^{i }}_{ s^{i+1}}     \big[{\mathbf{v}}_{\perp}(\tau ),  {\mathbf{v}}_{\parallel}(\tau ), {F}_{\parallel}(\tau )  \big]    \cdot \nabla H(\tau )  {\mathbf{x}}_{\perp}(\tau ) \mathrm{d}\tau   - \int^{ s^{i }}_{ s^{i+1}} D(  \tau )  \mathrm{d}\tau .
\end{split}
\end{equation}
Denote $$\partial = [\partial_{\mathbf{x}_{\perp}(s^{1})}, \partial_{\mathbf{x}_{\parallel}(s^{1})}, \partial_{\mathbf{v}_{\perp}(s^{1})}, \partial_{\mathbf{v}_{\parallel} (s^{1})}]= [\frac{\partial}{\partial \mathbf{x}_{\perp} (s^{1})}, \frac{\partial}{\partial \mathbf{x}_{\parallel}  ( s^{1})}, \frac{\partial}{\partial \mathbf{v}_{\perp}  ( s^{1})}, \frac{\partial}{\partial \mathbf{v}_{\parallel}  ( s^{1})}].$$ 
 We claim that, in a sense of distribution on $(s^{1}, \mathbf{x}_{\perp}( s^{1}), \mathbf{x}_{\parallel}(  s^{1}), \mathbf{v}_{\perp}(  s^{1}), \mathbf{v}_{\parallel}(  s^{1})) \in [0,\infty) \times (0,C_{\xi}) \times  (0,2\pi] \times (\delta, \pi-\delta) \times \mathbb{R}\times \mathbb{R}^{2}$,
 \begin{equation}\label{piece_d}
 \begin{split}
& \big[\partial \mathbf{x}_{\perp}( s^{i+1};s^{1},\mathbf{x}(s^{1}),\mathbf{v}(s^{1})), 
\partial \mathbf{x}_{\parallel} ( s^{i+1};s^{1},\mathbf{x}(s^{1}),\mathbf{v}(s^{1})),  \  \partial \mathbf{v}_{\parallel} ( s^{i+1};s^{1},\mathbf{x}(s^{1}),\mathbf{v}(s^{1}))\big]\\
 &= \sum_{\ell  } \mathbf{1}_{[t^{\ell+1},t^{\ell})} (s^{i+1})[\partial \mathbf{x}_{\perp} , \partial \mathbf{x}_{\parallel} ,     \partial \mathbf{v}_{\parallel}  \big],\\
& \partial\big[ \mathbf{v}_{\perp} ( s^{i+1};s^{1},\mathbf{x}(s^{1}),\mathbf{v}(s^{1}))\mathbf{x}_{\perp} ( s^{i+1};s^{1},\mathbf{x}(s^{1}),\mathbf{v}(s^{1}))\big] = \sum_{\ell  } \mathbf{1}_{[t^{\ell+1},t^{\ell})}(s^{i+1}) \big\{   \partial \mathbf{v}_{\perp}  \mathbf{x}_{\perp} + \mathbf{v}_{\perp}   \partial \mathbf{x}_{\perp} \big\},
 \end{split}
 \end{equation}
i.e. the distributional derivatives of $[ \mathbf{x}_{\perp}, \mathbf{x}_{\parallel} , \mathbf{v}_{\parallel}]$ and $\mathbf{v}_{\perp} \mathbf{x}_{\perp}$ equal the piecewise derivatives. 

\textit{Proof of }(\ref{piece_d}). Let $\phi(\tau^{\prime}, \mathbf{x}_{\perp}, \mathbf{x}_{\parallel}, \mathbf{v}_{\perp}, \mathbf{v}_{\parallel}) \in C^{\infty}_{c} ( [0,\infty) \times   (0,C_{\xi})\times \mathbb{S}^{2}\times\mathbb{R}\times\mathbb{R}^{2})$. Therefore $\phi\equiv 0$ when $\mathbf{x}_{\perp} < \delta$, $|v| > \frac{1}{\delta}$. For $\mathbf{x}_{\perp } \geq \delta$ we use the proof of Lemma \ref{chart_lemma}: For $x= \mathbf{\eta}(\mathbf{x}_{\parallel}) + \mathbf{x}_{\perp} [-\mathbf{n}(\mathbf{x}_{\parallel})],$
\[
|\mathbf{x}_{\perp}|  \ \lesssim_{\xi} \ \xi(x) = \xi(  \mathbf{\eta}(\mathbf{x}_{\parallel}) + \mathbf{x}_{\perp}[- \mathbf{n}(\mathbf{x}_{\parallel})] ) \ \lesssim_{\xi} \ |\mathbf{x}_{\perp}|,
\]
and therefore $\xi(x)\gtrsim_{\xi}\delta$ and $\alpha(t,x,v) \gtrsim_{\xi, E} \sqrt{ |\xi (x) |} \gtrsim_{\xi ,E }  \sqrt \delta.$
By the Velocity lemma, for $(x,v)\in \text{supp}(\phi)$
\[
\alpha(x^{\ell},v^{\ell}) \gtrsim_{\xi} e^{-C(|v|+1)|t^{1}-t^{\ell}|} \alpha(t,x,v) \gtrsim_{\xi} e^{-\frac{C}{\delta}(t-s)} \sqrt \delta \gtrsim_{\xi, E, |t-s|,\delta, \phi}1>0,
\]
where we used the fact that $\phi$ vanishes away from a compact subset $\text{supp}(\phi)$. Therefore $t^{\ell}(t,x,v)= t^{\ell}(t, \mathbf{x}_{\perp}, \mathbf{x}_{\parallel}, \mathbf{v}_{\perp}, \mathbf{v}_{\parallel})$ is $C^1$ with respect to $\mathbf{x}_{\perp}, \mathbf{x}_{\parallel}, \mathbf{v}_{\perp}, \mathbf{v}_{\parallel}$ locally on $\text{supp}(\phi)$ and therefore $\mathcal{M}=\{(\tau^{\prime}, \mathbf{x}, \mathbf{v}) \in  \text{supp}(\phi): \tau^{\prime} = t^{\ell}(t, \mathbf{x}, \mathbf{v})  \}$ is a $C^1$ manifold.

It suffices to consider the case $|\tau^{\prime}- t^{\ell}(t,x,v)| \ll1$. Denote $\partial_{\mathbf{e}}\in \{\partial_{\mathbf{x_{\perp}}}, \partial_{\mathbf{x}_{\parallel,1}}, \partial_{\mathbf{x}_{\parallel,2}}, \partial_{\mathbf{v_{\perp}}}, \partial_{\mathbf{v}_{\parallel,1}}, \partial_{\mathbf{v}_{\parallel,2}}\}$ and $n_{\mathcal{M}}= e_1$ to have
\begin{equation}\notag
\begin{split}
&\int_{\{(\tau^{\prime},\mathbf{x}, \mathbf{v}) \in \text{supp}(\phi)\}}[\partial_{\mathbf{e}} \mathbf{x}_{\perp}(\tau^{\prime};t,\mathbf{x}, \mathbf{v}), \partial_{\mathbf{e}} \mathbf{x}_{\parallel}(\tau^{\prime};t,\mathbf{x}, \mathbf{v}), \partial_{\mathbf{e}}\mathbf{v}_{\parallel}(\tau^{\prime};t,\mathbf{x}, \mathbf{v}) ]
\phi(\tau^{\prime},\mathbf{x}, \mathbf{v}) \mathrm{d} \mathbf{x} \mathrm{d}\mathbf{v} \mathrm{d}\tau^{\prime}\\
=& \int_{\tau^{\prime} < t^{\ell}} + \int_{\tau^{\prime} \geq t^{\ell}}\\
=& \int_\mathcal{M} \Big(\lim_{\tau^{\prime} \uparrow t^{\ell} } [\mathbf{x}_{\perp}(\tau^{\prime}), \mathbf{x}_{\parallel}(\tau^{\prime}), \mathbf{v}_{\parallel}(\tau^{\prime})]  - \lim_{\tau^{\prime} \downarrow t^{\ell}} [\mathbf{x}_{\perp}(\tau^{\prime}), \mathbf{x}_{\parallel}(\tau^{\prime}), \mathbf{v}_{\parallel} (\tau^{\prime}) ] \Big) \phi(\tau^{\prime}, \mathbf{x}, \mathbf{v})\{ \mathbf{e}\cdot n_{\mathcal{M}}\}  \mathrm{d}\mathbf{x} \mathrm{d}\mathbf{v}\\
&- \int_{\{\tau^{\prime} \neq t^{\ell}(t,\mathbf{x}, \mathbf{v})\}} [\mathbf{x}_{\perp}(\tau^{\prime}), \mathbf{x}_{\parallel}(\tau^{\prime}), \mathbf{v}_{\parallel}(\tau^{\prime})] \partial_{\mathbf{e}} \phi(\tau^{\prime}, \mathbf{x}, \mathbf{v}) \mathrm{d}\tau^{\prime} \mathrm{d}\mathbf{v} \mathrm{d}\mathbf{x} \\
=& - \int_{\{\tau^{\prime} \neq t^{\ell}(t,\mathbf{x}, \mathbf{v})\}} [\mathbf{x}_{\perp}(\tau^{\prime}), \mathbf{x}_{\parallel}(\tau^{\prime}), \mathbf{v}_{\parallel}(\tau^{\prime})] \partial_{\mathbf{e}} \phi(\tau^{\prime}, \mathbf{x}, \mathbf{v}) \mathrm{d}\tau^{\prime} \mathrm{d}\mathbf{v} \mathrm{d}\mathbf{x},
\end{split}
\end{equation}
where we used the continuity of $[ \mathbf{x}_{\perp}(\tau^{\prime};t,\mathbf{x}, \mathbf{v}),  \mathbf{x}_{\parallel}(\tau^{\prime};t,\mathbf{x}, \mathbf{v}),  \mathbf{v}_{\parallel}(\tau^{\prime};t,\mathbf{x}, \mathbf{v}) ]$ in terms of $\tau^{\prime}$ near $t^{\ell}(t,\mathbf{x}, \mathbf{v})$.

Note that $\mathbf{v}_{\perp}(\tau^{\prime};t, \mathbf{x}, \mathbf{v})$ is discontinuous around $|\tau^{\prime}- t^{\ell}| \ll1. (\lim_{\tau^{\prime} \downarrow t^{\ell}} \mathbf{v}_{\perp}(\tau^{\prime}) = - \lim_{\tau^{\prime} \uparrow t^{\ell}} \mathbf{v}_{\perp}(\tau^{\prime}))$ However with crucial $\mathbf{x}_{\perp}(\tau^{\prime})-$multiplication we have $\mathbf{x}_{\perp}(t^{\ell}) \mathbf{v}_{\perp}(t^{\ell})=0$ and therefore
\begin{equation}\notag
\begin{split}
&\int_{\{(\tau^{\prime},\mathbf{x}, \mathbf{v}) \in \text{supp}(\phi)\}} \partial_{\mathbf{e}} [ \mathbf{x}_{\perp}(\tau^{\prime};t,\mathbf{x},\mathbf{v}) \mathbf{v}_{\perp}(\tau^{\prime};t,\mathbf{x},\mathbf{v}) ]
\phi(\tau^{\prime},\mathbf{x}, \mathbf{v}) \mathrm{d} \mathbf{x} \mathrm{d}\mathbf{v} \mathrm{d}\tau^{\prime}\\
=& \int_{\tau^{\prime} < t^{\ell}} + \int_{\tau^{\prime} \geq t^{\ell}}\\
=& \int_\mathcal{M} \Big(\lim_{\tau^{\prime} \uparrow t^{\ell} } [\mathbf{x}_{\perp}(\tau^{\prime})  \mathbf{v}_{\perp}(\tau^{\prime})]  - \lim_{\tau^{\prime} \downarrow t^{\ell}} [\mathbf{x}_{\perp}(\tau^{\prime}) \mathbf{v}_{\perp} (\tau^{\prime}) ] \Big) \phi(\tau^{\prime}, \mathbf{x}, \mathbf{v})\{ \mathbf{e}\cdot n_{\mathcal{M}}\}  \mathrm{d}\mathbf{x} \mathrm{d}\mathbf{v}\\
&- \int_{\{\tau^{\prime} \neq t^{\ell}(t,\mathbf{x}, \mathbf{v})\}} [\mathbf{x}_{\perp}(\tau^{\prime})  \mathbf{v}_{\perp}(\tau^{\prime})] \partial_{\mathbf{e}} \phi(\tau^{\prime}, \mathbf{x}, \mathbf{v}) \mathrm{d}\tau^{\prime} \mathrm{d}\mathbf{v} \mathrm{d}\mathbf{x} \\
=& - \int_{\{\tau^{\prime} \neq t^{\ell}(t,\mathbf{x}, \mathbf{v})\}} [\mathbf{x}_{\perp}(\tau^{\prime};t,\mathbf{x},\mathbf{v})  \mathbf{v}_{\perp}(\tau^{\prime};t,\mathbf{x},\mathbf{v})] \partial_{\mathbf{e}} \phi(\tau^{\prime}, \mathbf{x}, \mathbf{v}) \mathrm{d}\tau^{\prime} \mathrm{d}\mathbf{v} \mathrm{d}\mathbf{x}.
\end{split}
\end{equation}
This completes the proof of (\ref{piece_d}).

Since $\mathbf{v}_{\perp}$
always is multiplied with $\mathbf{x}_{\perp }$ in (\ref{xv_mildform}), we may apply (\ref{piece_d}) and take
derivative inside each $\int_{s^{i+1}}^{s^{i}}$ of (\ref{xv_mildform}), separating the
main terms with $\partial_{\mathbf{e}}\mathbf{x}_{||}$ and $\partial_{\mathbf{e}}\mathbf{v}_{||},$ and treating the rest (underbraced terms) as forcing terms to obtain, for $\partial_{\mathbf{e}}\in \{\partial_{\mathbf{x_{\perp}}}, \partial_{\mathbf{x}_{\parallel,1}}, \partial_{\mathbf{x}_{\parallel,2}}, \partial_{\mathbf{v_{\perp}}}, \partial_{\mathbf{v}_{\parallel,1}}, \partial_{\mathbf{v}_{\parallel,2}}\}$,
%
\begin{equation}\label{Dode}
\begin{split}
&\partial_{\mathbf{e}}\mathbf{x}_{\parallel}(s^{i+1}) = \partial_{\mathbf{e}} \mathbf{x}_{\parallel}(s^{i}) - \int^{s^{i}}_{s^{i+1}} \partial_{\mathbf{e}} \mathbf{v}_{\parallel}(\tau ) \mathrm{d}\tau ,\\
&\partial \mathbf{v}_{\parallel}(s^{i+1}) = \partial_{\mathbf{e}} H(s^{i+1}) \mathbf{x}_{\perp}(s^{i+1}) + H(s^{i+1}) \underbrace{ \partial_{\mathbf{e}}\mathbf{x}_{\perp}(s^{i+1}) }+ \partial_{\mathbf{e}}\mathbf{v}_{\parallel}(s^{i }) - \partial_{\mathbf{e}}[H(\mathbf{x}_{\perp}, \mathbf{x}_{\parallel}, \mathbf{v}_{\parallel}) \mathbf{x}_{\perp}](s^{i+1})\\
& \ \ + \int^{s^{i}}_{s^{i+1}} \underbrace{  \partial_{\mathbf{e}}\mathbf{v}_{\perp}(\tau) }\partial_{\mathbf{x}_{\perp}}H(\tau)  \mathbf{x}_{\perp}(\tau)+ \partial_{\mathbf{e}}\mathbf{v}_{\parallel}(\tau)\cdot \nabla_{\mathbf{x}_{\parallel}}H(\tau)   \mathbf{x}_{\perp}(\tau) \mathrm{d}\tau  \\
& \ \ + \int^{ s^{i}}_{s^{i+1}}
\Big\{  \Big[  \underbrace{\partial_{\mathbf{e}}\mathbf{x}_{\perp}(\tau ) }\partial_{\mathbf{x}_{\perp}}H(\tau) + \partial_{\mathbf{e}}\mathbf{x}_{\parallel}(\tau ) \cdot \nabla_{\mathbf{x}_{\parallel}}H(\tau) + \partial_{\mathbf{e}}\mathbf{v}_{\parallel}(\tau) \cdot \nabla_{\mathbf{v}_{\parallel}} H(\tau) \Big]\mathbf{v}_{\perp}(\tau) \\
& \ \ \ \ \  + H(\tau)\underbrace{ \partial_{\mathbf{e}}\mathbf{v}_{\perp}(\tau) } +\underbrace{ \partial_{\mathbf{e}}\mathbf{x}_{\perp}(\tau)} \partial_{\mathbf{x}_{\perp}} D(\tau) + \partial_{\mathbf{e}}\mathbf{x}_{\parallel}(\tau) \cdot \nabla_{\mathbf{x}_{\parallel}} D(\tau)  + \partial_{\mathbf{e}}\mathbf{v}_{\parallel}(\tau) \nabla_{\mathbf{v}_{\parallel}} D(\tau)    \Big\}
\cdot \nabla_{\mathbf{v}_{\parallel}} H(\tau) \mathbf{x}_{\perp}(\tau) \mathrm{d}\tau  \\
& \ \ + \int^{ s^{i}}_{s^{i+1}}
\Big\{  \mathbf{v}_{\perp}(\tau) [  \underbrace{\partial_{\mathbf{e}}\mathbf{x}_{\perp}(\tau)}, \partial_{\mathbf{e}}\mathbf{x}_{\parallel}(\tau), \partial_{\mathbf{e}}\mathbf{v}_{\parallel}(\tau)] \cdot \nabla \partial_{\mathbf{x}_{\perp}} H(\tau)
+ \mathbf{v}_{\parallel}(\tau) \cdot  [ \underbrace{\partial_{\mathbf{e}}\mathbf{x}_{\perp}(\tau)}, \partial_{\mathbf{e}}\mathbf{x}_{\parallel}(\tau), \partial_{\mathbf{e}}\mathbf{v}_{\parallel}(\tau)] \cdot \nabla \nabla_{\mathbf{x}_{\parallel}} H(\tau) \\
& \ \ \ \ \ \ \ \ \ \ \ \
+ F_{\parallel}(\tau) \cdot  [ \underbrace{\partial_{\mathbf{e}}\mathbf{x}_{\perp}(\tau)}, \partial_{\mathbf{e}}\mathbf{x}_{\parallel}(\tau), \partial_{\mathbf{e}}\mathbf{v}_{\parallel}(\tau)] \cdot \nabla \nabla_{\mathbf{v}_{\parallel}} H(\tau)
\Big\}\mathbf{x}_{\perp}(\tau)  \mathrm{d}\tau  \\
& \ \ + \int^{ s^{i}}_{s^{i+1}} \big\{
\mathbf{v}_{\perp}(\tau) \partial_{\mathbf{x}_{\perp}} H(\tau) + \mathbf{v}_{\parallel}(\tau) \cdot \nabla_{\mathbf{x}_{\parallel}} H(\tau) + F_{\parallel}(\tau) \cdot \nabla_{\mathbf{v}_{\parallel}} H(\tau)
\big\} \underbrace{ \partial_{\mathbf{e}}\mathbf{x}_{\perp}(\tau)} \mathrm{d}\tau\\
& \ \ - \int^{ s^{i}}_{s^{i+1}}  \big[ \underbrace{ \partial_{\mathbf{e}}\mathbf{x}_{\perp}(\tau)}, \partial_{\mathbf{e}}\mathbf{x}_{\parallel}(\tau), \partial_{\mathbf{e}}\mathbf{v}_{\parallel}(\tau) \big] \cdot \nabla D(\tau) \mathrm{d}\tau.
\end{split}
\end{equation}
Now we use (\ref{middle}) to control the underbraced term of (\ref{Dode}). Notice that we cannot directly use (\ref{middle}) since now we fix the chart for whole $i-$th intermediate group but the estimate (\ref{middle}) is for the moving frame. (For clarity, we write the index for the chart for this part.) Note the time of bounces within the $i-$th intermediate group ($|t^{\ell_{i-1}}-t^{\ell_{i}}||v| \simeq L_{\xi}$) are
\[
t^{\ell_{i}+1} < s^{i+1} < t^{\ell_{i}} <  t^{\ell_{i}-1} < \cdots\cdots <t^{\ell_{i-1}+2}  < t^{\ell_{i-1}+1} < s^{i} < t^{\ell_{i-1}}.
\]

Now we apply (\ref{chart_changing}) and (\ref{middle}) to bound, for $\tau \in (s^{i+1},s^{i})$ and $\ell \in \{ \ell_{i}, \ell_{i}-1, \cdots, \ell_{i-1}+2, \ell_{i-1} +1 , \ell_{i-1}\}$
\begin{equation}\label{ell_1}
\begin{split}
&\frac{\partial (  \mathbf{x}_{\perp_{\ell}}(\tau ), \mathbf{x}_{\parallel_{\ell}}( \tau), \mathbf{v}_{\perp_{\ell}} ( \tau ), \mathbf{v}_{\parallel_{\ell}} ( \tau )   )}{\partial    (    \mathbf{x}_{\perp_{1}}(s^{1}),   \mathbf{x}_{\parallel_{1}}(s^{1})  , \mathbf{v}_{\perp_{1}}(s^{1}),   \mathbf{v}_{\parallel_{1}}(s^{1})   )}\\
=&\frac{\partial (  \mathbf{x}_{\perp_{\ell}}(\tau), \mathbf{x}_{\parallel_{\ell}}( \tau ), \mathbf{v}_{\perp_{\ell}} ( \tau ), \mathbf{v}_{\parallel_{\ell}} ( \tau )   )}{\partial (  \mathbf{x}_{\perp_{\ell_{i}}}(\tau ), \mathbf{x}_{\parallel_{\ell_{i}}}( \tau ), \mathbf{v}_{\perp_{\ell_{i}}} ( \tau ), \mathbf{v}_{\parallel_{\ell_{i}}} ( \tau )   )}
\frac{  \partial (   \mathbf{x}_{\perp_{\ell_{i}}}(\tau ), \mathbf{x}_{\parallel_{\ell_{i}}}( \tau ), \mathbf{v}_{\perp_{\ell_{i}}} ( \tau ), \mathbf{v}_{\parallel_{\ell_{i}}} ( \tau )   )  }{\partial    (    \mathbf{x}_{\perp_{1}}(s^{1}),   \mathbf{x}_{\parallel_{1}}(s^{1})  , \mathbf{v}_{\perp_{1}}(s^{1}),   \mathbf{v}_{\parallel_{1}}(s^{1})   )}\\
\lesssim&  \tiny{ e^{C|t-s||v|} \left\{ \mathbf{Id}_{6,6} + O_{\xi}(|\mathbf{p}^{\ell} - \mathbf{p}^{\ell_{i}}|)
\left[\begin{array}{ccc|ccc} 
0 & 0 & 0 & & & \\
0 & 1 & 1 & & \mathbf{0}_{3,3} & \\
0 & 1 & 1 &  & & \\ \hline
0 & 0 & 0 & 0 & 0 & 0 \\
0 & |v| & |v| & 0 & 1 & 1 \\
0 & |v| & |v| & 0 & 1 & 1
\end{array}
\right]
 \right\}  
  \left[\begin{array}{cccccc}
 \frac{|v| +1}{|\mathbf{v}_{\perp}^{1}|}  &  \frac{|v|+1}{|\mathbf{v}_{\perp}^{1}|}  &  \frac{|v|+1}{|\mathbf{v}_{\perp}^{1}|} & \min \{ \frac{1}{|v|},1 \} & \min \{ \frac{1}{|v|},1 \}  &\min \{ \frac{1}{|v|},1 \} \\
  \frac{|v|^{2}+1}{|\mathbf{v}_{\perp}^{1}|^{2}} &  \frac{|v|^{2}+1}{|\mathbf{v}_{\perp}^{1}|^{2}} &  \frac{|v|^{2}+1}{|\mathbf{v}_{\perp}^{1}|^{2}} & 
  \frac{1}{|\mathbf{v}_{\perp}^{1}|} &\min \{ \frac{1}{|v|},1 \} &\min \{ \frac{1}{|v|},1 \} \\
   \frac{|v|^{2}+1}{|\mathbf{v}_{\perp}^{1}|^{2}} &  \frac{|v|^{2}+1}{|\mathbf{v}_{\perp}^{1}|^{2}} &  \frac{|v|^{2}+1}{|\mathbf{v}_{\perp}^{1}|^{2}} & 
    \frac{1}{|\mathbf{v}_{\perp}^{1}|} &\min \{ \frac{1}{|v|},1 \} &\min \{ \frac{1}{|v|},1 \} \\
    \frac{|v|^{3}+1}{|\mathbf{v}_{\perp}^{1}|^{2}} &  \frac{|v|^{3}+1}{|\mathbf{v}_{\perp}^{1}|^{2}} &  \frac{|v|^{3}+1}{|\mathbf{v}_{\perp}^{1}|^{2}}
    & \frac{|v|+1}{|\mathbf{v}_{\perp}^{1}|} & 1 & 1\\
     \frac{|v|^{3}+1}{|\mathbf{v}_{\perp}^{1}|^{2}}&  \frac{|v|^{3}+1}{|\mathbf{v}_{\perp}^{1}|^{2}} &  \frac{|v|^{3}+1}{|\mathbf{v}_{\perp}^{1}|^{2}}
   & \frac{|v|+1}{|\mathbf{v}_{\perp}^{1}|} & 1 & 1\\    
      \frac{|v|^{3}+1}{|\mathbf{v}_{\perp}^{1}|^{2}}&  \frac{|v|^{3}+1}{|\mathbf{v}_{\perp}^{1}|^{2}} & \frac{|v|^{3}+1}{|\mathbf{v}_{\perp}^{1}|^{2}}
        & \frac{|v|+1}{|\mathbf{v}_{\perp}^{1}|} & 1 & 1\\
\end{array}\right] }  \\
\lesssim &e^{C|t-s||v|}  
  \left[\begin{array}{cccccc}
 \frac{|v| +1}{|\mathbf{v}_{\perp}^{1}|}  &  \frac{|v|+1}{|\mathbf{v}_{\perp}^{1}|}  &  \frac{|v|+1}{|\mathbf{v}_{\perp}^{1}|} & \min \{ \frac{1}{|v|},1 \} & \min \{ \frac{1}{|v|},1 \}  &\min \{ \frac{1}{|v|},1 \} \\
  \frac{|v|^{2}+1}{|\mathbf{v}_{\perp}^{1}|^{2}} &  \frac{|v|^{2}+1}{|\mathbf{v}_{\perp}^{1}|^{2}} &  \frac{|v|^{2}+1}{|\mathbf{v}_{\perp}^{1}|^{2}} & 
  \frac{1}{|\mathbf{v}_{\perp}^{1}|} &\min \{ \frac{1}{|v|},1 \} &\min \{ \frac{1}{|v|},1 \} \\
   \frac{|v|^{2}+1}{|\mathbf{v}_{\perp}^{1}|^{2}} &  \frac{|v|^{2}+1}{|\mathbf{v}_{\perp}^{1}|^{2}} &  \frac{|v|^{2}+1}{|\mathbf{v}_{\perp}^{1}|^{2}} & 
    \frac{1}{|\mathbf{v}_{\perp}^{1}|} &\min \{ \frac{1}{|v|},1 \} &\min \{ \frac{1}{|v|},1 \} \\
    \frac{|v|^{3}+1}{|\mathbf{v}_{\perp}^{1}|^{2}} &  \frac{|v|^{3}+1}{|\mathbf{v}_{\perp}^{1}|^{2}} &  \frac{|v|^{3}+1}{|\mathbf{v}_{\perp}^{1}|^{2}}
    & \frac{|v|+1}{|\mathbf{v}_{\perp}^{1}|} & 1 & 1\\
     \frac{|v|^{3}+1}{|\mathbf{v}_{\perp}^{1}|^{2}}&  \frac{|v|^{3}+1}{|\mathbf{v}_{\perp}^{1}|^{2}} &  \frac{|v|^{3}+1}{|\mathbf{v}_{\perp}^{1}|^{2}}
   & \frac{|v|+1}{|\mathbf{v}_{\perp}^{1}|} & 1 & 1\\    
      \frac{|v|^{3}+1}{|\mathbf{v}_{\perp}^{1}|^{2}}&  \frac{|v|^{3}+1}{|\mathbf{v}_{\perp}^{1}|^{2}} & \frac{|v|^{3}+1}{|\mathbf{v}_{\perp}^{1}|^{2}}
        & \frac{|v|+1}{|\mathbf{v}_{\perp}^{1}|} & 1 & 1\\
\end{array}\right] ,
\end{split}
\end{equation}
where we have used $|\mathbf{p}^{\ell} - \mathbf{p}^{\ell_{i}}| \lesssim 1.$


We plug in (\ref{Dode}) with (\ref{ell_1}) respectively with 
\[ 
| \partial_{\mathbf{x}_{\perp_1}}\mathbf{x}_{\perp}(\tau) | \lesssim \frac{|v|+1}{ |\mathbf v_\perp^1 | }, \, | \partial_{\mathbf{x}_{\perp_1}}\mathbf{v}_{\perp}(\tau) | \lesssim \frac{|v|^3+1}{ |\mathbf v_\perp^1 |^2}, \, | \partial_{\mathbf{v}_{\perp_1}}\mathbf{x}_{\perp}(\tau) | \lesssim \min\{ \frac{1}{|v|}, 1 \}, \, | \partial_{\mathbf{v}_{\perp_1}}\mathbf{v}_{\perp}(\tau) | \lesssim 1,
\]
and
\[ |  \nabla_{\mathbf{v}_{\parallel}} H(\tau) | \lesssim 1, \, |  \nabla_{\mathbf{x}_{\parallel}, \mathbf{x}_\perp} H(\tau) | \lesssim |v|+1, \, |  \nabla_{\mathbf{v}_{\parallel}} D(\tau) | \lesssim |v|+1, \, |  \nabla_{\mathbf{x}_{\parallel}, \mathbf{x}_\perp} D(\tau) | \lesssim |v|^2+1,
 \]
and use the fact that $|s^i -s^{i+1}| \lesssim \frac{1}{|v|+1}$ by the way we define $s^i$. Collecting terms with tedious but straightforward bounds, we summarize the results
as: for $s \in [s^{i+1}, s^{i}]$
\begin{equation}\label{ODE_group}
\begin{split}
\left[
\begin{array}{c}
|\frac{\partial  \mathbf{x}_{\parallel}(s )}{ \partial {\mathbf{x}_{\perp}} }|\\
 |  \frac{ \partial   \mathbf{v}_{\parallel}(s )}{  \partial {\mathbf{x}_{\perp}}} | 
\end{array}
\right]
&\lesssim_{\xi}
\left[
\begin{array}{c}
|\frac{\partial  \mathbf{x}_{\parallel}(s^{i})}{ \partial {\mathbf{x}_{\perp}} }| \\
 |  \frac{ \partial   \mathbf{v}_{\parallel}(s^{i})}{  \partial {\mathbf{x}_{\perp}}} |
 + |v| |\frac{\partial  \mathbf{x}_{\parallel}(s^{i})}{ \partial {\mathbf{x}_{\perp}} }|
\end{array}
\right]
+ 
\left[\begin{array}{c}
\int^{s^{i}}_{s } | \frac{ \partial \mathbf{v}_{\parallel} }{\partial{\mathbf{x}_{\perp}}}  | \\
 \int^{s^{i}}_{s }  ( |v|^{2}+1) |   \frac{ \partial \mathbf{x}_{\parallel} }{   \partial {\mathbf{x}_{\perp}}}   |+ ( |v|+1)
  | \frac{ \partial  \mathbf{v}_{\parallel} }{  \partial {\mathbf{x}_{\perp}}}|  
\end{array}
\right]
+ \left[\begin{array}{c}
0\\
  e^{C|v||t-s |}   \frac{|v|^{2}+1}{|\mathbf{v}_{\perp}^{1}|}  
\end{array}
\right],
\\
\left[
\begin{array}{c}
|\frac{\partial  \mathbf{x}_{\parallel}(s )}{ \partial {\mathbf{v}_{\perp}} }|\\
 |  \frac{ \partial   \mathbf{v}_{\parallel}(s )}{  \partial {\mathbf{v}_{\perp}}} | 
\end{array}
\right]
&\lesssim_{\xi}
\left[
\begin{array}{c}
|\frac{\partial  \mathbf{x}_{\parallel}(s^{i})}{ \partial {\mathbf{v}_{\perp}} }| \\
 |  \frac{ \partial   \mathbf{v}_{\parallel}(s^{i})}{  \partial {\mathbf{v}_{\perp}}} |
 + |v| |\frac{\partial  \mathbf{x}_{\parallel}(s^{i})}{ \partial {\mathbf{v}_{\perp}} }|
\end{array}
\right]
+ 
\left[\begin{array}{c}
\int^{s^{i}}_{s } | \frac{ \partial \mathbf{v}_{\parallel} }{\partial{\mathbf{v}_{\perp}}}  | \\
 \int^{s^{i}}_{s }   (|v|^{2} +1) |   \frac{ \partial \mathbf{x}_{\parallel} }{   \partial {\mathbf{v}_{\perp}}}   |+  (|v|+1)
  | \frac{ \partial  \mathbf{v}_{\parallel} }{  \partial {\mathbf{v}_{\perp}}}|  
\end{array}
\right]
+ \left[\begin{array}{c}
0\\
  e^{C|v||t-s |} 
\end{array}
\right].
\end{split}
\end{equation}
From \eqref{ODE_group} we have
\Be \label{beforegronwall}
\langle v \rangle |\frac{\partial  \mathbf{x}_{\parallel}(s )}{ \partial {\mathbf{x}_{\perp}} }| + | \frac{ \partial \mathbf v_\parallel (s) }{\partial \mathbf x_\perp} (s) | 
\lesssim   |  \frac{ \partial   \mathbf{v}_{\parallel}(s^{i})}{  \partial {\mathbf{x}_{\perp}}} |
 +\langle v \rangle |\frac{\partial  \mathbf{x}_{\parallel}(s^{i})}{ \partial {\mathbf{x}_{\perp}} }| + e^{C|v||t-s |}   \frac{|v|^{2}+1}{|\mathbf{v}_{\perp}^{1}|} +  \int^{s^{i}}_{s }  \langle v \rangle \left(\langle v \rangle|   \frac{ \partial \mathbf{x}_{\parallel} }{   \partial {\mathbf{x}_{\perp}}}   |+ 
  | \frac{ \partial  \mathbf{v}_{\parallel} }{  \partial {\mathbf{x}_{\perp}}}|   \right),
\Ee
from Gronwall inequality we get
\Be \label{gronwallpxparall} \begin{split}
\langle v \rangle |\frac{\partial  \mathbf{x}_{\parallel}(s )}{ \partial {\mathbf{x}_{\perp}} }| + | \frac{ \partial \mathbf v_\parallel (s) }{\partial \mathbf x_\perp} (s) | 
& \le C'_\xi e^{\int_s^{s^i} \langle v \rangle  d\tau}   \left(  |  \frac{ \partial   \mathbf{v}_{\parallel}(s^{i})}{  \partial {\mathbf{x}_{\perp}}} | + \langle v \rangle |\frac{\partial  \mathbf{x}_{\parallel}(s^{i})}{ \partial {\mathbf{x}_{\perp}} }| + e^{C|v||t-s |}  \frac{|v|^{2}+1}{|\mathbf{v}_{\perp}^{1}|}  \right)
\\ & \le  C(\xi)   \left(  |  \frac{ \partial   \mathbf{v}_{\parallel}(s^{i})}{  \partial {\mathbf{x}_{\perp}}} | + \langle v \rangle |\frac{\partial  \mathbf{x}_{\parallel}(s^{i})}{ \partial {\mathbf{x}_{\perp}} }| + e^{C|v||t-s |}  \frac{|v|^{2}+1}{|\mathbf{v}_{\perp}^{1}|}  \right)
\end{split}
\Ee 
Iterating \eqref{gronwallpxparall} we get
\Be \label{vperpparafinal} \begin{split}
\langle v \rangle |\frac{\partial  \mathbf{x}_{\parallel}(s )}{ \partial {\mathbf{x}_{\perp}} }| + | \frac{ \partial \mathbf v_\parallel (s) }{\partial \mathbf x_\perp} (s) | 
\le & C^2 \left( |  \frac{ \partial   \mathbf{v}_{\parallel}(s^{i-1})}{  \partial {\mathbf{x}_{\perp}}} | + \langle v \rangle |\frac{\partial  \mathbf{x}_{\parallel}(s^{i-1})}{ \partial {\mathbf{x}_{\perp}} }|  \right)  + (C^2 +C) e^{C|v||t-s |}  \frac{|v|^{2}+1}{|\mathbf{v}_{\perp}^{1}|} 
\\   \vdots &
\\ \le & C^{[\frac{|t-s||v|}{L_{\xi}}]}\left( |  \frac{ \partial   \mathbf{v}_{\parallel}(s^{1})}{  \partial {\mathbf{x}_{\perp}}} | + \langle v \rangle |\frac{\partial  \mathbf{x}_{\parallel}(s^{1})}{ \partial {\mathbf{x}_{\perp}} }|  \right) + (C^{[\frac{|t-s||v|}{L_{\xi}}]} + \cdots + C)e^{C|v||t-s |}  \frac{|v|^{2}+1}{|\mathbf{v}_{\perp}^{1}|} 
\\ \le & C^{[\frac{|t-s||v|}{L_{\xi}}]}\left( |  \frac{ \partial   \mathbf{v}_{\parallel}(s^{1})}{  \partial {\mathbf{x}_{\perp}}} | + \langle v \rangle |\frac{\partial  \mathbf{x}_{\parallel}(s^{1})}{ \partial {\mathbf{x}_{\perp}} }|  \right)  + C^{2[\frac{|t-s||v|}{L_{\xi}}] } e^{C|v||t-s |}  \frac{|v|^{2}+1}{|\mathbf{v}_{\perp}^{1}|} 
\\ \le & C^{C|t-s||v| } \left( |  \frac{ \partial   \mathbf{v}_{\parallel}(s^{1})}{  \partial {\mathbf{x}_{\perp}}} | + \langle v \rangle |\frac{\partial  \mathbf{x}_{\parallel}(s^{1})}{ \partial {\mathbf{x}_{\perp}} }| + \frac{|v|^{2}+1}{|\mathbf{v}_{\perp}^{1}|} \right)
\\ \le & C^{C |t-s| |v| }\frac{\langle v \rangle^{2}}{|\mathbf{v}_{\perp}^{1}|}.
\end{split} \Ee
And by the same argument as \eqref{beforegronwall} - \eqref{vperpparafinal}, we get from \eqref{ODE_group} that
\Be \label{vperpparafinal2}
\langle v \rangle |\frac{\partial  \mathbf{x}_{\parallel}(s )}{ \partial {\mathbf{v}_{\perp}} }| + | \frac{ \partial \mathbf v_\parallel (s) }{\partial \mathbf v_\perp} (s) | \le  C^{C |t-s| |v| }.
\Ee
Therefore, from \eqref{vperpparafinal} and \eqref{vperpparafinal2} we get
\Be \label{xparallderfinal}
\left[
\begin{array}{c}
|\frac{\partial  \mathbf{x}_{\parallel}(s )}{ \partial {\mathbf{x}_{\perp_1}} }|\\
 |  \frac{ \partial   \mathbf{x}_{\parallel}(s )}{  \partial {\mathbf{v}_{\perp_1}}} | 
\end{array}
\right]
\lesssim C^{C|t-s||v|}
\left[ \begin{array}{c}
\frac{\langle v \rangle}{|\mathbf v_\perp^1|} \\
\frac{1}{\langle v \rangle}
\end{array}
\right]
\Ee

With the estimate \eqref{xparallderfinal}, we refine (\ref{middle}) and \eqref{middle2} to give a final estimate for the case that some $\ell$ is \textit{Type I} or \textit{Type II}
:
  \begin{equation}\label{middle_refined}
 \begin{split}
 & \frac{\partial (s^{\ell_{*}}, \mathbf{x}_{\perp}(s^{\ell_{*}}), \mathbf{x}_{\parallel}(s^{\ell_{*}}), \mathbf{v}_{\perp}(s^{\ell_{*}}), \mathbf{v}_{\parallel}(s^{\ell_{*}})  )}{\partial (s^{1}, \mathbf{x}_{\perp}(s^{1}), \mathbf{x}_{\parallel}(s^{1}), \mathbf{v}_{\perp}(s^{1}), \mathbf{v}_{\parallel}(s^{1}) )}\\
 &
 \lesssim C^{C|v|(t-s)} \left[\begin{array}{c|cc|ccc}
 0 & 0 & \mathbf{0}_{1,2} & 0 & \mathbf{0}_{1,2} &\\ \hline
 \frac{|v|^2+1}{ |\mathbf{v}_{\perp}^{1}| }&   \frac{|v|+1}{|\mathbf{v}_{\perp}^{1}|} & \min\{ | \mathbf v_\perp^1|, \frac{ | \mathbf v_\perp^1|}{\langle v \rangle} \}  & \frac{1}{\langle v \rangle}& \frac{1}{\langle v \rangle}&\\
\frac{ |v|^3 +|v|}{ |\mathbf{v}_{\perp}^{1}|^2 } &    \frac{|v|+1}{|\mathbf{v}_{\perp}^{1}|}   &  1  &\frac{1}{\langle v \rangle}& \frac{1}{\langle v \rangle} \\ \hline
\frac{ |v|^4 +1}{ |\mathbf{v}_{\perp}^{1}|^2 } & \frac{|v|^{3}+1}{|\mathbf{v}_{\perp}^{1}|^{2}}   &   |v|+1  &  \frac{|v|}{|\mathbf{v}_{\perp}^{1}|} & O_{\xi}(1) \\
\frac{ |v|^4 +1}{ |\mathbf{v}_{\perp}^{1}|^2 }&   \frac{|v|^{2}+1}{|\mathbf{v}_{\perp}^{1}|}&  |v|+1 & O_{\xi}(1)&O_{\xi}(1) &
   \end{array}\right],
 \end{split}
 \end{equation}
and from (\ref{ss1}) and (\ref{s1t})
\begin{equation}\label{final_Dxv_grazing}
\begin{split}
 &\frac{\partial ( X_{\mathbf{cl}}(s;t,x,v),V_{\mathbf{cl}}(s;t,x,v))}{\partial (t,x,v)} \\
 &  \lesssim   C^{C|v|(t-s)}  \ \frac{\partial ( X_{\mathbf{cl}}(s), V_{\mathbf{cl}}(s))}{\partial (s^{\ell_{*}}, \mathbf{X}_{\mathbf{cl}}(s^{\ell_{*}}), \mathbf{V}_{\mathbf{cl}}(s^{\ell_{*}}))}
\left[\begin{array}{c|c|c} 
0 & \mathbf{0}_{1,3} & \mathbf{0}_{1,3} \\ \hline
 \frac{ |v|^3 +|v|}{ |\mathbf{v}_{\perp}^{1}|^2 } &   \frac{|v|+1}{|\mathbf{v}_{\perp}^{1}|} & \frac{1}{\langle v \rangle} \\ \hline
\frac{ |v|^4 +1}{ |\mathbf{v}_{\perp}^{1}|^2 } &   \frac{|v|^{3}+1}{|\mathbf{v}_{\perp}^{1}|^{2}} &  \frac{|v|+1}{|\mathbf{v}_{\perp}^{1}|} \end{array}\right]  \frac{   {\partial (s^{1}, \mathbf{x}_{\perp}(s^{1}), \mathbf{x}_{\parallel}(s^{1}), \mathbf{v}_{\perp}(s^{1}), \mathbf{v}_{\parallel}(s^{1}) )}       }{\partial (t,x,v)}\\
&  \lesssim C^{C|v|(t-s)} \left[\begin{array}{ccc} 
 |v| + |s^{\ell_{*}}-s| & 1 & |s^{\ell_{*}}-s| \\ 1 & |v| & 1 \end{array}\right]
\left[\begin{array}{c|c|c} 
0 & \mathbf{0}_{1,3} & \mathbf{0}_{1,3} \\ \hline
 \frac{ |v|^3 +|v|}{ |\mathbf{v}_{\perp}^{1}|^2 } &   \frac{|v|+1}{|\mathbf{v}_{\perp}^{1}|} & \frac{1}{\langle v \rangle} \\ \hline
\frac{ |v|^4 +1}{ |\mathbf{v}_{\perp}^{1}|^2 } &   \frac{|v|^{3}+1}{|\mathbf{v}_{\perp}^{1}|^{2}} &  \frac{|v|+1}{|\mathbf{v}_{\perp}^{1}|} \end{array}\right] 
 \left[\begin{array}{ccc} 1 & \mathbf{0}_{1,3} & \mathbf{0}_{1,3} \\  |t-s^{1}|^2 & 1 & |t-s^{1}| \\  |t-s^{1}| & |v| &  1  \end{array}\right]\\
 & \lesssim  C^{C|v|(t-s)}
  \left[\begin{array}{c|c|c} 
 \frac{ |v|^3 +|v|}{ |\mathbf{v}_{\perp}^{1}|^2 } &   \frac{|v|+1}{|\mathbf{v}_{\perp}^{1}|} & \frac{1}{\langle v \rangle} \\ \hline
  \frac{ |v|^4 +1}{ |\mathbf{v}_{\perp}^{1}|^2 } &  \frac{|v|^{3}+1}{|\mathbf{v}_{\perp}^{1}|^{2}} &   \frac{|v|+1}{|\mathbf{v}_{\perp}^{1}|}\end{array} \right].
\end{split}
\end{equation}
Finally from (\ref{final_Dxv_nongrazing}) and (\ref{final_Dxv_grazing}) we conclude, for all $\tau \in [s,t]$
\begin{equation}\notag
\begin{split}
\frac{\partial(  X_{\mathbf{cl}}(s;t,x,v), V_{\mathbf{cl}}(s;t,x,v))}{\partial (t,x,v)}
\leq
Ce^{C|v|(t-s)}    \left[\begin{array}{c|c|c}
 \frac{ |v|^3 +|v|}{ |\mathbf{v}_{\perp}^{1}|^2 } &   \frac{|v|}{|\mathbf{v}_{\perp}^{1}|} & \frac{1}{\langle v \rangle} \\ \hline
  \frac{ |v|^4 +1}{ |\mathbf{v}_{\perp}^{1}|^2 }  &   \frac{|v|^{3}}{|\mathbf{v}_{\perp}^{1}|^{2}} &   \frac{|v|}{|\mathbf{v}_{\perp}^{1}|}\end{array} \right]_{6\times 7}.
\end{split}
\end{equation}
From the Velocity lemma (Lemma \ref{velocitylemma}),
\begin{equation}\notag
\begin{split}
|\mathbf{v}_{\perp}^{1}| &= |v^{1}\cdot [-n(x^{1})] |= |V_{\mathbf{cl}}(t^{1};t,x,v) \cdot n(X_{\mathbf{cl}}(t^{1};t,x,v) ) |\\
&= \sqrt{\alpha(X_{\mathbf{cl}}(t^{1} ),V_{\mathbf{cl}}(t^{1} ))} \geq e^{\mathcal{C}|v| |t-t^{1}|}\alpha(t,x,v) \gtrsim \alpha(t,x,v),
\end{split}
\end{equation}
and this completes the proof.

\end{proof}

\section{Weighted $C^1$ estimate}
In this section, we put together all the results we get in previous sections and prove our main theorem.
\begin{proof}[\textbf{Proof of Theorem \ref{C1Main}}]
We use the approximation sequence (\ref{VPBsq1}) with (\ref{seqbd}). Due to \eqref{uniformlinfinitybddnoselfgeneratedpotential} we have 
\[
\sup_{m} \sup_{0 \leq  t \leq T}||  e^{\theta|v|^{2}}f^{m}(t) ||_{\infty}\lesssim_{\xi, T}  P(   || e^{\theta^{\prime}|v|^{2}}  f_{0}||_{\infty}).
\]

Now we claim that the distributional derivatives coincide with the piecewise derivatives. This is due to Proposition \ref{inflowexistence1} with an invariant property of $\Gamma(f,f)= \Gamma_{\mathrm{gain}}(f,f) - \nu(\sqrt{\mu}f)f:$ \textit{Assume $f^{m}(v)=f^{m-1}(\mathcal{O}v)$ holds for some orthonormal matrix. Then }
\begin{equation}\label{invariant}
\Gamma(f^{m}, f^{m})(v) = \Gamma(f^{m-1}, f^{m-1})(\mathcal{O}v).
\end{equation}
Denote 
\Be \label{nuell}
\nu^{m - \ell}(s): = \nu^{m - \ell}(s,X_{\mathbf{cl}}(s), V_{\mathbf{cl}}(s)) : = \nu(\sqrt \mu f^{m - \ell } )(s,X_{\mathbf{cl}}(s), V_{\mathbf{cl}}(s)) - \frac{V_{\mathbf{cl}}(s))}{2} \cdot E(s,X_{\mathbf{cl}}(s), V_{\mathbf{cl}}(s)).
\Ee
Using (\ref{invariant}), we apply Proposition 1 to have 
\begin{equation}\notag
\begin{split}
&f^{m}(t,x,v) \\
&= e^{- \int_{0}^{t} \sum_{\ell=0}^{\ell_{*} (0)} \mathbf{1}_{[t^{\ell+1} , t^{\ell} ) } (s)\nu^{m - \ell}(s) \mathrm{d}s  } f_{0}(X_{\mathbf{cl}}(0), V_{\mathbf{cl}}(0))\\
& \ \ +  \int_{0}^{t}\sum_{\ell=0 }^{\ell_{*}(0)}  \mathbf{1}_{[t^{\ell+1},t^{\ell})}(s) e^{-\int^{t}_{s}
\sum_{j=0}^{\ell_{*}(s)} \mathbf{1}_{[{t^{j+1}},{t^{j}})}(\tau)\nu^{m -j}(\tau )%
\mathrm{d}\tau }     \Gamma_{\mathrm{gain}} (f^{m-\ell}, f^{m-\ell}) (s,X_{\mathbf{cl}}(s),V_{\mathbf{cl}}(s)) \mathrm{d}%
s.
\end{split}
\end{equation}

 Now we consider the spatial and velocity derivatives. In the sense of distributions, we have for $\partial_{\mathbf{e}}\in  \{\nabla_{x}, \nabla_{v}\}$  
\begin{equation}\label{deriv_spec}
\begin{split}
  \partial_{\mathbf{e}} f^{m}(t,x,v) = \text{I} _{\mathbf{e}}+ \text{II}_{\mathbf{e}} + \text{III}_{\mathbf{e}}. \end{split}
\end{equation}
Here
\[
\text{I}_{\mathbf{e}} \ = \ e^{- \int^{t}_{0}   \sum_{\ell=0}^{\ell_{*}(0)} \mathbf{1}_{[{t^{\ell+1}} ,{t^{\ell}})}(s) \nu^{m - \ell}(s) \mathrm{d}s} \
 \partial_{\mathbf{e}}[ X_{\mathbf{cl}}(0),  V_{\mathbf{cl}}(0) ] \cdot \nabla_{x,v}f_{0}(X_{\mathbf{cl}}(0),V_{\mathbf{cl}}(0)),
\]
and
   \begin{equation}\notag
   \begin{split}
\text{II}_{\mathbf{e}}&= \int_{0}^{t}\sum_{\ell=0 }^{\ell_{*}(0)}  \mathbf{1}_{[t^{\ell+1},t^{\ell})}(s) e^{-\int^{t}_{s}
\sum_{j=0}^{\ell_{*}(s)} \mathbf{1}_{[{t^{j+1}},{t^{j}})}(\tau)\nu^{m -j}(\tau )%
\mathrm{d}\tau }     \partial_{\mathbf{e}} \big[    \Gamma_{\mathrm{gain}} (f^{m-\ell}, f^{m-\ell}) (s,X_{\mathbf{cl}}(s),V_{\mathbf{cl}}(s))\big] \mathrm{d}%
s\\
& -\int_{0}^{t}\sum_{\ell =0}^{\ell_{*}(0)}  \mathbf{1}_{[t^{\ell+1},t^{\ell})}(s) e^{-\int^{t}_{s}
\sum_{j} \mathbf{1}_{[{t^{j+1}},{t^{j}})}(\tau)\nu^{m -j}(\tau )%
\mathrm{d}\tau }  \int^{t}_{s}
\sum_{j=0}^{\ell_{*}(s)} \mathbf{1}_{[{t^{j+1}},{t^{j}})}(\tau) \partial_{\mathbf{e}} [\nu^{m -j}(\tau, X_{\mathbf{cl}}(\tau), V_{\mathbf{cl}}(\tau) )]%
\mathrm{d}\tau   \\  
& \ \ \ \ \  \times  \Gamma_{\text{gain}}(f^{m-\ell}, f^{m-\ell}) (s,X_{\mathbf{cl}}(s),V_{\mathbf{cl}}(s)) \mathrm{d}%
s  \\
&   - e^{- \int^{t}_{0}   \sum_{\ell=0 }^{\ell_{*}(0)} \mathbf{1}_{[{t^{\ell+1}} ,{t^{\ell}})}(s) \nu^{m - \ell}(s) \mathrm{d}s} \
   f_{0}(X_{\mathbf{cl}}(0),V_{\mathbf{cl}}(0)){  \int^{t}_{0}   \sum_{\ell=0 }^{\ell_{*}(0)} \mathbf{1}_{[{t^{\ell+1}} ,{t^{\ell}})}(s)  \partial_{\mathbf{e}} \big[  \nu^{m -\ell}(s, X_{\mathbf{cl}}(s), V_{\mathbf{cl}}(s))  \big] \mathrm{d}s},
\end{split}
\end{equation}
and
\begin{equation}\notag
\begin{split}
\mathbf{III}_{\mathbf{e}}=& \sum_{\ell=0}^{\ell_{*}(0)} \big[ -\partial_{\mathbf{e}} t^{\ell} 
\lim_{s \uparrow t^{\ell} } \nu^{m - \ell}(s) +  \partial_{\mathbf{e}} t^{\ell+1} \lim_{s\downarrow t^{\ell+1}} \nu^{m - \ell}(s)  \big]\\
& \ \ \ \ \ \  \times e^{-\int_{0}^{t}  \sum_{\ell=0}^{\ell_{*} (0)} \mathbf{1}_{[t^{\ell+1}, t^{\ell} )}(s) \nu^{m - \ell}(s)   }  \\
+& \sum_{\ell=0}^{\ell_{*}(0)}  \Big[\lim_{s\uparrow t^{\ell}} e^{- \int^{t}_{s} \sum_{j} \mathbf{1}_{[t^{j+1}, t^{j} ) } (\tau) \nu^{m -j}(\tau) \mathrm{d}\tau} \Gamma_{\mathrm{gain}}(f^{m-\ell}, f^{m-\ell}) (s,X_{\mathbf{cl}}(s), V_{\mathbf{cl}}(s))\\
& \ \ \ \ \ \ \ \ - \lim_{s\downarrow t^{\ell+1}} e^{- \int^{t}_{s} \sum_{j} \mathbf{1}_{[t^{j+1}, t^{j} ) } (\tau) \nu^{m -j}(\tau) \mathrm{d}\tau} \Gamma_{\mathrm{gain}}(f^{m-\ell}, f^{m-\ell}) (s,X_{\mathbf{cl}}(s), V_{\mathbf{cl}}(s))\\
+& \int_{0}^{t} \sum_{\ell} \mathbf{1}_{[t^{\ell+1}, t^{\ell} )}(s) \sum_{j=0}^{\ell_{*}(s)}\big[ -\lim_{\tau \downarrow t^{j}} \nu^{m -j}(\tau, X_{\mathbf{cl}}(\tau), V_{\mathbf{cl}}(\tau))  + \lim_{\tau \uparrow t^{j+1}} \nu^{m -j}(\tau, X_{\mathbf{cl}}(\tau), V_{\mathbf{cl}}(\tau))\big] \\
&  \ \ \ \ \  \ \ \ \times e^{-\int^{t}_{s} \sum_{j=0}^{\ell_{*} (s)}  \mathbf{1}_{[t^{j+1} , t^{j} ) }(\tau) \nu^{m -j} (\tau) \mathrm{d}\tau }\Gamma_{\mathrm{gain}}(f^{m-\ell} , f^{m-\ell})(s, X_{\mathbf{cl}}(s), V_{\mathbf{cl}}(s)).
\end{split}
\end{equation}
For $\mathbf{III}_{\mathbf{e}}$ we rearrange the summation and use (\ref{v_perp}), \eqref{nuell} and apply (\ref{invariant}) to get 
\begin{equation} \label{IIIe}
\begin{split}
\mathbf{III}_{\mathbf{e}} = & \sum_{\ell=0}^{\ell_{*}(0)} 
\Big[  - \nu^{m - \ell}(t^{\ell}, x^{\ell}, v^{\ell} ) + \nu^{m - \ell+1}( t^{\ell}, x^{\ell}, R_{x^{\ell}}v^{\ell})  \Big] \partial_{\mathbf{e}} t^{\ell}e^{-\int_{0}^{t}  \sum_{\ell=0}^{\ell_{*} (0)} \mathbf{1}_{[t^{\ell+1}, t^{\ell} )}(s) \nu^{m - \ell}(s)   } \\
+ &\sum_{\ell=0}^{\ell_{*}(0)} e^{- \int^{t}_{t^{\ell}} \sum_{j} \mathbf{1}_{[t^{j+1},t^{j}) }(\tau) \nu(\sqrt{\mu} f^{m-j}) (\tau) \mathrm{d}\tau  } 
\Big[ \Gamma_{\mathrm{gain}}( f^{m-\ell}, f^{m-\ell} )(t^{\ell}, x^{\ell}, v^{\ell}) -\Gamma_{\mathrm{gain}}( f^{m-\ell+1}, f^{m-\ell+1} )(t^{\ell}, x^{\ell}, R_{x^{\ell}}v^{\ell}) \Big]\\
+& \int_{0}^{t} \sum_{\ell} \mathbf{1}_{[t^{\ell+1}, t^{\ell} )}(s) ^{-\int^{t}_{s} \sum_{j=0}^{\ell_{*} (s)}  \mathbf{1}_{[t^{j+1} , t^{j} ) }(\tau) \nu^{m -j} (\tau) \mathrm{d}\tau }\Gamma_{\mathrm{gain}}(f^{m-\ell} , f^{m-\ell})(s, X_{\mathbf{cl}}(s), V_{\mathbf{cl}}(s))\\
&\times\sum_{\ell=0}^{\ell_{*}(s)} 
\Big[  - \nu^{m - \ell}(t^{\ell}, x^{\ell}, v^{\ell} ) + \nu^{m - \ell+1}( t^{\ell}, x^{\ell}, R_{x^{\ell}}v^{\ell})  \Big]\\
= & \sum_{\ell=0}^{\ell_{*}(0)} \left[ \frac{ R_{x^{\ell}}v^{\ell} - v^\ell  }{2}  \cdot E(t^\ell, x^\ell )\right]\partial_{\mathbf{e}} t^{\ell}e^{-\int_{0}^{t}  \sum_{\ell=0}^{\ell_{*} (0)} \mathbf{1}_{[t^{\ell+1}, t^{\ell} )}(s) \nu^{m - \ell}(s)   }
\\ +& \int_{0}^{t} \sum_{\ell} \mathbf{1}_{[t^{\ell+1}, t^{\ell} )}(s) ^{-\int^{t}_{s} \sum_{j=0}^{\ell_{*} (s)}  \mathbf{1}_{[t^{j+1} , t^{j} ) }(\tau) \nu^{m -j} (\tau) \mathrm{d}\tau }\Gamma_{\mathrm{gain}}(f^{m-\ell} , f^{m-\ell})(s, X_{\mathbf{cl}}(s), V_{\mathbf{cl}}(s))\\
&\times\sum_{\ell=0}^{\ell_{*}(s)} 
\Big[  \frac{ R_{x^{\ell}}v^{\ell} - v^\ell  }{2}  \cdot E(t^\ell, x^\ell )  \Big]
.
\end{split}
\end{equation}

\noindent\textit{Proof of (\ref{invariant}).} The proof is due to the change of variables 
\[
\tilde{u} = \mathcal{O} u, \ \ \tilde{\omega} = \mathcal{O} \omega, \ \ \ 
\mathrm{d}\tilde{u} = \mathrm{d}u , \ \ \mathrm{d}\tilde{\omega} = \mathrm{d}\omega.
\]
Note
\begin{equation}\notag
\begin{split}
&\Gamma(f^{m}, f^{m})(v)\\
=&\int_{\mathbb{R}^{3}} \int_{\mathbb{S}^{2}} |v-u|^{\kappa}q_{0}(\frac{v-u}{|v-u|}\cdot \omega) \sqrt{\mu(u)} \Big\{ f^{m}(u-[(u-v)\cdot \omega]\omega) f^{m}(v+ [(u-v)\cdot \omega]\omega) -f^{m} (u)f^{m} (v) \Big\}
\mathrm{d}\omega\mathrm{d}u\\
=&\int_{\mathbb{R}^{3}} \int_{\mathbb{S}^{2}}
 |\mathcal{O}v-\mathcal{O}u|^{\kappa}q_{0}(\frac{\mathcal{O}v-\mathcal{O}u}{|\mathcal{O}v-\mathcal{O}u|}\cdot \mathcal{O}\omega) \sqrt{\mu(\mathcal{O}u)} \\
&  \ \ \ \times \Big\{ f^{m-1}
(\mathcal{O} u -[( \mathcal{O}u- \mathcal{O}v)\cdot \mathcal{O} \omega] \mathcal{O}{\omega}) f^{m-1}(\mathcal{O}v+ [(\mathcal{O}u-\mathcal{O}v)\cdot \mathcal{O}\omega]\mathcal{O}\omega) -f^{m-1} ( \mathcal{O}u)f^{m-1} (\mathcal{O}v) \Big\}
\mathrm{d}\omega\mathrm{d}u\\
=& \int_{\mathbb{R}^{3}} \int_{\mathbb{S}^{2}} | \mathcal{O}v- \tilde{u}|^{\kappa}q_{0}(\frac{ \mathcal{O}v-  \tilde{u}}{|\mathcal{O}v-  \tilde{u}|}\cdot  \tilde{\omega}) \sqrt{\mu( \tilde{u})}\\
& \ \ \ \times \Big\{
f^{m-1} (  \tilde{u} - [ (   \tilde{u} - \mathcal{O}  {v}  ) \cdot  \tilde{\omega}  ]  \tilde{\omega} )
f^{m-1} ( \mathcal{O}v + [(\tilde{u} - \mathcal{O}v)]\cdot \tilde{\omega}  ) \tilde{\omega}
-f^{m-1} (\tilde{u}) f^{m-1}(\mathcal{O}v)
\Big\} \mathrm{d}\tilde{\omega} \mathrm{d}\tilde{u}\\
= & \ \Gamma(f^{m-1}, f^{m-1})(  \mathcal{O} v ).
\end{split}
\end{equation}

This proves (\ref{invariant}). Especially we can apply (\ref{invariant}) for the specular reflection BC (\ref{specularBC}) with $\mathcal{O}v= R_{x}v$ 

Using Lemma \ref{lemma_operator} and \eqref{uniformlinfinitybddnoselfgeneratedpotential}, we obtain for $\partial_{\mathbf{e}}\in  \{\nabla_{x}, \nabla_{v}\}$
\begin{equation}\notag
\begin{split}
\text{II}_{\mathbf{e}}& \lesssim \   P(|| e^{\theta|v|^{2}} f_{0} ||_{\infty})\int_{0}^{t}\sum_{\ell=0 }^{\ell_{*}(0)}  \mathbf{1}_{[t^{\ell+1},t^{\ell})}(s)    |\partial_{\mathbf{e}} X_{\mathbf{cl}}(s)| 
\int_{\mathbb{R}^{3}}  \frac{e^{-C_{\theta}  |V_{\mathbf{cl}}(s)-u|^{2}}}{|V_{\mathbf{cl}}(s)-u|^{2-\kappa}} | \nabla_{x} f^{m-\ell} (s,X_{\mathbf{cl}}(s),u  )| \mathrm{d}u \mathrm{d}s
\\
& \ \ + P(|| e^{\theta|v|^{2}} f_{0} ||_{\infty})\int_{0}^{t}\sum_{\ell=0 }^{\ell_{*}(0)}  \mathbf{1}_{[t^{\ell+1},t^{\ell})}(s)    |\partial_{\mathbf{e}} V_{\mathbf{cl}}(s)| 
\int_{\mathbb{R}^{3}}  \frac{e^{-C_{\theta}  |V_{\mathbf{cl}}(s)-u|^{2}}}{|V_{\mathbf{cl}}(s)-u|^{2-\kappa}} | \nabla_{v} f^{m-\ell} (s,X_{\mathbf{cl}}(s),u  )| \mathrm{d}u \mathrm{d}s
\\
& \ \ + tP(|| e^{\theta|v|^{2}} f_{0} ||_{\infty})  \langle v\rangle^{\kappa} e^{-\theta |v|^{2}} \left(\|  E \|_{L^\infty_{t,x}} + \| \nabla_x E \|_{L^\infty_{t,x}} \right) \left(  \sup_{0 \leq s\leq t} |\partial_{\mathbf{e}}V(s;t,x,v)| + \sup_{0 \leq s\leq t} |\partial_{\mathbf{e}}X(s;t,x,v)| \right).
\end{split}
\end{equation}

We shall estimate the following:
\[
e^{-\varpi \langle v\rangle t} \frac{ [\alpha(t,x,v)]^{\beta} }{\langle v\rangle^{b+1}} |\partial_{x} f(t,x,v)|, \ \ \ 
e^{-\varpi \langle v\rangle t} \frac{ [\alpha(t,x,v)]^{\beta-1} }{\langle v\rangle^{b-1}} |\partial_{v} f(t,x,v)|.
\]
 
 From (\ref{lemma_Dxv}), the Velocity lemma (Lemma \ref{velocitylemma}), Lemma \ref{localexslemma}, and $F^{m}\geq 0$ for all $m,$ with $\varpi \gg 1$
\begin{equation}\notag
\begin{split}
&e^{-\varpi \langle v\rangle t}  \frac{1}{\langle v\rangle^{b+1}}
 [\alpha(t,x,v)]^{\beta} \ \text{I}_{\mathbf{x}}\\
& \lesssim_{\xi,t}   e^{-\varpi \langle v \rangle t } \frac{1}{\langle v\rangle^{b+1 }}
 [\alpha(X_{\mathbf{cl}}(0),V_{\mathbf{cl}}(0)) ]^{\beta}e^{ 2\mathcal{C}|v|t} \\
 &  \ \ \ \times \Big\{   \frac{\langle v \rangle}{{\alpha(t,x,v)}} |\partial_{x} f_{0}(X_{\mathbf{cl}}(0), V_{\mathbf{cl}}(0))| +  \frac{\langle v\rangle^{3}}{\alpha^2(t,x,v)}  |\partial_{v} f_{0}(X_{\mathbf{cl}}(0), V_{\mathbf{cl}}(0))|  \Big\} \\
&\lesssim_{\xi,t}   
\big|\big| \  \frac{\langle v\rangle }{\langle v\rangle^{b+1}}
 {\alpha}^{\beta-1} \partial_{x} f_{0} \  \big|\big|_{\infty} + \big|\big| \   \frac{   \langle v\rangle^{3}}{\langle v\rangle^{b+1}}
 \alpha^{\beta-2} \partial_{v} f_{0}  \ \big|\big|_{\infty}\\
& \lesssim_{\xi,t} 
 \big|\big| \  \frac{ {\alpha}^{\beta-1}   }{\langle v\rangle^{b }}
\partial_{x} f_{0} \  \big|\big|_{\infty} + \big|\big| \   \frac{     \alpha^{\beta-2} }{\langle v\rangle^{b-2 }}
 \partial_{v} f_{0}  \ \big|\big|_{\infty}
 ,
\end{split}
\end{equation}
and
\begin{equation}\notag
\begin{split}
& e^{-\varpi \langle v\rangle t}  \frac{1}{\langle v\rangle^{b-1}}
[\alpha(t,x,v)]^{\beta-1} \ \text{I}_{\mathbf{v}}\\
& \lesssim_{\xi,t}  \ e^{-\varpi \langle v \rangle t } \frac{ 1}{\langle v\rangle^{b-1}}
 [\alpha(X_{\mathbf{cl}}(0),V_{\mathbf{cl}}(0))]^{\beta -1 } e^{2 \mathcal{C}|v|t}  \\
 &  \ \ \ \times  \Big\{ \frac{1}{\langle v \rangle} |\partial_{x} f_{0}(X_{\mathbf{cl}}(0), V_{\mathbf{cl}}(0))| +   \frac{\langle v\rangle }{{\alpha(t,x,v)}}  |\partial_{v} f_{0}(X_{\mathbf{cl}}(0), V_{\mathbf{cl}}(0))|  \Big\} \\
 & \lesssim_{\xi,t}\
\big|\big| \    \frac{\alpha^{\beta-1 } }{\langle v\rangle^{b}}
 \partial_{x} f_{0} \  \big|\big|_{\infty} + \big|\big| \  \frac{ 1}{\langle v\rangle^{b-2}}
{\alpha}^{ \beta-2}  \partial_{v} f_{0}  \ \big|\big|_{\infty},
\end{split}
\end{equation}
where we have used $\alpha(t,x,v) \lesssim_{\xi} |v|^{2}$ and the choice of $\varpi \gg 1.$

From Lemma \ref{localexslemma} and Lemma \ref{lemma_operator},  
\begin{equation}\notag
\begin{split}
 \text{II}_{\mathbf{e}} & \lesssim_{t}  P(|| e^{\theta|v|^{2}} f_{0} ||_{\infty} ) \int_{0}^{t} \mathrm{d}s \sum_{\ell=0}^{\ell_{*}(0)} \mathbf{1}_{[t^{\ell+1}, t^{\ell})}(s)    \int_{\mathbb{R}^{3}} \mathrm{d}u 
 \frac{e^{-C_{\theta} |u-V_{\mathbf{cl}}(s)|^{2} }}{|u-V_{\mathbf{cl}}(s)|^{2-\kappa}}
   \\
 &  \ \   \times
  \Big\{   |\partial_{\mathbf{e}} X_{\mathbf{cl}}(s)| |\partial_{x} f^{m-j}(s, X_{\mathbf{cl}}(s), u)| 
  + |\partial_{\mathbf{e}} V_{\mathbf{cl}}(s) |   \big(  1+  |\partial_{v} f^{m-j} (s, X_{\mathbf{cl}}(s), u)| \big)
  \Big\} .
\end{split}
\end{equation}
Now we use (\ref{lemma_Dxv}) to have 
\begin{equation}\notag
\begin{split}
& e^{-\varpi \langle v \rangle t}   \frac{ [\alpha(t,x,v)]^{\beta}}{\langle v\rangle^{b+1 }}
 \ \text{II}_{\mathbf{x}}  \lesssim_{t,\xi}   
P(||  e^{\theta|v|^{2}} f_{0} ||_{\infty} )  \\
& \ \ \ \times 
  \bigg\{ 
  \int_{0}^{t}  \int_{\mathbb{R}^{3}}  
  \frac{  e^{-C_{\theta}|V_{\mathbf{cl}}(s) -u|^{2}}}{ |u-V_{\mathbf{cl}}(s)|^{2-\kappa}}
 e^{-\varpi \langle v\rangle t }e^{\varpi \langle u\rangle s} e^{C|v||t-s|} \frac{|v| [\alpha(t,x,v)]^{\beta- \frac{1}{2}}}{ [\alpha(s,X_{\mathbf{cl}}(s),u)]^{\beta}} \frac{\langle u\rangle^{b+1} }{ \langle v\rangle^{b+1}}  \mathrm{d}u\mathrm{d}s
 \\
& \ \ \ \ \ \ \ \ \ \ \ \ \ \ \ \times  \sup_{m} \sup_{0 \leq s\leq t} \big|\big| e^{-\varpi \langle u \rangle s} \frac{[\alpha(s,X_{\mathbf{cl}} (s), u  )]^\beta}{ \langle u \rangle^{b+1}} \partial_{x} f^{m-j}(s,X_{\mathbf{cl}}(s),u)  \big|\big|_{\infty}  \\ 
& \ \ \ \ \ \ + \int_{0}^{t} \int_{\mathbb{R}^{3}}
  \frac{  e^{-C_{\theta}|V_{\mathbf{cl}}(s) -u|^{2}}}{ |u-V_{\mathbf{cl}}(s)|^{2-\kappa}} e^{- \varpi \langle v\rangle t} e^{\varpi \langle u\rangle s} e^{C|v||t-s|} 
\frac{ \langle u\rangle^{b}}{ \langle v\rangle^{b }}  \frac{|v|^{2} [\alpha(t,x,v)]^{\beta-1}}{ |u|[\alpha(s,X_{\mathbf{cl}}(s),u)]^{\beta- \frac{1}{2}}}\\
& \ \ \ \ \ \ \ \ \ \ \ \ \ \ \ \times  \sup_{m} \sup_{0 \leq s\leq t} \big|\big| e^{-\varpi \langle u \rangle s} \frac{|u|[\alpha(s,X_{\mathbf{cl}} (s), u  )]^{\beta-\frac{1}{2}}}{ \langle u \rangle^{b}} \partial_{v} f^{m-j}(s,X_{\mathbf{cl}}(s),u)  \big|\big|_{\infty}\bigg\}
.
\end{split}
\end{equation}
We first claim that 
\begin{equation}\label{exponent}
e^{-\varpi \langle v \rangle t}e^{\varpi \langle u \rangle s} e^{C|v|(t-s)} e^{-C^{\prime}|v-u|^{2} }\lesssim  e^{-\frac{\varpi  \langle v\rangle}{2} (t-s)}e^{C^{\prime\prime}  (s+s^{2})}   e^{- {C}^{\prime\prime}  |v-u|^{2}}.
\end{equation}
Using $\langle u \rangle   \leq  1+ |u| \leq 1+ |v|  + |u-v|  \leq 1+ \langle v\rangle  + |v-u|,$ we bound the first three exponents as
\begin{equation}\notag
\begin{split}
  -(\varpi-C) \langle v \rangle (t-s) - \varpi (\langle v\rangle - \langle u \rangle )s \leq -(\varpi-C)\langle v\rangle (t-s)+   \varpi |v-u|s + \varpi s.
\end{split}
\end{equation} 
Then we use a complete square trick, for $0 < \sigma \ll 1$
\begin{equation}\notag
\begin{split}
\varpi |v-u|s  
  = \frac{\sigma \varpi^{2}}{2} |v-u|^{2} + \frac{s^{2}}{2\sigma} - \frac{1}{2\sigma} \big[s- \sigma \varpi |v-u| \big]^{2} \leq \frac{\sigma \varpi^{2}}{2} |v-u|^{2} + \frac{s^{2}}{2\sigma} ,
\end{split}
\end{equation}
to bound the whole exponents of (\ref{exponent}) by
\begin{equation}\notag
\begin{split}
&-(\varpi -C) \langle v \rangle (t-s) + \varpi |v-u|s  -C^{\prime}|v-u|^{2} + \varpi s \\ 
&\leq -(\varpi-C) \langle v \rangle (t-s) -(C - \frac{\sigma \varpi^{2}}{2}) |v-u|^{2} + \frac{s^{2}}{2\sigma} +\varpi s  \\
& \leq - (\varpi-C)\langle  v\rangle (t-s) - C_{\sigma, \varpi}  |v-u|^{2} + C_{\sigma,\varpi}^{\prime} \big\{ s^{2} +s\big\}.
\end{split}
\end{equation}
Hence we prove the claim (\ref{exponent}) for $\varpi \gg 1.$

Now we use (\ref{exponent}) to bound 
\begin{equation}\label{AB}
\begin{split}
&e^{-\varpi \langle v \rangle t}  \frac{1}{\langle v\rangle^{ b+1}}
[\alpha(t,x,v)]^{\beta} \ \text{II}_{\mathbf{x}}\\
& \lesssim_{t,\xi} P( ||   e^{\theta|v|^{2}} f_{0} ||_{\infty} ) \times\\
& \ \times \bigg\{  \underbrace{  \int_{0}^{t} \int_{ \mathbb{R}^{3}}
 e^{-  \frac{\varpi \langle v \rangle}{2} (t-s) }  
\frac{ e^{-C_{\theta}^{\prime}|V_{\mathbf{cl}}(s)-u|^{2}} }{|V_{\mathbf{cl}}(s)-u|^{2-\kappa}}
 \frac{\langle u\rangle^{b+1}}{\langle v\rangle^{b+1}} \frac{ \langle v\rangle   [\alpha(t,x,v)]^{\beta-1}}{ [\alpha(s,X_{\mathbf{cl}}(s),u)]^{\beta}}
 \mathrm{d}u \mathrm{d}s }_{ \mathbf{(A)} }  \ 
\sup_{m}
 \sup_{0\leq s \leq t}  \big|\big| e^{-\varpi \langle v \rangle s}    \frac{\alpha^{\beta}}{\langle v\rangle^{b+1}}
 \partial_{x} f^{m}(s) \big|\big|_{\infty}\\
& \ \  + 
  \underbrace{  \int_{0}^{t} \int_{ \mathbb{R}^{3}}
 e^{-  \frac{\varpi \langle v \rangle}{2} (t-s) }   \frac{ e^{-C_{\theta}^{\prime}|V_{\mathbf{cl}}(s)-u|^{2}} }{|V_{\mathbf{cl}}(s)-u|^{2-\kappa}}\frac{\langle u\rangle^{b-1 }}{\langle v\rangle^{b+1 }} \frac{ \langle v \rangle^{3}  [\alpha(t,x,v)]^{\beta-2}}{[\alpha(s,X_{\mathbf{cl}}(s),u)]^{\beta-1}}
 \mathrm{d}u \mathrm{d}s }_{ \mathbf{(B)} }    \ 
\sup_{m}
 \sup_{0\leq s \leq t} \big|\big| e^{-\varpi \langle v \rangle s}    \frac{\alpha^{\beta-1}}{\langle v\rangle^{b -1 }}
 \partial_{v} f^{m}(s) \big|\big|_{\infty} \bigg\}
 .\end{split}
\end{equation}
For $\mathbf{(A)}$ we use (\ref{specular_nonlocal}) with $Z=\langle v\rangle\big[ \alpha(t,x,v) \big]^{\beta-1}$ and $l= \frac{\varpi}{2}$ and $r=b+1$. For $\mathbf{(B)}$ we use (\ref{specular_nonlocal}) with $\beta \mapsto\beta-1$ and $Z=\langle v\rangle\big[ \alpha(t,x,v) \big]^{\beta- {2}}$ and $l= \frac{\varpi}{2}$ and $r=b-1$. Then 
\[
\mathbf{(A)} , \  \mathbf{(B)} \   \ll \ 1.
\] 

Similarly, but with a different weight $e^{-\varpi \langle v\rangle t}  \frac{1}{\langle v\rangle^{b-1}}
 [\alpha(t,x,v)]^{\beta-1}$, we use (\ref{lemma_Dxv}) to have
\begin{equation}\notag
\begin{split}
 &e^{-\varpi \langle v \rangle t}   \frac{1}{\langle v\rangle^{b-1}}
[\alpha(t,x,v)]^{\beta-1} \ \text{II}_{\mathbf{v}} \\
&\lesssim_{t,\xi}   
P(|| e^{\theta|v|^{2}} f_{0} ||_{\infty} )  \times\\
& \ \  \times 
  \bigg\{ 
  \int_{0}^{t}  \int_{\mathbb{R}^{3}} 
  \frac{e^{-C|V_{\mathbf{cl}}(s) -u|^{2}} }{ |u-V_{\mathbf{cl}}(s)|^{2-\kappa}}e^{-\varpi \langle v\rangle t }e^{\varpi \langle u\rangle s} e^{C|v||t-s|} \frac{ [\alpha(t,x,v)]^{\beta-1} }{   [\alpha(s,X_{\mathbf{cl}}(s),u)]^{\beta}} \frac{\langle u\rangle^{b+1} }{ \langle v\rangle^{b}}  \mathrm{d}u\mathrm{d}s
 \\
& \ \ \ \ \ \ \ \ \ \ \ \ \ \ \ \times  \sup_{m} \sup_{0 \leq s\leq t} \big|\big| e^{-\varpi \langle u \rangle s} \frac{[\alpha(s,X_{\mathbf{cl}} (s), u  )]^\beta}{ \langle u \rangle^{b+1}} \partial_{x} f^{m }(s,X_{\mathbf{cl}}(s),u)  \big|\big|_{\infty}  \\
& \ \ \ \ \ \ 
+ \int_{0}^{t} \int_{\mathbb{R}^{3}}
 \frac{e^{-C|V_{\mathbf{cl}} (s)-u|^{2}} }{|u-V_{\mathbf{cl}}(s)|^{2-\kappa} }  e^{- \varpi \langle v\rangle t} e^{\varpi \langle u\rangle s} e^{C|v||t-s|} 
\frac{ \langle u\rangle^{b-1}}{ \langle v\rangle^{b-1 }}  \frac{\langle v \rangle [\alpha(t,x,v)]^{\beta-2}}{ [\alpha(s,X_{\mathbf{cl}}(s),u)]^{\beta- 1}}\\
& \ \ \ \ \ \ \ \ \ \ \ \ \ \ \ \times  \sup_{m} \sup_{0 \leq s\leq t} \big|\big| e^{-\varpi \langle u \rangle s} \frac{[\alpha(s,X_{\mathbf{cl}} (s), u  )]^{\beta-1}}{ \langle u \rangle^{b-1}} \partial_{v} f^{m }(s,X_{\mathbf{cl}}(s),u)  \big|\big|_{\infty}\bigg\}
.
\end{split}
\end{equation}
Again we use (\ref{exponent}) and (\ref{specular_nonlocal}) exactly as (\ref{AB}). Therefore for $0<\delta = \delta( || e^{\theta|v|^{2}} f_{0}||_{\infty}  ) \ll 1$
\begin{equation*} 
\begin{split}
&e^{-\varpi \langle v\rangle t} \frac{1}{\langle v\rangle^{b+1}} [\alpha(t,x,v)]^{\beta} \text{II}_{\mathbf{x}}
+e^{-\varpi \langle v\rangle t} \frac{|v|}{ \langle v\rangle^{b} } [\alpha(t,x,v)]^{\beta-1} \text{II}_{\mathbf{v}}\\
\lesssim & \ \delta \ \big\{ \sup_{m} \sup_{0 \leq s\leq t} \big|\big| e^{-\varpi \langle v \rangle s} \frac{ \alpha  ^\beta}{ \langle v \rangle^{b+1}} \partial_{x} f^{m }(s )  \big|\big|_{\infty}+   \sup_{m} \sup_{0 \leq s\leq t} \big|\big| e^{-\varpi \langle v \rangle s} \frac{ \alpha ^{\beta-1}}{ \langle v \rangle^{b-1}} \partial_{v} f^{m }(s)  \big|\big|_{\infty}\big\}.
\end{split}
\end{equation*}

Finally using $ \frac{ R_{x^{\ell}}v^{\ell} - v^\ell  }{2} =  \mathbf v_\perp^\ell $, the bound on $\p_{\mathbf e } t^\ell$ in \eqref{ptell} and \eqref{ptell2}, from (\ref{lemma_Dxv}), the Velocity lemma (Lemma \ref{velocitylemma})
\begin{equation}\notag
\begin{split}
&e^{-\varpi \langle v\rangle t}  \frac{1}{\langle v\rangle^{b+1}}
 [\alpha(t,x,v)]^{\beta} \ \text{III}_{\mathbf{x}} 
\\ & \lesssim e^{-\varpi \langle v\rangle t}  \frac{1}{\langle v\rangle^{b+1}}   [\alpha(t,x,v)]^{\beta} \| E \|_{L^\infty_{t,x}} e^{C\langle v \rangle t } \alpha(t,x,v) \ell_*(0) \sup_{0 \le \ell \le \ell_*(0) } | \p_x t^\ell  | + tP(|| e^{\theta|v|^{2}} f_{0} ||_{\infty})  \| E \|_{L^\infty_{t,x}} 
\\ & \lesssim \| E \|_{L^\infty_{t,x}} \alpha(t,x.v)^{\beta - 2 } +  tP(|| e^{\theta|v|^{2}} f_{0} ||_{\infty})  \| E \|_{L^\infty_{t,x}},
\end{split}
\end{equation}
for $\varpi \gg 1 $. And similarly
\begin{equation}\notag
\begin{split}
&e^{-\varpi \langle v\rangle t}  \frac{1}{\langle v\rangle^{b-1}}
 [\alpha(t,x,v)]^{\beta-1} \ \text{III}_{\mathbf{v}} 
\\ & \lesssim e^{-\varpi \langle v\rangle t}  \frac{1}{\langle v\rangle^{b-1}}   [\alpha(t,x,v)]^{\beta-1} \| E \|_{L^\infty_{t,x}} e^{C\langle v \rangle t } \alpha(t,x,v) \ell_*(0) \sup_{0 \le \ell \le \ell_*(0) } | \p_v t^\ell  | + tP(|| e^{\theta|v|^{2}} f_{0} ||_{\infty})  \| E \|_{L^\infty_{t,x}} 
\\ & \lesssim \| E \|_{L^\infty_{t,x}} \alpha(t,x.v)^{\beta - 2 } +  tP(|| e^{\theta|v|^{2}} f_{0} ||_{\infty})  \| E \|_{L^\infty_{t,x}}.
\end{split}
\end{equation}

Collecting all the terms, for $2< \beta < 3$ and $b > 1$ with $\varpi \gg1$ and $0 < \delta \ll1 $, we get 
\begin{equation}\notag
\begin{split}
&\sup_{m} \sup_{0 \leq s\leq t}||e^{-\varpi \langle v\rangle s}   \frac{\alpha ^{\beta} }{\langle v\rangle^{b+1 }}
\partial_{x}f^{m}(s )||_{\infty}
+ \sup_{m}  \sup_{0 \leq s\leq t}||e^{-\varpi \langle v\rangle s}  \frac{ \alpha ^{\beta-1} }{\langle v\rangle^{b-1}}
\partial_{v}f^{m}(s )||_{\infty}
\\
 & \lesssim 
 ||   \frac{ \alpha^{\beta-1}}{\langle v\rangle^{b}}  \partial_{x}f_{0} ||_{\infty}
   + ||  \frac{ \alpha^{\beta-2}}{\langle v\rangle^{b-2 }}\partial_{v }f_{0 }||_{\infty} 
     +  {P}(||   e^{\theta |v|^{2}} f_{0} ||_{\infty}).
\end{split}
\end{equation} 
 
We remark that this sequence $f^{m}$ is Cauchy in $L^{\infty}([0,T]\times \bar{\Omega}\times\mathbb{R}^{3})$ for $0< T\ll 1$. Therefore the limit function $f$ is a solution of the Boltzmann equation satisfying the specular reflection BC. On the other hand, due to the weak lower semi-continuity of $L^{p}$, $p>1$, we pass a limit $\partial f^{m} \rightharpoonup \partial f$ weakly in the weighted $L^{\infty}-$norm.

Now we consider the continuity of $e^{-\varpi \langle v\rangle t} \langle v\rangle^{-b-1}\alpha^{\beta} \partial_{x} f$ and $e^{-\varpi \langle v\rangle t}\langle v\rangle^{-b+1}\alpha^{\beta-1} \partial_{v}f$. Remark that  $e^{-\varpi \langle v\rangle t} \langle v\rangle^{-b-1}\alpha^{\beta} \partial_{x} f^m$ and $e^{-\varpi \langle v\rangle t}\langle v\rangle^{-b+1}\alpha^{\beta-1} \partial_{v}f^m$ satisfy all the conditions of Proposition \ref{inflowexistence1}. Therefore we conclude 
$$e^{-\varpi \langle v\rangle t} \langle v\rangle^{-b-1}\alpha^{\beta} \partial_{x} f^{m} \in C^{0}([0,T ] \times \bar{\Omega} \times\mathbb{R}^{3}), \ e^{-\varpi \langle v\rangle t}  \langle v\rangle^{-b+1}\alpha^{\beta-1} \partial_{v}f^{m}\in C^{0}([0,T ] \times \bar{\Omega} \times\mathbb{R}^{3}).$$
Now we follow $W^{1,\infty}$ estimate proof for $e^{-\varpi \langle v\rangle t} \langle v\rangle^{-b-1}\alpha^{\beta} [\partial_{x} f^{m+1}- \partial_{x}f^{m}]$ and $e^{-\varpi \langle v\rangle t} \langle v\rangle^{-b+1} \alpha^{\beta-1} [\partial_{v}f^{m+1}-\partial f^{m}]$ to show that $e^{-\varpi \langle v\rangle t} \langle v\rangle^{-b-1}\alpha^{\beta} \partial_{x} f^{m}$ and $e^{-\varpi \langle v\rangle t}\langle v\rangle^{-b+1}\alpha^{\beta-1} \partial_{v}f^{m}$ are Cauchy in $L^{\infty}$. Then we pass a limit $e^{-\varpi \langle v\rangle t} \langle v\rangle^{-b-1}\alpha^{\beta} \partial_{x} f^{m}\rightarrow e^{-\varpi \langle v\rangle t} \langle v\rangle^{-b-1}\alpha^{\beta} \partial_{x} f$ and $e^{-\varpi \langle v\rangle t}\langle v\rangle^{-b+1} \alpha^{\beta-1} \partial_{v}f^{m}\rightarrow e^{-\varpi \langle v\rangle t}\langle v\rangle^{-b+1} \alpha^{\beta-1} \partial_{v}f$ strongly in $L^{\infty}$ so that 
$e^{-\varpi \langle v\rangle t} \langle v\rangle^{-b-1}\alpha^{\beta} \partial_{x} f  \in C^{0}([0,T^{*}] \times \bar{\Omega} \times\mathbb{R}^{3})$ and $e^{-\varpi \langle v\rangle t} \langle v\rangle^{-b+1} \alpha^{\beta-1} \partial_{v}f \in C^{0}([0,T^* ] \times \bar{\Omega} \times\mathbb{R}^{3}).$\end{proof}

\textbf{Acknowledgements.} This research is partly support in part by National Science Foundation under award NSF DMS-1501031, DMS-1900923.

\end{document}